\documentclass{amsart}

\usepackage[hidelinks]{hyperref}
\usepackage{mathrsfs}
\usepackage{stmaryrd}
\usepackage{color}
\usepackage{scalerel}
\usepackage{enumerate}
\usepackage{amssymb}

\swapnumbers

\numberwithin{equation}{section}

\newtheorem{Thm}{Theorem}[section]
\newtheorem{Prop}[Thm]{Proposition}
\newtheorem{Lem}[Thm]{Lemma}
\newtheorem{Cor}[Thm]{Corollary}
\theoremstyle{definition}
\newtheorem{Def}[Thm]{Definition}
\newtheorem{Expl}[Thm]{Example}

\newcommand{\R}{\mathbb{R}}		
\newcommand{\Z}{\mathbb{Z}}		

\newcommand{\Hi}{X}				
\newcommand{\B}{\textbf{B}}		
\newcommand{\oB}{\textbf{U}}	
\newcommand{\Sp}{\textbf{S}}	
\newcommand{\J}{\mathcal J}		
\newcommand{\id}{\mbox{id}}		

\newcommand{\HS}{{\scriptscriptstyle \mathrm {HS}}}

\newcommand{\cF}{\mathscr{F}}	
\newcommand{\cH}{\mathscr H}	
\newcommand{\calH}{\mathcal H}	
\newcommand{\cL}{\mathscr L}	
\newcommand{\cP}{\mathscr{P}}	
\newcommand{\cR}{\mathscr{R}}	

\newcommand{\bg}{\mathbf{g}} 
\newcommand{\bh}{\mathbf{h}} 
\newcommand{\bv}{\mathbf{v}} 
\newcommand{\bw}{\mathbf{w}} 
\newcommand{\by}{\mathbf{y}} 
\newcommand{\bz}{\mathbf{z}}

\newcommand{\bM}{\operatorname{\mathbf{M}}}		
\newcommand{\bS}{\operatorname{\mathbf{S}}}		
\newcommand{\Exc}{\operatorname{Exc}}			
\newcommand{\exc}{\operatorname{\mathbf{exc}}}	

\newcommand{\bc}{\mathbf{c}}
\newcommand{\bG}{\mathbf{G}}
\newcommand{\bQ}{\mathbf{Q}}
\newcommand{\bb}{\mathbf{b}}
\newcommand{\bP}{\mathbf{P}}
\newcommand{\bV}{\mathbf{V}}
\newcommand{\bomega}{\pmb{\omega}}
\newcommand{\balpha}{\pmb{\alpha}}
\newcommand{\bbeta}{\pmb{\beta}}
\newcommand{\bnu}{\pmb{\nu}}
\newcommand{\dev}{\operatorname{\mathbf{dev}}}

\newcommand{\graph}{\operatorname{graph}}
\newcommand{\hdist}{d_{\textrm{H}}}
\newcommand{\spa}{\operatorname{span}}

\newcommand{\tr}{\operatorname{tr}}
\newcommand{\Lip}{\operatorname{Lip}}
\newcommand{\dist}{\operatorname{dist}}
\newcommand{\spt}{\operatorname{spt}}

\newcommand{\defl}{\mathrel{\mathop:}=}
\newcommand{\defr}{=\mathrel{\mathop:}}
\newcommand{\curr}[1]{[\![{#1}]\!]}
\newcommand{\res}{\scaleobj{1.7}{\llcorner}}

\newcommand{\be}{\mathbf{e}}
\newcommand{\rmexc}{\mathrm{exc}}

\makeatletter
\newcommand*{\cone}{%
  {%
    \mathpalette\@coneOf{\times}%
  }%
}
\newcommand*{\@coneOf}[2]{%
  \sbox0{$\m@th#1\mathsf{#2}$}%
  \mathsf{#2}%
  \kern-\wd0 %
  \mkern2.00mu\relax
  \nonscript\mkern0.25mu\relax
  \mathsf{#2}%
}
\makeatother

\def\Xint#1{\mathchoice
	{\XXint\displaystyle\textstyle{#1}}%
	{\XXint\textstyle\scriptstyle{#1}}%
	{\XXint\scriptstyle\scriptscriptstyle{#1}}%
	{\XXint\scriptscriptstyle\scriptscriptstyle{#1}}%
	\!\int}
\def\XXint#1#2#3{{\setbox0=\hbox{$#1{#2#3}{\int}$ }
		\vcenter{\hbox{$#2#3$ }}\kern-.6\wd0}}

\def\dashint{\Xint-}

\title[Partial regularity of almost minimizing $G$ chains]{Partial regularity of almost minimizing rectifiable $G$ chains in Hilbert space}

\author{Thierry De Pauw}
\address{School of Mathematical Sciences\\
Shanghai Key Laboratory of PMMP\\ 
East China Normal University\\
500 Dongchuang Road\\
Shanghai 200062\\
P.R. of China\\
and NYU-ECNU Institute of Mathematical Sciences at NYU Shanghai\\
3663 Zhongshan Road North\\
Shanghai 200062\\
China}
\curraddr{Universit\'e Paris Diderot\\ 
Sorbonne Universit\'e\\
CNRS\\ 
Institut de Math\'ematiques de Jussieu -- Paris Rive Gauche, IMJ-PRG\\
F-75013, Paris\\
France}
\email{thdepauw@math.ecnu.edu.cn,thierry.de-pauw@imj-prg.fr}

\author{Roger Z\"{u}st}
\address{University of Bern\\
Mathematical Institute\\
Alpeneggstrasse 22\\
3012 Bern\\
Switzerland}
\email{roger.zuest@math.unibe.ch}

\thanks{The first author was partially supported by the Science and Technology Commission of Shanghai (No. 18dz2271000). The second author was partially supported by grant P2EZP2{\textunderscore}148732 of the Swiss National Science Foundation. Both authors were partially supported by project ANR-12-BS01-0014-01 Geometrya.}

\begin{document}

\begin{abstract}
We adapt to an infinite dimensional ambient space E.R.\ Reifenberg's epiperimetric inequality and a quantitative version of D.\ Preiss' second moments computations to establish that the set of regular points of an almost mass minimizing rectifiable $G$ chain in $\ell_2$ is dense in its support, whenever the group $G$ of coefficients is so that $\{\|g\| : g \in G \}$ is discrete and closed.
\end{abstract}

\maketitle

\tableofcontents


\newpage

\section{Introduction}

Let $(X,|\cdot|)$ be a separable Hilbert space, and $(G,\|\cdot\|)$ be a complete normed Abelian group, and $m$ be a nonnegative integer. We consider $m$ dimensional rectifiable $G$ chains in $X$, $T \in \cR_m(X;G)$, introduced in \cite{PH}. Such $T$ is associated with an $m$ dimensional rectifiable Borel subset $M \subset X$ and a Borel measurable $G$ orientation $\bg(x)$ corresponding to $\cH^m$ almost every $x \in M$. Specifically at almost every $x \in M$ where $M$ admits an approximate $m$ dimensional tangent space $W \subset X$, $\bg(x)$ is an equivalence class $(\xi,g)$ where $\xi$ is an orientation of $W$ and $g \in G$. Here $(\xi,g)$ and $(\xi',g')$ are equivalent if either they are equal or $\xi$ and $\xi'$ are opposite orientations of $W$ and $g=-g'$. We set $\|\bg\|=\|(\xi,g)\|=\|g\|$. For the data consisting of $M$ and $\bg$ to correspond to a member $T \in \cR_m(X;G)$ we also require that its mass be finite:
\begin{equation*}
\bM(T) = \int_M \|\bg(x)\| d\cH^m(x) < \infty .
\end{equation*}
In \cite{PH} $m$ dimensional rectifiable $G$ chains are understood as members of the larger groups $\cF_m(X;G)$ of $m$ dimensional flat $G$ chains. This allows for introducing the standard tools of Geometric Measure Theory : Boundary Operator; Push-Forward by Lipschitzian Mappings; Restriction; Slicing by Lipschitzian Mappings; Support of a Chain; Convergence in Flat Norm; Constancy Theorem, see \cite{PH2}; Approximation by Polyhedral $G$ chains, see \cite{DePauw2}. We will use all of these in the present paper.

The support $\spt(T)$ of $T \in \cR_m(X;G)$ consists of those $x \in X$ such that $T \res \B(x,r) \neq 0$ for all $r > 0$. Without further restriction it may be the case that $\spt(T) = X$ ; this can be achieved for $T$ consisting of a mass convergent series of properly chosen $G$ oriented circles whose collection of centers is dense. We say $a \in \spt(T) \setminus \spt(\partial T)$ is a regular point of $T$ whenever there exists a neighborhood $U$ of $a$ in $X$ such that $\spt(T) \cap U$ is an embedded $m$ dimensional H\"older continuously differentiable submanifold of $X$. In this situation there exists a possibly smaller neighborhood $V \subset U$ of $a$ with the following property. There are $C \geq 0$, $\alpha > 0$ and $0 < \delta \leq \infty$ such that for every $x \in V$, every $0 < r < \min \{\delta,\dist(x,\spt(\partial T))\}$, and every $S \in \cR_m(X;G)$, if $\spt(S) \subset \B(x,r)$ and $\partial S = 0$, then
\begin{equation*}
\bM(T \res \B(x,r)) \leq (1 + C r^\alpha) \bM( T \res \B(x,r) + S) \,.
\end{equation*}
If $T$ verifies the property stated in the last sentence, we say $T$ is $(\bM,Cr^\alpha,\delta)$ minimizing in $V$. If the specific $C$, $\alpha$ and $\delta$ are irrelevant we simply say $T$ is almost mass minimizing in $V$. This class of geometric variational objects was introduced, and their regularity studied, by F.J.\ Almgren \cite{A} in the framework of subsets of $X=\ell_2^n$ rather than chains. Still in a finite dimensional ambient space, E.\ Bombieri \cite{B} adapted Almgren's definition to rectifiable $\Z$ chains, essentially as that given above, and studied their regularity following \cite{A}. 

Examples of $(\bM,Cr^\alpha,\delta)$ minimizing chains $T \in \cR_m(X;G)$ in $V$ encompass the case when $C=0$, $\delta = \infty$ and $V = X \setminus \spt(\partial T)$. These include mass minimizing $G$ chains in the sense that $\bM(T) = \inf \{ \bM(S) : S \in \cR_m(X;G) \text{ and } \partial S = \partial T \}$.  If $a$ is a regular point of such mass minimizing $T$ and $W_a$ is the tangent space of $\spt(T)$ at $a$ then $\spt(T)$ is, in a neighborhood of $a$, the (translated) graph of a smooth $f : W_a \to W_a^\perp$ (with $f(0)=0$). Furthermore, near the origin, $f$ satisfies the minimal surface equation in case $m +1 = \dim X$, the minimal surface system in case $m + 1 < \dim X < \infty$, and a corresponding infinite dimensional system of Partial Differential Equations in case $\dim X = \infty$. Since $Df(0)=0$ (as $W_a$ is tangent to $\spt(T)$ at $a$) the so-called blow-up $\hat{f}$ of $f$, i.e. the weak limit in the Hilbert Sobolev space of properly rescaled and renormalized versions of $f$, is harmonic : $\triangle \langle \hat{f} , e \rangle = 0$ whenever $e$ is a unit vector in $W_a^\perp$. A weak version of this observation applied to Lipschitzian maps $f$ that approximate the support of $T$ near points $a$ that verify additional ``closeness to flat'' assumptions, is at the heart of many proofs of Regularity Theorems for either minimizers of mass or stationary chains. This technique goes back to E. De Giorgi \cite{DG}. However in case $T$ is $(\bM,Cr^\alpha,\delta)$ minimizing, $a$ is a regular point of $\spt(T)$ and the latter is, near $a$, the graph of some H\"older continuously differentiable $f : W_a \to W_a^\perp$, then $f$ need not solve any Partial Differential Equation whatsoever, as indeed the graph of any such $f$ has the almost mass minimizing property. Our main result is as follows.

\begin{Thm}
\label{introthm}
Assume $G$ is such that $\{ \|g\| : g \in G\}$ is discrete and closed and $T \in \cR_m(\Hi;G)$ is almost mass minimizing in an open set $V \subset \Hi \setminus \spt(\partial T)$. It follows that there exists a relatively open set $U \subset \spt(T) \cap V$ which is an embedded $m$ dimensional H\"older continuously differentiable submanifold of $\Hi$, and which is dense in $\spt(T) \cap V$. If one further assumes that $\|g\|=1$ for all $g \in G \setminus \{0_G\}$ then $\cH^m(\spt(T) \cap V \setminus U) = 0$. 
\end{Thm}

We now review relevant earlier results in this vein.

\begin{itemize}
\item {\it The case of minimizers in the finite dimensional setting}. Here $T \in \cR_m(\ell_2^n;G)$ is mass minimizing. If $m+1=n$ and $G=\Z$, this is E.\ De Giorgi's Theorem \cite{DG} established in the framework of oriented frontiers. In case $m$ is arbitrary the result has been established by E.R.\ Reifenberg \cite{Reif,Reif3} in a different setting than ours. E.R. Reifenberg \cite{Reif2} considers some compact groups of coefficients, he considers sets rather than rectifiable chains and the boundary conditions are expressed by means of \v{C}ech homology groups, finally his sets minimize size rather than mass, i.e. Hausdorff measure not weighted by coefficients norm. His method differs from the analysis of blow-up set forth by E. De Giorgi and has inspired the present paper. F.J. Almgren establishes the result in the framework of rectifiable chains minimizing the integral of some elliptic integrands, \cite{A0}.
\item{\it The case of almost minimizers in the finite dimensional setting.} Here $T \in \cR_m(\ell_2^n;\Z)$ is almost mass minimizing. The regularity has been established by E.\ Bombieri \cite{B} following the scheme of proof set forth by F.J. Almgren \cite{A} in the framework of sets rather than chains.
\item {\it The case of minimizers in the infinite dimensional setting}. Here $T \in \cR_m(\ell_2;\Z)$ is mass minimizing. L.\ Ambrosio, C.\ De Lellis and T.\ Schmidt \cite{ADS} have established the result in the framework of ``currents in metric spaces'' \cite{AK2}.
\end{itemize}

A $C^1$ version of the regularity theorem holds when the bound $Cr^\alpha$ quantifying the almost minimizing property is replaced by a coarser bound that decays fast enough, see the remark after Theorem~\ref{reifflat2}. In Section~\ref{examples} we show that the discreteness of $G$ is necessary and give an example of a mass minimizing chain, with coefficients in a totally disconnected compact normed group, and without any regular point, Example~\ref{discreteness_example}. Proposition~\ref{submanifoldprop} shows that any chain induced by a $C^{1,\alpha}$ submanifold of $\Hi$ is almost mass minimizing, as we claimed above, and more generally the sum of finitely many $C^{1,\alpha}$ submanifolds that ``nicely'' intersect in a common set $\Sigma$ is almost mass minimizing as well. Taking for $\Sigma$ a Cantor set of positive $\cH^m$ measure we show that the set of regular points need not be co-null in general, Example~\ref{regularsetexample}.

One serious difficulty with working in infinite dimension is the lack of certain compactness results that are specific to a finite dimensional ambient space. Specifically, the essential Excess Decay Lemma is established in \cite{A} and \cite{B} by contradiction, a delicate argument based on compactness. Such reasoning seems to be bound to fail in $\ell_2$. We now turn to briefly describing the scheme of proof of Theorem \ref{introthm}.

Let $T \in \cR_m(X;G)$. We let $\|T\|$ denote the finite Borel measure in $X$ associated with $T$, i.e.\ $\|T\|(B) = \int_B \|\bg_T\| d\cH^m \res M_T$ where $M_T$ and $\bg_T$ are respectively the $m$ dimensional rectifiable set and the $G$ orientation of $T$. Let $a \in X$, $r > 0$ and $W \in \bG(X,m)$. Assume for definiteness $\Theta^m(\|T\|,a)=1$. We consider the following two quantities.
\begin{equation*}
\exc^m(\|T\|,a,r) = \frac{\|T\|(\B(a,r))}{\balpha(m)r^m} - 1
\end{equation*}
and
\begin{equation*}
\bbeta_2^2(\|T\|,a,r,W) = \frac{1}{r^{m+2}}\int_{\B(a,r)} \dist^2(y-a,W)d\|T\|(y) \,.
\end{equation*}
The first one we call the spherical excess. To get a sense of what these quantify, assume $\spt(T) \cap \B(a,r)$ coincides with $(a+\Gamma_f) \cap \B(a,r)$ where $\Gamma_f$ is the graph of some $f : W \to W^\perp$ with small Lipschitz constant so that $(a+\Gamma_f) \cap \B(a,r)$ is contained in a slab around $a+W$ whose height is small with respect to $r$. Upon considering cylindrical versions of $\exc^m$ and $\bbeta_2$, and calculating the Taylor expansion of the Hilbertian Jacobian of $f$, one can check that
\begin{equation*}
\exc^m(\|T\|,a,r) \cong \frac{1}{\balpha(m)r^m} \int_{W \cap \B(0,r)} \|Df\|_{\HS}^2 d\cH^m
\end{equation*}
(where $\|\cdot\|_{\HS}$ stands for the Hilbert-Schmidt norm), and
\begin{equation*}
\bbeta_2^2(\|T\|,a,r,W) \cong \frac{1}{r^{m+2}}\int_{W \cap \B(0,r)} |f|^2d\cH^m \,.
\end{equation*}

It is classical to check that if $T$ is mass minimizing then $\exc^m(\|T\|,a,r)$ is a nondecreasing function of $0 < r < \dist(a,\spt(\partial T))$, for every $a \not\in \spt(\partial T)$. In case $T$ is almost mass minimizing, the spherical excess is almost nondecreasing in the sense that $\exp[Cr^\alpha]\frac{\|T\|(\B(a,r))}{\balpha(m)r^m} - 1$ is nondecreasing. This is a useful fact used throughout the paper. The three steps of the proof are now the following.

\begin{enumerate}
\item[(H)] Identifying a set of hypotheses called ``closeness to flat'' that guarantee that $T$ is sufficiently close to flat in some ball $\B(a,r)$, both in the sense that $\spt(T) \cap \B(a,r)$ is close in Hausdorff distance to $(a+W) \cap \B(a,r)$ for some $W \in \bG(X,m)$, and in the sense that $\|T\|(\B(a,r))$ is close to the measure $\balpha(m)r^m$ of the $m$ dimensional disk $(a+W) \cap \B(a,r)$ weighted by the ``local coefficient'' of $T$ near $a$. The existence of such $W$ by no means implies that it is the ``tangent space'' to $\spt(T)$ at $a$ and the remaining part of the analysis consists in estimating how much these $W$ may vary as $r$ tends to 0. Another feature of the ``closeness to flat'' assumption is that it should be verified by an open dense set of points in the support. This is where the assumption that $G$ be discrete comes into play.
\item[(1)] Showing that, under (H), for $a \in \spt(T)$ the spherical excess decays fast enough, specifically $\exc^m(\|T\|,a,r) \leq Cr^\gamma$ for some $\gamma > 0$.
\item[(2)] Showing that, under (H), there exists $W_{a,r} \in \bG(X,m)$ (not necessarily the same as in the assumption (H)) such that $\bbeta_2^2(\|T\|,a,r,W_{a,r}) \leq C \exc^m(\|T\|,a,r)^\beta$. It will then ensue from (2) that $W_a = \lim_{r \to 0} W_{a,r}$ exists, and that $W_a$ and $W_b$ are close to one another according to how $a$ and $b$ are close. This is enough to finish off the proof.
\end{enumerate}

Let us now briefly comment on steps (1) and (2). To keep the notations short we assume $a=0$ and we abbreviate $\phi_T(r) = \|T\|(\B(0,r))$, $\be_T(r) = \exc^m(\|T\|,0,r)$, and for convenience we introduce the non normalized version of the spherical excess, namely $e_T(r) = \rmexc^m(\|T\|,0,r) = \balpha(m)r^m \exc^m(\|T\|,0,r) = \phi_T(r)-\balpha(m)r^m$. We ought to determine whether $\be_T(r) \leq C r^\gamma$, which is equivalent to $e_T(r) \leq Cr^{m(1+\epsilon)}$ where we have abbreviated $\epsilon = \gamma m^{-1}$. In case $T$ is mass minimizing $\be_T(r)$ is nondecreasing and nonnegative. In case $T$ is merely almost mass minimizing we would have to introduce a multiplicative factor $\exp[Cr^\alpha]$ in the computations, which we omit in these introductory remarks. Thus assuming that $e_T(r) \geq 0$, the growth $e_T(r) \leq Cr^{m(1+\epsilon)}$ will follow from the differential inequality
\begin{equation*}
m(1+\epsilon) \frac{d}{dr} \log r \leq \frac{d}{dr} \log e_T(r) \,,
\end{equation*}
which is thus equivalent to
\begin{equation*}
(1+\epsilon) e_T(r) \leq \frac{r}{m}e'_T(r) \,.
\end{equation*}
In estimating $rm^{-1} e'_T(r) = rm^{-1} ( \phi'_T(r) - m \balpha(m)r^{m-1}) = rm^{-1} \phi'_T(r) - \balpha(m)r^m$ we observe that
\begin{equation*}
\frac{r}{m} \phi'_T(r) \geq \frac{r}{m} \bM( \langle T , |\cdot| , r \rangle) = \bM( \curr 0 \cone \langle T , |\cdot| , r \rangle) \defr \phi_{C_r}(r)
\end{equation*}
is the measure in $\B(0,r)$ of the cone $C_r$ with vertex 0 and with base the slice $\langle T , |\cdot| , r \rangle = \partial ( T \res \B(0,r))$. In other words $C_r = \curr 0 \cone \partial ( T \res \B(0,r))$. Thus we are asking whether
\begin{equation*}
(1+\epsilon) e_T(r) \leq \phi_{C_r}(r) - \balpha(m)r^m = e_{C_r}(r) \,.
\end{equation*}
The inequality $e_T(r) \leq \theta e_{C_r}(r)$, for some $0 < \theta < 1$, is the quantitative improvement on the monotonicity inequality that we call the ``Reifenberg epiperimetric inequality''. In order to evoke its proof let us assume that $r=1$ and that $\spt(T) \cap Z_W(0,1) = \Gamma_{f_W} \cap Z_W(0,1)$ for some $W \in \bG(X,m)$ and some $f_W : W \to W^\perp$, where $Z_W(0,1)$ denotes the cylinder $\pi_W^{-1}(W \cap \B(0,1))$. We let $h_W$ be the homogeneous extension of degree 1, to $W \cap \B(0,1)$, of $f_W | _{\mathrm{Bdry}(W \cap \B(0,1))}$. The question is then (up to error terms due to the replacement of a ball by a cylinder, and which are small with respect to $e_T(1)$) whether the following holds:
\begin{equation}
\label{eeqq}
\int_{W \cap \B(0,1)} \|Df_W\|_{\HS}^2 d \cH^m \leq \theta \int_{W \cap \B(0,1)} \|Dh_W\|_{\HS}^2 d \cH^m \,.
\end{equation}
To see that this cannot be true in general it suffices to observe that if $h_W$ is linear and nonzero then it is harmonic and hence a solution of the Dirichlet problem: It minimizes its Dirichlet energy among those maps having the same boundary values and thus $f_W$ cannot have strictly less Dirichlet energy. This drawback can be overcome if we can make sure $h_W$ is far enough from being linear, which amounts to choosing $W$ initially so as to cancel the linear part of $h_W$. This should mean that, if possible, $h_W$ is orthogonal to linear functions in the Lebesgue space $\mathbf{L}_2(W \cap \B(0,1),\cH^m)$. In other words we seek for a $W \in \bG(X,m)$ such that
\begin{equation}
\label{eeeqqq}
\int_{W \cap \B(0,1)} \langle h_W(x) , v \rangle \langle x , u \rangle d\cH^m(x) = 0
\end{equation}
for every unit vectors $u \in W$ and $v \in W^\perp$. E.R. Reifenberg solves this problem by minimizing a functional over the Grassmannian $\bG(\ell^n_2,m)$. In the infinite dimensional setting the Grassmannian lacks the compactness required to make this work and instead we argue as follows. Abbreviating $C = C_1 = \curr 0 \cone \partial ( T \res \B(0,1))$, we consider the quadratic form 
\begin{equation*}
Q_C(x) = \int_{\B(0,1)} \langle x , y \rangle^2 d\|C\|(y) = \left\langle x , L_C(x) \right\rangle
\end{equation*}
where
\begin{equation*}
L_C(x) = \int_{\B(0,1)} y \langle x,y \rangle d\|C\|(y) \,.
\end{equation*}
We notice that $L_C$ is a positive self-adjoint compact operator on $X$ and therefore $X$ admits an orthonormal basis of eigenvectors $e_1,e_2,...$ of $L_C$, according to the spectral theorem. We order these as usual so that $\lambda_1 \geq \lambda_2 \geq ... \geq 0$ for the corresponding eigenvalues. In fact, since we assume $T$ (and hence $C$) to be sufficiently close to flat in $\B(0,1)$, $Q_C$ is close to a multiple of $|\pi_{W}(\cdot)|^2$. Thus if $\hat{W} \in \bG(X,m)$ is generated by $e_1,...,e_m$ then $\hat{W}$ is close to $W$ (appearing in (H)). Moreover $L_C(e_k)= \lambda_ke_k$ for every $k$, and therefore $\langle e_j , L_C(e_k) \rangle = 0$ for all $j=1,...,m$ and all $k=m+1,m+2,...$, which is \eqref{eeeqqq} up to some small error terms depending on the initial closeness to flat. With $W$ replaced by $\hat{W}$ it will be possible to find $f_{\hat{W}}$ with boundary values same as $h_{\hat{W}}$ and achieving the improvement of \eqref{eeqq}. Of course the graph of $f_{\hat{W}}$ may not meet at all the support of $T$ in $\B(0,1)$, yet
\begin{equation*}
\int_{\hat{W} \cap \B(0,1)} \|Df_{\hat{W}}\|_{\HS}^2 d \cH^m \gtrapprox e_T(1)
\end{equation*}
according to the almost mass minimizing property of $T$. We can now infer that $\be_T(1) \leq \theta \be_C(1)$. Our proof follows very closely the arguments of E.R.\ Reifenberg in \cite{Reif}. The essential differences are the averaging of the different layers to account for the elements of the normed group $G$ and the use of the quadratic form $Q_C$ as explained above.

At this point it is perhaps worth saying a word on how to obtain $f_W$ in \eqref{eeqq} provided \eqref{eeeqqq} holds. One way would be to find the optimal $f_W$ for the problem, i.e. the mapping $f_W$ each of whose coordinates is harmonic. This is indeed how E.R. Reifenberg proceeded. Nevertheless the computations that establish \eqref{eeqq}, involving identities for spherical harmonics, work as well for any homogeneous extension of $h_W$, of degree $1+t > 1$ (with $\theta$ depending on $t$). In other words, contrary to popular belief, there exist proofs of partial regularity for (almost) mass minimizing chains that do not involve at any stage the use of harmonic functions. Sure enough, after we realized this, we found out that C.B.\ Morrey had also reported this fact in \cite[Lemma~10.6.13]{Mor}.

Finally we turn to briefly discussing step (2) of our proof, i.e. the sort of Poincar\'e inequality with an exponent, $\bbeta_2^2(\|T\|,a,r,W) \leq C \exc^m(T,a,r)^\beta$ whenever (H) holds. In fact, according to the Ahlfors regularity of $\spt(T)$ (a consequence of almost monotonicity), this is equivalent to $\bbeta_\infty(\|T\|,a,r,W) \leq \exc^m(\|T\|,a,r)^\beta$ (for some different $0 < \beta < 1$). Readers familiar with regularity proofs will recognize here a height bound, usually obtained as a corollary of graphical Lipschitz approximation of the support, and that would definitely be one way to go about it for instance when $G=\Z$. Here, as we used a quadratic form $Q$ associated to $T$ already, we choose to be consistent and keep using it. This sort of moments computations has been set forth in \cite{P}, and a quantitative version was described in \cite{DePauw3}. We review the idea behind it and point out the differences with \cite{DePauw3}.

Assume for simplicity $a=0 \in \spt(T)$. We notice that according to Cavalieri's principle, information about the growth of $\|T\|(\B(0,r))$ is equivalent to information about the growth of
\begin{equation*}
\hat{V}(\|T\|,x,r) = \int_{\B(x,r)} \left( r^2 - |x-y|^2 \right)^2 d\|T\|(y) \,.
\end{equation*}
As its variable $x$ appears in the moving domain of integration, this function $\hat{V}$ does not immediately occur as being differentiable in $x$. We therefore introduce a slight variation of $\hat{V}$, the polynomial function
\begin{equation*}
V(\|T\|,x,r) = \int_{\B(0,r)} \left( r^2 - |x-y|^2 \right)^2 d\|T\|(y) = \sum_{k=0}^4 P_k(\|T\|,x,r)
\end{equation*}
which is expanded in a sum of $P_k$'s, polynomials in $x$ homogeneous of degree $k$. A simple computation yields
\begin{equation*}
\left| V(\|T\|,x,r) - \hat{V}(\|T\|,x,r) \right| \leq C \left( r^{m+1} |x|^3 + r^{m+2+\beta} |x|^2 \right) \,.
\end{equation*}
From this we infer that if we divide $V$ and $\hat{V}$ by $r^{m+2}$ the corresponding functions will be close up to $O(|x|^2)$ when $|x| \leq r$, and up to $o(|x|^2)$ when $|x| = o(1)r$, so that information about the variations of $\|T\|(\B(0,r))$, and in turn about the variations of $\hat{V}$, translates to information about the polynomials $P_k$. We then use the fact that $P_2$ normalizes to a dimensionless quantity when divided by $r^{m+2}$, and contains a term akin to $Q$ defined above. Specifically,
\begin{equation*}
P_2(\|T\|,x,r) = 4 Q_T(x,r) - 2 |x|^2 \int_{\B(0,r)} \left( r^2 - |y|^2 \right) d\|T\|(y) 
\end{equation*} 
where, as above,
\begin{equation*}
Q_T(x,r) = \int_{\B(0,r)} \langle x , y \rangle^2 d\|T\|(y) \,,
\end{equation*}
and if 
\begin{equation*}
\bP_2 = \frac{m+2}{\balpha(m)r^{m+2}} P_2 \text{ and } \bQ_T = \frac{m+2}{\balpha(m)r^{m+2}} Q_T
\end{equation*}
then, for instance,
\begin{equation}
\label{eeqq2}
\left| \tr \bQ_T(\cdot,r) - m \right| \leq C r^{\gamma'} \,,
\end{equation}
under (H), where $\gamma'$ is a function of $\gamma$ appearing in step (1). Assuming for now that the first moment $P_1$, which is a version of a center of mass of $\|T\|$ in $\B(0,r)$ that initially normalizes as $r^{m+1}$, can be made as small as $Cr^{\gamma'}$ when divided by $r^{m+2}$, we end up showing that
\begin{equation*}
\left| \bQ_T(x,r) - |x|^2 \right| \leq C r^{\gamma'} |x|^2
\end{equation*}
whenever $x \in \spt(T)$, $\Theta^m(\|T\|,x)=1$ and $|x|=r^{1+\gamma'}$. If we manage to find orthogonal vectors $x_1,...,x_m \in \spt(T)$ such that $|x_i| = r^{1+\gamma'}$ and $\Theta^m(\|T\|,x_i)=1$, then we infer from the inequality above, from the definition of $P_2$ and from \eqref{eeqq2} that
\begin{equation}
\label{eeqq3}
\frac{1}{r^{m+2}} \int_{\B(0,r)} \dist^2(y,W) d\|T\|(y) \leq C r^{\gamma'}
\end{equation}
and we will be done. In \cite{DePauw3} the orthogonal frame $x_1,...,x_m$ in $\spt(T)$ was found because $\spt(T)$ was shown to be Reifenberg flat under (H), and therefore a topological $m$ disk according to \cite{Reif}. The Reifenberg flatness however followed by a compactness argument not available in infinite dimension. Instead we take advantage of the fact that $T$ is rectifiable and hence for $\|T\|$ almost every $a$ we can assume we look at a scale $r(a)$ small enough to start with that $\spt(T) \cap \B(a,r)$ is close to its tangent space $W_a$ at $a$ (notice a tangent space exists because $\ell_2$ has the Radon-Nikod\'ym property, and closeness occurs in Hausdorff distance thanks to Ahlfors regularity). At that small scale $r(a)$ (and all smaller scales) we can find $x_1,...,x_m$ with the required properties. We then show in \ref{bootstraplem}, using a backward bootstrap argument that the same holds at {\it larger} scales $r$ provided the density ratio at scale $r$ is not much larger than 1. This can be done at neighboring points up the same largest scale $r_0$ thanks to almost monotonicity, see \ref{weakdini}. To summarize, we establish that \eqref{eeqq3} now holds at scales uniformly small $0 < r \leq r_0$ and that $\spt(T) \cap \B(a,r)$ is in fact Reifenberg flat, \ref{reifflat}. It then classically follows from \eqref{eeqq3} that $\spt(T) \cap \B(a,r)$ is $C^{1,\gamma''}$.

\section{Preliminary results for chains in a Hilbert space}

Throughout these notes $(\Hi,\langle\cdot,\cdot\rangle)$ is a separable Hilbert space with $\dim(X) > m$, and $T \in \cR_m(\Hi;G)$ is a chain with coefficients in a complete normed Abelian group $(G,\|\cdot\|)$. A norm on an Abelian group is a function $\|\cdot\| : G \to [0,\infty)$ such that for all $g,h \in G$,
\begin{enumerate}
	\item $\|-g\| = \|g\|$,
	\item $\|g + h\| \leq \|g\| + \|h\|$,
	\item $\|g\| = 0$ if and only if $g = 0_G$.
\end{enumerate}
The vector space norm on $\Hi$ is denoted by $|\cdot|$ to distinguish it from the one on $G$. Since the support of $T$ is a separable subset of $\Hi$, the statements we make remain true for any Hilbert ambient space because we can always restrict to a separable complete subspace thereof. The Grassmannian of $m$-dimensional subspaces of $\Hi$ is denoted by $\bG(\Hi,m)$. Given an $m$-plane $W \in \bG(\Hi,m)$ we denote by $\pi_W$ the orthogonal projection of $\Hi$ onto $W$. We often use these projections as push-forwards for chains. In order to speak about chains on a plane $W \in \bG(\Hi,m)$ it is necessary that we choose an orientation on $W$. This orientation is mostly not so important since we are interested in the mass of the chains which is independent of this choice. In case we work with a collection of planes that are close to each other, we implicitly assume that these orientations are compatible in the sense that the orthogonal projection from one to the other is orientation preserving. 

With $\B(x,r)$ and $\oB(x,r)$ we denote the closed, respectively open, ball of radius $r \geq 0$ around $x \in \Hi$. Sometimes we may also use $\B(r) \defl \B(0,r)$, $\oB(r) \defl \oB(0,r)$, $\B \defl \B(0,1)$ and $\oB \defl \oB(0,1)$. Similarly, for $A \subset \Hi$ we denote by $\B(A,r)$ and $\oB(A,r)$ the closed, respectively open, tubular neighborhood of radius $r$ around $A$. For intersections we use $\B_A(x,r) \defl \B(x,r) \cap A$ and $\oB_A(x,r) \defl \oB(x,r) \cap A$. The cylinder around $x \in \Hi$ of radius $r$ above an $m$-dimensional plane $W \in \bG(\Hi,m)$ is given by $Z_W(x,r) \defl {\pi_W}^{-1}(\B_W(\pi_W(x),r))$, or in special cases $Z_W(r) \defl Z_W(0,r)$ and $Z_W \defl Z_W(0,1)$. $\balpha(m)$ denotes the volume of the unit ball in $\R^m$. For two nonempty subsets $A,A' \subset \Hi$ the \emph{Hausdorff distance} is
\begin{equation*}
\begin{split}
\hdist(A,A') & \defl \inf \{r > 0 : A \subset \B(A',r), A'\subset \B(A,r) \} \\
& = \max \left\{ \sup_{x \in A} \dist(x,A') , \sup_{x' \in A'} \dist(x',A) \right\} \,.
\end{split}
\end{equation*}
In the rest of this section we describe some of the basic tools we need in the process.

\subsection{Grassmannian and Hausdorff distance}

The Grassmannian $\bG(\Hi,m)$ can be equipped with a complete metric. Two natural definitions of a metric are shown to be equivalent in the following lemma. Here $\|\cdot\|$ denotes the operator norm.

\begin{Lem}
\label{lemma.norms.1}
Let $m \geq 1$ be an integer and $V_1,V_2 \in \bG(\Hi,m)$. The following hold:
\begin{enumerate}
\item $\|\pi_{V_1}-\pi_{V_2}\| = \hdist(V_1 \cap \B(0,1) , V_2 \cap \B(0,1))$;
\item $\|\pi_{V_1}-\pi_{V_2}\| = \max \{|v - \pi_{V_2}(v)| : v \in V_1 \cap \B(0,1)\}$;
\item $\|\pi_{V_2^\perp} \circ \pi_{V_1}\| \leq \|\pi_{V_1}-\pi_{V_2}\|$.
\end{enumerate}
\end{Lem}

\begin{proof}
We start with the trivial observation that $\dist(v,V_i \cap \B(0,1)) = \dist(v,V_i)$ whenever $v \in \B(0,1)$, $i=1,2$, which will be used repeatedly without further notice.
Next we observe that
\begin{equation}
\label{eq.dieudonne.1}
\max \{ \dist(v,V_2) : v \in V_1 \cap \B(0,1) \} \leq \|\pi_{V_1}-\pi_{V_2}\| \,.
\end{equation}
Indeed suppose $v \in V_1 \cap \B(0,1)$, then 
\begin{equation*}
\dist(v,V_2)  = |\pi_{V_2^\perp}(v)| = |\id_{\Hi}(v) - \pi_{V_2}(v)| = |\pi_{V_1}(v)-\pi_{V_2}(v)| 
 \leq \|\pi_{V_1}-\pi_{V_2}\|
\end{equation*}
which establishes \eqref{eq.dieudonne.1}. Swapping $V_1$ and $V_2$ in \eqref{eq.dieudonne.1} we obtain 
\begin{equation}
\label{eq.dieudonne.3}
\hdist(V_1 \cap \B(0,1) , V_2 \cap \B(0,1)) \leq \|\pi_{V_1}-\pi_{V_2}\| \,.
\end{equation}
In order to prove conclusion (1) it remains to prove the reverse inequality:
\begin{equation}
\label{eq.dieudonne.2}
\|\pi_{V_1}-\pi_{V_2}\| \leq \max \left\{ \max_{v \in V_1 \cap \B(0,1)}\dist(v,V_2) , \max_{v \in V_2 \cap \B(0,1)}\dist(v,V_1)  \right\} \,.
\end{equation}
We prove inequality \eqref{eq.dieudonne.2} first in case $m=1$ and $\operatorname{dim}(\Hi) = 2$. We leave to the reader the simple computation showing that, in this special case, the map $\Hi \cap \bS(0,1) \rightarrow \R : x \mapsto |\pi_{V_1}(x) - \pi_{V_2}(x)|$ is constant. Therefore
\begin{equation}
\label{eq.dieudinne.4}
\|\pi_{V_1}-\pi_{V_2}\| = |\pi_{V_1}(v) - \pi_{V_2}(v)| = |\pi_{V_2^\perp}(v)| = \dist(v,V_2) \,,
\end{equation}
for every $v \in V_1 \cap \bS(0,1)$.

Next we assume $m \geq 2$ and $\operatorname{dim}(\Hi) < \infty$. We choose $v^*$ which maximizes $\bS(0,1) \rightarrow \R : v \mapsto |\pi_{V_1}(v)-\pi_{V_2}(v)|^2$, so that $\| \pi_{V_1} - \pi_{V_2} \| = | \pi_{V_1}(v^*) - \pi_{V_2}(v^*) |$. According to the Lagrange Multiplier Theorem there exists $\lambda \in \R$ such that $\lambda v^* = \pi_{V_1}(v^*) - \pi_{V_2}(v^*)$. If $\lambda = 0$ then $\|\pi_{V_1}-\pi_{V_2}\|=0$ and \eqref{eq.dieudonne.2} clearly holds. We henceforth assume $\lambda \neq 0$. If $\pi_{V_1}(v^*)=0$ then $v^* \in V_2$, consequently $\|\pi_{V_1}-\pi_{V_2}\|=|\pi_{V_2}(v^*)| = |v^*|=1=|v^*-\pi_{V_1}(v^*)| \leq \max \{ \dist(v,V_1) : v \in V_2 \cap \B(0,1)\}$, thus \eqref{eq.dieudonne.2} holds. Similarly if $\pi_{V_2}(v^*)=0$ then $\|\pi_{V_1}-\pi_{V_2}\| \leq \max \{ \dist(v,V_2) : v \in V_1 \cap \B(0,1)\}$, and \eqref{eq.dieudonne.2} holds as well. Assuming now that $\pi_{V_1}(v^*) \neq 0 \neq \pi_{V_2}(v^*)$ we define $V \defl \operatorname{span}\{\pi_{V_1}(v^*),\pi_{V_2}(v^*)\}$, so that $v \in V$. We note $\dim V = 2$ (assuming if possible that $\dim V = 1$, it would ensue $\pi_{V_1}(v^*)=t.\pi_{V_2}(v^*)$ for some $t \neq 0$, and in turn $v^* = \lambda^{-1}(t-1)\pi_{V_2}(v^*) \in V_2$, thus also $v^* = \pi_{V_2}(v^*) = t^{-1}\pi_{V_1}(v^*) \in V_1$, whence $\lambda v^* = v^* - v^* = 0$, a contradiction), as well as $\dim V \cap V_1 = 1 = \dim V \cap V_2$. In other words we may apply the previous 2 dimensional case applies to $V$, $V \cap V_1$ and $V \cap V_2$. To this end we notice that $\pi_{v_i}(v) = \pi_{V \cap V_i}(v)$ whenever $v \in V$ (it suffices to observe that $\pi_{V_i}(v) \in V$ for $v = \pi_{V_j}(v^*)$, $j=1,2$, which is obvious if $j=i$ whereas if $j=2$, $i=1$ then $\pi_{V_2}(\pi_{V_1}(v^*)) = (\lambda+1) \pi_{V_2}(v^*) \in V$, and the other case is similar). Thus,
\begin{align*}
\|\pi_{V_1}-\pi_{V_2}\| & = |\pi_{V_1}(v^*) - \pi_{V_2}(v^*)| \\
& = | \pi_{V \cap V_1}(v^*) - \pi_{V \cap V_2} (v^*) | \\
& \leq \|\pi_{V  \cap V_1} \restriction V - \pi_{V \cap V_2} \restriction V \| \\
& \leq \max \left\{ \max_{v \in V \cap V_1 \cap \B(0,1)}\dist(v,V \cap V_2) , \max_{v \in V \cap V_2 \cap \B(0,1)}\dist(v,V \cap V_1)  \right\} \\
& = \max \left\{ \max_{v \in V \cap V_1 \cap \B(0,1)}\dist(v,V_2) , \max_{v \in V \cap V_2 \cap \B(0,1)}\dist(v,V_1)  \right\} \\
& \leq \max \left\{ \max_{v \in V_1 \cap \B(0,1)}\dist(v,V_2) , \max_{v \in V_2 \cap \B(0,1)}\dist(v,V_1)  \right\} \\
& \leq \| \pi_{V_1} - \pi_{V_2} \| \,.
\end{align*}
This completes the proof of \eqref{eq.dieudonne.2} in case $\dim(\Hi) < \infty$ and we leave it to the reader to check it also holds when $\dim(\Hi) = \infty$. Conclusion (1) now readily follows, whereas conclusion (2) holds because it holds in case $\dim X = 2$ and all inequalities above are equalities.

In order to prove (3) we simply notice that
\begin{align*}
\|\pi_{V_2^\perp} \circ \pi_{V_1} \| & = \| (\id_{\Hi} - \pi_{V_2}) \circ \pi_{V_1} \| \\
& = \| (\id_{\Hi} - \pi_{V_2}) \circ (\pi_{V_1} - \pi_{V_2}) + (\id_{\Hi}-\pi_{V_2}) \circ \pi_{V_2}) \| \\
& \leq \|\pi_{V_2^\perp} \| \; \|\pi_{V_1}-\pi_{V_2}\| \,.
\end{align*}
\end{proof}

We will need some estimates on the closeness of different planes that approximate some set $S \subset \Hi$. Such estimates are established in the results below.

\begin{Def}
Let $m > 0$ be an integer, $S \subset X$, $x \in S$, $r > 0$ and $\epsilon > 0$. We say that $S$ is \emph{$\epsilon$ flat at $(x,r)$} if there exists $V \in \bG(\Hi,m)$ such that
\begin{equation*}
\hdist \left( S \cap \B(x,r) , (x + V) \cap \B(x,r) \right) \leq \epsilon r \, .
\end{equation*}
We also define
\begin{equation*}
\bG(S,x,r,\epsilon) \defl \{ V \in \bG(\Hi,m): \hdist(S \cap \B(x,r) , (x+V) \cap \B(x,r)) \leq \epsilon r \} \, .
\end{equation*}
\end{Def}

The dimension $m$ does not appear either in the notation or in the terminology since it will always be clear from the context.
We leave it to the reader to check the trivial and useful fact that $\bG(X,m) \subset \bG(S,x,r,\epsilon)$ whenever $\epsilon \geq 1$.

\begin{Lem}
\label{lemma.reif.angle.1.prelim}
Assume that
\begin{enumerate}
	\item $S \subset X$, $x \in S$, $R > 0$, $\epsilon > 0$ and $0 < \lambda \leq 1$;
	\item $V \in \bG(S,x,R,\epsilon)$.
\end{enumerate}
Then $V \in \bG(S,x,\lambda R, 2 \epsilon \lambda^{-1})$.
\end{Lem}

\begin{proof}
If $\lambda \leq \epsilon$ then there is nothing to prove, according to the remark right before the Lemma. Thus we assume from now on that $\epsilon < \lambda$. Let $\zeta \in S \cap \B(x,\lambda R) \subset S \cap \B(x,R)$. Define $\xi \defl x + \pi_V(\zeta-x)$. Since $|\xi-x| = |\pi_V(\zeta-x)| \leq |\zeta-x| \leq \lambda R$ one has $\xi \in (x+V) \cap \B(x,\lambda R)$. Furthermore,
\begin{equation*}
|\xi-\zeta| = \dist(\zeta , x +V) \leq \dist(\zeta, (x+V) \cap \B(x,R)) \leq \epsilon R 
\end{equation*}
by assumption (2). It follows that
\begin{equation*}
S \cap \B(x,\lambda R) \subset \B \big[ (x+V) \cap \B(x,\lambda R) , \left( \epsilon \lambda^{-1} \right) \lambda R \big] \,.
\end{equation*}
\par 
Let now $\xi \in (x+V) \cap \B(x,\lambda R)$ and $\epsilon < \hat{\epsilon} < \lambda$. Define $\xi' = x + \lambda^{-1}(\lambda - \hat{\epsilon})(\xi-x)$ and observe that $|\xi'-x| = \lambda^{-1} (\lambda - \hat{\epsilon})|\xi-x| \leq (\lambda - \hat{\epsilon})R \leq R$. In particular $\xi' \in (x+V) \cap \B(x,R)$ and according to assumption (2) there exists $\zeta \in S \cap \B(x,R)$ such that $|\xi' - \zeta| \leq \hat{\epsilon} R$. Note that
\begin{equation*}
|\zeta-x| \leq |\zeta - \xi'| + |\xi'-x| \leq \hat{\epsilon}R + (\lambda - \hat{\epsilon})R = \lambda R \,,
\end{equation*}
whence in fact $\zeta \in S \cap \B(x,\lambda R)$, and also
\begin{multline*}
|\zeta-\xi| \leq |\zeta - \xi'| + |\xi'-\xi| \leq \hat{\epsilon} R + \left| 1 - \lambda^{-1}(\lambda - \hat{\epsilon}) \right| . |\xi-x| \\
\leq \hat{\epsilon}R + \hat{\epsilon} \lambda^{-1}\lambda R = 2 \hat{\epsilon} R \,.
\end{multline*}
Therefore
\begin{equation*}
(x+V) \cap \B(x,\lambda R) \subset \B \big[ S \cap \B(x,\lambda R), \left( 2 \hat{\epsilon}\lambda^{-1} \right) \lambda R \big]
\end{equation*}
and the conclusion follows from the arbitrariness of $\hat{\epsilon} > \epsilon$.
\end{proof}

\begin{Cor}[Same center, different scales]
\label{cor.reif.angle.1}
Assume that
\begin{enumerate}
	\item $S \subset X$, $x \in S$, $\epsilon > 0$ and $0 < r < R$;
	\item $V_{r} \in \bG(S,x,r,\epsilon)$ and $V_{R} \in \bG(S,x,R,\epsilon)$.
\end{enumerate}
Then
\begin{equation*}
\hdist(V_{r} \cap \B(0,1), V_{R} \cap \B(0,1)) \leq \epsilon \left(1 + 2Rr^{-1}\right) \,.
\end{equation*}
\end{Cor}

\begin{proof}
We put $\lambda = rR^{-1}$ and we infer from Lemma~\ref{lemma.reif.angle.1.prelim} that
$
	\hdist \bigl( (x+V_{R}) \cap \B(x,r) , S \cap \B(x,r) \bigr) \leq 2\epsilon\lambda^{-1}r \,.
$
Since also $\hdist \bigl( (x+V_r) \cap \B(x,r) , S \cap \B(x,r) \bigr) \leq \epsilon r$,
it follows from the triangle inequality for the Hausdorff distance, and from its invariance under translation, that
$
	\hdist \bigl( V_{r} \cap \B(0,r) , V_{R} \cap \B(0,r) \bigr) \leq \epsilon (1+2\lambda^{-1})r \,.
$
\end{proof}

\begin{Lem}[Different centers, same scale]
	\label{lemma.reif.angle.2}
Assume that
\begin{enumerate}
	\item $S \subset X$, $x_1,x_2 \in S$, $\epsilon > 0$, $\nu > 1$, $R > 0$, $0 < \lambda < 1$, $|x_1-x_2| \leq (1-\lambda)R$;
	\item $1-\lambda+\epsilon+\nu^{-1} < 1$;
	\item $V_i \in \bG(S,x_i,R,\epsilon)$, $i=1,2$.
\end{enumerate}
Then
\begin{equation*}
	\hdist(V_1 \cap \B(0,1), V_2 \cap \B(0,1)) \leq 3 \epsilon \nu \,.
\end{equation*}
\end{Lem}

\begin{proof}
We let the relation $R = \nu r$ define $r > 0$. Choose $\hat{\epsilon} > \epsilon$ such that $1-\lambda+\hat{\epsilon}+\nu^{-1} < 1$. Let $h_1 \in V_1 \cap \B(0,r)$ and define $\xi_1 \defl x_1 + h_1 \in (x_1 + V_1) \cap \B(x_1,r) \subset (x_1 + V_1) \cap \B(x_1,R)$. Choose $y \in S \cap \B(x_1,R)$ such that $|y-\xi_1| \leq \hat{\epsilon} R$, according to hypothesis (3). Observe that
\begin{equation*}
|y-x_2| \leq |y-\xi_1| + |\xi_1-x_1| + |x_1-x_2| \leq \hat{\epsilon}R + \nu^{-1}R + (1-\lambda) R
\leq R
\end{equation*}
and define $\xi_2 \defl x_2 + \pi_{V_2}(y-x_2)$. Since $y \in S \cap \B(x_2,R)$ it follows from assumption (3) that $|y-\xi_2| \leq \epsilon R$. Define also $c_2 \defl x_2 + \pi_{V_2}(x_1-x_2)$. Since $x_2 \in S \cap \B(x_1,(1-\lambda)R) \subset S \cap \B(x_2,R)$ it follows again from assumption (3) that $|c_2-x_1| \leq \epsilon R$. Finally we let $h_2 \defl \xi_2-c_2 \in V_2$. Since $|h_2| = |\pi_{V_2}(y-x_1)| \leq |y-x_1| \leq r$ one has $h_2 \in V_2 \cap \B(0,r)$. Furthermore,
\begin{equation*}
\begin{split}
|h_1-h_2| & = |\xi_1-x_1-\xi_2 + c_2| \\
& \leq |\xi_1-y| + |y-\xi_2| + |x_1-c_2| \\
& \leq \hat{\epsilon} R + \epsilon R + \epsilon R \\
& \leq 3 \hat{\epsilon} \nu r
\end{split}
\end{equation*}
so that
\begin{equation*}
V_1 \cap \B(0,r) \subset \B \big[ V_2 \cap \B(0,r) , 3 \hat{\epsilon} \nu r \big] \,.
\end{equation*}
The conclusion follows from the arbitrariness of $\hat{\epsilon} > \epsilon$ and Lemma \ref{lemma.norms.1}(2).
\end{proof}

\subsection{Jacobians in Hilbert space}

Let $V \in \bG(\Hi,m)$ be an $m$-plane. For a linear map $L : V \to V^\perp$ we define $\J L$ to be the non-negative number such that 
\[
1 + (\J L)^2 = \det(\id_V + L^*L) = \det(\bar L^* \bar L) \,,
\]
where $\bar L = \id_V + L : V \to \Hi$ is the direct sum of $\id_V$ and $L$. If $(e_1,\dots,e_m)$ is an orthonormal basis of $V$, we can define the matrix $M_{ij} \defl \langle L e_i, L e_j\rangle$ and for any $K \subset \{1,\dots, m\}$ we denote by $M_K$ the submatrix of $M$ corresponding to the indices in $K$. The determinant in question calculates as
\begin{align*}
1 + (\J L)^2 & = \det(\delta_{ij} + M_{ij}) \\
& = \sum_{\sigma \in S_m} (-1)^{|\sigma|} \prod_{i=1}^m(\delta_{i\sigma(i)} + M_{i\sigma(i)} ) \\
 & = 1 + \sum_{K \neq \emptyset} \det(M_K) \,.
\end{align*}
The Hilbert-Schmidt norm of $L$ is given by
\[
\|L\|_{\HS} \defl \tr(L^* L)^\frac{1}{2} \,.
\]
The following identity relates the Hilbert-Schmidt norm to $\J L$,
\begin{equation}
\label{jacobian_sum}
(\J L)^2 = \|L\|_{\HS}^2 + \sum_{\#K \geq 2} \det(M_K) \,.
\end{equation}
Further, if $0 \leq \lambda_1^2 \leq \cdots \leq \lambda_m^2$ are the eigenvalues of $L^*L$ and $\|L\|$ is the usual operator norm of $L$, then
\begin{equation}
\label{norms}
\Lip(L) = \|L\| = \lambda_m \leq (\lambda_1^2 + \cdots + \lambda_m^2)^\frac{1}{2} = \|L\|_{\HS} \leq \sqrt{m}\|L\| \,.
\end{equation}
If $f : A \subset V \to V^\perp$ is a map that is differentiable at $x \in A$ we also use the abbreviation $\J f_x$ instead of $\J Df_x$.

\subsection{Chains, Slicing and tangent planes}
\label{chains_subsection}

With $\cP_m(\Hi;G)$ and $\cR_m(\Hi;G)$ we denote the $m$-dimensional polyhedral, respectively rectifiable, chains in the Hilbert space $X$ with coefficients in $G$. We refer to \cite{PH} and \cite{PH2} for precise definitions and properties of these spaces. Here is a sketch and some basic results we will use later on. Consider a rectifiable $G$-chain $T \in \cR_m(\Hi;G)$. According to the definition, $T$ can be represented by a sequence of bi-Lipschitz maps $\gamma_i : K_i \to \Hi$ defined on compact sets $K_i \subset \R^m$ and measurable maps $\bg_i : K_i \to G$ such that $\gamma_i(K_i) \cap \gamma_j(K_j) = \emptyset$ for $i \neq j$. The set $M_T \defl \bigcup_i \gamma_i(K_i)$ is $\cH^m$-rectifiable and $\bg : \Hi \to G$ is defined by $\bg(x) \defl 0_G$ for $x \notin M_T$ and $\bg(x) \defl \bg_i(\gamma_i^{-1}(x))$ if $x \in \gamma_i(K_i)$. Clearly, $\bg$ is $\cH^m$-measurable and it is further assumed to be integrable, respectively that the corresponding measure $\|T\| \defl \cH^m \res \|\bg\|$ is finite. The \emph{mass} of $T$ is the total measure $\bM(T) = \|T\|(\Hi)$. For any Borel set $B \subset \Hi$, the restriction $T \res B$ is well defined by the weight function $\chi_B \bg : \Hi \to G$. Obviously,
\[
\bM(T \res B) = \int_{B \cap M_T} \|\bg(x)\| \, d\cH^m(x) \,.
\]
Since the metric differential (in the sense of Kirchheim \cite{Kirch}) of each $\gamma_i$ is represented by a scalar product at almost every point, we can further assume that the $\gamma_i$ are close to isometries, i.e.\ given a $\lambda > 1$ we can find parametrizations $\gamma_i$ as above such that $\max\{\Lip(\gamma_i),\Lip(\gamma_i^{-1})\} \leq \lambda$, see \cite[Lemma~4]{Kirch} and \cite[Lemma~3.2.2]{Fed}.

The \emph{$m$-density} of $\|T\|$ at a point $x$ is defined by
\begin{align*}
\Theta^m(\|T\|,x) & \defl \lim_{r \downarrow 0} \frac{\|T\|(\B(x,r))}{\balpha(m) r^m} \,,
\end{align*}
in case the limit exists. It can be shown that for $\|T\|$-a.e.\ point $x \in \Hi$, there holds $\Theta^m(\|T\|,x) = \|\bg(x)\|$.

Let $f : \Hi \to \R^n$ be a Lipschitz map for some $n \leq m$. Following \cite[Section~3.7]{PH}, the slice $\langle T,f,y\rangle$ is an element of $\cR_{m-n}(\Hi;G)$ for almost every $y \in \R^n$. In particular, for almost every $y$, the set $M_T \cap f^{-1}\{y\}$ is $\cH^{m-n}$-rectifiable and has a $G$-orientation given by $\pm \bg$. On a given chart $K_i$ define $T_i$ as the $m$-dimensional $G$-chain equipped with $G$-orientation $\bg_i$. For almost all $y$ it follows from \cite[Theorem~3.8.1]{PH} that
\[
\langle T,f,y\rangle \res \gamma_i(K_i) = \langle T \res \gamma_i(K_i),f,y\rangle = \gamma_{i\#}\langle T_i,f \circ \gamma_i,y\rangle \,.
\]
Hence by the coarea formula \cite[Theorem~3.2.12]{Fed},
\begin{align*}
\int_{\R^n}\bM(\langle T,f,y\rangle & \res \gamma_i(K_i)) \, dy \\
 & = \int_{\R^n} \bM(\gamma_{i\#}\langle T_i,f \circ \gamma_i,y\rangle) \, dy \\
 & \leq \Lip({\gamma_i})^{m-n} \int_{\R^n} \bM(\langle T_i,f \circ \gamma_i,y\rangle) \, dy \\
 & = \Lip({\gamma_i})^{m-n} \int_{\R^n} \int_{K_i \cap (f \circ \gamma_i)^{-1}\{y\}} \|\bg_i(z)\| \, d\cH^{m-n}(z)\, dy \\
 & = \Lip({\gamma_i})^{m-n} \int_{K_i} \|\bg_i(x)\| C_n(D(f \circ \gamma_i)_x) \, dx \,,
\end{align*}
where $C_n(L)$ denotes the coarea factor of a linear map $L : \R^m \to \R^n$. In \cite[Lemma 9.3]{AK} or the proof of \cite[Lemma~3.2.10]{Fed} it is shown that in case the kernel $K$ of $L$ has dimension $m-n$, then $C_n(L) = J_m(q)/J_{m-n}(p|_K)$ (otherwise $C_n(L) = 0$), where $q = L + p : \R^m \to \R^m$ is linear and $p$ is injective on $K$. In the setting of Hilbert spaces we have $J_k(q)J_{k'}(q') = J_{k + k'}(q \oplus q')$ in case $q : W \to \R^k$, $q' : W' \to \R^{k'}$ and $W \oplus W'$ is a direct sum of Hilbert spaces of dimension $k$ and $k'$ respectively. Hence, $C_n(L) = J_n(L|_{K^\perp})$ and in particular $C_n(L) \leq \|L\|^n = \Lip(L)^n$. Applied to the situation above and assuming that $\max\{\Lip(\gamma_i),\Lip(\gamma_i^{-1})\} \leq \lambda$ we get
\begin{align*}
\int_{\R^n}\bM(\langle T,f,y\rangle \res \gamma_i(K_i)) \, dy & \leq \lambda^{m-n} \Lip(f \circ \gamma_i)^n \int_{K_i} \|\bg_i(x)\| \, dx \\
 & \leq \lambda^m \Lip(f)^n \bM({\gamma_i^{-1}}_\# (T \res \gamma_i(K_i))) \\
 & \leq \lambda^{2m} \Lip(f)^n \bM(T \res \gamma_i(K_i)) \,.
\end{align*}
Summing over all $i$ and taking the limit for $\lambda \to 1$ we get
\begin{equation}
\label{slicemass}
\int_{\R^n}\bM(\langle T,f,y\rangle) \, dy \leq \Lip(f)^n \bM(T) \,.
\end{equation}

Another important tool we need is the cone construction. For a rectifiable chain $T \in \cR_{m-1}(\Hi;G)$ one can construct the cone over $T$ with center $x \in \Hi$, $\curr x \cone T \in \cR_m(\Hi;G)$. In case $T$ is a polyhedral chain we can write $T = \sum_i g_i \curr{S_i}$ for finitely many oriented simplices $S_i$. If $x = 0$, then $\curr 0 \cone T = \sum_i g_i \curr{S_i'}$ where $S_i' \defl \{tx : t \in [0,1], x \in S_i\}$ and an appropriate orientation. For general rectifiable chains we may define $\curr 0 \cone T \defl \psi_\# (\curr{0,1} \times T)$ where $\psi(t,x) = tx$. If $B \subset \partial \B(0,r)$ has $\cH^{m-1}$-finite measure, then the set $B' = \{tx : t \in [0,1], x \in B\}$ satisfies $\cH^m(B') = \frac{r}{m}\cH^{m-1}(B)$ as one can verify easily. Hence if $\spt(T) \subset \partial \B(0,r)$ and $\partial T = 0$, then $\partial (\curr 0 \cone T) = T$ and
\begin{equation}
\label{cone_equality}
\bM(\curr 0 \cone T) = \frac{r}{m}\bM(T) \,.
\end{equation}
In case $B \subset \Hi\setminus \oB(0,r)$ we get $\cH^m(B' \cap \B(0,r)) \leq \frac{r}{m}\cH^{m-1}(B)$ since the orthogonal projection of $\Hi\setminus\B(0,r)$ onto $\partial \B(0,r)$ is $1$-Lipschitz. Hence if $\spt(T) \subset \Hi\setminus\oB(0,r)$ and $\partial T = 0$, then $\spt(\partial(\curr 0 \cone T)) \subset \Hi \setminus \oB(0,r)$ and
\begin{equation}
\label{cone_inequality}
\bM(\curr 0 \cone T) \leq \frac{r}{m}\bM(T) \,.
\end{equation}

The basic idea in order to establish a regularity result for some chain $T$ is to show that $\spt(T)$ can be uniformly approximated by planes. Rectifiable measures in $\R^n$ have weak tangent planes almost everywhere. It is then a classical observation that this tangent plane is actually a tangent plane for the support of the measure in case the measure is Ahlfors regular. Here we show this fact for chains in $\Hi$. For this it is crucial that Hilbert spaces have the Radon-Nikod\'ym property, which allows to differentiate Lipschitz maps $\R^m \to \Hi$ at almost every point. For a Radon measure $\phi$ in $\Hi$ and $V \in \bG(\Hi,m)$ we define
\begin{align*}
\bbeta_\infty(\phi,x,r,V) & \defl r^{-1}\sup\{|\pi_{V^\perp}(y - x)| : y \in \spt(\phi) \cap \B(x,r)\} \,.
\end{align*}

\begin{Lem}
	\label{tangentplane}
Let $\Hi$ be a Hilbert space, $T \in \cR_m(\Hi;G)$ and let $U \subset \spt(T)$ be a relatively open set with $\dist(U,\spt(\partial T)) \geq r_0 > 0$. Assume that there are constants $c_2 \geq c_1 > 0$ such that $\|T\|$ is Ahlfors regular in $U$ in the sense that for all $x \in U$ and $0 < r < r_0$,
\begin{equation}
\label{density_bound}
c_1 \balpha(m) r^m \leq \|T\|(\B(x,r)) \leq c_2 \balpha(m) r^m \,.
\end{equation}
Then for $\|T\|$-a.e.\ $x \in U$ there is a plane $W_x \in \bG(\Hi,m)$ such that
\[
\lim_{r \downarrow 0} \bbeta_\infty(\|T\|,x,r,W_x) = 0 \,.
\]
Moreover, there is some $r_x > 0$ such that for all $0 < r < r_x$, there holds $\bbeta_\infty(\|T\|,x,r,W_x) \leq 2^{-1}$ and
\[
\pi_{W_x\#} \left(T \res (\B(x,r) \cap Z_{W_x}(x,2^{-1}r))\right) = g_x \curr{\B_{W_x}(\pi_{W_x}(x),2^{-1}r)} \,,
\]
for some $g_x \in G$ with $\|g_x\| = \Theta^m(\|T\|,x)$.
\end{Lem}

\begin{proof}
The bounds on volumes of balls \eqref{density_bound} imply that $c_1 \leq \Theta^{m*}(\|T\|,x) \leq c_2$ for all $x \in U$ and by a standard result of measure theory
\begin{equation}
\label{measurecompare}
c_1\cH^m(B) \leq \|T\|(B) \leq 2^mc_2\cH^m(B)
\end{equation}
for all Borel measurable subsets $B \subset U$, see e.g.\ \cite[Theorem~2.4.3]{AT}. As in Subsection~\ref{chains_subsection} let $\gamma^i : K_i \to U$  be countably many bi-Lipschitz maps defined on compact sets $K_i \subset \R^m$ with pairwise disjoint images and $\cH^m(U \setminus \bigcup_i \gamma^i(K_i)) = 0$. Let $\bg_i : K_i \to G$ be $\cH^m$-measurable functions such that $T \res U = \sum_{i = 1}^\infty \gamma^i_\#(\bg_i \curr{K_i})$ and $\|\bg_i\| \in [c_1,c_2]$. The $\cH^m$-measurable function $\bg : \Hi \to G$ is defined to be $\bg_i \circ (\gamma^i)^{-1}$ on $\gamma^i(K_i)$ and equal to $0_G$ on the complement of the union of all these sets. We can further assume that for each $i$ the inverse of $\gamma^i$ is extended to a Lipschitz map $\psi^i : \Hi \to \R^m$. Kirszbraun's Theorem allows to extend $\gamma^i$ to a Lipschitz map $\gamma^i : \R^m \to \Hi$. In what follows we don't need a sharp control on the Lipschitz constant of these extensions and just assume that $\Lip(\gamma^i),\Lip(\psi^i) \leq L$ for some $L > 1$. 
The extended maps $\gamma^i$ are differentiable almost everywhere since Hilbert spaces have the Radon-Nikod\'ym property. For a proof of this see for example \cite[1.2]{DePauw4}. The Radon-Nikod\'ym property implies as in a standard proof of Rademacher's Theorem that each $\gamma^i$ is differentiable at almost every point of $\R^m$. With the help of the area formula we see that for $\cH^m$-almost every point $x \in U$ there is an index $i$ and $y \in K_i$ such that
\begin{enumerate}
	\item $\gamma^i(y) = x$,
	\item $D\gamma^i_y$ exists and has rank $m$,
	\item $\Theta^m(\cH^m\res K_i,y) = \Theta^m(\cH^m\res \gamma^i(K_i),x) = \Theta^m(\cH^m\res U,x) = 1$,
	\item $\|\bg_i\| : K_i \to \R$ has an approximate point of continuity at $y$.
\end{enumerate}
Let $W_x \in \bG(\Hi,m)$ be the image of $D\gamma^i_y$. The differentiability of $\gamma^i$ at $y$ and the fact that $\gamma^i$ is $L$-bi-Lipschitz on $K_i$ imply that for any $0 < \epsilon < \frac{1}{4}$ there is a $0 < r_\epsilon < r_0$ such that $\B(x,r_\epsilon) \cap \spt(T) \subset U$ and for all $0 < r < r_\epsilon$,
\begin{equation}
\label{bilipestimate}
\dist(\gamma^i(\B(y,Lr)),x + W_x) \leq \epsilon r \,, \; \mbox{and} \; \gamma^i(K_i \setminus \B(y,Lr)) \subset U\setminus\B(x,r) \,,
\end{equation}
and with (3),
\begin{align}
\nonumber
(1 - 3^{-1}(2^mc_2)^{-1}c_1 \epsilon^m)\balpha(m)r^m & \leq \cH^m(\gamma^i(K_i) \cap \B(x,r)) \\
\nonumber
 & \leq \cH^m(U \cap \B(x,r)) \\
\label{densityest}
 & \leq (1 + 3^{-1}(2^mc_2)^{-1}c_1 \epsilon^m)\balpha(m)r^m\,.
\end{align}
The inclusions in \eqref{bilipestimate} imply
\begin{equation}
\label{measureflat}
\gamma^i(K_i) \cap \B(x,r) \subset N(x,\epsilon r) \,,
\end{equation}
where $N(x,t) \defl \{z \in U : \dist(z,x + W_x) \leq t\}$ for $t \geq 0$. Assume by contradiction that for some $0 < r < r_\epsilon$ there is a point
\begin{equation}
\label{assumepointoutside}
x' \in U \cap \B(x,r/2) \setminus N(x,2\epsilon r) \,.
\end{equation}
Then obviously $\B(x',\epsilon r) \subset \B(x,r)$ and $\dist(\B(x',\epsilon r),x + W_x) > \epsilon r$ as well as $\cH^m(U \cap \B(x',\epsilon r)) \geq (2^mc_2)^{-1}c_1\balpha(m) (\epsilon r)^m$ by \eqref{density_bound} and \eqref{measurecompare}. Hence with the first estimate of \eqref{densityest} and \eqref{measureflat},
\begin{align*}
\cH^m(U \cap \B(x,r)) & \geq \cH^m(\gamma^i(K_i) \cap \B(x,r)) + \cH^m(U \setminus N(x,\epsilon r)) \\
 & > (1 + 3^{-1}(2^mc_2)^{-1}c_1 \epsilon^m)\balpha(m)r^m \,.
\end{align*}
This contradicts the last estimate of \eqref{densityest}. Thus \eqref{assumepointoutside} is wrong and $\limsup_{r \downarrow 0} \bbeta_\infty(\|T\|,x,r,W_x) \leq 4\epsilon$ for all $\epsilon > 0$. This proves the first statement.

To establish the statement about the projections, let $r_x \defl 2^{-1} r_{8^{-1}}$. Then $\bbeta_\infty(\|T\|,x,r,W_x) \leq 2^{-1}$ for all $0 < r < r_x$. For almost every such $r$, $\langle T,d_x,r \rangle = \partial (T \res \B(x,r)) \in \cR_{m-1}(\Hi;G)$, where $d_x(z) = |x-z|$. Since $\spt(\langle T,d_x,r \rangle) \subset \partial \B(x,r) \cap N(x,2^{-1}r)$ we get that $|\pi_{W_x}(z-x)| > 2^{-1}r$ for all $z \in \spt(\langle T,d_x,r \rangle)$. Hence the constancy theorem \cite[Theorem~6.4]{PH2} implies that for $0 < s \leq 2^{-1}r$,
\[
\pi_{W_x\#} \bigl(T \res (\B(x,r) \cap Z_{W_x}(x,s)\bigr) = \bigl(\pi_{W_x\#} (T \res \B(x,r))\bigr) \res \B(\pi_{W_x}(x),s) \,,
\]
and this chain is equal to $g_{r,s}\curr{\B(\pi_{W_x}(x),s)}$ for some $g_{r,s} \in G$. The map $(r,s) \mapsto g_{r,s}$ is locally constant by the constancy theorem and since the domain $\{(r,s) \in (0,r_x) \times (0,r_x) : 2s \leq r \}$ is connected, this function is constant. Therefore we find a unique $g_x \in G$ as in the second statement of the lemma. It remains to show that $\Theta^m(\|T\|,x) = \|g_x\|$.

Since $\|T\|(\B(x,r)) = \sum_{j}\int_{\B(x,r) \cap \gamma^j(K_j)} \|\bg_j \circ (\gamma^j)^{-1}\| \, d\cH^m$ and $\|\bg_j\| \in [c_1,c_2]$, it follows directly from (3) and (4) that $\|\bg\|$ is $\cH^m$-approximately continuous at $x$ and 
\begin{equation}
\label{massdensity}
\Theta^m(\|T\|,x) = \Theta^m(\|T\| \res \gamma^i(K_i),x) = \|\bg(x)\| \,.
\end{equation}
Consider the $L$-Lipschitz function $f : \Hi \to \R$ given by
\[
f(z) \defl \max\{|\psi^i(z) - \psi^i(x)|, L^{-1}|z - x|\} \,.
\]
If $z \in \gamma^i(K_i)$, then $f(z) = |\psi^i(z) - \psi^i(x)|$ and also
\[
\B(x,L^{-1}r) \subset \{f \leq r\} \subset \B(x,Lr) \,.
\]
As before, $\langle T,f,r \rangle = \partial (T \res \{f \leq r\}) \in \cR_{m-1}(\Hi;G)$ for almost every $r > 0$. With (3) and the coarea formula
we find a sequence $r_k \in [2^{-k-1},2^{-k}]$ such that
\[
\lim_{k \to \infty} \frac{\cH^{m-1}(\partial \B(y,r_k) \cap K_i)}{\cH^{m-1}(\partial \B(y,r_k))} = 1 \,,
\]
as well as $\langle T,f,r_k \rangle = \partial (T \res \{f \leq r_k\}) \in \cR_{m-1}(\Hi;G)$ for all $k$. From this it follows that $\frac{1}{r_k}\hdist(\partial \B(y,r_k) \cap K_i,\partial \B(y,r_k)) \to 0$. Because $\gamma^i : \R^m \to X$ is Lipschitz, differentiable at $y$ and $\lim_{r \downarrow 0} \bbeta_\infty(\|T\|,x,r,W_x) = 0$, it follows from (3) that $\frac{1}{r_k} \dist(\spt(\langle T,f,r_k \rangle),\gamma^i(K_i)) \to 0$ and further $\frac{1}{r_k}\hdist\bigl(\gamma^i(\partial \B(y,r_k) \cap K_i),\spt(\langle T,f,r_k \rangle)\bigr) \to 0$. Thus
\[
\frac{1}{r_k}\hdist(\partial \B(y,r_k),\spt(\partial T_{k})) \to 0 \,,
\]
for the push-forward $T_k \defl \psi^i_\#(T \res \{f \leq r_k\}) \in \cR_{m}(\R^m;G)$. The constancy theorem implies now that there are $g_k \in G$ and a sequence $0 < s_k < r_k$ with $s_k/r_k \to 1$ and
\begin{equation}
\label{constancyapplication}
T_{k} \res \B(y,s_k) = g_k \curr{\B(y,s_k)} \,.
\end{equation}
From the construction of the push-forward as in \cite[3.5]{PH}, it follows that
\[
T_{k} = \bg_i \curr{K_i \cap \B(y,r_k)} + \psi^i_\#(T \res A_k) \,,
\]
for some $\cH^m$-measurable set $A_k \subset U \setminus \gamma^i(K_i)$. From (3) and \eqref{massdensity} it follows that $\frac{1}{r_k^m}\cH^m(\psi^i(A_k)) \to 0$ and also $\frac{1}{r_k^m}\bM(\psi^i_\#(T \res A_k)) \to 0$. With \eqref{constancyapplication} we conclude that $\bg_i = g_k$ on $\B(y,r_k/2) \setminus \psi^i(A_k)$. Since $r_k \in [2^{-k-1},2^{-k}]$, this shows that $g_k = g_{k+1}$ if $k$ is large enough. Hence $\bg_i$ is approximately equal to some $g_x' \in G$ at $y \in K_i$ with $\|g_x'\| = \|\bg(x)\| = \Theta^m(\|T\|,x)$ by (4) and \eqref{massdensity}. It follows that $\frac{1}{r_k^m}\bM(T_k - g_x'\curr{\B(y,r_k)}) \to 0$.
Thus,
$\frac{1}{r_k^m}\bM(T \res \{f \leq r_k\} - \gamma^i_\#(g_x'\curr{\B(y,r_k)})) \to 0$ and further
\begin{equation}
\label{differentialprojection1}
\frac{1}{r_k^m}\bM\bigl(\pi_{W_x\#}(T \res \{f \leq r_k\}) - (\pi_{W_x}\circ\gamma^i)_\#(g_x'\curr{\B(y,r_k)})\bigr) \to 0 \,. 
\end{equation}
From the fact that $\pi_{W_x} \circ \gamma^i$ is differentiable at $y$ and \eqref{constancyapplication}, 
\begin{equation}
\label{differentialprojection2}
\bM\bigl(g_x'\curr{x' + D(\pi_{W_x} \circ \gamma^i)_y (\B(0,1))} - (\eta_{r_k} \circ \pi_{W_x}\circ\gamma^i)_\#(g_x'\curr{\B(y,r_k)})\bigl) \to 0 \,,
\end{equation}
for $k \to \infty$, where $\eta_{r}(z) \defl \frac{1}{r}(z-x') + x'$ and $x' = \pi_{W_x}(x)$. We already know that $\pi_{W_x}$ projects $T$ around $x$ to a chain with weight $g_x$, hence $g_x = g_x'$ follows from \eqref{differentialprojection1} and \eqref{differentialprojection2}. This concludes the proof.
\end{proof}

The following lemma shows that the group element associated with a projection doesn't change for projections to nearby planes. This is a direct consequence of the constancy theorem.

\begin{Lem}
	\label{differentplane}
Let $T \in \cR_m(\Hi;G)$ be rectifiable chain. Assume that there is an $m$-plane $V \in \bG(\Hi,m)$ such that,
\begin{enumerate}
	\item $\spt(T) \subset \B(0,1)$ and $\spt(\partial T) \subset \partial\B(0,1)$,
	\item $|\pi_{V^\perp}(x)| \leq \frac{1}{5}$ for $x \in \spt(T)$,
	\item $\pi_{V\#} (T \res Z_V(0,\frac{1}{2})) = g_0 \curr{\B_V(0,\frac{1}{2})}$.
\end{enumerate}
Then for any $W \in \bG(\Hi,m)$ that satisfies $\|\pi_W - \pi_V\| \leq \frac{1}{5}$,
\[
\pi_{W\#} (T \res Z_W(0,2^{-1})) = g_0 \curr{\B_W(0,2^{-1})} \,.
\]
We assume that $W$ and $V$ are given orientations such that the projection $\pi_W : W \to V$ is orientation preserving.
\end{Lem}

\begin{proof}
There is a Lipschitz family of $m$-planes $t \mapsto W_t$ connecting $V$ with $W$ such that $\|\pi_{W_t} - \pi_{W_s}\| \leq L |s - t|$ and $|\pi_{V}(w) - w| \leq 5^{-1}|w|$ for all $s,t \in [0,1]$ and $w \in W_t$. With Lemma~\ref{lemma.norms.1}, the latter is equivalent to $\|\pi_V - \pi_{W_t}\| \leq 5^{-1}$ for all $t$. Such a family of planes is for example given by $\operatorname{im}(L_t)$ where $L_t : W \to \Hi$ is defined by $L_t(w) \defl \pi_V(w) + t\pi_{V^\perp}(w)$. With a reparametrization, we may assume that $W_t = V$, $W_{1-t} = W$ for $t \leq 4^{-1}$. Let $S \defl \curr{0,1} \times T \in \cR_{m+1}(\R \times \Hi;G)$ and $\psi : [0,1] \times \Hi \to [0,1] \times V$ be the Lipschitz map given by $\psi(t,x) = (t,\pi_V \circ \pi_{W_t}(x))$. We identify $[0,1] \times V$ isometrically (in an orientation preserving way) with $[0,1] \times \R^m \subset \R^{m+1}$. The boundary of $S$ is given by, see \cite[4.1.8]{Fed} and the definition of the boundary in \cite[Section~5.1]{PH},
\begin{equation}
\label{bounderyidentity}
\partial S = \curr 1 \times T - \curr 0 \times T - \curr{0,1} \times \partial T \,.
\end{equation}
Let $x \in \spt(T)$. By assumption, $x$ can be written as $x = v + v^\perp$ with $v \in \B_V(0,1)$ and $|v^\perp| \leq 5^{-1}$. It follows by the closeness of $W_t$ to $V$,
\begin{align}
	\nonumber
|\pi_{W_t^\perp}(x)| & = |\pi_{W_t}(x) - x| \leq |\pi_{W_t}(v) - v| + |\pi_{W_t}(v^\perp) - v^\perp| \\
	\label{verticaldistnewplane}
 & \leq \frac{|v|}{5} + \frac{2}{5} \leq \frac{3}{5} \,.
\end{align}
In particular if $|x| = 1$,
\begin{align}
 \nonumber
|\pi_V\circ\pi_{W_t}(x)| & \geq |\pi_V(x)| - |\pi_V(\pi_{W_t}(x) - x)| \geq 1 - |\pi_{V^\perp}(x)| - |\pi_{W_t^\perp}(x)| \\
\label{iteratedist}
 & \geq 1 - \frac{4}{5} \geq \frac{1}{5} \,.
\end{align}
		
Fix some $t_0 \in (0,\frac{1}{4})$ and let $B \defl (t_0,1 - t_0) \times \oB_V(0,\frac{1}{5})$. Because of \eqref{bounderyidentity} and \eqref{iteratedist}, the set $B$ lies outside the support of the boundary $\partial \psi_\# S$. The constancy theorem implies that $(\psi_\# S) \res B = g \curr B$ for some $g \in G$. For almost all $t \in [t_0,1-t_0]$ we have for the projection $\pi(t,x) \defl t$,
\begin{equation*}
\langle (\psi_\# S) \res B, \pi, t \rangle = \langle g \curr{B} , \pi , t \rangle 
 = g(\curr t \times \curr{\oB_V(0,5^{-1})}) \,.
\end{equation*}
Considering $t_0 < t < 4^{-1}$ we infer $g = g_0$.
The same slicing argument shows that 
\[
	((\pi_V \circ \pi_{W_t})_\# T) \res \oB_V(0,5^{-1}) = g_0 \curr{\oB_V(0,5^{-1})} \,,
\]
for almost all $t \in [t_0,1-t_0]$. $\pi_V : W_t \to V$ is bi-Lipschitz and orientation preserving (by assumption). Hence we get for some small $\rho > 0$ (such that $\oB_{W_t}(0,\rho) \subset {\pi_V}^{-1}(\oB_V(0,\frac{1}{5}))$),
\[
	(\pi_{W_t\#} T) \res \oB_{W_t}(0,\rho) = (\pi_{W_t\#} T) \res \oB_{W_t}(0,\rho) = g_0 \curr{\oB_{W_t}(0,\rho)} \,.
\]
Since $(\frac{3}{5})^2 + (\frac{1}{2})^2 < 1$ and \eqref{verticaldistnewplane}, we again see by the constancy theorem on $W$ that the conclusion of the lemma holds for $W = W_1$.

\end{proof}

\subsection{Excess over a plane}
\label{excess_section}

For a chain $T \in \cR_m(\Hi; G)$, an $m$-plane $V \in \bG(\Hi,m)$ and an $\cH^m$-measurable set $B \subset V$ the excess over $B$ is defined by
\[
\Exc(T,V,B) \defl \bM\left(T\res \pi_V^{-1}(B)\right) - \bM\left(\pi_{V\#}(T \res \pi_V^{-1}(B))\right) \,.
\]
Since $\Lip(\pi_V) = 1$, this number is non-negative and for disjoint Borel measurable sets $B,B' \subset V$ it is $\Exc(T,V,B) + \Exc(T,V,B') = \Exc(T,V,B \cup B')$. From this it follows that $\Exc$ is subadditive on the Borel $\sigma$-algebra on $V$. We also use the notation
\[
\Exc(T,V,x,r) \defl \Exc(T,V,\B_V(\pi_V(x),r)) \,,
\]
and for $x$ at the origin we abbreviate $\Exc_r(T,V)$ for $\Exc(T,V,0,r)$. We will occasionally calculate the excess for different masses. For example $\bS(T)$ denotes the \emph{size} of $T$ and is just the Hausdorff measure of the underlying set $M_T$. Alternatively, $\bS(T)$ is the mass with respect to the discrete norm on $G$ taking only the values $0$ and $1$. The excess with respect to $\bS$ is denoted by $\Exc(T,V,B,\bS)$, and similarly if we work with another mass. Here we give some first implications of small excess. For a given $T \in \cR_m(\Hi;G)$ and $V \in \bG(\Hi,m)$ we denote by $E_2$ the set of points $x \in \B_V(0,1)$ for which $M_x \defl {\pi_V}^{-1}\{x\} \cap M_T$ contains more than one point. One can show that this set is Lebesgue measurable and hence $E_2$ contains a Borel set $E_2'$ with $\cL^m(E_2\setminus E_2') = 0$.

\begin{Lem}
[{\cite[Lemma~5]{Reif}}]
\label{multiple_values_estimate}
Let $T \in \cR_m(\Hi;G)$ and assume there is an $m$-plane $V \in \bG(\Hi,m)$ such that $\Exc_1(T,V,\bS) \leq \epsilon$. Then,
\[
\cH^m(E_2) \leq \int_{E_2} \# M_x \, dx  \leq 2 \epsilon \,.
\]
\end{Lem}

\begin{proof}
The Jacobian of a map on an $m$-rectifiable set is bounded by the $m$th power of its Lipschitz constant. The area formula leads to,
\begin{align*}
\int_{E_2'} \# M_x \, dx & \leq 2\int_{E_2'} \# M_x - 1 \, dx \\
 & = 2\int_{{\pi_V}^{-1}(E_2') \cap M_T} J(\pi_V|M_T)(z) \, d\cH^m(z) - 2 \cH^m(E_2') \\
 & \leq 2(\cH^m({\pi_V}^{-1}(E_2') \cap M_T) - \cH^m(E_2'))\\
 & = 2 \Exc(T,V,E_2',\bS) \leq 2 \epsilon \,.
\end{align*}
\end{proof}

A similar estimate is obtained by summing over the normed group elements in the preimage.

\begin{Lem}
\label{multiple_values_mass_estimate}
Let $T \in \cR_m(\Hi;G)$ and assume there is an $m$-plane $V \in \bG(\Hi,m)$ with $\spt(\partial T) \cap Z_V = \emptyset$. Let $g_0 \in G$ be such that $\pi_{V\#} (T \res Z_V) = g_0 \curr{\B_V(0,1)}$, then
\[
\sum_{y \in M_x} \|\bg(y)\| \geq \|g_0\|
\]
for almost every $x \in \B_V(0,1)$. If further
\[
\max\{\Exc_1(T,V),\|g_0\|\Exc_1(T,V,\bS)\} \leq \|g_0\|\epsilon \,,
\]
then,
\[
\int_{E_2} \sum_{y \in M_x} \|\bg(y)\| \, dx  \leq 3 \|g_0\|\epsilon \,.
\]
\end{Lem}

\begin{proof}
The slices $\langle T,\pi_V,x \rangle$ are defined for almost every $x$ and concentrated on the finite set $M_x$, i.e.\ $\langle T,\pi_V,x \rangle = \sum_{y \in M_x} \pm \bg(y)\curr{y}$ with $+$ if $\pi_V : \operatorname{Tan}(M_T,y) \to V$ is orientation preserving and $-$ otherwise. By \cite[Theorem 3.8.1]{PH} we have
\[
\sum_{y \in M_x} \pm \bg(y) \curr x = \pi_{V\#} \langle T,\pi_V,x \rangle = \langle \pi_{V\#} T,\pi_V,x \rangle = g_0\curr x
\]
for almost all $x \in \B_V(0,1)$ since on $\B_V(0,1)$ the chain $\pi_{V\#} T$ is a multiple of $g_0$. The first statement of the lemma now follows from the triangle inequality of the norm. By the assumption on the mass excess,
\begin{align*}
& \int_{E_2'} \sum_{y \in M_x} \|\bg(y)\| - \|g_0\| \, dx \\
 & \qquad = \int_{{\pi_V^{-1}}(E_2') \cap M_T} \|\bg(z)\|J(\pi_V|M_T)(z) \, d\cH^m(z) - \|g_0\|\cH^m(E_2') \\
 & \qquad \leq \int_{{\pi_V^{-1}}(E_2') \cap M_T} \|\bg(z)\| \, d\cH^m(z) - \|g_0\|\cH^m(E_2') \\
 & \qquad = \Exc(T,V,E_2') \leq \|g_0\|\epsilon \,.
\end{align*}
By Lemma~\ref{multiple_values_estimate} it is also $\int_{E_2'} \|g_0\| dx \leq 2 \|g_0\|\epsilon$, hence the result follows.
\end{proof}

Often we assume that we are away from the boundary and close to a plane, i.e.\ by truncating, translating and scaling we are in the situation where $\spt(\partial T) \cap Z_V = \emptyset$. By the constancy theorem for $G$ chains \cite[Theorem~6.4]{PH2} it follows that there is an element $g_0 \in G$ such that $\pi_{V\#}(T \res Z_V) = g_0 \curr{\B_V(0,1)}$. Next we give conditions such that the size excess is small if the mass excess is small.

\begin{Lem}
\label{sizebound}
Let $T \in \cR_m(\Hi;G)$, $V \in \bG(\Hi,m)$ and $\epsilon > 0$ such that
\begin{enumerate}
	\item $\spt(\partial T) \cap Z_V = \emptyset$,
	\item $\pi_{V\#}(T \res Z_V) = g_0 \curr{\B_V(0,1)}$ for some $g_0 \in G \setminus \{0_G\}$,
	\item $\Theta^m(\|T\|,x) \geq \frac{3}{4}\|g_0\|$ for $\|T\|$-almost every $x \in \Hi$,
	\item $\Exc_1(T,V) \leq \|g_0\|\epsilon$. 
\end{enumerate}
Then $\Exc_1(T,V,\bS) \leq 5\epsilon$.
\end{Lem}

\begin{proof}
Let $M_T \subset Z_V$ be an $\cH^m$-rectifiable set with $\|T\|(Z_V \setminus M_T) = 0$ and $\bg(x) = 0_G$ for $x \notin M_T$. Since $\Theta^m(\|T\|,x) = \|\bg(x)\|$ for $\|T\|$-almost every $x \in \Hi$, assumption (3) is equivalent to $\|\bg(x)\| \geq \frac{3}{4}\|g_0\|$ for $\|T\|$-almost every $x \in \Hi$.

As before let $E_2$ be the set of those $x \in \B_V(0,1)$ for which $M_x \defl \pi_V^{-1}\{x\} \cap M_T$ contains more than one point. For almost every $x \in \B_V(0,1)\setminus E_2$, the set $M_x$ consists of exactly one point $y$ for which $\|\bg(y)\| = \|g_0\|$. Using the area formula and the first part of Lemma~\ref{multiple_values_mass_estimate},
\begin{align*}
\|g_0\|\epsilon & \geq \int_{M_T} \|\bg(x)\| \, d\cH^m(x) - \|g_0\| \balpha(m) \\
 & \geq \int_{M_T} \|\bg(x)\| J(\pi_V|M_T)(x) \, d\cH^m(x) - \int_{\B_V(0,1)} \|g_0\| \, dx \\
  & = \int_{\B_V(0,1)} \sum_{y \in M_x} \|\bg(y)\| - \|g_0\| \, dx \\
  & = \int_{E_2} \sum_{y \in M_x} \|\bg(y)\| - \|g_0\| \, dx \\
  & \geq \int_{E_2} \frac{6}{4}\|g_0\| - \|g_0\| \, dx \\
	& = \frac{1}{2}\|g_0\| \cH^m(E_2) \,.
\end{align*}
If $E \subset \B_V(0,1)$ is Borel measurable with $\cH^m(E) \leq \delta$, then
\begin{align*}
\|g_0\|\cH^m(M_T \cap \pi_V^{-1}(E)) & \leq \frac{4}{3}\bM(T \res \pi_V^{-1}(E)) \\
 & \leq \frac{4}{3}(\Exc(T, V, E) + \|g_0\|\cH^m(E)) \\
 & \leq \frac{4}{3}\|g_0\|(\epsilon + \delta) \,.
\end{align*}
Let $B_2 \supset E_2$ be a Borel set with $\cL^m(B_2 \setminus E_2) = 0$ and set $B_1 \defl \B_V(0,1) \setminus B_2$, . As shown above, it is $\cH^m(B_2) \leq 2\epsilon$ and hence
\begin{align*}
\Exc_1(T,V,\bS) & = \cH^m(M_T) - \balpha(m) \\
 & = \cH^m(M_T \cap \pi_V^{-1}(B_1)) - \cH^m(B_1) \\
 & \quad + \cH^m(M_T \cap \pi_V^{-1}(B_2)) - \cH^m(B_2) \\
 & \leq \|g_0\|^{-1}\Exc(T, V,B_1) + \cH^m(M_T \cap \pi_V^{-1}(B_2)) \\
 & \leq \epsilon + \frac{4}{3}(\epsilon + 2\epsilon) \leq 5 \epsilon \,.
\end{align*}
\end{proof}

\subsection{Quadratic forms associated to a chain}
\label{betaprelim}

In order for a chain to be close to a submanifold near some point, it is necessary that it can locally be well approximated by planes. The quadratic form associated to a $G$-chain we introduce below will actually be used twice in these notes. First in the proof of the epiperimetric inequality in Section~3, and second, as part of the moments computations in Section~4.

For $T \in \cR_m(\Hi;G)$, $x \in \Hi$ and $r > 0$ we define the quadratic form
\[
Q(T,x,r)(y) \defl \frac{m+2}{\balpha(m) r^{m+2}} \int_{\B(x,r)} \langle z - x, y \rangle^2 \, d\|T\|(z) \,.
\]
The reason for this particular normalization is that in case $T = g\curr{V}$ for some oriented $m$-plane $V \in \bG(\Hi,m)$, then $Q(T,x,r)(y) = \|g\||\pi_V(y)|^2$ for all $x \in V$ and $r > 0$. This is demonstrated in Lemma~\ref{closeplanes} below. With $L_{Q} : \Hi \to \Hi$ we denote the self-adjoint operator associated to a quadratic form $Q$, i.e.\ $Q(y) = \langle y, L_Q(y) \rangle$. We now show that in case a chain is close to a plane $V$ near $x$, then $L_{Q(T,x,r)}$ is close to the orthogonal projection $\pi_V$.

\begin{Lem}
\label{closeplanes}
There is a constant $\bc_{\theThm}(m) > 0$ with the following property. Let $T \in \cR_m(\Hi;G)$ and $x \in \spt(T)$. Assume there is an $m$-plane $V \in \bG(\Hi,m)$ and $0 < \rho < 1$, $0 < \epsilon < 1$ with
\begin{enumerate}
	\item $\spt(\partial T) \cap Z_V(x,r) = \emptyset$,
	\item $\pi_{V\#} (T \res Z_V(x,r)) = g_0 \curr{\B_V(\pi_V(x),r)}$ for some $g_0 \in G \setminus \{0_G\}$,
	\item $|\pi_{V^\perp}(y - x)| \leq \rho r$ for all $y \in \spt(T) \cap Z_V(x,r)$,
	\item $\Theta^m(\|T\|,x') \geq \frac{3}{4}\|g_0\|$ for $\|T\|$-a.e. $x'$,
	\item $\Exc(T,V,x,r) \leq \|g_0\|\epsilon r^m$.
\end{enumerate}
Then
\begin{align*}
\tr(Q(T,x,r)) - m\|g_0\| & \leq \bc_{\theThm}\|g_0\|(\rho^2 + \epsilon) \,, \\
\|L_{Q(T,x,r)} - \|g_0\|\pi_V\| & \leq \bc_{\theThm}\|g_0\|(\rho + \epsilon) \,.
\end{align*}
\end{Lem}

\begin{proof}
By translating and rescaling $T$ we can assume that $x = 0$ and $r = 1$. This is justified by the scaling factor in the definition of $Q(T,x,r)$. Let $(e_i)_{i \geq 1}$ be an orthonormal basis of $\Hi$ such that $e_1,\dots,e_m \in V$. We define the quadratic form $Q'$ on $\Hi$ by
\[
Q'(y) \defl \frac{m+2}{\balpha(m)} \|g_0\| \int_{\B_V(0,1)} \langle y, z\rangle^2  \, dz \,.
\]
$Q'$ is indeed a quadratic form with corresponding bilinear form $Q'(y,y') = \frac{m+2}{\balpha(m)} \|g_0\| \int_{\B_V(0,1)} \langle y, z\rangle\langle y', z\rangle \, dz$. If $y \in V$ with $|y| = 1$, the symmetry of $Q'$ implies that
\begin{align}
\nonumber
\int_{\B_V(0,r)} \langle y, z\rangle^2  \, dz & = \frac{1}{m}\sum_{i=1}^m \int_{\B_V(0,r)} \langle e_i, z\rangle^2  \, dz \\
\nonumber
 & = \frac{1}{m} \int_{\B_V(0,r)} |z|^2  \, dz \\
 \nonumber
 & = \frac{1}{m} \int_0^r s^2 \cH^{m-1}(\Sp^{m-1}(0,s)) \,ds \\
 \label{flatbeta}
 & = \frac{1}{m} \int_0^r s^{m+1} \balpha(m) m \, dr = \frac{\balpha(m)}{m + 2} r^{m+2} \,.
\end{align}
Therefore, $Q'(y) = \|g_0\||\pi_V(y)|^2$ for all $y \in \Hi$ and by the formula $2 Q'(y,y') = Q'(y+y') - Q'(y) - Q'(y')$ we get
\[
Q'(y,y') = \|g_0\|\langle \pi_V (y), \pi_V (y') \rangle = \langle y, \|g_0\|\pi_V (y') \rangle
\]
for all $y,y' \in \Hi$. Hence, $L_{Q'} = \|g_0\|\pi_V$. We show next that $Q'$ is close to $Q(T,0,1)$. 

By assumption $\spt(T \res Z_V(\sqrt{1 - \rho^2})) \subset \spt(T \res \B(0,1))$ and hence the constancy theorem implies $(\pi_{V\#} (T \res \B(0,1))) \res \B_V(0,\sqrt{1-\rho^2}) = g_0 \curr{\B_V(0,\sqrt{1-\rho^2})}$. Thus with the smallness of the excess,
\[
\|g_0\|\balpha(m)(1 - \rho^2)^\frac{m}{2} \leq \bM(T \res \B(0,1)) \leq \|g_0\|(\balpha(m) + \epsilon) \,.
\]
As before let $M_T \subset \spt(T) \cap Z_V$ be a $\sigma$-compact $\cH^m$-rectifiable set on which $\|T \res Z_V\|$ is concentrated, and $\bg : M_T \to G$ is a measurable function characterizing $T$.
For $y \in \B(0,1)$ write the orthogonal decomposition with respect to $V$ as $y = y_v + y_v^\perp$. The area formula leads to
\begin{align*}
& \int_{\B(0,1)} \langle y,z \rangle^2 \, d\|T\|(z) = \int_{\B(0,1) \cap M_T} \|\bg(z)\| \langle y,z \rangle^2 \, d\cH^m(z) \\
 & \qquad\qquad \geq \int_{\B(0,1) \cap M_T} \|\bg(z)\| \langle y,z \rangle^2 J(\pi_V|M_T)(z) \, d\cH^m(z) \\
 & \qquad\qquad \geq \int_{\B_V(0,\sqrt{1-\rho^2})} \sum_{z \in \pi_V^{-1}(z') \cap M_T} \|\bg(z)\| \langle y,z \rangle^2 \, dz' \\
 & \qquad\qquad = \int_{\B_V(0,\sqrt{1-\rho^2})} \sum_{z \in \pi_V^{-1}(z') \cap M_T} \|\bg(z)\| (\langle y_v,z_v \rangle + \langle y_v^\perp,z_v^\perp \rangle)^2 \, dz' \\
 & \qquad\qquad \geq \int_{\B_V(0,\sqrt{1-\rho^2})} \left( \|g_0\|\langle y_v,z' \rangle^2 - 2 \rho\sum_{z \in \pi_V^{-1}(z') \cap M_T} \|\bg(z)\|  \right)\, dz' \,.
\end{align*}
Similarly,
\begin{align*}
\int_{\B_V(0,1)} \sum_{z \in \pi_V^{-1}(z') \cap M_T} \|\bg(z)\| \, dz' & = \int_{M_T} \|\bg(z)\| J(\pi_V|M_T) \, d\cH^m(z) \\
 & \leq \int_{M_T} \|\bg(z)\| \, d\cH^m(z) \\
 & = \bM(T \res Z_V) \\
 & \leq \|g_0\|(\epsilon + \balpha(m)) \,.
\end{align*}
With \eqref{flatbeta},
\begin{align*}
Q(T,0,1)(y) & = \frac{m+2}{\balpha(m)}\int_{\B(0,1)} \langle y,z \rangle^2 \, d\|T\|(z) \\
 & \geq \frac{m+2}{\balpha(m)} \|g_0\|\left[ \int_{\B_V(0,\sqrt{1-\rho^2})} \langle y_v,z' \rangle^2 \, dz' - (\epsilon + \balpha(m))2\rho \right] \\
 & = \|g_0\|(1-\rho^2)^\frac{m+2}{2}|y_v|^2 - \frac{m+2}{\balpha(m)} \|g_0\|(\epsilon + \balpha(m))2\rho \\
 & \geq \|g_0\||\pi_V(y)|^2 - \bc\|g_0\|\rho \,,
\end{align*}
for some $\bc(m) > 0$. To obtain an upper bound, a similar calculation shows,
\begin{align}
\nonumber
\int_{\B(0,1)} \langle y,z \rangle^2 \, d\|T\|(z) & \leq \int_{M_T} \|\bg(z)\|\langle y,z \rangle^2 \, d\cH^m(z) \\
\nonumber
 & = \int_{M_T} \|\bg(z)\|(\langle y_v,z_v \rangle + \langle y_v^\perp,z_v^\perp \rangle)^2 \, d\cH^m(z) \\
\label{upperboundonbeta}
 & \leq \int_{M_T} \|\bg(z)\|(\langle y_v,z_v \rangle^2 + 2\rho + \rho^2) \, d\cH^m(z) \,.
\end{align}
Lemma~\ref{multiple_values_estimate} and Lemma~\ref{sizebound} imply that $\cH^m(E_2) \leq 10 \epsilon$ where $E_2 \subset \B_V(0,1)$ is the set of those $x$ for which $M_x = \pi_V^{-1}\{x\} \cap M_T$ contains more than one point. By the first part of Lemma~\ref{multiple_values_mass_estimate} we can find a Borel set $E_1$ such that $M_x = \{y_x\}$ and $\|\bg(y_x)\| = \|g_0\|$ for $x \in E_1$ and $E_1 \cup E_2$ has full measure in $\B_V(0,1)$. We can find countable many pairwise disjoint compact sets $K_i \subset {\pi_V}^{-1}(E_1) \cap M_T$ such that $\pi_V : K_i \to K_i' \defl \pi_V(K_i) \subset \B_V(0,1)$ is bi-Lipschitz with inverse $\varphi_i$ such that $\sup J\varphi_i \leq 1 + \epsilon$ and the remaining set $M' \defl M_T \cap {\pi_V}^{-1}(E_1) \setminus \cup_i K_i$ satisfies $\cH^m(\pi_V(M')) = 0$. By replacing $E_1$ with $\cup_i K_i'$ if necessary, we can assume that $M' = \emptyset$. The compact sets $K_i'$ are pairwise disjoint since every $x \in E_1$ has only one preimage in $M_T$. Using the additivity of the excess,
\begin{align*}
\|g_0\|\epsilon \geq \Exc(T,V,E_1) & = \sum_i \int_{K_i} \|\bg(z)\| \, d\cH^m(z) - \int_{K_i'} \|g_0\| \, dz' \\
 & = \|g_0\|\sum_i \int_{K_i'} \left( J\varphi_i(z') - 1 \right) \, dz' \\
 & \geq 4^{-1} \|g_0\|\sum_i \int_{K_i'} (\langle y_v,z' \rangle^2 + 3\rho)(J\varphi_i(z') - 1) \, dz' \,.
\end{align*}
Similarly,
\begin{equation*}
\label{restset}
\bM(T \res {\pi_V}^{-1}(\B_V(0,1) \setminus E_1)) \leq \|g_0\|(\cH^m(E_2) + \epsilon) \leq 11\|g_0\|\epsilon \,.
\end{equation*}
Combining these estimates with \eqref{upperboundonbeta} leads to
\begin{align*}
\int_{\B(0,1)} \langle y,z \rangle^2 \, d\|T\|(z) & \leq \int_{M_T} \|\bg(z)\|(\langle y_v,z_v \rangle^2 + 2\rho + \rho^2) \, d\cH^m(z) \\
 & \leq \int_{{\pi_V}^{-1}(E_1) \cap M_T} \|\bg(z)\|(\langle y_v,z_v \rangle^2 + 3\rho) \, d\cH^m(z) \\
 & \quad + \bM(T \res {\pi_V}^{-1}(\B_V(0,1) \setminus E_1))(1 + 3\rho) \\
 & \leq \|g_0\| \sum_i \int_{K_i'} (\langle y_v,z' \rangle^2 + 3\rho)J\varphi_i(z') \, dz' + 44\|g_0\|\epsilon \\
 & \leq \|g_0\| \sum_i \int_{K_i'} \left( \langle y_v,z' \rangle^2 + 3\rho\right)  \, dz' + 48\|g_0\|\epsilon \\
 & \leq \|g_0\| \int_{\B_V(0,1)} \langle y_v,z' \rangle^2 \, dz' + 3\|g_0\|\balpha(m)\rho + 48\|g_0\|\epsilon \,.
\end{align*}
Hence, $Q(T,0,1)(y) \leq Q'(y) + \bc\|g_0\|(\rho + \epsilon)$ for some $\bc(m) > 0$. And with the lower bound we obtain $|Q(T,0,1)(y) - Q'(y)| \leq \bc\|g_0\|(\rho + \epsilon)|y|^2$ for all $y \in \Hi$. It follows directly that
\[
|\langle y', (L_{Q(T,0,1)} - L_{Q'}) (y) \rangle| = |Q(T,0,1)(y',y) - Q'(y',y)| \leq \bc \|g_0\|(\rho + \epsilon) \,,
\]
for all $y,y' \in \B(0,1)$. If we set $y' = (L_{Q(T,0,1)} - L_{Q'}) (y)/|(L_{Q(T,0,1)} - L_{Q'}) (y)|$ we get the result about the projection.

Next we estimate the trace of the quadratic form. By the monotone convergence theorem,
\begin{align*}
\tr{Q(T,0,1)} & = \sum_{i \geq 1} Q(T,0,1)(e_i) = \frac{m+2}{\balpha(m)} \int_{\B(0,1)} \sum_{i \geq 1} \langle z, e_i \rangle^2 \, d\|T\|(z) \\
 & = \frac{m+2}{\balpha(m)} \int_{\B(0,1)} |z|^2 \, d\|T\|(z) \,.
\end{align*}
With similar calculations as before,
\begin{align*}
\int_{\B(0,1)} |z|^2 \, d\|T\|(z) & \leq \int_{M_T} \|\bg(z)\|(|z_v|^2 + \rho^2) \, d\cH^m(z) \\
 & \leq \int_{{\pi_V}^{-1}(E_1) \cap M_T} \|\bg(z)\|(|z_v|^2 + \rho^2) \, d\cH^m(z) \\
 & \quad + \bM(T \res {\pi_V}^{-1}(\B_V(0,1) \setminus E_1))(1 + \rho^2) \\
 & \leq \|g_0\| \sum_i \int_{K_i'} (|z'|^2 + \rho^2)J\varphi_i(z') \, dz' + 22\|g_0\|\epsilon \\
 & \leq \|g_0\| \sum_i \int_{K_i'} |z'|^2 + \rho^2 \, dz' + 24\|g_0\|\epsilon \\
 & \leq \|g_0\| \int_{\B_V(0,1)} |z'|^2 \, dz' + \balpha(m)\|g_0\|\rho^2 + 24\|g_0\|\epsilon \,.
\end{align*}
Since $\tr{Q'} = m\|g_0\|$, the result follows.
\end{proof}

It is easy to check that the quadratic forms $Q(T,x,r)$ are compact. Without loss of generality assume that $x = 0$ and $r = 1$. If $(y_i)$ is a sequence in $\Hi$ converging weakly to $y$, then there is a finite constant $C > 0$ such that $|y_i| \leq C$ for all $i$. For any $z \in \B(0,1)$ we get $\langle y_i, z\rangle^2 \leq C^2$. So, by Lebesgue's Dominated Convergence Theorem and using the fact that $\|T\|$ is a finite measure,
\begin{align*}
\lim_{i \to \infty} Q(T,0,1)(y_i) & = \lim_{i \to \infty} \frac{m+2}{\balpha(m)}\int_{\B(0,1)} \langle y_i, z \rangle^2 \, d\|T\|(z) \\
 & = \frac{m+2}{\balpha(m)}\int_{\B(0,1)} \lim_{i \to \infty} \langle y_i, z \rangle^2 \, d\|T\|(z) \\
 & = \frac{m+2}{\balpha(m)}\int_{\B(0,1)} \langle y, z \rangle^2 \, d\|T\|(z) = Q(T,x,r)(y) \,.
\end{align*}
It follows from the spectral theorem that there is an orthonormal basis $(e_k) \subset \Hi$ of eigenvectors of $L_{Q(T,x,r)}$. The only accumulation point of the set of eigenvalues $(\lambda_k)$ corresponding to $(e_k)$ is 0 and we can assume that this sequence is ordered as $\lambda_1 \geq \lambda_2 \geq \cdots \geq 0$. Let $W_{x,r} \in \bG(\Hi,m)$ be the $m$-plane spanned by $e_1, \dots, e_m$.

\begin{Lem}
\label{newplaneclosetoold}
Assume $T$ satisfies the same assumptions as in Lemma~\ref{closeplanes} for $x \in \Hi$, $r > 0$ and an $m$-plane $V\in \bG(\Hi,m)$. If $\bc_{\ref{closeplanes}}(\rho + \epsilon) < 1$, then $|\lambda_i - \|g_0\|| \leq \bc_{\ref{closeplanes}}\|g_0\|(\rho + \epsilon)$ for $i = 1, \dots, m$ and $W_{x,r}$ is close to $V$ in the sense that 
\[
\left\|\pi_{W_{x,r}} - \pi_V\right\| \leq 2\sqrt{m}\bc_{\ref{closeplanes}} (\rho + \epsilon) \,.
\]
\end{Lem}

\begin{proof}
Again we assume that $x = 0$ and $r = 1$. Let $v$ be a unit vector in $V$ orthogonal to $V' \defl \spa\{e_1, \dots, e_{m-1}\}$. Such a vector exists since orthogonal projection $\pi_{V'}: V \to V'$ has nontrivial kernel. We can write $v = \sum_{i \geq m} v_i e_i$ with $1 = |v|^2 = \sum_{i \geq m} v_i^2$. By Lemma~\ref{closeplanes} we get $\|g_0\||v| - |L_{Q(T,x,r)}(v)| \leq \bc_{\ref{closeplanes}}\|g_0\| (\rho + \epsilon)|v|$ and further,
\[
\|g_0\|^2(1 - \bc_{\ref{closeplanes}}(\rho + \epsilon))^2 \leq |L_{Q(T,x,r)}(v)|^2 = \sum_{i \geq m} |\lambda_i v_i|^2 \leq \lambda_m^2 \,.
\]
Hence, $\|g_0\|(1 - \bc_{\ref{closeplanes}}(\rho + \epsilon)) \leq \lambda_m \leq \dots \leq \lambda_1$. To obtain the upper bound we use again Lemma~\ref{closeplanes},
\begin{align*}
\lambda_1 & \leq |\|g_0\|\pi_V(e_1) - \lambda_1 e_1| + \|g_0\||\pi_V (e_1)| \\
 & \leq |\|g_0\|\pi_V(e_1) - L_{Q(T,x,r)}(e_1)| + \|g_0\| \\
 & \leq \bc_{\ref{closeplanes}}\|g_0\|(\rho + \epsilon) + \|g_0\| \,.
\end{align*}
This shows the first estimate of the lemma. For $i = 1, \dots, m$, this implies that
\begin{align*}
\|g_0\||e_i - \pi_V(e_i)| & \leq |\|g_0\|e_i - \lambda_i e_i| + |L_{Q(T,x,r)}(e_i) - \|g_0\|\pi_V(e_i)| \\
 & \leq 2 \bc_{\ref{closeplanes}}\|g_0\|(\rho + \epsilon) \,,
\end{align*}
and therefore $|w - \pi_V(w)| \leq 2\sqrt{m}c|w|$ for all $w \in W_{x,r}$. With Lemma~\ref{lemma.norms.1}, the proof is finished.
\end{proof}

\section{Reifenberg's epiperimetric inequality for a polyhedral cone}
\label{section.epiperimetric}

This section essentially contains the results covered by Reifenberg in \cite{Reif}. In the setting of manifolds they can also be found in the book \cite{Mor} of Morrey. The proofs here often contain a bit more detail, and where appropriate we use the tools for chains such as slicing and push-forwards to clarify the arguments. The main object in this section is a polyhedral cone $P \in \cP_m(\Hi;G)$ in a Hilbert space $\Hi$ with coefficients in a normed Abelian group $G$. We assume that $\spt(P)$ has finite diameter but we could as well work with the infinite cone generated by $P$. The main statement of this section is Theorem~\ref{masscomparison2} which essentially states that if $P$ is close to an $m$-plane $V \in \bG(\Hi,m)$, then there is a comparison surface $S \in \cP_m(\Hi;G)$ with $\partial (S \res \B(0,1)) = \partial (P \res \B(0,1))$ and
\[
\bM(S \res \B(0,1)) - \|g_0\|\balpha(m) \leq \lambda(\bM(P \res \B(0,1)) - \|g_0\|\balpha(m)) \,,
\]
for some constant $0 < \lambda(m) < 1$ and some $g_0 \in G$ representing the group element of the projection from $P$ to $V$. We will apply this theorem in Proposition~\ref{rundifferentialequation} to a rectifiable chain by approximating it with polyhedral chains which is justified by \cite{DePauw2}.

\subsection{Assumptions and preliminaries}
\label{assumptions}

Until further notice we assume that $0 \leq \rho< \frac{1}{2}$, $0\leq\epsilon < \frac{1}{2}$ and $V \in \bG(\Hi,m)$ is an $m$-plane with:
\begin{enumerate}
	\item \label{general_pos} $P$ is in general position with respect to $V$, i.e.\ $P$ can be written as a finite formal sum $\sum_i g_i \curr{S_i}$ of oriented $m$-simplices $S_i$ and $g_i \in G$ such that $\pi_V$ restricted to $S_i$ is one-to-one and orientation preserving.
	\item \label{boundary} $\spt(\partial P) \cap Z_V(2) = \emptyset$.
	\item \label{homogeneous} $P \res Z_V$ is $1$-homogeneous in its domain, i.e.\ $\eta_{r\#} (P \res Z_V) = P \res Z_V(r)$ where $\eta_r(p) = rp$ is the scaling by $r \in [0,2]$ in $\Hi$.
	\item \label{projection} $\pi_{V\#} (P \res Z_V) = g_0 \curr{\B_V(0,1)}$ for some $g_0 \in G \setminus \{0_G\}$.
	\item \label{lowerbound} $\|g_i\| \geq \frac{3}{4}\|g_0\|$.
	\item \label{mass_ball} $\Exc_1(P,V) \leq \|g_0\|\epsilon$.
	\item \label{dist_ball} $|\pi_{V^\perp}(x)| \leq \rho$ for all $x \in \spt(P) \cap Z_V$.
	\item $\epsilon \leq \rho^{6m}$.
\end{enumerate}
Combining \eqref{projection} and \eqref{mass_ball} we get
\[
\bM(P \res Z_V) - \|g_0\|\balpha(m) \leq \|g_0\|\epsilon \,.
\]
Similarly, it follows directly from Lemma~\ref{sizebound} that
\begin{equation}
\label{sizeboundprelim}
\Exc_1(P,V,\bS) = \cH^m(\spt(P) \cap Z_V) - \balpha(m) \leq 5\epsilon \,.
\end{equation}
For the most part of this section we could replace \eqref{lowerbound} by a bound like $\Exc_1(P,V,\bS) \leq \epsilon$. The properties for $P$ stated above are scaling invariant with respect to $\|g_0\|$ and by replacing the group norm $\|\cdot\|$ with $\|\cdot\|/\|g_0\|$ if necessary we may assume that $\|g_0\| = 1$.

By the general position assumption \eqref{general_pos} we can write $P = \sum_i g_i \curr{S_i}$ for finitely many almost disjoint oriented $m$-simplices $S_i$ and $g_i \neq 0_G$. Because $P$ is $1$-homogeneous we can assume that $S_i$ is the convex hull of $S_i' \cup \{0\}$ where $S_i'$ is some $(m-1)$-simplex in $\Hi \setminus Z_V(2)$. It follows that $\spt(P) = \bigcup_i S_i$ and therefore the set $P_x \defl \pi_V^{-1}(x) \cap \spt(P)$ is finite for all $x \in \B_V(0,1)$. Any such simplex $S_i$ can be uniquely expressed as the graph of an affine map $\by^i : \pi_V(S_i) \to V^\perp$. Let $I_x$ be the collection of all $i$ for which $x \in \pi_V(S_i)$. Note that for almost every $x \in \B_V(0,1)$, $x + \by^i(x)$ is in the interior of $S_i$ for every $i \in I_x$. Further, for almost every $x \in \B_V(0,1)$,
\[
\sum_{i \in I_x} g_i = g_0 \,,
\]
since we assume that $\pi_V : S_i \to V$ is orientation preserving. The map $\by^i$ is defined in a neighborhood of $x$ if $x + \by^i(x)$ is in the interior of $S_i$. As in Subsections~\ref{excess_section} and ~\ref{betaprelim} let $E_1$ be the points in $\B_V(0,1)$ where $\pi_V$ has only one preimage in $\spt(P)$ lying in the interior of some simplex $S_i$ and $E_2 = \B_V(0,1) \setminus E_1$ be the complement. For a point $x \in E_1$ we write $\by(x)$ for the only element in $V^\perp$ with $x + \by(x) \in \spt(P)$. Some immediate consequences of Lemma~\ref{multiple_values_estimate}, Lemma~\ref{multiple_values_mass_estimate} and \eqref{sizeboundprelim} are
\begin{equation}
\label{multiple_values_estimate_prelim}
\cH^m(E_2) \leq \int_{E_2} \# I_x \, dx  \leq 10 \epsilon \,,
\end{equation}
\begin{equation}
\label{multiple_values_mass_estimate_prelim}
\int_{E_2} \sum_{i \in I_x} \|g_i\| \, dx  \leq 15 \|g_0\|\epsilon \,.
\end{equation}

Further, by our assumptions on $P$ and the area formula, there holds for any Borel set $B \subset V$,
\[
\cH^m(\spt(P) \cap \pi_V^{-1}(B)) = \sum_i \cH^m(S_i \cap \pi_V^{-1}(B)) = \int_B \sum_{i \in I_x} (1 + (\J \by^i_x)^2)^\frac{1}{2} \, dx \,,
\]
and
\[
\bM\left(P \res \pi_V^{-1}(B)\right) = \sum_i \bM\left(g_i \curr{S_i} \res \pi_V^{-1}(B)\right) = \int_B \sum_{i \in I_x} \|g_i\|\left(1 + (\J \by^i_x)^2\right)^\frac{1}{2} \, dx \,.
\]
We know that $\langle P, \pi_V, x \rangle$ is in $\cP_0(\Hi;G)$ for almost every $x \in V$ and because of the observation above we get that
\[
\langle P, \pi_V, x \rangle = \sum_{i \in I_x} g_i \curr{x + \by^i(x)} \,.
\]
As in \cite{Reif} we often use orthogonal coordinates $x \in V$ and $y \in V^\perp$.

 We define for any $x \in \B_V(0,1)$,
\begin{equation*}
\|D\by\|(x) \defl\sum_{i \in I_x} \|D\by^i_x\| \,,
\end{equation*}
and similarly $\|D\by\|_{\HS}(x)$ is defined.

Here are some basic estimates we will need later on.
\begin{Lem}
[{\cite[Lemma~1]{Reif}}]
\label{squares}
Let $a,b \geq 0$, $a_1, \dots, a_k \geq 0$, $\lambda_1, \dots, \lambda_k \geq 0$ with $\lambda_1 + \dots + \lambda_k = 1$ and $\delta \in (0,1]$. Then
\begin{align}
\label{squaresa}
\left(1 + a^2 + b^2\right)^{\frac{1}{2}} & \leq  \left(1 + a^2\right)^{\frac{1}{2}} + \left(\left(1 + b^2\right)^{\frac{1}{2}} - 1\right) \,, \\
\label{convex_comb}
\left(1 + (\lambda_1 a_1 + \dots +  \lambda_ka_k)^2\right)^{\frac{1}{2}} & \leq \lambda_1\left(1 + a_1^2\right)^{\frac{1}{2}} + \dots + \lambda_k\left(1 + a_k^2\right)^{\frac{1}{2}} \,, \\
\label{squaresc}
\left(1 + \delta^2 a^2\right)^{\frac{1}{2}} - 1 & \leq \delta\left(\left(1 + a^2\right)^{\frac{1}{2}} - 1\right) \,, \\
\label{squaresd}
(1 + ab)^{\frac{1}{2}} - 1 & \leq \frac{1}{2} \delta^{-2}b^2 + \delta\left(\left(1 + a^2\right)^{\frac{1}{2}} - 1\right) \,.
\end{align}
\end{Lem}

\begin{proof}
Abbreviating $\alpha^2 = 1 + a^2$, $\beta^2 = 1 + b^2$ and squaring both sides of \eqref{squaresa} we notice it is equivalent to $\alpha + \beta \leq 1 + \alpha \beta$, which readily holds as $\alpha \geq 1$ and $\beta \geq 1$.
Next, introducing $h(t) = \sqrt{1+t^2}$, $t \geq 0$, we notice \eqref{convex_comb} expresses the convexity of $h$. Letting $g(s) = s^{-1}\int_0^s h' = s^{-1}(h(s)-h(0))$ we note \eqref{squaresc} is equivalent to $g$ being nondecreasing. Computing $g'$ this in turn is equivalent to $s^{-1}\int_0^s h' \leq h'(s)$, a consequence of the fact that $h'$ is nondecreasing. In order to establish \eqref{squaresd} we notice $h$ is nondecreasing, convex, $h(t) \leq 1 + \frac{t^2}{2}$ and we apply \eqref{squaresc} as follows : $h(\sqrt{ab}) -1 \leq h \left( \frac{1}{2} \frac{b}{\delta} + \frac{1}{2} \delta a \right) - 1 \leq \frac{1}{2} \left( h \left( \frac{b}{\delta} \right) - 1 \right) + \frac{1}{2} \left( h(\delta a) - 1 \right) \leq \frac{1}{4} \frac{b^2}{\delta^2} + \frac{1}{2} \delta \left( h(\delta a) - 1 \right)$.
\end{proof}

Now, let $W$ be a $k$-dimensional subspace of $V$ and $W^\perp$ its orthogonal complement in $V$. Then for all $w' \in W^\perp$ the slice $\langle P, \pi_{W^\perp}, w'\rangle$ is defined and an element of $\cP_{k}(\Hi;G)$. Writing $x = w + w' \in W + W^\perp$ for some $x \in \B_V(0,1)$, we have by iterated slicing and the assumption on $P$,
\[
\langle\langle P, \pi_{W^\perp}, w'\rangle, \pi_W, w\rangle = \langle P, (\pi_{W^\perp},\pi_W), (w',w)\rangle = \langle P, \pi_V, x \rangle  \,.
\]
So, for some fixed $w' \in W^\perp \cap \B_V(0,1)$ the slice $P' \defl \langle P, \pi_{W^\perp}, w'\rangle$ satisfies $\chi \langle P', \pi_W, w\rangle = \chi \langle P, \pi_V, x \rangle = g_0$ (where $\chi$ denotes the augmentation map) for all $w \in W$ with $w + w' \in \B_V(0,1)$. Further,
\[
\partial P' = (-1)^{m-k}\langle \partial P, \pi_{W^\perp}, w'\rangle \,,
\]
and hence
\[
\spt(\partial P') \cap Z_V = \spt(\langle \partial P, \pi_{W^\perp}, w'\rangle) \cap Z_V \subset \spt(\partial P) \cap Z_V = \emptyset \,.
\]
As a result $\spt(\partial (P' \res Z_V)) \subset \partial Z_V$.

\subsection{Mass estimate of the averaged cone}
\label{massest}

Let $[x_1,x_2]$ be the straight line segment connecting two different points $x_1,x_2 \in \B_V(0,1)$. Let $W$ be the oriented $(m-1)$-dimensional subspace of $V$ orthogonal to $x_2 - x_1$. By the general position assumption the set $\spt(P) \cap \pi_V^{-1}([x_1,x_2])$ is a union of finitely many straight lines each lying in some simplex $S_i$. We define the truncated slice,
\[
P(x_1,x_2) \defl \langle P, \pi_W, w\rangle \res \pi_V^{-1}([x_1,x_2]) \,,
\]
where $w \in W \cap \B_V(0,1)$ is the vector with $\pi_W(x_2)=\pi_W(x_1) = w$. By the observations above and some general facts about slices, $P(x_1,x_2) \in \cP_1(\Hi;G)$, $\pi_{V\#} P(x_1,x_2) = g_0 \curr{x_1,x_2}$ (given the right orientation on $W$) and hence $\pi_{V\#} \partial P(x_1,x_2) = g_0 \curr{x_2} - g_0 \curr{x_1}$.

\begin{Lem}
[{\cite[Lemma~2]{Reif}}]
\label{distance_estimate}
If $x_1,x_2 \in E_1$, then
\[
|\by(x_1) - \by(x_2)| \leq \int_{[x_1,x_2]} \|D\by\|(x) \, d\cH^1(x) \,.
\]
\end{Lem}

\begin{proof}
As noted above, $P(x_1,x_2)$ is a well-defined non-zero polyhedral $1$-chain with support in $\spt(P) \cap \pi_V^{-1}([x_1,x_2])$ and $\partial P(x_1,x_2) = g_0 \curr {x_2 + \by(x_2)} - g_0 \curr {x_1 + \by(x_1)}$. We want to show that there is a curve in $\spt(P(x_1,x_2))$ connecting $x_1 + \by(x_1)$ and $x_2 + \by(x_2)$. As a polyhedral $1$-chain, $P(x_1,x_2)$ is a finite sum of loops $g_i \curr{L_i}$ and curves $g_j \curr{C_j}$ connecting $x_1 + \by(x_1)$ with $x_2 + \by(x_2)$. For any loop, $\pi_{V\#}(g_i \curr {L_i}) = 0$ by the constancy theorem, since $\partial (g_i \curr {L_i}) = 0$. Since $\pi_{V\#} \sum_j g_j \curr{C_j} = g_0 \curr{x_1,x_2} \neq 0$, the set of $C_j$'s is not empty.

Thus there is a piecewise linear curve $\gamma : [0,L] \to \spt(P(x_1,x_2))$ connecting $x_1 + \by(x_1)$ and $x_2 + \by(x_2)$ that is injective and satisfies $\|\gamma'\| = 1$ almost everywhere. By the area formula
\[
\int_0^L f(t)|(\pi_V \circ \gamma)'(t)| \, dt = \int_{[x_1,x_2]} \sum_{t \in (\pi_V \circ \gamma)^{-1}(x)} f(t) \, d\cH^1(x) \,,
\]
for all measurable functions $f$. Set $f(t) \defl |(\pi_V \circ \gamma)'(t)|^{-1}$. If $\gamma(t)$ is in some simplex $S_{i(t)}$ in a neighborhood of $t$, then by \eqref{norms}
\[
f(t) = |\pi_V(\gamma'(t))|^{-1} \leq \Lip\bigl(\by^{i(t)}\bigr) = \bigl\|D\by^{i(t)}\bigr\| \,.
\]
Finally, 
\begin{align*}
|\by(x_1) - \by(x_2)| \leq L & = \int_0^L f(t)|(\pi_V \circ \gamma)'(t)| \, dt \\
 & \leq \int_{[x_1,x_2]} \sum_{t \in (\pi_V \circ \gamma)^{-1}(x)} \bigl\|D\by^{i(t)}_x\bigr\| \, d\cH^1(x) \\
 & \leq \int_{[x_1,x_2]} \sum_{i \in I_x}\bigl\|D\by^i_x\bigr\| \, d\cH^1(x) \,,
\end{align*}
and by definition this is exactly the integral of $\|D\by\|$.
\end{proof}

\begin{Lem}
[{\cite[Lemma~3]{Reif}}]
\label{homogen}
There is a constant $\bc_{\theThm}(m) > 0$ with the following property. If $f$ is a measurable nonnegative homogeneous function of degree $0$ on $V$ (i.e.\ $f(tx) = f(x)$ for all $t > 0$), then
\[
\int_{\B_V(0,1)} dx \int_{\B_V(x,\rho|x|)} dx' \int_{[x,x']} dz \, \rho^{-m}|x|^{-m} f(z) \leq \bc_{\theThm} \rho \int_{\B_V(0,1)} f(x) \, dx \,.
\]
\end{Lem}

\begin{proof}
Let
\begin{align*}
Z &\defl \{|x| \leq 1\} \times \{|x'| \leq 1\} \times \{\lambda \in [0,\rho]\} \subset V^2 \times \R \,, \\
Z' &\defl \{(x,x',x'') \in V^3 : |x| \leq 1, |x-x'| \leq \rho|x|, x'' \in [x,x']\} \subset V^3 \,,
\end{align*}
and define $\varphi : Z \to Z'$ by $\varphi(x,x',\lambda) \defl (x,x + \rho |x|x',x + \lambda |x|x')$ and $g : Z \to \R$ by $g(x,x',\lambda) \defl f(x + \lambda|x|x')$. $\varphi$ is bi-Lipschitz onto its image and $J\varphi_{(x,x',\lambda)} = \rho^m |x|^{m+1}|x'| \leq \rho^m |x|^{m}$ holds almost everywhere on $Z$. By the area formula,
\begin{align*}
\int_{|x| \leq 1} \int_{|x-x'| \leq \rho |x|} \int_{z \in [x,x']} \rho^{-m}|x|^{-m} f(z) & \leq \int_{|x| \leq 1} \int_{|x'| \leq 1} \int_{\lambda \in [0,\rho]} g(x,x',\lambda) \\
 & = \int_{|x'| \leq 1} \int_{\lambda \in [0,\rho]} \int_{|x| \leq 1} g(x,x',\lambda) \\
 & \leq \balpha(m) \rho \max_{|x'|\leq 1, \lambda \in [0,\rho]} \int_{|x| \leq 1} f(x + \lambda|x|x') \,.
\end{align*}
By assumption $\lambda \leq \rho < \frac{1}{2}$. For fixed $\lambda$ and $x'$, the map $\psi: x \mapsto x + \lambda|x|x'$ is a bijection of $\B_V(0,1)$ onto a subset of $\B_V(0,2)$ with
\begin{equation}
\label{psi}
 D \psi_x(v) = v + \frac{1}{|x|}\langle x, v\rangle \lambda x' \,.
\end{equation}
Hence $\det D\psi_x = 1 + \lambda \tfrac{1}{|x|} \langle x, x' \rangle \in [\tfrac{1}{2},\tfrac{3}{2}]$ since we assume that $\rho \leq \tfrac{1}{2}$. Therefore,
\begin{align*}
\int_{|x| \leq 1} f(x + \lambda|x|x') & = \int_{\B_V(0,1)} f(\psi(x))\, dx \\
 & = \int_{\psi(\B_V(0,1))} f(x) (\det D\psi_x)^{-1} \, dx \\
 & \leq 2 \int_{\B_V(0,2)} f(x) \, dx \,.
\end{align*}
Finally,
\[
\int_{\B_V(0,2)} f(x) \, dx = \int_{\B_V(0,1)} 2^{m}f(2x) \, dx = \int_{\B_V(0,1)} 2^mf(x)\, dx\,,
\]
and the statement follows.
\end{proof}

The next estimate and its proof is directly taken from \cite{Reif}.

\begin{Lem}
[{\cite[Lemma~4]{Reif}}]
\label{two_terms}
Let $\mu$ be a finite measure on $E$ and $f \in L^\infty(E,\mu)$ a positive function. Then
\[
\min\left\{\left(\int_E f \, d\mu\right)^2 , \left(\int_E 1 \, d\mu\right)^2 \right\} \leq 3\left(\int_E 1 \, d\mu\right) \left(\int_E (1 + f^2)^\frac{1}{2} - 1 \, d\mu \right) \,.
\]
\end{Lem}

\begin{proof}
The Cauchy-Schwarz inequality implies that
\begin{align*}
\left(\int_E f \, d\mu\right)^2 & = \left(\int_E \frac{f}{((1 + f^2)^{\frac{1}{2}} + 1)^{\frac{1}{2}}}((1 + f^2)^{\frac{1}{2}} + 1)^{\frac{1}{2}} \, d\mu\right)^2 \\
 & \leq \left(\int_E \frac{f^2}{(1 + f^2)^{\frac{1}{2}} + 1} \, d\mu\right)\left(\int_E (1 + f^2)^{\frac{1}{2}} + 1 \, d\mu\right) \\
 & = \left(\int_E (1 + f^2)^{\frac{1}{2}} - 1 \, d\mu\right)\left(\int_E 2 + ((1 + f^2)^{\frac{1}{2}} - 1) \, d\mu\right) \,.
\end{align*}
If $\int_E (1 + f^2)^{\frac{1}{2}} - 1 \, d\mu \leq \int_E 1 \, d\mu$, it follows that
\begin{align*}
\left(\int_E f \, d\mu\right)^2 & \leq \left(\int_E (1 + f^2)^{\frac{1}{2}} - 1 \, d\mu\right)\left(\int_E 3 \, d\mu\right) \,,
\end{align*}
and otherwise,
\begin{align*}
\left(\int_E 1 \, d\mu\right)^2 & \leq \left(\int_E 1 \, d\mu\right) \left(\int_E (1 + f^2)^\frac{1}{2} - 1 \, d\mu \right) \,.
\end{align*}
\end{proof}

We define an average function $\bar \by : \B_V(0,2) \to V^\perp$ by
\[
\bar \by(x) \defl \sum_{i \in I_x} \delta_i \by^i(x) \defl \Bigl(\sum_{i \in I_x} \|g_i\|\Bigr)^{-1} \sum_{i \in I_x}  \|g_i\| \by^i(x) \,.
\]
Note that for $\cH^m$-a.e.\ $x \in \B_V(0,1)$ it is $\|g_0\| \leq \sum_{i \in I_x} \|g_i\|$. Obviously, $|\bar \by(x)| \leq \rho$ for $\cH^m$-a.e.\ $x \in \B_V(0,1)$ and for $x \in E_1$ we have $\bar \by(x) = \by(x)$. A suitable smoothed version of $\by$ is obtained by the map $\bv : \B_V(0,1) \to V^\perp$ with
\[
\bv(x) \defl (\balpha(m) \rho^m |x|^m)^{-1} \int_{\B_V(x,\rho |x|)} \bar \by(z) \, dz \,.
\]
Because $P$ is $1$-homogeneous, so are $\bar \by$ and $\bv$. Here are first estimates of $\bv$.

\begin{Lem}
[{\cite[Lemma~7]{Reif}}]
\label{dist_to_v}
For all $x \in \B_V(0,1)$,
\[
|\bv(x)| \leq 2\rho \,,
\]
and there is a constant $\bc_{\theThm}(m) > 0$ such that
\[
\int_{\B_V(0,1)} \sum_{i \in I_x}|\bv(x) -  \by^i(x)|^2 \, dx \leq \bc_{\theThm} \rho^2 \epsilon \,.
\]
\end{Lem}

\begin{proof}
By the definition of $\bar \by$ and the assumption on the height bound of $P$ it is clear that $|\bar \by(x)| \leq \rho$ for $x \in \B_V(0,1)$. Fix some $x \in \B_V(0,\frac{1}{2})$. To obtain shorter formulas we abbreviate $B_{x,\rho} \defl \B_V(x,\rho|x|)$ and $|B_{x,\rho}| = \cH^m(B_{x,\rho})$. Clearly, $B_{x,\rho} \subset \B_V(0,1)$ and hence
\begin{align}
\nonumber
|\bv(x)| & \leq (\balpha(m) \rho^m |x|^m)^{-1} \int_{\B_V(x,\rho|x|)} |\bar \by(z)| \, dz \\
\label{bvbound}
 & \leq |B_{x,\rho}|^{-1} |B_{x,\rho}| \rho = \rho \,.
\end{align}
This immediately implies the first statement of the lemma. By \eqref{bvbound} and \eqref{multiple_values_estimate_prelim} we have on $E_2 \subset \B_V(0,1)$,
\begin{align}
\nonumber
\int_{E_2} \sum_{i \in I_x} |\bv(x) -  \by^i(x)|^2 \, dx & \leq \int_{E_2} \sum_{i \in I_x} 4\max\{|\bv(x)|,|\by^i(x)|\}^2 \, dx \,. \\
\nonumber
 & \leq 4 \rho^2 \int_{E_2} \# I_x \, dx \\
\label{e2bound}
 & \leq 40 \epsilon \rho^2.
\end{align}
Using that $E_2$ is a cone, it follows with \eqref{multiple_values_estimate_prelim} and the bounds $\sqrt\epsilon \leq \rho^{m} \leq 1$ that
\begin{align}
\nonumber
|B_{x,\rho}|^{-1}\cH^m(E_2 \cap \B_V(x,\rho|x|)) & \leq |B_{x,\rho}|^{-1}\cH^m(E_2 \cap \B_V(0,(1+\rho)|x|)) \\
\nonumber
 & \leq |B_{x,\rho}|^{-1}10\epsilon(1+\rho)^m|x|^m \\
\label{e2ballbound}
 & \leq 10 \cdot 2^{m}\balpha(m)^{-1} \sqrt{\epsilon} \,.
\end{align}
With \eqref{e2bound} and \eqref{e2ballbound} it follows that there is some $\bc(m) > 0$ with
\begin{align*}
& \int_{\B_V(0,\frac{1}{2})} \sum_{i \in I_x} |\bv(x) - \by^i(x)|^2 \, dx \\
& \qquad \leq \bc \epsilon \rho^2 + \int_{E_1 \cap \B_V(0,\frac{1}{2})} \sum_{i \in I_x} \Bigl[ |B_{x,\rho}|^{-1} \int_{B_{x,\rho}} \bar \by(x') \, dx' -  \by^i(x) \Bigr]^2 \, dx \\
 & \qquad \leq \bc \epsilon \rho^2 + \int_{E_1 \cap \B_V(0,\frac{1}{2})} \Bigl[ |B_{x,\rho}|^{-1} \int_{B_{x,\rho}} |\bar \by(x') - \by(x)| \, dx' \Bigr]^2 \,dx \\
& \qquad \leq \bc \epsilon \rho^2 + \int_{E_1 \cap \B_V(0,\frac{1}{2})} \Bigl[ \Bigl| |B_{x,\rho}|^{-1} \int_{E_1 \cap B_{x,\rho}} |\by(x') - \by(x)| \, dx' \Bigr| + \bc \sqrt{\epsilon} \rho \Bigr]^2 \, dx \,. \\
\end{align*}
Because $(a + b)^2 \leq 2(a^2 + b^2)$ we get that the expression above is bounded by
\begin{equation}
\label{term}
\bc' \epsilon \rho^2 + 2 \int_{E_1 \cap \B_V(0,\frac{1}{2})} \left[ |B_{x,\rho}|^{-1} \int_{E_1 \cap B_{x,\rho}} |\by(x') - \by(x)| \, dx' \right]^2 \, dx \,,
\end{equation}
for some $\bc'(m) > 0$. Since $|x'| \leq 2|x| \leq 1$ for $x' \in B_{x,\rho}$ and $\by$ is $1$-homogeneous,
\begin{align*}
\int_{E_1 \cap B_{x,\rho}} |\by(x') - \by(x)| \, dx' & \leq \int_{B_{x,\rho}}  3 \rho |x| \, dx' \leq 3\left( 1 + \frac{1}{m} \right) \int_{B_{x,\rho}} dx' \int_{[x,x']} 1 \,.
\end{align*}
By Lemma~\ref{distance_estimate},
\[
\int_{E_1 \cap B_{x,\rho}} |\by(x') - \by(x)| \, dx' \leq  \int_{B_{x,\rho}} \, dx' \int_{[x,x']}\|D\by\| \,,
\]
and hence the second term of \eqref{term} is bounded by
\begin{equation}
\label{term2}
\bc(m) \int_{E_1} \left[ |B_{x,\rho}|^{-1} \min\left\{\int_{B_{x,\rho}} dx' \int_{[x,x']} \|D\by\|, \int_{B_{x,\rho}} dx'\int_{[x,x']} 1 \right\} \right]^2 dx \,.
\end{equation}
Hence we can apply Lemma~\ref{two_terms} and \eqref{term2} is smaller or equal
\begin{align*}
& 3 \bc(m) \int_{E_1}|B_{x,\rho}|^{-2} \left(\int_{B_{x,\rho}}dx' \int_{[x,x']} 1 \right)\left(\int_{B_{x,\rho}}dx' \int_{[x,x']} (1+\|D\by\|^2)^\frac{1}{2} - 1 \right) \\
& \quad \leq 3 \bc(m) \int_{E_1}|B_{x,\rho}|^{-2}|B_{x,\rho}| \rho |x| \left(\int_{B_{x,\rho}} dx' \int_{[x,x']} (1+\|D\by\|^2)^\frac{1}{2} - 1 \right) \\
& \quad \leq 3 \bc(m) \int_{\B_V(0,1)} |B_{x,\rho}|^{-1} \rho |x| \int_{B_{x,\rho}} dx'\int_{[x,x']} (1+\|D\by\|^2)^\frac{1}{2} - 1 \,.
\end{align*}
Because $\|D\by\|$ is $0$-homogeneous, so is $(1+\|D\by\|^2)^\frac{1}{2} - 1$ and by Lemma~\ref{homogen} the term above is smaller than or equal to
\[
3 \bc(m) \balpha(m)^{-1}\bc_{\ref{homogen}} \rho^2 \int_{\B_V(0,1)} (1+\|D\by\|^2)^\frac{1}{2} - 1 \,.
\]
Using $(1 + (\sum_i a_i)^2)^\frac{1}{2} \leq \sum_i (1 + a_i^2)^\frac{1}{2}$, $\|D\by^i_x\| \leq \J \by^i_x$ for almost all $x$ and \eqref{sizeboundprelim} we see that
\begin{align*}
\int_{\B_V(0,1)} (1+\|D\by\|^2)^\frac{1}{2} - 1 & = \int_{\B_V(0,1)} \left(1 + \biggl(\, \sum_{i \in I_x}\|D\by^i_x\|\biggr)^2\right)^\frac{1}{2} - 1 \, dx\\
 & \leq \int_{\B_V(0,1)} \sum_{i \in I_x} \left(1 + \|D\by^i_x\|^2\right)^\frac{1}{2} - 1 \, dx \\
 & \leq \int_{\B_V(0,1)} \sum_{i \in I_x} \left(1 + (\J \by^i_x)^2\right)^\frac{1}{2} - 1 \, dx \\
 & = \cH^m(\spt(P) \cap Z_V) - \balpha(m) \leq 5\epsilon \,.
\end{align*}
Plugging this back into \eqref{term} gives that $\int_{\B_V(0,\frac{1}{2})} \sum_{i \in I_x} |\bv(x) - \by^i(x)|^2 \, dx \leq \bc'' \rho^2 \epsilon$ for some $\bc''(m) > 0$. By homogeneity the lemma follows.
\end{proof}

Next we want to estimate the differential $D\bv_x$. It will also become clear from the following proof that $\bv$ is differentiable at all $x \neq 0$.

\begin{Lem}
[{\cite[Lemma~8]{Reif}}]
\label{derivative_estimate}
There is a constant $\bc_{\theThm}(m) > 0$ such that for all $x \in \B_V(0,1) \setminus \{0\}$,
\[
\|D\bv_x\|_{\HS} \leq (1 + \bc_{\theThm}\rho)(\balpha(m) \rho^m |x|^m)^{-1} \int_{\B_V(x,\rho |x|)} \|D\bar \by_z\|_{\HS} \, dz \,.
\]
\end{Lem}

\begin{proof}
By definition
\[
\bv(x) = (\balpha(m) \rho^m |x|^m)^{-1} \int_{\B_V(x,\rho |x|)} \bar \by(z) \, dz = \balpha(m)^{-1}\int_{\B_V(0,1)} \bar \by(x + z \rho |x|) \, dz \,.
\]
The map $\bar \by$ is piecewise linear and in particular differentiable outside a set of measure zero. Let $v \in B_V(0,1)$ and $t \geq 0$. The Lebesgue Dominated Convergence Theorem implies that
\begin{align}
\nonumber
& \frac{1}{t}(\bv(x + tv) - \bv(x)) \\
\nonumber
 & \qquad =  \balpha(m)^{-1}\int_{\B_V(0,1)} \frac{1}{t}(\bar \by(x + tv + z \rho |x + tv|) - \bar \by(x + z \rho |x|)) \, dz \\
\label{differentialexpr}
 & \qquad \to \balpha(m)^{-1}\int_{\B_V(0,1)} D \bar \by_{x + z \rho |x|}(v + z \rho \langle x, v\rangle) \, dz \,,
\end{align}
for $t \to 0$. Hence with \eqref{norms},
\begin{equation}
\label{random.new.1}
 |D\bv_x(v)| \leq \balpha(m)^{-1} \int_{\B_V(0,1)} |D \bar \by_{x + z \rho |x|}(v)| + \rho \|D \bar \by_{x + z \rho |x|}\|_{\HS} \, dz \,.
\end{equation}
For an orthonormal basis $(e_1,\dots,e_m)$ of $V$, using the multivariate Jensen inequality
\[
\biggl(\Bigl(\dashint a_1\Bigr)^2 + \cdots + \Bigl(\dashint a_m\Bigr)^2 \biggr)^\frac{1}{2} \leq \dashint \left(a_1^2 + \cdots + a_m^2 \right)^\frac{1}{2} \,,
\]
to \eqref{random.new.1} leads to
\begin{align*}
\|D\bv_x\|_{\HS} & = \biggl(\sum_i |D\bv_x(e_i)|^2\biggr)^\frac{1}{2} \\
 & \leq  \balpha(m)^{-1}\int_{\B_V(0,1)} \biggl(\sum_i\left(|D \bar \by_{x + z \rho |x|}(e_i)| + \rho \|D \bar \by_{x + z \rho |x|}\|_{\HS} \right)^2\biggr)^\frac{1}{2} \, dz \\
 & \leq \balpha(m)^{-1}\int_{\B_V(0,1)} \left(\|D \bar \by_{x + z \rho |x|}\|_{\HS}^2 + (m\rho^2 + 2 \sqrt{m}\rho) \|D \bar \by_{x + z \rho |x|}\|_{\HS}^2\right)^\frac{1}{2} \, dz \\
 & \leq (1 + \sqrt{m}\rho) \balpha(m)^{-1} \int_{\B_V(0,1)} \|D \bar \by_{x + z \rho |x|}\|_{\HS} \, dz \\
 & \leq (1 + \sqrt{m}\rho) (\balpha(m) \rho^m |x|^m)^{-1} \int_{\B_V(x,\rho |x|)} \|D \bar \by_{z}\|_{\HS} \, dz \,.
\end{align*}
\end{proof}

In the next Lemma we estimate the derivatives of $\bv$ pointwise.

\begin{Lem}
[{\cite[Lemma~9]{Reif}}]
\label{derivative_estimate2}
There is a constant $\bc_{\theThm}(m) > 0$ such that
$\|D\bv_x\|_{\HS} \leq \bc_{\theThm} \epsilon^{\frac{1}{3}}$ and hence $|\bv(x)| \leq \bc_{\theThm} \epsilon^{\frac{1}{3}}|x|$ for all $x \in \B_V(0,1)$.
\end{Lem}

\begin{proof}
Fix a point $x \in \B_V(0,1)$. For shorter formulas we abbreviate $B_{x,\rho} \defl \B_V(x,\rho |x|)$. The Cauchy-Schwarz inequality implies
\begin{align}
\nonumber
& \int_{B_{x,\rho}} \|D \bar \by\|_{\HS} = \int_{B_{x,\rho}} (\|D \bar \by\|_{\HS}^2)^\frac{1}{2} \\
\nonumber
 & \, = \int_{B_{x,\rho}} \left((1 + \|D \bar \by\|_{\HS}^2)^\frac{1}{2} - 1\right)^\frac{1}{2} \left((1 + \|D \bar \by\|_{\HS}^2)^\frac{1}{2} + 1 \right)^\frac{1}{2} \\
\nonumber
 & \, \leq \biggl(\int_{B_{x,\rho}} (1 + \|D \bar \by\|_{\HS}^2)^\frac{1}{2} - 1\biggr)^\frac{1}{2}\biggl(\int_{B_{x,\rho}} (1 + \|D\bar\by\|_{\HS}^2)^\frac{1}{2} + 1\biggr)^\frac{1}{2} \\
\label{derivative_estimate2a}
 & \, = \biggl(\int_{B_{x,\rho}} (1 + \|D \bar \by\|_{\HS}^2)^\frac{1}{2} - 1\biggr)^\frac{1}{2}\biggl(2|B_{x,\rho}| + \int_{B_{x,\rho}} (1 + \|D \bar \by\|_{\HS}^2)^\frac{1}{2} - 1\biggr)^\frac{1}{2} \,.
\end{align}
By definition $\bar \by(z) = \sum_{i \in I_z} \delta_i(z) \by^i(z)$, where $\delta_i(z) = \|g_i\|(\sum_{j \in I_x} \|g_j\|)^{-1} \leq \|g_i\|\|g_0\|^{-1}$. For almost every $\xi \in \B_V(0,1)$ there holds,
\[
\|D \bar \by_\xi\|_{\HS} =  \Bigl\|D \Bigl(\sum_{i \in I_\xi} \delta_i\by^i\Bigr)_\xi\Bigr\|_{\HS} = \Bigl\|\sum_{i \in I_\xi} \delta_i(\xi)D\by^i_\xi\Bigr\|_{\HS} \leq \sum_{i \in I_\xi} \delta_i(\xi)\|D\by^i_\xi\|_{\HS} \,.
\]
Using Lemma~\ref{squares} and the fact that $\|D\by^i_z\|_{\HS} \leq \J \by^i_z$ for almost all $z$,
\begin{align}
\nonumber
\int_{B_{x,\rho}} (1 + \|D \bar \by\|_{\HS}^2)^\frac{1}{2} - 1 & \leq \int_{B_{x,\rho}} \biggl(1 + \Bigl(\sum_{i \in I_z} \delta_i(z) \|D\by^i_z\|_{\HS} \Bigr)^2\biggr)^\frac{1}{2} - 1 \, dz \\
\nonumber
 & \leq \frac{1}{\|g_0\|}\int_{B_{x,\rho}} \sum_{i \in I_z} \|g_i\| \left(1 + \|D\by_z^i\|_{\HS}^2\right)^\frac{1}{2} - \|g_0\| \, dz \\
\nonumber
 & \leq \frac{1}{\|g_0\|}\int_{\B_V(0,(1+\rho)|x|)} \sum_{i \in I_z} \|g_i\|\left(1 + (\J \by_z^i)^2\right)^\frac{1}{2} - \|g_0\| \, dz \\
\nonumber
 & = (1+\rho)^m|x|^m\|g_0\|^{-1}\Exc_1(P,V) \\
\label{derivative_estimate2b}
 & \leq 2^m|x|^m\epsilon \,.
\end{align}
Hence
\begin{align*}
\int_{B_{x,\rho}} \|D \bar \by\|_{\HS} & \leq (2^m|x|^m \epsilon)^\frac{1}{2}(2^m|x|^m \epsilon + 2|B_{x,\rho}|)^\frac{1}{2} \\
 & = (2^m|x|^m \epsilon)^\frac{1}{2}(2^m|x|^m\epsilon + 2\balpha(m)\rho^m|x|^m)^\frac{1}{2} \\
 & \leq \bc(m)|x|^m \epsilon^\frac{1}{2} \,,
\end{align*}
for some constant $\bc(m) > 0$. By Lemma~\ref{derivative_estimate} and because $\epsilon^\frac{1}{6} \leq \rho^{m}$ we get for $x \neq 0$,
\begin{align*}
\|D\bv_x\|_{\HS} & \leq (1 + \bc_{\ref{derivative_estimate}}\rho)(\balpha(m)\rho^m|x|^m)^{-1}\bc|x|^m\epsilon^\frac{1}{2} \\
 & = \bc(1 + \bc_{\ref{derivative_estimate}}\rho)\balpha(m)^{-1}\epsilon^{-\frac{1}{6}}\epsilon^\frac{1}{2} \\
 & = \bc'\epsilon^\frac{1}{3} \,,
\end{align*}
for some constant $\bc'(m) > 0$.
\end{proof}

Because $\|D\bv_x\| \leq \|D\bv_x\|_{\HS} \leq \sqrt{m} \|D\bv_x\|$ the statement in the lemma above is true for both norms. Further estimates on the integral of this differential is needed. In the following lemma the particular definition of the averaging function $\bar\by(x) = (\sum_{j \in I_x}\|g_j\|)^{-1}\sum_{i \in I_x}\|g_i\|\by^i(x)$ is used in an essential way.

\begin{Lem}
[{\cite[Lemma~10]{Reif}}]
\label{integral_of_dv}
There is a constant $\bc_{\theThm}(m) > 0$ such that
\begin{align*}
\int_{\B_V} \sum_{i \in I_x} \|g_i\|(1 + \|D\bv_x\|_{\HS}^2)^\frac{1}{2} \, dx & \leq \bc_{\theThm}\|g_0\|\rho\epsilon + \int_{\B_V} \sum_{i \in I_x} \|g_i\|(1 + \|D \by^i_x\|_{\HS}^2)^\frac{1}{2} \, dx \,.
\end{align*}
\end{Lem}

\begin{proof}
For $x \in \B_V$ abbreviate again $B_{x,\rho} \defl \B_V(x,\rho |x|)$. Combining \eqref{derivative_estimate2a} and \eqref{derivative_estimate2b} as obtained in the proof of Lemma~\ref{derivative_estimate2} above with the estimate $\epsilon \leq \rho^{m+1}$, we get,
\begin{align}
\nonumber
& \biggl(\int_{B_{x,\rho}} \|D\bar \by\|_{\HS}\biggr)^2 \\
\nonumber
& \quad \leq \biggl(\int_{B_{x,\rho}} (1 + \|D\bar \by\|_{\HS}^2)^\frac{1}{2} - 1 \biggr)\biggl(2|B_{x,\rho}| + \int_{B_{x,\rho}} (1 + \|D\bar \by\|_{\HS}^2)^\frac{1}{2} - 1 \biggr) \\
\nonumber
 & \quad \leq \biggl(\int_{B_{x,\rho}} (1 + \|D\bar \by\|_{\HS}^2)^\frac{1}{2} - 1\biggr)\biggl(2\balpha(m) \rho^m |x|^m + 2^m|x|^m \epsilon \biggr) \\
\label{integral_of_dva}
 & \quad \leq 2\balpha(m)\rho^m|x|^m(1 + \bc\rho) \int_{B_{x,\rho}} (1 + \|D\bar \by\|_{\HS}^2)^\frac{1}{2} - 1 \,,
\end{align}
for some constant $\bc(m) > 0$. Using the fact that $(1 + a^2)^\frac{1}{2} \leq 1 + \tfrac{1}{2}a^2$, it follows from Lemma~\ref{derivative_estimate} and \eqref{integral_of_dva} that for $\bc' = \max\{\bc_{\ref{derivative_estimate}},\bc\}$,
\begin{align}
\nonumber
& \int_{\B_V} (1 + \|D\bv\|_{\HS}^2)^\frac{1}{2} - 1 \leq \frac{1}{2} \int_{\B_V} \|D\bv\|_{\HS}^2 \\
\nonumber
 & \qquad \qquad \leq \frac{1}{2} \int_{\B_V} \biggl[(1 + \bc_{\ref{derivative_estimate}}\rho)(\balpha(m) \rho^m |x|^m)^{-1} \int_{B_{x,\rho}} \|D\bar \by\|_{\HS} \biggr]^2 \\
\nonumber
 & \qquad \qquad \leq \frac{1}{2} \int_{\B_V} 2(1 + \bc'\rho)^3(\balpha(m) \rho^m |x|^m)^{-1} \int_{B_{x,\rho}}(1 + \|D \bar \by\|_{\HS}^2)^\frac{1}{2} - 1 \\
\nonumber
 & \qquad \qquad = \int_{\B_V} (1 + \bc'\rho)^3 \balpha(m)^{-1} \int_{\B_V} (1 + \|D \bar \by_{x + \rho |x| z}\|_{\HS}^2)^\frac{1}{2} - 1 \, dz \\
\label{integral_of_dvb}
 & \qquad \qquad \leq (1 + \bc'\rho)^3  \max_{z \in \B_V} \int_{\B_V} (1 + \|D \bar \by_{x + \rho |x| z}\|_{\HS}^2)^\frac{1}{2} - 1 \, dx \,.
\end{align}
As in \eqref{psi}, for fixed $z \in \B_V(0,1)$ the map $\psi : x \mapsto x + \rho |x| z$ is a bijection of $\B_V(0,1)$ and a subset of $\B_V(0,1 + \rho)$ with $D \psi_x(v) = v + \frac{\rho}{|x|}\langle x, v\rangle z$ and $\det D\psi_x = 1 + \tfrac{\rho}{|x|} \langle x, z \rangle$ for $x \neq 0$. It is $\det D\psi_x \geq 1 - \rho$ and since we assume that $\rho \leq \tfrac{1}{2}$ we obtain
\[
\det D(\psi^{-1})_\xi \leq \tfrac{1}{1 - \rho} \leq 1 + 2\rho \,,
\]
for almost all $\xi \in \psi(\B_V(0,1))$. By the homogeneity of $\|D\bar\by\|_{\HS}$,
\begin{align}
\nonumber
& \int_{\B_V(0,1)} (1 + \|D \bar \by_{x + \rho |x| z}\|_{\HS}^2)^\frac{1}{2} - 1 \, dx \\
\nonumber
 & \qquad \qquad = \int_{\psi(\B_V(0,1))} ((1 + \|D \bar \by_\xi\|_{\HS}^2)^\frac{1}{2} - 1) \det D(\psi^{-1})_\xi \, d\xi \\
\nonumber
 & \qquad \qquad \leq (1 + 2\rho) \int_{\B_V(0,1 + \rho)} (1 + \|D \bar \by_\xi\|_{\HS}^2)^\frac{1}{2} - 1 \, d\xi \\
\nonumber
 & \qquad \qquad = (1 + 2\rho) (1 + \rho)^m \int_{\B_V(0,1)} (1 + \|D \bar \by_\xi\|_{\HS}^2)^\frac{1}{2} - 1 \, d\xi \\
\label{integral_of_dvc}
 & \qquad \qquad \leq (1 + \bc''\rho) \int_{\B_V(0,1)} (1 + \|D \bar \by_\xi\|_{\HS}^2)^\frac{1}{2} - 1 \, d\xi \,,
\end{align}
for some constant $\bc''(m) > 0$. By the excess bound in term of $\epsilon$, we obtain as in \eqref{derivative_estimate2b},
\begin{align}
\nonumber
\int_{\B_V} (1 + \|D \bar \by\|_{\HS}^2)^\frac{1}{2} - 1 & \leq \int_{\B_V} \biggl(1 + \Bigl(\sum_{i \in I_\xi} \delta_i(\xi) \|D\by^i_\xi\|_{\HS} \Bigr)^2\biggr)^\frac{1}{2} - 1 \, d\xi \\
\nonumber
 & \leq \int_{\B_V} \sum_{i \in I_\xi} \delta_i(\xi) \left(1 + \|D\by_\xi^i\|_{\HS}^2\right)^\frac{1}{2} - 1 \, d\xi \\
\nonumber
 & \leq \|g_0\|^{-1}\int_{\B_V} \sum_{i \in I_\xi} \|g_i\|\left(1 + (\J \by_\xi^i)^2\right)^\frac{1}{2} - \|g_0\| \, d\xi \\
\label{integral_of_dvd}
 & = \|g_0\|^{-1}\Exc_1(P,V) \leq \epsilon \,.
\end{align}
Combining \eqref{integral_of_dvb},\eqref{integral_of_dvc} and \eqref{integral_of_dvd} we see that for $\bc''' = \max\{\bc'',\bc'\}$,
\begin{align*}
\int_{\B_V} (1 + \|D\bv\|_{\HS}^2)^\frac{1}{2} - 1 
 & \leq (1 + \bc'''\rho)^4 \int_{\B_V} (1 + \|D \bar \by_\xi\|_{\HS}^2)^\frac{1}{2} - 1 \, d\xi \\
 & \leq  \bc'''' \rho \epsilon + \int_{\B_V} (1 + \|D \bar \by_\xi\|_{\HS}^2)^\frac{1}{2} - 1 \, d\xi \,,
\end{align*}
for some constant $\bc''''(m) > 0$. With \eqref{convex_comb} of Lemma~\ref{squares} we obtain,
\begin{align*}
\int_{\B_V} (1 + \|D\bv\|_{\HS}^2)^\frac{1}{2} & \leq \bc'''' \rho \epsilon + \int_{\B_V} \biggl(\sum_{j \in I_\xi} \|g_j\|\biggr)^{-1}\sum_{i \in I_\xi}\|g_i\| (1 + \|D \by^i_\xi\|_{\HS}^2)^\frac{1}{2} \, d\xi \,. 
\end{align*}
Applying \eqref{multiple_values_mass_estimate_prelim}, Lemma~\ref{derivative_estimate2} and $\epsilon^\frac{2}{3} \leq \rho$ to the estimate above leads to,
\begin{align*}
& \int_{\B_V} \sum_{i \in I_x} \|g_i\| \left(\left(1 + \|D\bv_x\|^2 \right)^\frac{1}{2} - 1\right) \, dx \\
& \qquad \leq \|g_0\|\int_{E_1} \left(1 + \|D\bv_x\|_{\HS}^2 \right)^\frac{1}{2} - 1 \, dx + \int_{E_2} \sum_{i \in I_x} \|g_i\| \|D\bv_x\|_{\HS}^2 \, dx \\
& \qquad \leq \|g_0\|\int_{\B_V} \left(1 + \|D\bv_x\|_{\HS}^2 \right)^\frac{1}{2} - 1 \, dx + 15\|g_0\| \epsilon\bc_{\ref{derivative_estimate2}}\epsilon^\frac{2}{3} \\
& \qquad \leq 15\bc_{\ref{derivative_estimate2}}\|g_0\| \epsilon^\frac{5}{3} + \bc''''\|g_0\|\rho\epsilon \\
& \qquad \quad + \|g_0\|\int_{\B_V} \biggl(\sum_{j \in I_x} \|g_j\|\biggr)^{-1} \sum_{i \in I_x} \|g_i\|\left((1 + \|D \by^i_x\|_{\HS}^2)^\frac{1}{2} - 1\right) \, dx \\
& \qquad \leq (15\bc_{\ref{derivative_estimate2}} + \bc'''')\|g_0\|\rho\epsilon + \int_{\B_V} \sum_{i \in I_x} \|g_i\|\left((1 + \|D \by^i_x\|_{\HS}^2)^\frac{1}{2} - 1\right) \, dx \,.
\end{align*}
Adding $\int_{\B_V}\sum_{i \in I_x} \|g_i\|$ to both sides, the lemma follows.
\end{proof}




By assumption \eqref{projection} we have $\pi_{V\#}(P \res Z_V) = g_0\curr{\B_V(0,1)}$. Define the rectifiable $G$-chain $T' \defl (\id_V + \bv)_\#(g_0\curr{\B_V(0,1)}) \in \cR_m(\Hi;G)$. Lemma~\ref{dist_to_v} implies that
\[
\hdist((\spt(P) \cup \spt(T')) \cap Z_V, \B_V(0,1)) \leq 2\rho < 1 \,.
\]
Let $\varphi : \Hi \to \Hi$ be some Lipschitz map with the following properties: for $x \in V$ and $y \in V^\perp$,
\begin{enumerate}
\item $\varphi(x+y) = x + y$, if $|x| \geq \frac{3}{4}$ or $|y| \geq 2$,
\item $\varphi(x+y) = x + \bv(x)$, if $|x| \leq \tfrac{1}{2}$ and $|y| \leq 1$,
\item $\varphi(x+y) = x + (4|x| - 2)y + (3 - 4|x|)\bv(x))$, if $\tfrac{1}{2} \leq |x| \leq \tfrac{3}{4}$ and $|y| \leq 1$,
\end{enumerate}
Define $T \defl \varphi_\# P \in \cR_m(\Hi;G)$. It is easy to check that
\begin{enumerate}
\item $\partial (T \res Z_V(\frac{3}{4})) = \partial (P \res Z_V(\frac{3}{4}))$,
\item $\pi_{V\#}(T \res Z_V) = g_0\curr{\B_V(0,1)}$,
\item $T \res Z_V(\tfrac{1}{2}) = T' \res Z_V(\tfrac{1}{2})$,
\item for $\frac{1}{2} \leq |x| \leq \frac{3}{4}$, 
\[
\spt(T) \cap \pi_V^{-1}(x) = \bigcup_{i \in I_x} \left\{x + (3 - 4|x|)\bv(x) + (4|x| - 2)\by^i(x)\right\} \,.
\]
\end{enumerate}

\begin{Lem}
[{\cite[Lemma~11]{Reif}}]
\label{mass_of_z}
There are constants $\bc_{\theThm}(m),\rho_{\theThm}(m),\epsilon_{\theThm}(m) > 0$ such that if $0 < \rho \leq \rho_{\theThm}(m)$ and $0 < \epsilon \leq \epsilon_{\theThm}(m)$ (recall we also always assume $\epsilon \leq \rho^{6m}$), then
\[
\int_{A} \sum_{i \in I_x} \|g_i\| (1 + (\J \bz^i_x)^2)^\frac{1}{2} \, dx \leq \bc_{\theThm}\|g_0\|\rho^\frac{1}{2}\epsilon + \int_{A} \sum_{i \in I_x} \|g_i\| (1 + (\J \by^i_x)^2)^\frac{1}{2} \, dx \,,
\]
where $\bz^i \defl (4r - 2)\by^i + (3 - 4r)\bv$, $r(x) \defl |x|$ and $A \defl \{x \in V : \frac{1}{2} \leq |x| \leq \frac{3}{4}\}$.
\end{Lem}

\begin{proof}
We start by noticing that $0 \leq 4|x|-2 \leq 1$ as well as $0 \leq 3-4|x| \leq 1$, two facts we will freely use without further reference in the course of this proof.
For almost all $x \in A$ it is
\begin{equation}
\label{random.new.20}
D\bz^i_x(v) = (4|x| - 2)D\by^i_x(v) + (3 - 4|x|)D\bv_x(v) + \frac{4\langle x,v \rangle}{|x|} (\by^i(x) - \bv(x)) \,.
\end{equation}
Fix a point $x$ for which all the differentials in the formula above exist and consider the matrix $M^i_{kl} \defl \langle D\bz^i_x(e_k), D\bz^i_x(e_l) \rangle$ for some orthonormal basis $(e_1,\dots, e_m)$ of $V$. Applying \eqref{jacobian_sum},
\[
(\J \bz^i_x)^2 = \sum_{K \neq \emptyset} \det(M_K^i) \,.
\]
For simplicity we assume that the $e_k=e^i_k$'s are eigenvectors of ${D\by^i_x}^*D\by^i_x$ with eigenvalues $(\lambda^i_k)^2$, for $\lambda^i_k \geq 0$, and $\Lambda^i$ is the corresponding diagonal matrix (i.e.\ $\Lambda^i_{kl} = \langle D\by^i_x(e^i_k), D\by^i_x(e^i_l) \rangle = (\lambda^i_k)^2 \delta_{kl}$). We can write
\begin{align*}
M^i_{kl} & = (4|x| - 2)^2 \Lambda^i_{kl} + (3 - 4|x|)^2 \langle D\bv_x(e^i_k), D\bv_x(e^i_l) \rangle + R^i_{kl} \,,
\end{align*}
where
\begin{equation}
\label{random.new.11}
\begin{split}
|R^i_{kl}| & \leq 16|\by^i(x) - \bv(x)|^2 + 8\|D\bv_x\||\by^i(x) - \bv(x)| \\
 & \quad + (\|D\bv_x\| + 4|\by^i(x) - \bv(x)|)(\lambda^i_k + \lambda^i_l)\,.
\end{split}
\end{equation}
Next we estimate the minors of $M^i$ corresponding to some nonempty subset $K \subset \{1,\dots,m\}$. Using \eqref{random.new.20} to expand each $M^i_{kk}=\det(M^i_{\{k\}})$ we find that for subsets with one element $K = \{k\}$,
\begin{equation}
\label{random.new.12}
\begin{split}
\|D\bz^i_x\|_{\HS}^2 &= \sum_{k=1}^m \det\left(M^i_{\{k\}}\right) \\
&\leq  (4|x|-2)^2 \sum_{k=1}^m  \Lambda^i_{kk} + (3-4|x|)^2 \sum_{k=1}^m |D\bv_x(e_k)|^2 \\
& \quad\quad\quad\quad +2(4|x|-2)(3-4|x|)\sum_{k=1}^m  \langle D\by_x^i(e_k),D\bv_x(e_k) \rangle  \\
& \quad + 16 |\by^i(x)-\bv(x)|^2 + 8 |\by^i(x)-\bv(x)| \left( \|D\bv_x\| + \|D\by^i_x\|\right) \\
& = \bigg( (4|x|-2) \|D\by_x^i\|_{\HS} + (3-4|x|) \|D\bv_x\|_{\HS} \bigg)^2 \\
& \quad + 16 |\by^i(x)-\bv(x)|^2 + 8 |\by^i(x)-\bv(x)| \left( \|D\bv_x\| + \|D\by^i_x\|\right) \,.
\end{split}
\end{equation}
Furthermore the exists a constant $\bc=\bc(m) > 0$ such that if $\# K \geq 2$ then,
\begin{equation}
\label{random.new.10}
\begin{split}
\det(M^i_K) & \leq (4|x| - 2)^2 \det(\Lambda^i_K) + \bc\|D\bv_x\|^4 \\
 & \quad + \bc(|\by^i(x) - \bv(x)|^2 + \|D\bv_x\||\by^i(x) - \bv(x)|) \\
 & \quad + \bc(|\by^i(x) - \bv(x)| + \|D\bv_x\|) (\J \by^i_x)^2 \,,
\end{split}
\end{equation}
We now indicate how to obtain the above inequality:
\begin{multline*}
\det(M_K^i) = \sum_{\sigma \in S_q} (-1)^{|\sigma|} \prod_{k \in K} M^i_{k \sigma(k)} \\
= \sum_{\sigma \in S_q} (-1)^{|\sigma|} \prod_{k \in K} \left(  (4|x|-2)^2 \Lambda^i_{k \sigma(k)} +(3-4|x|)^2 \langle D\bv_x(e^i_k),D\bv_x(e^i_{\sigma(k)})\rangle +R^i_{k \sigma(k)} \right)\\
= \sum_{\sigma \in S_q} (-1)^{|\sigma|} \sum_{K_1,K_2,K_3} \left( \prod_{k \in K_1} (4|x|-2)^2 \Lambda^i_{k\sigma(k)} \right)\\\left( \prod_{k \in K_2} (3-4|x|)^2 \langle D\bv_x(e^i_k),D\bv_x(e^i_{\sigma(k)})\rangle \right) 
\left( \prod_{k \in K_3} R^i_{k\sigma(k)} \right) \,,
\end{multline*}
where the last sum extends over all partitions $\{K_1,K_2,K_3\}$ of $K$. Letting $q \defl \#K$, the term corresponding to $K_2=K_3=\emptyset$ equals $(4|x| - 2)^{2q} \det(\Lambda^i_K)$ when $\sigma = \mathrm{id}_K$ and zero otherwise, whereas the term corresponding to $K_1=K_3=\emptyset$ is bounded above by $(3-4|x|)^{2q} \|D\bv_x\|^{2q}$ regardless of $\sigma$. These are respectively bounded above by $(4|x| - 2)^2 \det(\Lambda^i_K)$ and by $\|D\bv_x\|^4$ since $0 \leq \det(\Lambda^i_K)$, since $q \geq 2$, and since we may assume $\|D\bv_x\| \leq \bc_{\ref{derivative_estimate2}}(m)\epsilon^\frac{1}{3} \leq 1$ provided $\epsilon$ is small enough depending upon $\epsilon$. Thus these two terms, corresponding to all $\sigma \in S_q$, are accounted for in the first line of \eqref{random.new.10}. In order to bound from above the other terms, we start by observing that if $\emptyset \neq \hat{K} \subset \{1,\ldots,m\}$ then 
\[
\prod_{k \in \hat{K}}(\lambda^i_k)^2 \leq \det(I + \Lambda^i) - 1 = (\J \by^i_x)^2 \,,
\]
Furthermore a product $\prod_{k \in \hat{K}} \Lambda^i_{k \sigma(k)}$ is either zero (when $\sigma \neq \mathrm{id}_K$) or equal to $\prod_{k \in \hat{K}} (\lambda^i_k)^2$, thus in both cases it is less than or equal to $ (\J \by^i_x)^2$.
 Now the remaining terms corresponding to $K_3=\emptyset$ have necessarily $K_1 \neq \emptyset \neq K_2$ and are therefore bounded above by 
\begin{equation*}
\left( \prod_{k \in K_1} \Lambda^i_{k\sigma(k)} \right)\left( \prod_{k \in K_2}\|D\bv_x\|^2 \right) \leq \|D\bv_x\|  (\J \by^i_x)^2 \,,
\end{equation*}
and are accounted for in the last line of \eqref{random.new.10}. Next we consider the terms corresponding to $K_3 \neq \emptyset$ and we define $a^i_x$ and $b^i_x$ in the obvious way, according to \eqref{random.new.11}, so that $|R^i_{kl}| \leq a^i_x + b^i_x (\lambda^i_k + \lambda^i_l)$ and we notice that we may assume $0 \leq a^i_x \leq 1$ and $0 \leq b^i_x \leq 1$ provided $\epsilon$ is small enough depending on $m$. Now
\begin{multline*}
\left| \prod_{k \in K_3} R^i_{k \sigma(k)} \right| \leq \prod_{k \in K_3} \big( a^i_b + b^i_x(\lambda^i_k + \lambda^i_{\sigma(k)} ) \big) \\ = \sum_{K_{3,1},K_{3,2}} (a^i_x)^{\# K_{3,1}} (b^i_x)^{\# K_{3,2}} \prod_{k \in K_{3,2}} (\lambda^i_k + \lambda^i_{\sigma(k)})\\
\leq \begin{cases}
a^i_x & \text{ if } K_{3,2} = \emptyset \\
b^i_x \prod_{k \in K_{3,2}} (\lambda^i_k + \lambda^i_{\sigma(k)}) & \text{ if } K_{3,2} \neq \emptyset 
\end{cases}\,,
\end{multline*}
where the sum extends over partitions $\{K_{3,1},K_{3,2}\}$ of $K_3$.
Thus the terms corresponding to $K_3 \neq \emptyset$, $K_{3,2} =\emptyset$ and $K_1=\emptyset$ are accounted for in the second line of \eqref{random.new.10}, whereas the terms where $K_3 \neq \emptyset$, $K_{3,2} = \emptyset$ and $K_1 \neq \emptyset$ are accounted for in the third line of \eqref{random.new.10} because $|\by^i(x) - \bv(x)| \leq 1$. It remains to consider the terms for which $K_{3,2} \neq \emptyset$. The following sum extends over partitions $\{L_1,L_2\}$ of $K_{3,2}$:
\begin{multline*}
\prod_{k \in K_{3,2}} (\lambda^i_k + \lambda^i_{\sigma(k)}) = \sum_{L_1,L_2} \left( \prod_{k \in L_1} \lambda^i_k \right) \left( \prod_{k \in \sigma(L_2)} \lambda^i_{k} \right) \\ =\sum_{L_1,L_2} \left( \prod_{k \in L_1 \cap \sigma(L_2)} (\lambda^i_k)^2 \right) \left( \prod_{k \in L_1 \ominus \sigma(L_2)} \lambda^i_{k} \right) \\
\leq \sum_{L_1,L_2} \left( \prod_{k \in L_1 \cap \sigma(L_2)} (\lambda^i_k)^2 \right) \left(\prod_{\substack{k \in L_1 \ominus \sigma(L_2)\\\lambda^i_k \geq 1}} (\lambda^i_{k})^2 \right) \\
\end{multline*}
We first analyze the various subcases occurring as $K_1 = \emptyset$. If $L_1 \cap \sigma(L_2) \neq \emptyset$, the corresponding term is bounded above by $b_x^i (\J \by_x^i)^2$ and thus is accounted for in the third line of \eqref{random.new.10}. The same occurs if $L_1 \cap \sigma(L_2) = \emptyset$ but there exists $k \in L_1 \ominus \sigma(L_2)$ such that $\lambda^i_k \geq 1$. Now if $\lambda^i_k < 1$ for all $k \in L_1 \ominus \sigma(L_2)$, and $L_1 \cap \sigma(L_2) = \emptyset$, we pick two $\lambda^i_{k_1},\lambda^i_{k_2} \in  L_1 \ominus \sigma(L_2)$ with $k_1 \neq k_2$ (this is possible since $\#( L_1 \ominus \sigma(L_2)) \geq 2$ because $L_1 \cap \sigma(L_2) = \emptyset$) and we observe that we can bound a factor $\lambda^i_{k_1} \lambda^i_{k_2} \leq \frac{1}{2} \left( (\lambda^i_{k_1})^2 + (\lambda^i_{k_2})^2 \right) \leq (\J\by_x^i)^2$ and, again, the corresponding term is bounded above by $b_x^i (\J \by_x^i)^2$ and thus is accounted for in the third line of \eqref{random.new.10}.

Finally it remains only to consider the case when $K_{3,2} \neq \emptyset$ and $K_1 \neq \emptyset$. We observe these terms are bounded above by 
\begin{equation*}
b^i_x \left( \prod_{k \in K_1} \Lambda^i_{k \sigma(k)} \right) \left( \prod_{k \in L_1 \cap \sigma(L_2)} (\lambda^i_k)^2 \right) \left( \prod_{\substack{k \in L_1 \ominus \sigma(L_2)\\\lambda^i_k \geq 1}} (\lambda^i_{k})^2 \right)
\end{equation*}
Clearly $K_1 \cap L_1 \subset K_1 \cap K_3 = \emptyset$. If $\sigma_{K_1} \neq \mathrm{id}_{K_1}$ then the second factor above (corresponding to $k \in K_1$) vanishes, thus we may assume $K_1 \cap \sigma(L_2) = \emptyset$ but then the product is of the type $b^i_x \prod_{k \in \hat{K}} (\lambda^i_k)^2$ with $\hat{K} \neq \emptyset$ and \eqref{random.new.10} is established.

Summing up \eqref{random.new.10} over all $K$ with $\# K \geq 2$ we find that
\begin{equation*}
\begin{split}
(\J\bz_x^i)^2 - \|D\bz_x^i\|^2_{\HS} & = \sum_{\substack{K \subset \{1,\ldots,m\}\\\#K \geq 2}} \det(M^i_K) \\
& \leq (4|x|-2)^2 \big( (\J\by_x^i)^2 - \|D\by_x^i\|^2_{\HS} \big) + \bc'\|D\bv_x\|^4 \\
 & \quad + \bc'(|\by^i(x) - \bv(x)|^2 + \|D\bv_x\||\by^i(x) - \bv(x)|) \\
 & \quad + \bc'(|\by^i(x) - \bv(x)| + \|D\bv_x\|) (\J \by^i_x)^2 \,,
\end{split}
\end{equation*}
and with the help of \eqref{random.new.12}, recalling that $\|D\by_x^i\|_{\HS} \leq \J\by_x^i$, 
\begin{align*}
(\J \bz^i_x)^2 & \leq (4|x|-2)^2 \big( (\J\by_x^i)^2 - \|D\by_x^i\|^2_{\HS} \big) \\
& \quad\quad\quad + \bigg( (4|x|-2) \|D\by_x^i\|_{\HS} + (3-4|x|) \|D\bv_x\|_{\HS} \bigg)^2 + \bc'\|D\bv_x\|^4 \\
 & \quad\quad\quad + \bc'(|\by^i(x) - \bv(x)|^2 + \|D\bv_x\||\by^i(x) - \bv(x)|) \\
 & \quad\quad\quad + \bc'(|\by^i(x) - \bv(x)| + \|D\bv_x\|) (\J \by^i_x)^2 \\
 & \quad\quad\quad + \bc' |\by^i(x) - \bv(x)| \|D\by_x^i\| \\
 & \leq \bigg( (4|x|-2) \J\by_x^i + (3-4|x|) \|D\bv_x\|_{\HS} \bigg)^2 + \bc'\|D\bv_x\|^4 \\
  & \quad\quad\quad + \bc'(|\by^i(x) - \bv(x)|^2 + \|D\bv_x\||\by^i(x) - \bv(x)|) \\
 & \quad\quad\quad + \bc'(|\by^i(x) - \bv(x)| + \|D\bv_x\|) (\J \by^i_x)^2 \\
  & \quad\quad\quad + \bc' |\by^i(x) - \bv(x)| \|D\by_x^i\| \\
 & = \bigg( (4|x|-2) \J\by_x^i + (3-4|x|) \|D\bv_x\|_{\HS} \bigg)^2 + \sum_{j=1}^6 R^{ij}(x) \,,
\end{align*}
for some $\bc'(m) > 0$, where the $R^{ij}(x)$, $j=1,\ldots,6$, are readily defined by the last equality.

Now from the above inequality, applying seven times \eqref{squaresa}
 of Lemma~\ref{squares}, and one time \eqref{convex_comb}, we see that
 \begin{multline}
 \label{random.new.30}
 \left( 1 + (\J\bz_x^i)^2 \right)^\frac{1}{2}  \leq (4|x|-2) \left( 1 + (\J\by_x^i)^2 \right)^\frac{1}{2} + (3-4|x|) \left( 1 + \|D\bv_x\|^2_{\HS} \right)^\frac{1}{2} \\ + \sum_{j=1}^6 \left( \left( 1 + R^{ij}(x) \right)^\frac{1}{2} - 1 \right) \,.
 \end{multline}
Since the maps $x \mapsto \|g_i(x)\|$, $x \mapsto \|D\bv_x\|_{\HS}$ and $x \mapsto \J \by^i_x$ are $0$-homogeneous for all $i$ we obtain, thanks to Lemma~\ref{integral_of_dv},
\begin{align*}
& \int_{|x| = r} \sum_{i \in I_x} \|g_i\| \left(1 + \|D\bv_x\|_{\HS}^2\right)^\frac{1}{2} \, d\cH^{m-1}(x) \\
 & \qquad \qquad \qquad = \frac{d}{dr} \int_{\B_V(0,r)} \sum_{i \in I_x} \|g_i\| \left(1 + \|D\bv_x\|_{\HS}^2\right)^\frac{1}{2} \, dx \\
 & \qquad \qquad \qquad = mr^{m-1} \int_{\B_V(0,1)} \sum_{i \in I_x} \|g_i\| \left(1 + \|D\bv_x\|_{\HS}^2\right)^\frac{1}{2} \, dx \\
 & \qquad \qquad \qquad \leq mr^{m-1} \int_{\B_V(0,1)} \sum_{i \in I_x} \|g_i\| \left(1 + (\J \by^i_x)^2\right)^\frac{1}{2} \, dx + mr^{m-1}\bc_{\ref{integral_of_dv}}\|g_0\|\rho\epsilon \\
 & \qquad \qquad \qquad = \int_{|x| = r} \sum_{i \in I_x} \|g_i\| \left(1 + (\J \by^i_x)^2\right)^\frac{1}{2} \, d\cH^{m-1}(x) + mr^{m-1}\bc_{\ref{integral_of_dv}}\|g_0\|\rho\epsilon \,.
\end{align*}
Therefore,
\begin{equation*}
\begin{split}
\int_A & \sum_{i \in I_x} \|g_i\|\bigg( (4|x|-2) \left( 1 + (\J\by_x^i)^2 \right)^\frac{1}{2} + (3-4|x|) \left( 1 + \|D\bv_x\|^2_{\HS} \right)^\frac{1}{2} \bigg) dx \\
& = \int_{\frac{1}{2}}^\frac{3}{4} \bigg( (4r-2) \int_{|x|=r} \sum_{i \in I_x} \|g_i\|  \left( 1 + (\J\by_x^i)^2 \right)^\frac{1}{2} d\cH^{m-1}(x) \\
& \quad\quad\quad+ (3-4r) \int_{|x|=r} \sum_{i \in I_x} \|g_i\|  \left( 1 + \|D\bv_x\|_{\HS}^2 \right)^\frac{1}{2} d\cH^{m-1}(x) \bigg)  dr \\
& \leq \int_{\frac{1}{2}}^\frac{3}{4} \bigg( (4r-2) \int_{|x|=r} \sum_{i \in I_x} \|g_i\|  \left( 1 + (\J\by_x^i)^2 \right)^\frac{1}{2} d\cH^{m-1}(x) \\
& \quad\quad\quad + (3-4r)\int_{|x| = r} \sum_{i \in I_x} \|g_i\| \left(1 + (\J \by^i_x)^2\right)^\frac{1}{2} \, d\cH^{m-1}(x) \\
& \quad\quad\quad + (3-4r)mr^{m-1}\bc_{\ref{integral_of_dv}}\|g_0\|\rho\epsilon \bigg) dr\\
& \leq m\bc_{\ref{integral_of_dv}}\|g_0\|\rho\epsilon + \int_A \sum_{i \in I_x} \|g_i\|  \left( 1 + (\J\by_x^i)^2 \right)^\frac{1}{2} dx \,.
\end{split}
\end{equation*}
Thus it follows from \eqref{random.new.30} that
\begin{multline*}
\int_A \sum_{i \in I_x} \|g_i\| \left( 1 + (\J\bz_x^i)^2 \right)^\frac{1}{2} dx \leq m\bc_{\ref{integral_of_dv}}\|g_0\|\rho\epsilon + \int_A \sum_{i \in I_x} \|g_i\|  \left( 1 + (\J\by_x^i)^2 \right)^\frac{1}{2} dx \\ + \sum_{j=1}^6 \int_A \sum_{i \in I_x} \|g_i\|\left( \left( 1 + R^{ij}(x) \right)^\frac{1}{2} - 1 \right) dx \,.
\end{multline*}
In order to complete the proof of the Lemma it is therefore sufficient to establish that
\begin{equation*}
\int_A \sum_{i \in I_x} \|g_i\|\left( \left( 1 + R^{ij}(x) \right)^\frac{1}{2} - 1 \right) dx \leq \bc'''(m) \|g_0\|\rho^\frac{1}{2} \epsilon
\end{equation*}
for some $\bc'''(m)>0$ and all $j=1,\ldots,6$.
 
Clearly,
\begin{align*}
 \left(1 + R^{i1}(x) \right)^\frac{1}{2} - 1 & \leq R^{i1}(x) \leq \bc' \|D\bv_x\|^4 \leq \bc'\bc_{\ref{derivative_estimate2}}^4\epsilon^\frac{4}{3} \,, \\
 \left(1 + R^{i2}(x) \right)^\frac{1}{2} - 1 & \leq R^{i2}(x) \leq \bc'|\by^i(x) - \bv(x)|^2 \,.
\end{align*}
Setting $\delta \defl \rho^\frac{1}{2} < 1$ we obtain with \eqref{squaresd} of Lemma~\ref{squares},
\begin{align*}
 \left(1 + R^{i3}(x) \right)^\frac{1}{2} - 1 & = \left(1 + \bc'\|D\bv_x\||\by^i(x) - \bv(x)| \right)^\frac{1}{2} - 1 \,, \\
  & \leq  \frac{1}{2} \delta^{-2}\bc'^2|\by^i(x) - \bv(x)|^2 + \delta\left(\left(1 + \|D\bv_x\|^2 \right)^\frac{1}{2} - 1\right) \,, \\
  & \leq  \frac{1}{2} \bc'^2\rho^{-1}|\by^i(x) - \bv(x)|^2 + \rho^\frac{1}{2}\left(\left(1 + \|D\bv_x\|^2 \right)^\frac{1}{2} - 1\right) \,.
\end{align*}
Setting $\delta \defl \bc' |\by^i(x)-\bv(x)| \leq \bc' 3 \rho \leq 1$ provided $\rho$ is small enough depending on $m$, we infer from \eqref{squaresc} of Lemma~\ref{squares} that
\begin{align*}
 \left(1 + R^{i4}(x) \right)^\frac{1}{2} - 1 & = \left(1 + \bc'|\by^i(x) - \bv(x)|(\J \by^i_x)^2 \right)^\frac{1}{2} - 1 \\
  & \leq (3 \bc' \rho)^\frac{1}{2} \left( \left( 1 + (\J \by^i_x)^2\right)^\frac{1}{2} - 1\right) \,.
\end{align*}
Similarly, if $\epsilon$ is small enough such that $\bc'\bc_{\ref{derivative_estimate2}}\epsilon^\frac{1}{3} \leq 1$, then Lemma~\ref{derivative_estimate2} implies that $\bc'\|D\bv_x\| \leq 1$ and it follows from \eqref{squaresc} of Lemma~\ref{squares} that
\begin{align*}
 \left(1 + R^{i5}(x) \right)^\frac{1}{2} - 1 & = \left(1 + \bc'\|D\bv_x\| (\J \by^i_x)^2 \right)^\frac{1}{2} - 1 \\
  & \leq (\bc'\bc_{\ref{derivative_estimate2}})^\frac{1}{2}\epsilon^\frac{1}{6}\left(\left(1 + (\J \by^i_x)^2 \right)^\frac{1}{2} - 1\right) \,.
\end{align*}
Finally, applying again \eqref{squaresd} of Lemma~\ref{squares} with $\delta \defl \rho^\frac{1}{2}$,
\begin{align*}
 \left(1 + R^{i6}(x) \right)^\frac{1}{2} - 1 & = \left(1 + \bc'|\by^i(x) - \bv(x)| \|D\by_x^i\|\right)^\frac{1}{2} - 1 \,, \\
  & \leq  \frac{1}{2} \delta^{-2}\bc'^2|\by^i(x) - \bv(x)|^2 + \delta\left(\left(1 + \|D\by_x^i\|^2 \right)^\frac{1}{2} - 1\right) \,, \\
  & \leq  \frac{1}{2} \bc'^2\rho^{-1}|\by^i(x) - \bv(x)|^2 + \rho^\frac{1}{2}\left(\left(1 + (\J\by_x^i)^2 \right)^\frac{1}{2} - 1\right) \,.
\end{align*}

We next integrate these over $x \in A$. By \eqref{multiple_values_mass_estimate_prelim},
\begin{align*}
\int_{A} \sum_{i \in I_x} \|g_i\| \left[\left(1 + R^{i1}(x) \right)^\frac{1}{2} - 1\right] \, dx & \leq \bc'\epsilon^\frac{4}{3}\int_{A} \sum_{i \in I_x} \|g_i\| \, dx \\
 & = \bc'\epsilon^\frac{4}{3} \int_{E_1} \|g_0\| \, dx + \bc'\epsilon^\frac{4}{3} 15\|g_0\|\epsilon \\
 & \leq (\balpha(m) + 15)\bc' \|g_0\|\epsilon^\frac{4}{3} \,.
\end{align*}

For the remaining terms note that by \eqref{multiple_values_mass_estimate_prelim} and Lemma~\ref{dist_to_v},
\begin{align*}
\int_{A} \sum_{i \in I_x} \|g_i\| |\bv(x) - \by^i(x)|^2 \, dx & \leq \int_{E_1} \|g_0\| |\by(x) - \bv(x)|^2 \, dx + 9\rho^2\int_{E_2} \sum_{i \in I_x} \|g_i\| \, dx \\
 & \leq (\bc_{\ref{dist_to_v}} + 135)\|g_0\| \rho^2 \epsilon \,.
\end{align*}
Further by Lemma~\ref{integral_of_dv},
\begin{align*}
& \int_{A} \sum_{i \in I_x} \|g_i\| \left(\left(1 + \|D\bv_x\|^2 \right)^\frac{1}{2} - 1\right) \, dx \\
& \qquad \qquad \leq \bc_{\ref{integral_of_dv}}\|g_0\|\rho\epsilon + \int_{\B_V} \sum_{i \in I_x} \|g_i\| \left(\left(1 + \|D\by^i_x\|^2 \right)^\frac{1}{2} - 1\right) \, dx \\
 & \qquad \qquad \leq \bc_{\ref{integral_of_dv}}\|g_0\|\rho\epsilon + \int_{\B_V} \sum_{i \in I_x} \|g_i\| \left(1 + \|\J\by^i_x\|^2 \right)^\frac{1}{2} - \|g_0\| \, dx \\
 & \qquad \qquad \leq (\bc_{\ref{integral_of_dv}}\rho + 1)\|g_0\|\epsilon \,,
\end{align*}
and obviously,
\begin{align*}
\int_{A} \sum_{i \in I_x} \|g_i\| \left(\left(1 + (\J \by^i_x)^2 \right)^\frac{1}{2} - 1\right) \, dx & \leq \int_{\B_V} \sum_{i \in I_x} \|g_i\| \left(1 + (\J \by^i_x)^2 \right)^\frac{1}{2} - \|g_0\| \, dx \\
 & \leq \|g_0\|\epsilon \,.
\end{align*}
Combining these estimates we obtain
\begin{align*}
\int_{A} \sum_{i \in I_x} \|g_i\| \left[\left(1 + R^{i1}(x) \right)^\frac{1}{2} - 1\right] \, dx & \leq \bc''(m) \|g_0\|\epsilon^\frac{4}{3} \,,\\
\int_{A} \sum_{i \in I_x} \|g_i\| \left[\left(1 + R^{i2}(x) \right)^\frac{1}{2} - 1\right] \, dx & \leq \bc''(m) \|g_0\|\rho^2 \epsilon \,,\\
\int_{A} \sum_{i \in I_x} \|g_i\| \left[\left(1 + R^{i3}(x) \right)^\frac{1}{2} - 1\right] \, dx & \leq \bc''(m) \|g_0\|\rho^\frac{1}{2} \epsilon \,,\\
\int_{A} \sum_{i \in I_x} \|g_i\| \left[\left(1 + R^{i4}(x) \right)^\frac{1}{2} - 1\right] \, dx & \leq \bc''(m) \|g_0\|\rho^\frac{1}{2} \epsilon \,,\\
\int_{A} \sum_{i \in I_x} \|g_i\| \left[\left(1 + R^{i5}(x) \right)^\frac{1}{2} - 1\right] \, dx & \leq \bc''(m) \|g_0\|\epsilon^\frac{7}{6} \,, \\
\int_{A} \sum_{i \in I_x} \|g_i\| \left[\left(1 + R^{i6}(x) \right)^\frac{1}{2} - 1\right] \, dx & \leq \bc''(m) \|g_0\| \rho^\frac{1}{2} \epsilon\,,
\end{align*}
for some constant $\bc''(m) > 0$. Since we assume that $\epsilon \leq \rho^{6m} \leq \rho^{6} \leq 1$, the sum of all the integrals involving the $R$-terms is bounded by $\bc'''(m) \|g_0\|\rho^\frac{1}{2} \epsilon$ for some constant $\bc'''(m) > 0$. 
\end{proof}

\begin{Lem}
[{\cite[Lemma~12]{Reif}}]
\label{jacobianofu}
There is a constant $\bc_{\theThm}(m) > 0$ such that
\[
\int_{\B_V(0,1)} (1 + (\J \bv_x)^2)^\frac{1}{2} \, dx \leq \int_{\B_V(0,1)} (1 + \|D\bv_x\|_{\HS}^2)^\frac{1}{2}\, dx + \bc_{\theThm}\epsilon^\frac{4}{3} \,.
\]
\end{Lem}

\begin{proof}
By Lemma~\ref{derivative_estimate2} we know that $\|D\bv_x\| \leq \bc_{\ref{derivative_estimate2}} \epsilon^\frac{1}{3}$. Hence letting $M$ denote the matrix of $D\bv_x^* D\bv_x$ with respect to some orthonormal basis of $V$,
\[
(\J \bv_x)^2 = \|D\bv_x\|_{\HS}^2 + \sum_{\#K \geq 2} \det(M_K) \leq \|D\bv_x\|_{\HS}^2 + \bc \epsilon^\frac{4}{3} \,,
\]
for some constant $\bc(m) > 0$, and the statement follows by Lemma~\ref{squares}, i.e.\ by integrating the estimate
\[
(1 + (\J \bv_x)^2)^\frac{1}{2} \leq (1 + \|D\bv_x\|_{\HS}^2 + \bc \epsilon^\frac{4}{3})^\frac{1}{2} \leq (1 + \|D\bv_x\|_{\HS}^2)^\frac{1}{2} + \bc \epsilon^\frac{4}{3} \,.
\]
\end{proof}

Combining the results of this subsection we get:

\begin{Prop}
\label{masscomparison1}
There is a constant $\bc_{\theThm}(m) > 0$ with the following property. If $\rho \leq \rho_{\ref{mass_of_z}}$ and $\epsilon \leq \epsilon_{\ref{mass_of_z}}$, then
\begin{align*}
\bM(T \res Z_V(2^{-1})) & \leq \bc_{\theThm}\|g_0\|\rho\epsilon + \bM(P \res Z_V(2^{-1})) \,, \\
\bM(T) & \leq \bc_{\theThm}\|g_0\|\rho^\frac{1}{2}\epsilon + \bM(P) \,.
\end{align*}
\end{Prop}

\begin{proof}
By the construction of $T$, it is $T \res Z_V(\frac{3}{4})^c = P \res Z_V(\frac{3}{4})^c$. By Lemma~\ref{mass_of_z}
\[
\bM\left(T \res \{{\scriptstyle \frac{1}{2}} \leq |\pi_V(x)| \leq {\scriptstyle \frac{3}{4}}\}\right) \leq \bM\left(P \res \{{\scriptstyle \frac{1}{2}} \leq |\pi_V(x)| \leq {\scriptstyle \frac{3}{4}}\}\right) + \bc_{\ref{mass_of_z}}\|g_0\|\rho^\frac{1}{2}\epsilon \,.
\]
Further, by the definition of $T$, Lemma~\ref{jacobianofu} and Lemma~\ref{integral_of_dv} (note that $\epsilon^\frac{1}{3} \leq \rho$) applied to the integration domain $\B_V(0,1)$,
\begin{align*}
\bM\left(T \res Z_V(2^{-1})\right) & = \bM\left((\id_V + \bv)_\#(g_0 \curr{\B_V(0,2^{-1})} \right) \\
 & = \|g_0\| \int_{\B_V(0,2^{-1})} (1 + (\J \bv_x)^2)^\frac{1}{2} \, dx \\
 & \leq \|g_0\| \int_{\B_V(0,2^{-1})} (1 + \|D\bv_x\|_{\HS}^2)^\frac{1}{2}\, dx + \bc_{\ref{jacobianofu}}\|g_0\|\epsilon^\frac{4}{3} \\
 & \leq \int_{\B_V(0,2^{-1})} \sum_{i \in I_x} \|g_i\| (1 + \|D \by^i_x\|_{\HS}^2)^\frac{1}{2} \, dx + (\bc_{\ref{jacobianofu}} + \bc_{\ref{integral_of_dv}})\rho\|g_0\|\epsilon \\
 & \leq \int_{\B_V(0,2^{-1})} \sum_{i \in I_x} \|g_i\| (1 + (\J \by^i_x)^2)^\frac{1}{2} \, dx + (\bc_{\ref{jacobianofu}} + \bc_{\ref{integral_of_dv}})\rho\|g_0\|\epsilon \\
 & = \bM\left(P \res Z_V(2^{-1})\right) + (\bc_{\ref{jacobianofu}} + \bc_{\ref{integral_of_dv}})\rho\|g_0\|\epsilon \,.
\end{align*}
Adding up the masses of these two regions gives the result.
\end{proof}

\subsection{Plane selection via a quadratic form}
\label{secplaneselection}

Additionally to the assumptions on $P$ in the beginning of Subsection~\ref{assumptions} we further assume that $\rho \leq \rho_{\ref{mass_of_z}}$ (without mentioning it repeatedly) and that $\epsilon \leq \epsilon_{\ref{mass_of_z}}$. With $T_\bv \defl (\id_V + \bv)_\# (g_0 \curr{\B_V(0,2)})$ we denote the cone generated by the graph of $\bv$ over $V$ and weight $g_0$. Let $Q = Q(T_\bv,0,1)$ be the quadratic form associated to the chain $T_\bv$ as defined in Subsection~\ref{betaprelim}, i.e.\
\[
Q(y) \defl \frac{m+2}{\balpha(m)}\|g_0\|\int_{\B(0,1) \cap \spt(T_\bv)} \langle y, x \rangle^2 \, d\cH^m(x) \,.
\]



As noted before Lemma~\ref{newplaneclosetoold}, the quadratic form $Q$ is compact and hence $\Hi$ has an orthonormal basis $(e_k)_k$ of eigenvectors of $Q$. Let $W \in \bG(\Hi,m)$ be the $m$-plane spanned by $e_1,\dots,e_m$ corresponding to the $m$ largest eigenvalues of $Q$.

\begin{Lem}
\label{newplane}
There are constants $\bc_{\theThm}(m) > 0$ and $0 < \epsilon_{\theThm} \leq \epsilon_{\ref{mass_of_z}}$ with $\bc_{\theThm} (\epsilon_{\theThm})^\frac{1}{3} < \frac{1}{20}$ and the following property. If $0 < \epsilon < \epsilon_{\theThm}$, then:
\begin{enumerate}
	\item $\|\pi_W - \pi_V\| \leq \bc_{\theThm} \epsilon^\frac{1}{3}$.
	\item $\pi_W : \spt(T_\bv) \to W$ is injective and the map $\bw : \B_W(0,1) \to W^\perp$ with $\graph(\bw) = \spt(T_\bv \res Z_W)$ is $1$-homogeneous, continuously differentiable in $\B_W(0,1) \setminus \{0\}$ with $\|D\bw_x\| \leq \bc_{\theThm}\epsilon^\frac{1}{3}$.
	\item For $x \in \B_W(0,1)$, $y \in W^\perp$ with $|x|,|y| \leq 1$,
		\begin{equation}
		\label{oepsilonofu}
			\left|\int_{\Sp_W} \langle y, \bw(z)\rangle \langle x, z\rangle \, d\cH^{m-1}(z)\right| \leq \bc_{\theThm}\epsilon \,,
		\end{equation}
\end{enumerate}
\end{Lem}

\begin{proof}
By Lemma~\ref{derivative_estimate2} we have a height bound for $T_\bv$ over $\B_V(0,1)$ of $\bc_{\ref{derivative_estimate2}} \epsilon^\frac{1}{3}$ and by Proposition~\ref{masscomparison1} the excess is bounded by $\Exc_1(T_\bv,V) \leq 2^m\bc_{\ref{masscomparison1}}\|g_0\|\rho\epsilon + \|g_0\|\epsilon$. The first statement is now a direct consequence of Lemma~\ref{newplaneclosetoold} and there is some $\bc(m) > 0$ such that $\|\pi_W - \pi_V\| \leq \bc \epsilon^\frac{1}{3}$. We may assume that $\epsilon$ is small enough such that $\max\{\bc,\bc_{\ref{derivative_estimate2}}\} \epsilon^\frac{1}{3} \leq \frac{1}{4}$. For $v \in V$ we have $|v - \pi_W(v)| \leq \bc \epsilon^\frac{1}{3}|v| \leq \frac{1}{2}|v|$, respectively, $\frac{1}{2}|v| \leq |\pi_W(v)|$. Now, consider two points $p,p' \in \spt(T_\bv)$, i.e.\ $p = v + \bv(v)$ and $p' = v' + \bv(v')$ for $v,v' \in V$. By Lemma~\ref{derivative_estimate2}, $\bv$ is Lipschitz with constant $\bc_{\ref{derivative_estimate2}}\epsilon^\frac{1}{3}$ and hence
\begin{align*}
|\pi_{W^\perp}(p - p')| & = |p - p' - \pi_W(p - p')| \\
 & \leq |v - v' - \pi_W(v - v')| + |\bv(v) - \bv(v') - \pi_W(\bv(v) - \bv(v'))| \\
 & \leq \bc\epsilon^\frac{1}{3}|v - v'| + 2\bc_{\ref{derivative_estimate2}}\epsilon^\frac{1}{3}|v - v'| \\
 & \leq 2(\bc + 2\bc_{\ref{derivative_estimate2}})\epsilon^\frac{1}{3}|\pi_W(v - v')| \,.
\end{align*}
Also
\begin{align*}
|\pi_W(p - p')| & \geq |\pi_W(v - v')| - |\pi_W(\bv(v) - \bv(v'))| \\
 & \geq |\pi_W(v - v')| - \bc_{\ref{derivative_estimate2}}\epsilon^\frac{1}{3}|v - v'| \\
 & \geq (1 - 2\bc_{\ref{derivative_estimate2}}\epsilon^\frac{1}{3})|\pi_W(v - v')| \\
 & \geq \frac{1}{2}|\pi_W(v - v')| \,.
\end{align*}
Combined we get
\[
|\pi_{W^\perp}(p - p')| \leq \bc'\epsilon^\frac{1}{3}|\pi_W(p - p')| \,,
\]
for some constant $\bc'(m) > 0$. This in turn implies that $\bw$ is well defined and Lipschitz with the constant as stated in (2). Since $\pi_W \circ (\id_V + \bv)$ is of class $C^1$ outside the origin, it follows from the inverse function theorem that the same holds for $\bw$ and $\|D\bw_x\| \leq \Lip(\bw) \leq \bc'\epsilon^\frac{1}{3}$. This shows (2).

Let $D \defl \B(0,1) \cap \spt(T_\bv)$. For $x \in W$ and $y \in W^\perp$ with $|x|,|y| \leq 1$,
\begin{align*}
0 & = \frac{\balpha(m)}{\|g_0\|(m+2)}Q(x,y) = \int_{D} \langle x, z\rangle \langle y, z\rangle \, d\cH^m(z) \\
 & = \int_{\pi_W(D)} \langle x, w + \bw(w)\rangle \langle y, w + \bw(w) \rangle (1 + (\J \bw_w)^2)^\frac{1}{2} \, dw \\
 & = \int_{\pi_W(D)} \langle x, w\rangle \langle y, \bw(w) \rangle (1 + (\J \bw_w)^2)^\frac{1}{2} \, dw \,.
\end{align*}
Because of $\|D\bw_x\| \leq \bc'\epsilon^\frac{1}{3}$, we note as in the proof of Lemma~\ref{jacobianofu} that
\begin{align*}
1 & \leq (1 + (\J \bw_z)^2)^\frac{1}{2} \leq 1 + \bc''\epsilon^\frac{2}{3} \,,
\end{align*}
for some constant $\bc''(m) > 0$. Hence $|s(w)| \leq \bc''\epsilon^\frac{2}{3}$ for $s(w) \defl (1 + (\J \bw_w)^2)^\frac{1}{2} - 1$ and $w \in \B_W(0,1)$. Further, $\cH^m(\B_W(0,1) \setminus \pi_W(D)) \leq \bc' \epsilon^\frac{2}{3}$ since the height bound obtained above implies $\B_W(0,\sqrt{1-\bc'^2\epsilon^\frac{2}{3}}) \subset \pi_W(D)$ and $|\bw(w)| \leq \bc'\epsilon^\frac{1}{3} \leq 1$ if $\epsilon$ is small enough. Hence,
\begin{align*}
 & \biggl|\int_{\B_W(0,1)} \langle y, \bw(w)\rangle \langle x, w \rangle \, dw \biggr| \\
 & \qquad \leq \biggl|\int_{\B_W(0,1)} \langle y, \bw(w)\rangle \langle x, w \rangle (1 + s(w)) \, dw \biggr| \\
 & \qquad \quad + \biggl|\int_{\B_W(0,1)} \langle y, \bw(w)\rangle \langle x, w \rangle s(w) \, dw \biggr| \\
  & \qquad \leq \biggl|\int_{\B_W(0,1) \setminus \pi_W(D)} \langle y, \bw(w)\rangle \langle x, w \rangle (1 + s(w)) \, dw \biggr| + \balpha(m) \|\bw\|_\infty \|s\|_\infty \\
  & \qquad \leq \cH^m(\B_W(0,1) \setminus \pi_W(D)) \|\bw\|_\infty 2 + \balpha(m) \|\bw\|_\infty \|s\|_\infty \\
	& \qquad \leq \bc''' \epsilon \,,
\end{align*}
for some constant $\bc'''(m) > 0$. By the homogeneity of $\bw$ the last statement of the lemma follows.
\end{proof}

Let
\begin{equation}
\label{zeroharm}
\bw_0 \defl \frac{1}{\balpha(m) m} \int_{\Sp_W} \bw(x) \, d\cH^{m-1}(x) \in W^\perp \,,
\end{equation}
and note that $\balpha(m) m = \cH^{m-1}(\Sp_W)$. Instead of taking the harmonic extension of $\bw$ over $\Sp_W$ as in \cite{Reif}, we follow \cite{Mor} and define $\bh : \B_W(0,1) \to W^\perp$ by 
\[
\bh(tx) \defl \bw_0 + t^2\left(\bw(x) - \bw_0\right) \,,
\]
for $|x|=1$ and $t \in [0,1]$. For a point $x \in \Sp_W$ we denote with $D_S \bw_x$ the derivative of the restriction of $\bw$ to $\Sp_W$ at the point $x$. The following calculations are contained in \cite{Mor}.

\begin{Lem}
\label{energyformulas}
There is a constant $\bc_{\theThm}(m) > 0$ such that if $0 < \epsilon < \epsilon_{\ref{newplane}}$, then
\[
\int_{\B_W(0,1)} (1 + (\J \bh)^2)^\frac{1}{2} \leq \bc_{\theThm}\epsilon^\frac{4}{3} + \int_{\B_W(0,1)} (1 + \|D\bh\|_{\HS}^2)^\frac{1}{2} \,,
\]
and further
\begin{align*}
\int_{\B_W(0,1)} \|D\bh\|_{\HS}^2 & = \frac{1}{m+2} \int_{\Sp_W} 4|\bw(x) - \bw_0|^2  + \|D_S\bw_x\|_{\HS}^2 \, d\cH^{m-1}(x) \,, \\
\int_{\B_W(0,1)} \|D\bw\|_{\HS}^2 & = \frac{1}{m} \int_{\Sp_W} |\bw(x)|^2  + \|D_S\bw_x\|_{\HS}^2 \, d\cH^{m-1}(x) \,.
\end{align*}
\end{Lem}

\begin{proof}
By definition and the $1$-homogeneity of $\bw$ it is
\[
\bh(x) = \bw_0 + |x|^2 \left(\bw\left(x/|x|\right) - \bw_0\right) = \bw_0\left(1 - |x|^2\right) + |x|\bw(x) \,.
\]
Hence, 
\[
D\bh_x(w) = - 2\langle x, w \rangle \bw_0 + \langle x/|x|, w \rangle\bw(x) + |x| D\bw_x(w) \,.
\]
If $x \neq 0$ and $e_1, \dots, e_m$ is an orthonormal basis of $W$ with $e_1 = |x|^{-1}x$, then $|x|D\bw_x(e_1) = \bw(x)$ and further
\begin{align*}
D\bh_x(e_1) & = -2 |x|\bw_0 + \bw(x) + |x|D\bw_x(e_1) = 2(\bw(x) - |x|\bw_0) \\
 & = 2|x|(\bw(x/|x|) - \bw_0) \,.
\end{align*}
For $k \geq 2$ by the $1$-homogeneity of $\bw$,
\[
D\bh_x(e_k) = |x|D\bw_x(e_k) = |x| D_S \bw_{|x|^{-1}x}(e_k) \,.
\]
Hence,
\begin{align*}
\|D\bh_x\|_{\HS}^2 & = \sum_k |D\bh_x(e_k)|^2 \\
 & = |D\bh_x(e_1)|^2 + |x|^2 \sum_{k \geq 2} |D_S \bw_{x/|x|}(e_k)|^2 \\
 & = 4|x|^2|\bw(x/|x|) - \bw_0|^2 + |x|^2 \|D_S \bw_{x/|x|}\|_{\HS}^2 \,,
\end{align*}
and in particular with the definition of $\bw_0$ and Lemma~\ref{newplane},
\[
\|D\bh_x\|_{\HS}^2 \leq 16 \mathrm{Lip}(\bw)^2  + \|D\bw_x\|_{\HS}^2 \leq 17\bc_{\ref{newplane}} \epsilon^\frac{2}{3} \,.
\]
The first statement follows now exactly as in Lemma~\ref{jacobianofu}. Similarly,
\[
\|D\bw_x\|_{\HS}^2 = |\bw(x/|x|)|^2 + \|D_S \bw_{x/|x|}\|_{\HS}^2 \,.
\]
Integrating gives
\begin{align*}
\int_{\B_W(0,1)} \|D\bh\|_{\HS}^2 & = \int_0^1 dr \int_{\Sp_W(r)} \|D\bh_x\|_{\HS}^2 \, d\cH^{m-1}(x) \\
 & = \int_0^1 r^2 \, dr \int_{\Sp_W(r)} 4|\bw(x/r) - \bw_0|^2 + \|D_S \bw_{x/r}\|_{\HS}^2 \, d\cH^{m-1}(x)\\
 & = \int_0^1 r^{m + 1} \, dr \int_{\Sp_W} 4|\bw(x) - \bw_0|^2 + \|D_S \bw_x\|_{\HS}^2 \, d\cH^{m-1}(x) \\
 & = \frac{1}{m+2} \int_{\Sp_W} 4|\bw(x) - \bw_0|^2  + \|D_S\bw_x\|_{\HS}^2 \, d\cH^{m-1}(x) \,.
\end{align*}
And similarly,
\begin{align*}
\int_{\B_V(0,1)} \|D\bw\|_{\HS}^2 & = \int_0^1 r^{m-1} \, dr \, \int_{\Sp_W} |\bw(x)|^2 + \|D_S \bw_{x}\|_{\HS}^2 \, d\cH^{m-1}(x) \\
 & =  \frac{1}{m} \int_{\Sp_W} |\bw(x)|^2 + \|D_S \bw_{x}\|_{\HS}^2 \, d\cH^{m-1}(x) \,.
\end{align*}
\end{proof}

It is $L^2(\Sp_W) = \oplus_{i=0}^\infty \calH_i(\Sp_W)$ where $\calH_i(\Sp_W)$ denotes the Hilbert space of harmonic polynomials of degree $i$ on $\Sp_W = \Sp^{m-1}$. Hence for a fixed orthonormal basis $(e_l)_{l > m}$ of $W^\perp$ we can write the restriction $w^{l} = \langle \bw, e_l\rangle$ as a sum $\sum_{i \geq 0}w^l_i p^l_i$ with $w^l_i \in \R$ and $p^l_i \in \calH_i(\Sp_W)$ of unit $L^2$-norm. It is understood that the partial sums converge in $L^2$ to $w^{l} : \B_W(0,1) \to \R$. Every $p^l_i$ is the restriction of a harmonic homogeneous polynomial of degree $i$ (the extension is just given by $p^l_i(rx) = r^ip^l_i(x)$ for $r \geq 0$), which is for simplicity also denoted by $p^l_i$. Thus
\begin{align*}
\int_{\Sp_W} |\bw(x)|^2 & = \int_{\Sp_W} \sum_{m < l} |w^l(x)|^2 = \sum_{m < l} \int_{\Sp_W} |w^l(x)|^2 \\
 & = \sum_{m < l} \sum_{i = 0}^\infty \int_{\Sp_W} (w_i^l)^2|p_i^l(x)|^2 = \sum_{m < l} \sum_{i = 0}^\infty (w_i^l)^2 \\
 & \leq \balpha(m)m\bc_{\ref{newplane}}^2 \epsilon^\frac{2}{3} \,,
\end{align*}
if we assume that $\epsilon$ is small enough. In particular, for each fixed $i$, $\sum_{m < l} (w^l_i)^2 < \infty$. Furthermore, since $\calH_i(\Sp_W)$ is finite dimensional, there is a constant $C(i,m)$ depending only on $i$ and $m$ such that $|p_i^l(x)| \leq C(i,m)$ for all $|x| \leq 1$. Hence,
\[
\sup_{|x| \leq 1} \biggl| \sum_{m < N < l \leq N'} w_i^l p_i^l(x) e_l \biggr|^2 \leq C(i,m)^2\sum_{m < N < l \leq N'} (w_i^l)^2  \to 0 \,,
\]
as $N,N' \to \infty$,
and the partial sums
\[
\biggl(x \mapsto \sum_{m < l \leq N} w_i^l p_i^l(x) e_l\biggr)_{N > m}
\]
converge uniformly with respect to the norm in $\Hi$ for $N \to \infty$ to some $\bw_i : \B_W(0,1) \to W^\perp$. Similarly, the norms of the derivatives up to some fixed order of these partial sums converge uniformly as well, hence $\bw_i$ is smooth because the partial sums are. Note that $\bw_0$ is the constant map with value as defined in \eqref{zeroharm}. Finally, it is not hard to see that the partial sums $\sum_i \bw_i$ converge to $\bw$ in $L^2(\Sp_W,W^\perp)$.

\begin{Lem}
\label{oepsilonlinear}
There is a constant $\bc_{\theThm}(m) > 0$ with the following property. If $0 < \epsilon < \epsilon_{\ref{newplane}}$, then $\|\bw_1\|_\infty \defl \sup_{|x| \leq 1}|\bw_1(x)| \leq \bc_{\theThm}(m)\epsilon$.
\end{Lem}

\begin{proof}
For $i \neq 1$, $x \in W$ and $y \in W^\perp$,
\[
\int_{\Sp_W} \langle y, \bw_i(z) \rangle \langle x, z \rangle \, d\cH^{m-1}(z) = 0 \,,
\]
since $\calH_1 \perp \calH_i$ if $i \neq 1$ and $z \mapsto \langle x,z \rangle \in \calH_1$. Because $\sum_i \bw_i$ converge to $\bw$ in $L^2$ and \eqref{oepsilonofu} we get
\[
\left| \int_{\Sp_W} \langle y, \bw_1(z) \rangle \langle x, z \rangle \, d\cH^{m-1}(z) \right| \leq \bc_{\ref{newplane}}\epsilon \,,
\]
for all $x \in W$ and $y \in W^\perp$ with $|x|,|y| \leq 1$. Since $w^l_1 p^l_1 : W \to \R$ is a linear functional it is represented by some vector $w_l \in W$ with $\langle e_l, \bw_1(z)\rangle = \langle w_l, z\rangle$. If we set $x = \sum_k \kappa^k e_k$,  $y = \sum_l\lambda^l e_l$ and $w_l = \sum_{k'}w_l^{k'}e_{k'} \in \Hi$ it follows
\begin{align*}
\int_{\Sp_W} \langle y, \bw_1(z) \rangle \langle x, z \rangle \, d\cH^{m-1}(z) & = \sum_{k,l,l'} \int_{\Sp_W} \langle \lambda^l e_l, \langle w_{l'}, z \rangle e_{l'} \rangle \langle \kappa^k e_k, z \rangle \, d\cH^{m-1}(z) \\
 & = \sum_{k,l} \int_{\Sp_W} \kappa^k \lambda^l \langle w_l, z \rangle \langle e_k, z \rangle \, d\cH^{m-1}(z) \\
 & = \sum_{k,k',l} \int_{\Sp_W} \kappa^k \lambda^l w_l^{k'} \langle e_{k'}, z \rangle \langle e_k, z \rangle \, d\cH^{m-1}(z) \\
 & = \sum_{k,l} \int_{\Sp_W} \kappa^k \lambda^l w_l^{k} \langle e_{k}, z \rangle^2 \, d\cH^{m-1}(z) \\
 & = \sum_{k,l} \bc(m) \kappa^k \lambda^l w_l^{k} = \bc(m) \sum_{k,l,l'} \langle \lambda^l e_l, \kappa^k w_{l'}^{k} e_{l'} \rangle \\
 & = \bc(m) \sum_{l,l'} \langle \lambda^l e_l, \langle w_{l'}, x \rangle e_{l'} \rangle = \bc(m) \, \langle y, \bw_1(x) \rangle \,,
\end{align*}
for some $\bc(m) > 0$ depending only on $m$. Hence letting $y = \bw_1(x)$ and referring to \eqref{oepsilonofu},
\[
|\bw_1(x)| = \left\langle \frac{\bw_1(x)}{|\bw_1(x)|}, \bw_1(x)\right\rangle\leq \bc(m)^{-1} \bc_{\ref{newplane}} \epsilon \,.
\]
Since $x$ is arbitrary the proof is complete.
\end{proof}

For $p,q \in \calH_k(\Sp^{m-1})$ it is a standard exercise using the divergence theorem to show that (with $\nabla$ we understand the gradient in $\R^m$)
\[
\int_{\Sp^{m-1}} \nabla p \cdot \nabla q \, d\cH^{m-1} = k(m + 2k - 2) \int_{\Sp^{m-1}} pq \, d\cH^{m-1} \,,
\]
whereas for harmonic polynomials $p$ and $q$ of different degree both integrals are zero. With $\partial_r p = k p$ we get (for $|x| = 1$)
\[
\nabla p \cdot \nabla p = \|Dp\|_{\HS}^2 = k^2 p^2 + \|D_S p\|_{\HS}^2 \,,
\]
and hence
\begin{align*}
\int_{\Sp^{m-1}} \|D_S p\|_{\HS}^2 \, d\cH^{m-1} & = k(m + 2k - 2) \int_{\Sp^{m-1}} p^2 \, d\cH^{m-1} - k^2 \int_{\Sp^{m-1}} p^2 \, d\cH^{m-1} \\
 & = k(m + k - 2) \int_{\Sp^{m-1}} p^2 \, d\cH^{m-1} \,.
\end{align*}
More generally, if $f = \sum_k f_k p_k$ is a finite sum of homogeneous harmonic polynomials $p_k$ of degree $k$ and $L^2$-norm $1$ we get as in \cite{Mor}
\begin{equation}
\label{harmoniccoordinates}
\int_{\Sp^{m-1}} \|D_S f\|_{\HS}^2 \, d\cH^{m-1} = \sum_k f_k^2 k(m + k - 2) \,.
\end{equation}
The restriction of any polynomial to the sphere is the finite sum of homogeneous harmonic polynomials and the formula remains true for $f$ in $C^1(\Sp^{m-1},\R)$ by approximation.

Let $\bw = \bw_0 + \bw_1 + \sum_{k \geq 2} \bw_k$ as before with convergence in $L^2(\Sp^{m-1},W^\perp)$, whereas with uniform convergence on $\Sp^{m-1}$,
\[
\bw_k = \sum_{l > n} w_k^l p_k^l e_l \,.
\]
Note that $p_0^l = \cH^{m-1}(\Sp^{m-1})^{-\frac{1}{2}} = (\balpha(m) m)^{-\frac{1}{2}}$ by normalization. As in \cite{Mor}, although there without the error term in $\epsilon$, we get the following estimate.

\begin{Lem}
\label{harmoniccoordinates2}
There is some constant $\bc_{\theThm}(m) > 0$ such that if $0 < \epsilon < \epsilon_{\ref{newplane}}$, then
\[
\int_{\B_W(0,1)} \|D\bh\|_{\HS}^2 \leq \bc_{\theThm}\epsilon^2 + \frac{2m}{2m+1} \int_{\B_W(0,1)} \|D\bw\|_{\HS}^2 \,.
\]
\end{Lem}

\begin{proof}
With Lemma~\ref{energyformulas} and \eqref{harmoniccoordinates} we get
\begin{align*}
\int_{\B_W(0,1)} \|D\bh\|_{\HS}^2 & = \frac{1}{m+2} \int_{\Sp_W} 4|\bw(x) - \bw_0|^2  + \|D_S\bw_x\|_{\HS}^2 \, d\cH^{m-1}(x) \\
 & = \frac{1}{m+2} \biggl[4 \sum_{l > m, k \geq 1} (w^l_k)^2 + \sum_{l > m, k \geq 1} (w^l_k)^2 k(m + k - 2)\biggr] \,.
\end{align*}
By Lemma~\ref{oepsilonlinear} it is
\[
\sum_{l > m} (w^l_1)^2 = \|\bw_1\|_{{\scriptscriptstyle \mathrm{L^2}}}^2 \leq \bc'\|\bw_1\|_\infty^2 \leq \bc\epsilon^2 \,,
\]
for some constants $\bc'(m) > 0$ and $\bc(m) = \bc_{\ref{oepsilonlinear}}(m) \bc'(m) > 0$. Hence,
\begin{align*}
\int_{\B_W(0,1)} \|D\bh\|_{\HS}^2 & \leq \bc\epsilon^2 + \frac{1}{m+2} \sum_{l > m, k \geq 2} (w^l_k)^2 [4 + k(m + k - 2) ] \,.
\end{align*}
Similarly by Lemma \ref{energyformulas},
\begin{align*}
\int_{\B_W(0,1)} \|D\bw\|_{\HS}^2 & =\frac{1}{m} \sum_{l > n, k \geq 0} (w^l_k)^2 [1 + k(m + k - 2) ] \,.
\end{align*}
As in \cite{Mor},
\begin{align*}
& \int_{\B_W(0,1)} \|D\bh\|_{\HS}^2 \leq \bc\epsilon^2 + \frac{1}{m+2} \sum_{l > m, k \geq 2} (w^l_k)^2 [4 + k(m + k - 2) ] \\
 & \qquad = \bc\epsilon^2 + \frac{1}{m+2} \sum_{l > m, k \geq 2} (w^l_k)^2 [1 + k(m + k - 2)] \left[1 + \frac{3}{1 + k(m + k - 2)}\right] \\
 & \qquad \leq \bc\epsilon^2 + \frac{m}{m+2} \left[1 + \frac{3}{1 + 2m}\right] \int_{\B_W(0,1)} \|D\bw\|_{\HS}^2 \\
 & \qquad = \bc\epsilon^2 + \frac{2m}{2m+1} \int_{\B_W(0,1)} \|D\bw\|_{\HS}^2 \,.
\end{align*}
\end{proof}

\subsection{A competitor better than the cone}

In Subsection~\ref{massest} we defined the comparison surface $T$ and in Proposition~\ref{masscomparison1} we saw that $\bM(T)$ can't be much bigger than $\bM(P)$. We now modify $T$ on the cylinder $Z_V(\frac{1}{4})$ to get a new chain $S$. We work on a $4$-times smaller scale than in the last subsection, so we replace the $\bh$ constructed there by $\bh_4 \defl 4^{-1} \bh(4x)$. The new chain $S$ is defined by
\begin{enumerate}
\item $S \res Z_W(\frac{1}{4})^c \defl T \res Z_W(\frac{1}{4})^c$,
\item $S \res Z_W(\frac{1}{4}) \defl (\id_W + \bh_4)_\#(g_0 \curr{\B_W(0,\frac{1}{4})})$.
\end{enumerate}
The plane $W$ is close to $V$ by Lemma~\ref{newplane}. So we can assume that $\epsilon$ is small enough such that
\begin{equation}
\label{otherplaneintersection}
\spt(T) \cap Z_W(4^{-1}) \subset \spt(T) \cap Z_V(2^{-1}) \,, \quad \mbox{and}
\end{equation}
\begin{equation}
\label{bounderywithball}
S \res \B(0,1)^c = T \res \B(0,1)^c = P \res \B(0,1)^c \,.
\end{equation}
By the construction of $\bh$ and because $\bw$ is $1$-homogeneous we have $\bh_4(x) = \bw(x)$ for $|x| = \frac{1}{4}$ and hence $\partial P = \partial T = \partial S$. Note that the comparison surface $S$ is obtained via a push-forward from $T$ and hence also from $P$, i.e.\ $S = \psi_\# P$ for some Lipschitz map $\psi : \Hi \to \Hi$ with $\spt(\psi)$ contained in a neighborhood of $0$ (for example, $\B_{V}(0,1) + \B_{V^\perp}(0,2\rho)$ would be okay).

Before we give the main result, we compare the mass over the planes $W$ and $V$. The Lemma is essentially contained in \cite{Reif}.

\begin{Lem}
[{\cite[Lemma~14]{Reif}}]
\label{planeselection}
There are constants $\bc_{\theThm}(m) > 0$ and $0 < \epsilon_{\theThm} \leq \epsilon_{\ref{newplane}}$ such that if $0 < \rho, \epsilon < \epsilon_{\theThm}$, then
\[
\bM(T_\bv \res Z_W) \leq \bM(T_\bv \res Z_V)
\]
implies
\[
\bigl|\bM(P \res Z_W) - \bM(P \res Z_{V}) - [\bM(T_\bv \res Z_W) - \bM(T_\bv \res Z_{V})]\bigr| \leq \bc_{\theThm}\|g_0\|\rho\epsilon \,.
\]
\end{Lem}

\begin{proof}
As noted in \cite[Lemma~13]{Reif}, an elementary calculation shows that
\begin{align*}
\bM(T_\bv \res Z_W) & = \int_{\B_{V}} \|g_0\|(1 + (\J \bv_x)^2)^\frac{1}{2} \left( \frac{|x|^2 + |\bv(x)|^2 - k(x)^2}{|x|^2} \right)^{-\frac{m}{2}} dx \,, \\
\bM(P \res Z_W) & = \int_{\B_{V}} \sum_{i \in I_x} \|g_i\|(1 + (\J \by^i_x)^2)^\frac{1}{2} \left( \frac{|x|^2 + |\by^i(x)|^2 - k^i(x)^2}{|x|^2} \right)^{-\frac{m}{2}} dx \,,
\end{align*}
where $k(x) \defl |\pi_{W^\perp}(x + \bv(x))|$ and $k^i(x) \defl |\pi_{W^\perp}(x + \by^i(x))|$. 
It is $|\by^i(x)| \leq \rho$ for $|x| \leq 1$ by the assumption on $P$. Using $\epsilon \leq \rho^3$,
\begin{align*}
k^i(x) & = |x + \by^i(x) - \pi_W(x + \by^i(x))| \leq 2\rho + |x - \pi_W(x)| \leq 2\rho + \bc_{\ref{newplane}}\epsilon^\frac{1}{3} \\
 & \leq (2 + \bc_{\ref{newplane}})\rho \,.
\end{align*}
Similarly, $k(x) \leq (2 \bc_{\ref{derivative_estimate2}} + \bc_{\ref{newplane}}) \epsilon^\frac{1}{3}$. For $i \in I_x$ let
\begin{align*}
\delta(x) & \defl \left( \frac{|x|^2 + |\bv(x)|^2 - k(x)^2}{|x|^2} \right)^{-\frac{m}{2}} - 1 \,, \\
\delta^i(x) & \defl \left( \frac{|x|^2 + |\by^i(x)|^2 - k^i(x)^2}{|x|^2} \right)^{-\frac{m}{2}} - 1 \,.
\end{align*}
The Taylor polynomial approximation for $t \in [-\frac{1}{2},\frac{1}{2}]$ implies that $(1 + t)^{-\frac{m}{2}} = 1 -\frac{m}{2}t + R(t)$ and $R(t) = \frac{m(m+2)}{8}(1 + \xi)^{-\frac{m+2}{2}}t^2$ for some $\xi \in [-\frac{1}{2},\frac{1}{2}]$. So, if $\epsilon$ and $\rho$ are small, there is a constant $\bc(m) > 0$ such that 
\begin{align}
\label{diff1}
\left|\delta(x) - \frac{m}{2}\frac{k(x)^2 - |\bv(x)|^2}{|x|^2}\right| & \leq \bc\epsilon^\frac{4}{3} \,, \\
\label{diff2}
\left|\delta^i(x) - \frac{m}{2}\frac{k^i(x)^2 - |\by^i(x)|^2}{|x|^2}\right| & \leq \bc \rho^4 \,.
\end{align}
Hence, for some $\bc'(m) > 0$, both $|\delta^i(x)|$ and $|\delta(x)|$ are bounded by $\bc' \rho^2$. It is
\begin{align*}
\bM(P \res Z_W) - \bM(P \res Z_{V}) & = \int_{\B_{V}} \sum_{i \in I_x} \|g_i\|(1 + (\J \by^i_x)^2)^\frac{1}{2} \delta^i(x) \, dx \\
 & = \int_{\B_{V}} \sum_{i \in I_x} \|g_i\|\delta^i(x) \, dx \\
 & \quad + \int_{\B_{V}} \sum_{i \in I_x} \|g_i\|\delta^i(x)\left((1 + (\J \by^i_x)^2)^\frac{1}{2} - 1\right) \, dx \,,
\end{align*}
and therefore
\begin{equation}
\label{perror}
\biggl|\bM(P \res Z_W) - \bM(P \res Z_{V}) - \int_{\B_{V}} \sum_{i \in I_x} \|g_i\|\delta^i(x) \, dx \biggr| \leq \bc'\|g_0\|\rho^2\epsilon \,,
\end{equation}
by the excess estimate on $P$ and the bound on $|\delta^i(x)|$ above. Similarly it follows from Lemma~\ref{integral_of_dv} and Lemma~\ref{jacobianofu},
\begin{align}
\nonumber
& \left|\bM(T_\bv \res Z_W) - \bM(T_\bv \res Z_{V}) - \int_{\B_{V}} \|g_0\|\delta(x) \, dx \right| \\
\nonumber
& \quad \leq \int_{\B_{V}} |\delta(x)| \|g_0\|\left((1 + (\J \bv_x)^2)^\frac{1}{2} - 1\right) \, dx \\
\nonumber
& \quad \leq \bc'\rho^2\biggl((\bc_{\ref{integral_of_dv}} + \bc_{\ref{jacobianofu}})\|g_0\|\epsilon + \int_{\B_{V}} \sum_{i \in I_x} \|g_i\|(1 + (\J \by^i_x)^2)^\frac{1}{2} - \|g_0\| \, dx\biggr) \\
\label{tprimeerror}
& \quad \leq \bc'' \|g_0\|\rho^2\epsilon \,,
\end{align}
for some constant $\bc''(m) > 0$. Since we assume that $\bM(T_\bv \res Z_W) \leq \bM(T_\bv \res Z_V)$ it follows from \eqref{tprimeerror} that $\int_{\B_{V}} \|g_0\| \delta(x) \, dx \leq \bc'' \|g_0\|\rho^2\epsilon$. Then by \eqref{diff1},
\begin{align*}
\|g_0\|\int_{\B_{V}} \frac{m}{2}\frac{k(x)^2 - |\bv(x)|^2}{|x|^2} \, dx & \leq \|g_0\| \int_{\B_{V}} \delta(x) + \bc\epsilon^\frac{4}{3} \, dx \\
 & \leq (\bc'' + \balpha(m)\bc)\|g_0\|\rho^2\epsilon \,.
\end{align*}
By $1$-homogeneity of $k$ and $\bv$,
\begin{align}
\nonumber
\|g_0\|\int_{\B_{V}} k(x)^2 - |\bv(x)|^2 \, dx & = \|g_0\|\frac{m}{m+2} \int_{\B_{V}} \frac{k(x)^2 - |\bv(x)|^2}{|x|^2} \, dx \\
\nonumber
 & \leq \frac{2}{m} (\bc'' + \balpha(m)\bc)\|g_0\|\rho^2\epsilon \\
\label{kminusbvest}
 & \leq \bc'''\|g_0\|\rho\epsilon \,,
\end{align}
for some $\bc'''(m) > 0$. Because of \eqref{multiple_values_estimate_prelim}, Lemma~\ref{derivative_estimate2}, Lemma~\ref{integral_of_dv} and the assumption on $P$ and assuming $\epsilon$ is small,
\begin{align}
\nonumber
\|g_0\|\int_{\B_{V}} \sum_{i \in I_x} |\bv(x)|^2 \, dx & \leq \|g_0\|\int_{E_2} \#I_x |\bv(x)|^2 \, dx + \|g_0\|\int_{E_1} |\bv(x)|^2 \, dx \\
\nonumber
 & \leq 10\bc_{\ref{derivative_estimate2}}\|g_0\|\epsilon^\frac{5}{3} + \|g_0\|\int_{\B_{V}} |\bv(x)|^2 \, dx \\
\nonumber
 & \leq 10\bc_{\ref{derivative_estimate2}}\|g_0\|\epsilon + 4\|g_0\|\int_{\B_{V}} (1 + |\bv(x)|^2)^\frac{1}{2} - 1 \, dx \\
\nonumber
 & \leq 10\bc_{\ref{derivative_estimate2}}\|g_0\|\epsilon + 4\|g_0\|\int_{\B_{V}} (1 + \|D\bv_x\|_{\HS}^2)^\frac{1}{2} - 1 \, dx \\
\nonumber
 & \leq (10\bc_{\ref{derivative_estimate2}} + 4\bc_{\ref{integral_of_dv}})\|g_0\|\epsilon \\
\nonumber
 & \quad + 4\int_{\B_{V}} \sum_{i \in I_x}\|g_i\|(1 + (\J \by^i_x)^2)^\frac{1}{2} - \|g_0\| \, dx \\
\label{integrationofbv}
 & \leq \bc''''\|g_0\|\epsilon \,,
\end{align}
for some $\bc''''(m) > 0$. Using a similar calculation, the bound $k(x) \leq (2 \bc_{\ref{derivative_estimate2}} + \bc_{\ref{newplane}}) \epsilon^\frac{1}{3}$ and \eqref{kminusbvest}, $\bc''''$ can be assumed big enough such that also
\begin{align}
\label{boundonk}
\|g_0\|\int_{\B_{V}} \sum_{i \in I_x} k(x)^2 \, dx \leq \bc''''\|g_0\|\epsilon \,.
\end{align}

There is a Lipschitz constant $L(m) > 0$ for which $|(1 + s)^{-\frac{m}{2}} - (1 + t)^{-\frac{m}{2}}| \leq L|s - t|$ if $|s|,|t|\leq \frac{1}{2}$. The assumption on $P$ and the $1$-homogeneity of $\bv$,$\by^i$,$k$ and $k^i$ imply with \eqref{multiple_values_estimate_prelim}, \eqref{multiple_values_mass_estimate_prelim}, \eqref{diff1} and \eqref{diff2},
\begin{align*}
& \int_{\B_{V}} \biggl|\|g_0\|\delta(x) - \sum_{i \in I_x} \|g_i\|\delta^i(x)\biggr| \, dx \\
& \quad = \|g_0\|\int_{\B_{V}} |(1 - \#I_x) \delta(x)| \, dx + \int_{\B_{V}} \biggl|\sum_{i \in I_x} \|g_0\|\delta^i(x) - \|g_i\|\delta^i(x)\biggr| \, dx  \\
& \quad \quad + \|g_0\|\int_{\B_{V}} \sum_{i \in I_x} |\delta(x) - \delta^i(x)| \, dx  \\
& \quad \leq \bc'\rho^2 \|g_0\| \int_{E_2} \#I_x \, dx + \bc'\rho^2 \int_{E_2} \#I_x\|g_0\| + \sum_{i \in I_x} \|g_i\| \, dx \\
& \quad \quad + L\|g_0\|\int_{\B_{V}} \sum_{i \in I_x} \biggl|\frac{ k(x)^2 - |\bv(x)|^2}{|x|^2} - \frac{k^i(x)^2 - |\by^i(x)|^2}{|x|^2}\biggr| \, dx \\
& \quad \leq 35\bc'\|g_0\|\rho^2\epsilon + 3L\|g_0\|\int_{\B_{V}} \sum_{i \in I_x} \left||\bv(x)|^2 - |\by^i(x)|^2\right| + \left|k(x)^2 - k^i(x)^2\right| \, dx \,.
\end{align*}
Instead of the factor $3$ above we could use $\frac{m+2}{m}$. Since $|k(x) - k^i(x)| \leq |\bv(x) - \by^i(x)|$ by the definition of $k$ and $k^i$, there holds
\begin{align*}
\left||\bv(x)|^2 - |\by^i(x)|^2\right| & + \left|k(x)^2 - k^i(x)^2\right| \\
& \leq |\bv(x) - \by^i(x)|^2 + 2|\bv(x)||\bv(x) - \by^i(x)| \\
& \quad + |k(x) - k^i(x)|^2 + 2k(x)|k(x) - k^i(x)| \\
& \leq 2|\bv(x) - \by^i(x)|^2 + 2(|\bv(x)| + k(x))|\bv(x) - \by^i(x)| \,.
\end{align*}
Hence by Lemma~\ref{dist_to_v}, the fact that $2ab \leq \rho a^2 + \rho^{-1} b^2$, \eqref{integrationofbv} and \eqref{boundonk},
\begin{align*}
& \left| \int_{\B_{V}} \|g_0\|\delta(x) - \sum_{i \in I_x} \|g_i\|\delta^i(x) \, dx \right| \\
& \qquad \leq 35\bc'\|g_0\|\rho^2\epsilon  \\
& \qquad \quad + 3L\|g_0\|\int_{\B_{V}} \sum_{i \in I_x} 2|\bv(x)||\bv(x) - \by^i(x)| + 2k(x)|\bv(x) - \by^i(x)| \, dx \\
& \qquad \leq 35\bc'\|g_0\|\rho^2\epsilon + 3L\|g_0\|\int_{\B_{V}} \sum_{i \in I_x} \rho |\bv(x)|^2 + \rho k(x)^2 + 2\rho^{-1} |\bv(x) - \by^i(x)|^2 \, dx \\
& \qquad \leq (35\bc' + 6\bc''''L + 6\bc_{\ref{dist_to_v}}L)\|g_0\|\rho\epsilon \,.
\end{align*}
If we apply this in \eqref{perror} and \eqref{tprimeerror}, the lemma follows.
\end{proof}

Next we give another technical lemma that will be used in the main result of this section. It essentially tells that if a rather flat cone is far away from some plane, then the cylindrical excess over this plane is large.

\begin{Lem}
\label{newplaneheight}
Let $C \in \cP_m(\Hi;G)$ be an infinite (or large enough) cone with center $0$. Assume that there is some $0 < \delta < \frac{1}{5}$ and $V \in \bG(\Hi,m)$ such that,
\begin{enumerate}
	\item $\spt(\partial (C \res Z_V(0,2))) \subset \Hi \setminus Z_V(0,2)$,
	\item $\pi_{V\#} (C \res Z_V) = g_0 \curr{\B_V(0,1)}$,
	\item $|\pi_{V^\perp}(x)| \leq \delta$ for $x \in \spt(C) \cap Z_V$.
\end{enumerate}
If for some $\lambda \geq 2$ and $U \in \bG(\Hi,m)$ there holds
\[
\frac{1}{5} > \|\pi_U - \pi_V\| > 6\lambda\delta \,,
\]
then $|\pi_{U^\perp}(x)| > 2(\lambda-2)\delta$ for some $x \in \spt(C) \cap Z_U$ and
	\[
	\Exc_1(C,U) \geq \|g_0\|\frac{\balpha(m)}{3^{m+1}}\lambda^2\delta^2 \,.
	\]
\end{Lem}

\begin{proof}
Note first that Lemma~\ref{differentplane} implies $\pi_{U\#} (C \res Z_U) = g_0 \curr{\B_U(0,1)}$ for the same group element $g_0 \in G$. Let $A$ be the set of points $u \in \B_U(0,1)$ for which $|\pi_{V^\perp}(u)| \leq 2\lambda\delta$. $A$ is closed, contains the origin and by the triangle inequality it is also convex. By assumption and Lemma~\ref{lemma.norms.1} there is some $u \in \B_U(0,1)$ such that $|\pi_{V^\perp}(u)| = |\pi_V(u) - u| > 6\lambda\delta$. Hence $\frac{1}{3}u \notin A$. By the hyperplane separation theorem there is a half-plane $U_0$ of $U$ that contains $\frac{1}{3}u$ and is disjoint from $A$. Hence there is some ball $A' \defl \B_U(u_0, \frac{1}{3}) \subset \B_U(0,1) \setminus A$. Let $x \in \spt(C) \cap Z_U$ with $\pi_U(x) \in A'$ and $|\pi_{U^\perp}(x)| \leq 2\lambda\delta$. Then $|x| \leq 1 + 2\lambda\delta \leq 2$ and hence $|\pi_{V^\perp}(x)| \leq \dist(V,\spt(C) \cap Z_V(0,2)) \leq 2\delta$. Further,
\begin{align*}
|\pi_{V^\perp}(x)| + |\pi_{U^\perp}(x)| & \geq |\pi_{V^\perp}(x)| + |\pi_{V^\perp}(\pi_{U^\perp}(x))| \geq |\pi_{V^\perp}(x - \pi_{U^\perp}(x))| \\
 & \geq |\pi_{V^\perp}(\pi_U(x))| > 2\lambda\delta \,,
\end{align*}
and because $\lambda \geq 2$,
\[
|\pi_{U^\perp}(x)| > 2\lambda\delta - 2\delta \geq \lambda\delta \,.
\]
This shows the first conclusion and further that if $x \in \spt(C)$ with $\pi_U(x) \in A'$, then $|\pi_{U^\perp}(x)| \geq 2\delta$. 

Set $C' = \sum_i g_i\curr{S_i}$ as a finite sum over oriented simplices $S_i$ for which one vertex is at the origin and the other vertices outside $Z_U$ such that $C \res Z_U = C' \res Z_U$. Let $A''$ be the subset of those points $u \in A'$ for which the set ${\pi_U}^{-1}\{u\} \cap \spt(C')$ is finite and contained in the interior of any simplex it intersects. $\pi_U(A'')$ has full measure in $A'$. For $x \in A''$ let $I_x$ be the collection of those $i$ for which $x \in \pi_U(S_i)$ and as we did for $P$ in the first sections, let $\by^i : U \to U^\perp$ be the affine map such that $x' + \by^i(x')$ is in $S_i$ if $x'$ is close to $x$. From Lemma~\ref{multiple_values_mass_estimate} it follows that $\sum_{i \in I_x} \|g_i\| \geq \|g_0\|$. It is $\J \by^i_x \geq \|D\by^i_x\| \geq |\by^i(x)| \geq 2\delta$ for $x \in A''$ by \eqref{jacobian_sum} and \eqref{norms}. Using $(1 + a^2)^\frac{1}{2} \geq 1 + \frac{1}{3}a^2$ for $a \leq 1$ we get
\begin{align*}
\Exc_1(C,U) & \geq \Exc(C,U,A'') = \int_{A''} \sum_{i \in I_x}\|g_i\|(1 + (\J \by^i_x)^2)^\frac{1}{2} - \|g_0\| \, dx \\
 & \geq \int_{A''} \sum_{i \in I_x}\|g_i\|(1 + (\lambda\delta)^2)^\frac{1}{2} - \|g_0\| \, dx \\
 & \geq \|g_0\| \cH^m(A'') \left[(1 + (\lambda\delta)^2)^\frac{1}{2} - 1\right] \\
 & \geq \|g_0\|\frac{\balpha(m)}{3^{m+1}}\lambda^2\delta^2 \,.
\end{align*}
\end{proof}

Below is the main result of this section, the epiperimetric inequality of Reifenberg \cite{Reif} adapted to $G$-chains in Hilbert spaces.

\begin{Thm}
\label{masscomparison2}
There is a constant $\epsilon_{\theThm}(m) > 0$ with the following property. Let $P \in \cP_m(\Hi;G)$ be a cone with center $0$ and $V$ be an $m$-plane such that:
\begin{enumerate}
	\item $\spt(\partial P) \cap Z_V(2) = \emptyset$,
	\item $\pi_{V\#} (P \res Z_V) = g_0 \curr{\B_V(0,1)}$ for some $g_0 \in G \setminus \{0_G\}$,
	\item $\|\bg(x)\| \geq \frac{3}{4}\|g_0\|$ for $\|P\|$-almost every $x \in \Hi$
	\item $|\pi_{V^\perp}(x)| < \epsilon_{\theThm}$ for $x \in \spt(P) \cap Z_V$,
	\item $\Exc_1(P,V) < \|g_0\|\epsilon_{\theThm}$.
\end{enumerate}
Then there is some $S \in \cR_m(\Hi;G)$ with $\partial(S \res \B(0,1)) = \partial(P \res \B(0,1))$ and 
\[
\bM(S \res \B(0,1)) - \|g_0\|\balpha(m) \leq \lambda_{\theThm}\bigl(\bM(P \res \B(0,1)) - \|g_0\|\balpha(m)\bigr)
\]
for $\lambda_{\theThm} \defl \frac{2m+1 - 4^{-m-1}}{2m+1} < 1$.
\end{Thm}

\begin{proof}
As noted in the beginning of this section, the group norm can be normalized such that $\|g_0\| = 1$. Let $\epsilon_{\theThm} > 0$ be small enough such that 
\[
24 \lambda \sqrt{\epsilon_{\theThm}} < (2\epsilon_{\theThm})^\frac{1}{6m} \leq \epsilon_{\ref{newplane}} < \frac{1}{20} \,.
\]
where $\lambda^2 \defl 2\balpha(m)^{-1}3^{m+1} \geq \frac{3^m}{2^m} 6 \geq 2^2$. Let $V' \in \bG(\Hi,m)$ be such that
\begin{equation*}
\Exc_1(P,V') < 2 \inf \left\{ \Exc_1(P,U) : \|\pi_V - \pi_{U}\| < \tfrac{1}{10} \right\} \,.
\end{equation*}
Then $\epsilon \defl \Exc_1(P,V') < 2\epsilon_{\theThm} \leq \epsilon_{\ref{newplane}}$ and hence $\|\pi_V - \pi_{V'}\| \leq 6\lambda \sqrt{\epsilon_{\theThm}} < \frac{1}{20}$ by Lemma~\ref{newplaneheight}. Since $\spt(P) \cap Z_{V'} \subset \B(0,2)$, there holds $|\pi_{V'^\perp}(x)| \leq 24\lambda\sqrt{\epsilon_{\theThm}} < \rho \defl (2\epsilon_{\theThm})^\frac{1}{6m} \leq \epsilon_{\ref{newplane}}$. In particular
\begin{equation*}
\Exc_1(P,V') < 2 \inf \left\{ \Exc_1(P,U) : \|\pi_{V'} - \pi_{U}\| < \tfrac{1}{20} \right\} \,,
\end{equation*}
and Lemma~\ref{newplane} implies $\|\pi_{V'} - \pi_{W}\| < \tfrac{1}{20}$ for the plane $W$ used in the construction of the competitor $S$. By changing $V'$ slightly we can assume that $P$ is in general position with respect to $V'$ ($\Exc_1(P,U)$ is continuous in $U$ and the $m$-planes for which $P$ is in general position are dense in $\bG(\Hi,m)$). Hence $\epsilon$, $\rho$ and $V'$ satisfy all the conditions at the beginning of Subsection~\ref{assumptions} and the subsequent estimates. For the rest of this proof we identify $V'$ with $V$. Let $S$ be the comparison surface as described in the beginning of this subsection. Lemma~\ref{energyformulas} and Lemma~\ref{harmoniccoordinates2} imply
\begin{align*}
\bM(S \res Z_W(4^{-1})) - 4^{-m}\balpha(m) & = 4^{-m}\int_{\B_W} (1 + (\J \bh_x)^2)^\frac{1}{2} - 1 \, dx \\
 & \leq \bc_{\ref{energyformulas}}\epsilon^\frac{4}{3} + 4^{-m}\int_{\B_W} (1 + \|D\bh_x\|_{\HS}^2)^\frac{1}{2} - 1 \, dx \\
 & \leq \bc_{\ref{energyformulas}}\epsilon^\frac{4}{3} + 4^{-m}\int_{\B_W} \frac{1}{2}\|D\bh_x\|_{\HS}^2 \, dx \\
 & \leq (\bc_{\ref{energyformulas}} + \bc_{\ref{harmoniccoordinates2}})\epsilon^\frac{4}{3} \\
 & \quad + 4^{-m} \frac{2m}{2m+1}\int_{\B_W} \frac{1}{2}\|D\bw_x\|_{\HS}^2 \, dx \,. \\
\end{align*}
Because $\frac{1}{2}a^2 \leq \frac{1}{8}a^4 + (1 + a^2)^\frac{1}{2} - 1$ for $a \in [0,1]$ and $\|D\bw_x\|^4_{\HS} \leq \bc_{\ref{newplane}}\epsilon^\frac{4}{3}$, we get for some $\bc(m) > 0$,
\begin{align*}
& \bM(S \res Z_W(4^{-1})) - 4^{-m}\balpha(m) \\
 & \qquad\qquad \leq \bc\epsilon^\frac{4}{3} + 4^{-m} \frac{2m}{2m+1}\int_{\B_W} (1 + \|D\bw_x\|_{\HS}^2)^\frac{1}{2} - 1 \, dx \\
 & \qquad\qquad \leq \bc\epsilon^\frac{4}{3} + 4^{-m} \frac{2m}{2m+1}\int_{\B_W} (1 + (\J \bw_x)^2)^\frac{1}{2} - 1 \, dx \\
 & \qquad\qquad = \bc\epsilon^\frac{4}{3} + \frac{2m}{2m+1} \left[\bM(T \res Z_W(4^{-1})) - 4^{-m}\balpha(m) \right] \,.
\end{align*}
Hence,
\begin{align*}
\bM(S \res Z_V) & = \bM(T \res (Z_V \setminus Z_W(4^{-1}))) + \bM(S \res Z_W(4^{-1})) \\
 & \leq \bc\epsilon^\frac{4}{3} + \bM(T \res (Z_V \setminus Z_W(4^{-1}))) \\
 & \quad + \frac{2m}{2m+1}\bM(T \res Z_W(4^{-1})) + \frac{4^{-m}\balpha(m)}{2m+1} \\
 & = \bc\epsilon^\frac{4}{3} + \bM(T \res Z_V) - \frac{1}{2m+1} \left[ \bM(T \res Z_W(4^{-1})) - 4^{-m}\balpha(m) \right] \,.
\end{align*}
If $\bM(T \res Z_W(4^{-1})) \geq \bM(T \res Z_{V}(4^{-1}))$ we estimate using Proposition~\ref{masscomparison1},
\begin{align*}
\bM(S \res Z_{V}) & \leq \bc\epsilon^\frac{4}{3} + \bM(T \res Z_{V}) - \frac{1}{2m+1} \left[ \bM(T \res Z_{W}(4^{-1})) - 4^{-m}\balpha(m) \right] \\
 & \leq \bc\epsilon^\frac{4}{3} + \bM(T \res (Z_{V} \setminus Z_{V}(4^{-1}))) \\
 & \quad + \frac{2m}{2m+1}\bM(T \res Z_{V}(4^{-1})) + \frac{4^{-m}\balpha(m)}{2m+1} \\
 & \leq (\bc + \bc_{\ref{masscomparison1}})\rho^\frac{1}{2}\epsilon + \bM(P \res (Z_{V} \setminus Z_{V}(4^{-1}))) \\
 & \quad + \frac{2m}{2m+1}\bM(P \res Z_{V}(4^{-1})) + \frac{4^{-m}\balpha(m)}{2m+1} \\
 & = (\bc + \bc_{\ref{masscomparison1}})\rho^\frac{1}{2}\epsilon + \bM(P \res Z_{V}) \\
 & \quad - \frac{1}{2m+1} \left[ \bM(P \res Z_{V}(4^{-1})) - 4^{-m}\balpha(m) \right] \\
 & = (\bc + \bc_{\ref{masscomparison1}})\rho^\frac{1}{2}\epsilon + \frac{2m + 1 - 4^{-m}}{2m+1}\bM(P \res Z_{V}) + \frac{4^{-m}}{2m+1}\balpha(m) \,.
\end{align*}
On the other hand, if $\bM(T \res Z_W(4^{-1})) \leq \bM(T \res Z_{V}(4^{-1}))$ it follows from Lemma~\ref{planeselection} and the almost minimality of the plane ${V}$ with respect to the excess of $P$,
\begin{align*}
\Exc_{4^{-1}}(T,{V}) & \leq \bc_{\ref{planeselection}}\rho\epsilon + \Exc_{4^{-1}}(T,W) + \Exc_{4^{-1}}(P,{V}) - \Exc_{4^{-1}}(P,W) \\
 & \leq \bc_{\ref{planeselection}}\rho\epsilon + \Exc_{4^{-1}}(T,W) + 2^{-1}\Exc_{4^{-1}}(P,{V}) \,.
\end{align*}
Hence, with the same derivations as in the other case above,
\begin{align*}
\bM(S \res Z_{V}) & \leq \bc\epsilon^\frac{4}{3} + \bM(T \res Z_{V}) - \frac{1}{2m+1} \Exc_{4^{-1}}(T,W) \\
 & \leq (\bc + \bc_{\ref{planeselection}})\rho\epsilon + \bM(T \res Z_{V}) \\
 & \qquad \qquad - \frac{1}{2m+1}\left[ \Exc_{4^{-1}}(T,{V}) - 2^{-1}\Exc_{4^{-1}}(P,{V}) \right] \\
 & \leq (\bc + \bc_{\ref{masscomparison1}} + \bc_{\ref{planeselection}})\rho^\frac{1}{2}\epsilon + \frac{2m + 1 - 4^{-m}}{2m+1}\bM(P \res Z_{V}) \\
 & \qquad \qquad + \frac{4^{-m}}{2m+1}\balpha(m) + \frac{2^{-1}4^{-m}}{2m+1}\Exc_1(P,{V}) \,.
\end{align*}
By subtracting $\balpha(m)$ from these estimates we obtain in both cases
\[
\Exc_1(S,{V}) \leq \bc'\rho^\frac{1}{2}\epsilon + \frac{2m + 1 - 2^{-1}4^{-m}}{2m+1}\Exc_1(P,{V}) \,,
\]
for some $\bc'(m) > 0$. If we assume that $\epsilon_{\theThm}$ is small enough such that for $\rho = (2\epsilon_{\theThm})^\frac{1}{6m}$,
\[
\bc' \rho^\frac{1}{2} + \frac{2m + 1 - 2^{-1}4^{-m}}{2m+1} \leq \frac{2m + 1 - 4^{-m - 1}}{2m+1} = \lambda,
\]
we obtain $\Exc_1(S,{V}) \leq \lambda \, \Exc_1(P,{V})$. Finally, by \eqref{bounderywithball},
\begin{align*}
\bM(S \res \B(0,1)) - \balpha(m) & = \bM(S \res Z_{V}) - \bM(P \res (Z_{V} \setminus \B(0,1))) - \balpha(m) \\
 & \leq \lambda \left[\bM(P \res \B(0,1)) + \bM(P \res (Z_{V} \setminus \B(0,1))) - \balpha(m)\right] \\
 & \quad - \bM(P \res (Z_{V} \setminus \B(0,1))) \\
 & \leq \lambda (\bM(P \res \B(0,1)) - \balpha(m)) \,.
\end{align*}
\end{proof}

\section{Moments computations and good approximations by planes}
\label{momentscomputations}

\subsection{Nearly monotonic, almost monotonic and epiperimetric chains}

Let $\phi$ be a finite Borel measure on $\Hi$. For $x \in \Hi$ and $r > 0$ we define
\begin{align*}
\exc_*^m(\phi,x,r) & \defl \sup \Biggl\{\biggl(\frac{\phi(\B(x,t))}{\balpha(m) t^m} - \frac{\phi(\B(x,s))}{\balpha(m) s^m}\biggr)_- : 0 < s \leq t \leq r \Biggr\} \,, \\
\exc^{m*}(\phi,x,r) & \defl \sup \Biggl\{\biggl(\frac{\phi(\B(x,t))}{\balpha(m) t^m} - \frac{\phi(\B(x,s))}{\balpha(m) s^m}\biggr)_+ : 0 < s \leq t \leq r \Biggr\} \,, \\
\exc^m(\phi,x,r) & \defl \max\{\exc_*^m(\phi,x,r),\exc^{m*}(\phi,x,r)\} \,.
\end{align*}

Throughout the rest of these notes we use the term \emph{gauge} or \emph{gauge function} for an increasing function $\xi : (0,\delta] \to \R_+$ for some $\delta > 0$ with $\lim_{r \to 0} \xi(r) = 0$. Let $T \in \cR_m(\Hi;G)$. For a subset $A \subset \Hi \setminus \spt(\partial T)$ we say that $T$ is \emph{nearly monotonic} in $A$ if there is a gauge $\xi_* : (0,\delta] \to \R_+$ such that
\[
\exc_*^m(\|T\|,x,r) \leq \xi_*(r) \,,
\]
for all $x \in A$ and $0 < r < \min\{\delta, \dist(x,\spt(\partial T))\}$. We say that $T$ is \emph{almost monotonic} in $A$ if there is a gauge $\xi : (0,\delta] \to \R_+$ such that
\[
r \mapsto \exp(\xi(r))\frac{\|T\|(\B(x,r))}{\balpha(m) r^m}
\]
is an increasing function for all $x \in A$ and all $0 < r < \min\{\delta, \dist(x,\spt(\partial T))\}$.

The following Lemma is an adaptation of \cite[Lemma~3.2.3]{DePauw3} to the setting of rectifiable chains Hilbert spaces. The proof goes through unchanged.

\begin{Lem}
	\label{densityexist}
If $T \in \cR_m(\Hi;G)$ is nearly monotonic in $A \subset \Hi \setminus \spt(\partial T)$, then the density $\Theta^m(\|T\|,x)$ exists for all $x \in A$ and is finite. Further, the function $x \mapsto \Theta^m(\|T\|,x)$ is upper semicontinuous in $A$.
\end{Lem}

\begin{proof}
Let $\xi_* : (0,\delta] \to \R_+$ be a gauge with respect to which $T$ is nearly monotonic. For $x \in A$ and $0 < r < \min\{\delta, \dist(x,\spt(\partial T))\}$ we abbreviate $f_x(r) \defl \balpha(m)^{-1}r^{-m}\|T\|(\B(x,r))$. By assumption
\[
-\xi_*(R) \leq f_x(R) - f_x(r) \,,
\]
for all $x \in A$ and $0 < r \leq R < \min\{\delta, \dist(x,\spt(\partial T))\}$. Therefore,
\[
-\xi_*(R) \leq \liminf_{r \downarrow 0}(f_x(R) - f_x(r)) = f_x(R) - \limsup_{r \downarrow 0}f_x(r) \,,
\]
and, in turn,
\[
0 \leq \liminf_{R \downarrow 0}f_x(R) - \limsup_{r \downarrow 0}f_x(r) \,.
\]
This shows that the densities $\Theta^m(\|T\|,x)$ exist for all $x \in A$. Let $(x_i)$ be a sequence in $A$ with $\lim_{i\to\infty}x_i = x \in A$ and $|x_i - x| < 2^{-1} \min\{\delta, \dist(x,\spt(\partial T))\}$ for all $i$. Since $\|T\|$ is a finite Borel measure, the function $\rho \mapsto \|T\|(\B(x,\rho))$ is continuous in all but countably many $r > 0$. If we pick a point of continuity $0 < r < 2^{-1}\min\{\delta, \dist(x,\spt(\partial T))\}$, then
\begin{align*}
\Theta^m(\|T\|,x_i) & \leq f_{x_i}(r) + \xi_*(r) \\
 & \leq f_{x}(r + |x_i - x|)\Biggl(1 + \frac{|x_i-x|}{r}\Biggr)^m + \xi_*(r) \,,
\end{align*}
for all $i$. Hence for all such $r$,
\[
\limsup_{i \to \infty}\Theta^m(\|T\|,x_i) \leq f_x(r) + \xi_*(r)  \,.
\]
Taking the limit for $r \to 0$, we see that $\limsup_{i \to \infty}\Theta^m(\|T\|,x_i) \leq \Theta^m(\|T\|,x)$.
\end{proof}

The following observation is essentially \cite[Lemma~3.4.1]{DePauw3}. For the reader's convenience we repeat the proof here.

\begin{Lem}
	\label{nearlymonotonic}
If $T \in \cR_m(\Hi;G)$ is almost monotonic in $A \subset \Hi \setminus \spt(\partial T)$ with gauge $\xi : (0,\delta] \to \R_+$, then $T$ is nearly monotonic in $A$ with gauge $\xi_* \defl \bc_{\theThm}\xi$, for some constant $\bc_{\theThm}(m,\delta,\xi(\delta),\bM(T)) > 0$.
\end{Lem}

\begin{proof}
For $x \in A$ and $0 < r < \min\{\delta, \dist(x,\spt(\partial T))\}$ we abbreviate again $f_x(r) \defl \balpha(m)^{-1} r^{-m}\|T\|(\B(x,r))$. Because $T$ is almost monotonic,
\[
f_x(r) \leq \bc_1 \defl \exp(\xi(\delta))\frac{\bM(T)}{\balpha(m) \delta^m} \,.
\]
Let $\bc_2 \defl \xi(\delta)^{-1}(\exp(\xi(\delta)) - 1)$. Then for $x \in A$ and $0 < r \leq s \leq \delta$,
\begin{align*}
f_x(r) & \leq \exp(\xi(r))f_x(r) \leq \exp(\xi(s))f_x(s) \\
	& \leq (1 + \bc_2\xi(s))f_x(s) \leq f_x(s) + \bc_1\bc_2\xi(s) \,.
\end{align*}
\end{proof}

A rather simple consequence of nearly monotonicity and lower density bounds is the compactness of the support away from the boundary and Ahlfors-regularity of the associated measure. The latter observation will be important in connection with Lemma~\ref{tangentplane} about tangent planes.

\begin{Lem}
	\label{compactlem}
Let $T \in \cR_m(\Hi;G)$ and $U \subset \Hi$ be an open set with $\dist(U,\spt(\partial T)) > r_0 > 0$. Assume that there is a constant $\theta > 0$ such that:
\begin{enumerate}
	\item $\Theta^m(\|T\|,x) \geq \theta$ for $\|T\|$-almost all $x \in \operatorname{cl}(U)$.
	\item $\exc^m_*(\|T\|,x,r_0) \leq \frac{\theta}{2}$ for all $x \in \operatorname{cl}(U)$.
	\item $T$ is nearly monotonic in $\operatorname{cl}(U)$.
\end{enumerate}
Then $\spt(T) \cap \operatorname{cl}(U)$ is compact and for all $0 < r < r_0$ and $x \in \spt(T) \cap \operatorname{cl}(U)$ there holds
\[
\frac{\theta}{2} \leq \frac{\|T\|(\B(x,r))}{\balpha(m)r^m} \leq \bc(\bM(T),r_0,\theta,m) \,.
\]
\end{Lem}

\begin{proof}
As a closed subset of a complete space, $\spt(T)\cap \operatorname{cl}(U)$ is itself complete. It remains to show that this set is totally bounded. By assumption, for $\|T\|$-almost all $x \in \operatorname{cl}(U)$ and all $0 < r < r_0$,
\begin{equation}
\label{densboundone}
	\theta \leq \Theta^m(\|T\|,x) \leq \frac{\|T\|(\B(x,r))}{\balpha(m) r^m} + \frac{\theta}{2} \,,
\end{equation}
and hence $\frac{\theta}{2}\,\balpha(m) r^m \leq \|T\|(\B(x,r))$. The nearly monotonicity of $T$ and Lemma \ref{densityexist} show that \eqref{densboundone} holds for every $x \in \spt(T) \cap \operatorname{cl}(U)$.
Let $x_1,\dots,x_k$ be points in $\spt(T) \cap \operatorname{cl}(U)$ with $d(x_i,x_j) > 2r$ for $i \neq j$. Then
\[
	k \,\frac{\theta}{2}\,\balpha(m) r^m \leq \sum_{i = 1}^k \|T\|(\B(x_i,r)) \leq \bM(T) \,.
\]
Thus there is an upper bound on the number of points that are $2r$-separated. Taking a maximal collection of such points it follows that $\spt(T) \cap \operatorname{cl}(U)$ can be covered by $\lfloor\bM(T)(\frac{\theta}{2}\balpha(m) r^m)^{-1}\rfloor$ closed balls of radius $2r$. This is true for any $r < r_0$ and $\spt(T) \cap \operatorname{cl}(U)$ is therefore totally bounded and hence compact.

The first inequality of the second conclusion follows from \eqref{densboundone}. Finally,
for all $x \in \spt(T) \cap \operatorname{cl}(U)$ and $0 < r < r_0$,
\[
\frac{\|T\|(\B(x,r))}{\balpha(m) r^m} \leq \frac{\|T\|(\B(x,r_0))}{\balpha(m) r_0^m} + \exc^m_*(\|T\|,x,r_0) \leq \frac{\bM(T)}{\balpha(m) r_0^m} + \frac{\theta}{2} \,.
\]
\end{proof}

\subsection{Moments computations}

Let $\phi$ be a finite Borel measure on $\Hi$. As defined in the introduction of the seminal paper \cite{P} by Preiss and later used in \cite{DePauw3}, we define some integrals for $r > 0$, compare with \cite[\S 4.1]{DePauw3}. First,
\[
V(\phi,x,r) \defl \int_{\B(0,r)} (r^2 - |x-y|^2)^2 \, d\phi(y) \,.
\]
This can be written as a sum $V(\phi,x,r) = \sum_{k=0}^4 P_k(\phi,x,r)$, where
\begin{align*}
P_0(\phi,x,r) & \defl \int_{\B(0,r)} (r^2 - |y|^2)^2 \, d\phi(y) \,, \\
P_1(\phi,x,r) & \defl 4\biggl\langle x, \int_{\B(0,r)} y(r^2 - |y|^2) \, d\phi(y) \biggr \rangle \,, \\
P_2(\phi,x,r) & \defl 4 \int_{\B(0,r)} \langle x,y \rangle^2 \, d\phi(y) - 2|x|^2\int_{\B(0,r)} r^2 - |y|^2 \, d\phi(y) \,, \\
P_3(\phi,x,r) & \defl -4|x|^2 \int_{\B(0,r)} \langle x,y \rangle \, d\phi(y) \,, \\
P_4(\phi,x,r) & \defl |x|^4 \phi(\B(0,r)) \,.
\end{align*}
We further define
\begin{align*}
b(\phi,r) & \defl \int_{\B(0,r)} y(r^2 - |y|^2) \, d\phi(y), \\
Q(\phi,r)(x) & \defl \int_{\B(0,r)} \langle x,y \rangle^2 \, d\phi(y) \,.
\end{align*}
We have already encountered the quadratic form $Q$ in Subsection~\ref{betaprelim} although with a particular renormalization. This same renormalization of the quantities above is what we define next and use boldface letters for those. Let $\bnu(m) \defl \frac{\balpha(m)}{m+2}$ and define:
\begin{align*}
\bV(\phi,x,r) & \defl \bnu(m)^{-1}r^{-m-2}V(\phi,x,r) \,, \\
\bP_k(\phi,x,r) & \defl \bnu(m)^{-1}r^{-m-2}P_k(\phi,x,r), \, k=0,\dots,4 \,, \\
\bb(\phi,r) & \defl \bnu(m)^{-1}r^{-m-2}b(\phi,r) \,, \\
\bQ(\phi,r) & \defl \bnu(m)^{-1}r^{-m-2}Q(\phi,r) \,. \\
\end{align*}
Further let
\[
\bomega(m,q) \defl \int_{\B^m(0,1)}(1 - |y|^2)^q \, d\cL^m(y), \quad q=0,1,2,\dots \,.
\]
The following simple identities will be important,
\begin{equation}
\label{omegaidentity2}
\frac{\balpha(m)}{m+2} = \bnu(m) = \frac{\bomega(m,0) - \bomega(m,1)}{m} = \frac{\balpha(m) - \bomega(m,1)}{m} \,,
\end{equation}
and hence
\begin{equation}
\label{omegaidentity}
\bomega(m,1) = \frac{2\balpha(m)}{m+2} = 2\bnu(m) \,.
\end{equation}

The following results in this section are from \cite{DePauw3}, the proofs below are almost the same with the minor difference that $\xi$ is a continuous increasing function for which we don't require that $\lim_{r \to 0} \xi(r) = 0$. This will be important when applied in Lemma~\ref{bootstraplem}, where only small bounds on the spherical excess are assumed. For the reader's convenience we repeat the proofs here.

\begin{Lem}{\cite[Lemma~4.1.1]{DePauw3}}
\label{densityintegralbound1}
Let $x \in \Hi$ and $r , \epsilon > 0$ be such that
\[
\left| \frac{\phi(\B(x,\rho))}{\balpha(m) \rho^m} - 1 \right| \leq \epsilon \,,
\]
for every $0 < \rho < r$. Then for every $q=0,1,2,\dots$ one has
\[
\left|\int_{\B(x,r)} (r^2 - |x-y|^2)^q \, d\phi(y) - \bomega(m,q) r^{2q + m} \right| \leq \epsilon \bomega(m,q) r^{2q+m} \,.
\]
\end{Lem}

\begin{proof}
It suffices to observe that with Cavalieri's principle,
\begin{align*}
\int_{\B(x,r)}(r^2 - |x-y|^2)^q \, d\phi(y) & = \int_0^{r^{2q}} \phi\left(\B\left(x,\bigl(r^2 - t^\frac{1}{q}\bigr)^\frac{1}{2}\right)\right) \, d\cL^1(t) \\
 & \leq (1 + \epsilon) \int_0^{r^{2q}} \cL^m\left(\B\left(0,\bigl(r^2 - t^\frac{1}{q}\bigr)^\frac{1}{2}\right)\right) \, d\cL^1(t) \\
 & = (1 + \epsilon) \int_{\B^m(0,r)} (r^2 - |y|^2)^q \, d\cL^m(y) \\
 & = (1 + \epsilon) \bomega(m,q) r^{2q+m} \,.
\end{align*}
The other inequality is proved exactly the same way.
\end{proof}

Next is an a priori bound on the trace of $\bQ(\phi,r)$.

\begin{Lem}{\cite[Lemma~4.1.2]{DePauw3}}
\label{densityintegralbound2}
Let $r , \epsilon > 0$ be such that
\[
\left| \frac{\phi(\B(0,\rho))}{\balpha(m) \rho^m} - 1 \right| \leq \epsilon \,,
\]
for every $0 < \rho < r$. Then
\[
|\tr \bQ(\phi,r) - m| \leq \epsilon(m+4) \,.
\]
\end{Lem}

\begin{proof}
Let $e_1,e_2,\dots$ be an orthonormal basis of $\Hi$. Then
\begin{align*}
\tr Q(\phi,r) & = \sum_{i \geq 1} Q(\phi,r)(e_i) = \int_{\B(0,r)} |y|^2 \, d\phi(y) \,.
\end{align*}
Therefore we have with \eqref{omegaidentity2} and Lemma~\ref{densityintegralbound1},
\begin{align*}
\bigl|m\bnu(m)r^{m+2} - \tr Q(\phi,r)\bigr| & = \Bigl|m\bnu(m)r^{m+2} - \int_{\B(0,r)} |y|^2 \, d\phi(y)\Bigr| \\
 & \leq \Bigl|\bomega(m,1)r^{m+2} - \int_{\B(0,r)} (r^2 - |y|^2) \, d\phi(y)\Bigr| \\
 & \quad + \Bigl|\bomega(m,0)r^{m+2} - r^2\phi(\B(0,r)) \Bigr| \\
 & \leq \epsilon\bomega(m,1)r^{m+2} + \epsilon \bomega(m,0)r^{m+2} \,.
\end{align*}
Dividing both sides by $\bnu(m)r^{m+2}$ gives the result.
\end{proof}

In Proposition~\ref{quadratformprop} the quadratic form $\bQ(\phi,r)$ is controlled in terms of the excess. In order to do so we need some bound on the length of $\bb(\phi,r)$ in terms of the spherical excess. Define
\begin{align*}
\hat V(\phi,x,r) & \defl \int_{\B(x,r)} (r^2 - |x-y|^2)^2 \, d\phi(y) \,.
\end{align*}

\begin{Lem}{\cite[Lemma~4.2.1]{DePauw3}}
\label{hatnohatdiff}
There is a constant $\bc_{\theThm}(m) > 0$ with the following property. Whenever $2|x| < r$, then
\begin{align*}
\bigl| V(\phi,x,r)-\hat V(\phi,x,r)\bigr| \leq \bc_{\theThm}(m) \bnu(m)r^m \bigl(& r|x|^3(r^{-m}\phi(\B(0,r))) \\
 & + r^2|x|^2\exc^{m*}(\phi,0,2r)\bigr) \,.
\end{align*}
\end{Lem}

\begin{proof}
The following statements are easy to check:
\begin{enumerate}
	\item $(\B(x,r) \setminus \B(0,r)) \cup (\B(0,r) \setminus \B(x,r)) \subset \B(x,r + |x|) \setminus \B(x,r-|x|)$;
	\item If $y \in (\B(x,r) \setminus \B(0,r)) \cup (\B(0,r) \setminus \B(x,r))$, then $|r^2 - |x-y|^2| \leq 3r|x|$;
	\item $\B(0,r - 2|x|) \subset \B(x,r-|x|) \subset \B(x,r + |x|) \subset \B(0,r + 2|x|)$.
\end{enumerate}
The statements (1) and (3) are obvious. To see (2), note that because of (1),
\[
0 \leq r - |x| \leq |x-y| \leq r + |x| \,,
\]
whence
\[
r^2 - 2r|x| + |x|^2 \leq |x-y|^2 \leq r^2 + 2r|x| + |x|^2 \,.
\]
Statement (2) now follows by noting that $|x|^2 \leq r|x|$. Using these properties we see that
\begin{align}
\nonumber
\bigl| V(\phi,x,r)-\hat V(\phi,x,r)\bigr| & \leq 9r^2|x|^2\phi\bigl[(\B(x,r) \setminus \B(0,r)) \cup (\B(0,r) \setminus \B(x,r))\bigr] \\
\nonumber
 & \leq 9r^2|x|^2\bigl(\phi(\B(x,r + |x|)) - \phi(\B(x,r - |x|))\bigr) \\
\label{vestimate}
 & \leq 9r^2|x|^2\bigl(\phi(\B(0,r + 2|x|)) - \phi(\B(0,r - 2|x|))\bigr) \,.
\end{align}
Since $r + 2|x| < 2r$, we have that
\[
\frac{\phi(\B(0,r + 2|x|))}{\balpha(m)(r + 2|x|)^m} \leq \frac{\phi(\B(0,r))}{\balpha(m)r^m} + \exc^{m*}(\phi,0,2r) \,,
\]
so that
\[
\phi(\B(0,r + 2|x|)) \leq \left(1 + \frac{2|x|}{r}\right)^m\phi(\B(0,r)) + \balpha(m)(r + 2|x|)^m\exc^{m*}(\phi,0,2r) \,.
\]
Similarly,
\[
\frac{\phi(\B(0,r - 2|x|))}{\balpha(m)(r - 2|x|)^m} \geq \frac{\phi(\B(0,r))}{\balpha(m)r^m} - \exc^{m*}(\phi,0,r) \,,
\]
so that
\[
\phi(\B(0,r - 2|x|)) \geq \left(1 - \frac{2|x|}{r}\right)^m\phi(\B(0,r)) - \balpha(m)(r - 2|x|)^m\exc^{m*}(\phi,0,r) \,.
\]
From this we deduce that
\begin{align*}
\phi(\B(0,r + 2|x|)) & - \phi(\B(0,r - 2|x|)) \\
& \leq \phi(\B(0,r))\left(\left(1 + \frac{2|x|}{r}\right)^m - \left(1 - \frac{2|x|}{r}\right)^m\right) \\
& \quad + (1 + 2^m)\balpha(m)r^m\exc^{m*}(\phi,0,2r) \\
 & \leq \phi(\B(0,r))m2^m\frac{2|x|}{r} + (1 + 2^m)\balpha(m)r^m\exc^{m*}(\phi,0,2r) \,.
\end{align*}
Plugging this into \eqref{vestimate} gives the result.
\end{proof}

\begin{Def}
Given $x_1,x_2 \in \Hi$ and $r > 0$ we define the \emph{deviation} as
\[
\dev^m(\phi,x_1,x_2,r) \defl \frac{\phi(\B(x_1,r)) - \phi(\B(x_2,r))}{r^m} \,.
\]
\end{Def}

\begin{Lem}{\cite[Lemma~4.2.3]{DePauw3}}
\label{hathatdiff}
There is a constant $\bc_{\theThm}(m) > 0$ such that
\begin{align*}
\hat V(\phi,x,r)-\hat V(\phi,0,r) & \leq \bc_{\theThm}(m) \bnu(m)r^{m+4}\bigl(\dev^m(\phi,x,0,r)  \\
 & \qquad + \exc^{m}_*(\phi,x,r) + \exc^{m*}(\phi,0,r)\bigr) \,, \\
\bigl|\hat V(\phi,x,r)-\hat V(\phi,0,r)\bigr| & \leq \bc_{\theThm}(m) \bnu(m)r^{m+4}\bigl(|\dev^m(\phi,x,0,r)|  \\
 & \qquad + \exc^{m}(\phi,x,r) + \exc^{m}(\phi,0,r)\bigr) \,.
\end{align*}
\end{Lem}

\begin{proof}
As in the proof of Lemma~\ref{densityintegralbound1},
\begin{align*}
\hat V(\phi,x,r) & = \int_{\B(x,r)} (r^2 - |x-y|^2)^2 \, d\phi(y) \\
 & = \int_0^{r^2} \phi\left(\B\left(x,\bigl(r^2 - t^\frac{1}{2}\bigr)^\frac{1}{2}\right)\right) \, d\cL^1(t) \\
 & = \int_0^{r} \phi(\B(x,\rho))4\rho(r^2-\rho^2) \, d\cL^1(\rho) \,.
\end{align*}
Similarly for $\hat V(\phi,0,r)$, so that
\begin{equation}
\label{firstvdifference}
\hat V(\phi,x,r) - \hat V(\phi,0,r) = \int_0^{r} (\phi(\B(x,\rho)) - \phi(\B(0,\rho)))4\rho(r^2-\rho^2) \, d\cL^1(\rho) \,.
\end{equation}
For $0 < \rho < r$ we have
\[
\frac{\phi(\B(x,\rho))}{\balpha(m)\rho^m} \leq \frac{\phi(\B(x,r))}{\balpha(m)r^m} + \exc^{m}_*(\phi,x,r) \,,
\]
so that
\begin{equation}
\label{phirhobound}
\phi(\B(x,\rho)) \leq \rho^m\frac{\phi(\B(x,r))}{r^m} + \balpha(m)\rho^m\exc^{m}_*(\phi,x,r) \,.
\end{equation}
Similarly,
\begin{equation}
\label{phirhobound2}
-\phi(\B(0,\rho)) \leq -\rho^m\frac{\phi(\B(0,r))}{r^m} + \balpha(m)\rho^m\exc^{m*}(\phi,0,r) \,.
\end{equation}
One also checks that
\begin{equation}
\label{phirhobound3}
\int_0^r 4\rho^{m+1}(r^2-\rho^2) \, d\cL^1(\rho) = 8(m+2)^{-1}(m+4)^{-1}r^{m+4} \,.
\end{equation}
Plugging \eqref{phirhobound} and \eqref{phirhobound2} into \eqref{firstvdifference} and using \eqref{phirhobound3} yields the first estimate. To obtain the second it suffices to apply the first one with $x$ and $0$ swapped.
\end{proof}

Next we obtain a trivial bound on $|\bb(\phi,r)|$ due to the normalization.

\begin{Lem}
{\cite[Lemma~4.3.1]{DePauw3}}
\label{trivialfirstmombound}
There is a constant $\bc_{\theThm}(m) > 0$ with the following property. If $\Theta^m(\phi,0) = 1$, then
\[
|\bb(\phi,r)| \leq 2r(1 + \exc^m(\phi,0,r)) \,.
\]
\end{Lem}

\begin{proof}
It suffices to apply Lemma~\ref{densityintegralbound1}:
\begin{align*}
|b(\phi,r)| & \leq \int_{\B(0,r)} |y|(r^2 - |y|^2) \, d\phi(y) \\
 &  \leq r(1 + \exc^m(\phi,0,r))\bomega(m,1)r^{m+2} \,,
\end{align*}
and divide by $\bnu(m)r^{m+2}$.
\end{proof}

We will also need to control the deviation in the following way.

\begin{Lem}
{\cite[Lemma~4.3.2]{DePauw3}}
\label{trivialdevbound}
Assume that $\Theta^m(\phi,0) = 1$, $0 < r \leq R$ and $|x| = \epsilon R$ for some $0 < \epsilon \leq 1$. Then
\begin{align*}
\balpha(m)^{-1}\dev^m(\phi,x,0,r) & \leq m2^{m-1}\epsilon(1 + \exc^m(\phi,0,2R)) \\
 & \quad + \exc^{m*}(\phi,0,2R) + \exc^{m}_*(\phi,x,R) \,.
\end{align*}
\end{Lem}

\begin{proof}
It suffices to compute:
\begin{align*}
\balpha(m)^{-1}\dev^m(\phi,x,0,r) & = \frac{\phi(\B(x,r))}{\balpha(m)r^m} - \frac{\phi(\B(0,r))}{\balpha(m)r^m}  \\
 & \leq \frac{\phi(\B(x,R))}{\balpha(m)R^m} + \exc^m_*(\phi,x,R) - \frac{\phi(\B(0,r))}{\balpha(m)r^m} \\
 & \leq \frac{(R + |x|)^m}{R^m} \frac{\phi(\B(0,R + |x|))}{\balpha(m)(R + |x|)^m} - \frac{\phi(\B(0,r))}{\balpha(m)r^m} \\
 & \quad + \exc^m_*(\phi,x,R) \\
 & \leq \left((1 + \epsilon)^m - 1\right) \frac{\phi(\B(0,R + |x|))}{\balpha(m)(R + |x|)^m} \\
 & \quad + \frac{\phi(\B(0,R + |x|))}{\balpha(m)(R + |x|)^m}  - \frac{\phi(\B(0,r))}{\balpha(m)r^m} + \exc^m_*(\phi,x,R) \\
 & \leq m2^{m-1}\epsilon(1 + \exc^m(\phi,0,R + |x|)) \\
 & \quad + \exc^{m*}(\phi,0,R + |x|) + \exc^{m}_*(\phi,x,R)
\end{align*}
\end{proof}

The following is an improvement on Lemma~\ref{trivialfirstmombound}. Note that compared with \cite[Proposition~4.3.3]{DePauw3} we don't assume that $\lim_{t \to 0}\xi(t) = 0$.

\begin{Lem}
{\cite[Proposition~4.3.3]{DePauw3}}
\label{firstmombound}
There is a constant $\bc_{\theThm}(m) > 0$ with the following property. Let $0 < 2\sqrt{r} \leq r_0 \leq 1$ and assume there is a continuous increasing function $\xi : (0,r_0] \to [0,1]$ with
\begin{enumerate}
	\item $\Theta^m(\phi,0) = 1$,
	\item $\exc^{m}_*(\phi,x,\rho) \leq \xi(\rho)$ for $0 < \rho \leq \sqrt r$ and $x \in \B(0,r)$,
	\item $\exc^{m}(\phi,0,\rho) \leq \xi(\rho)$ for $0 < \rho \leq 2 \sqrt r$.
\end{enumerate}
Then
\[
|\bb(\phi,r)| \leq \bc_{\theThm}(m) r \max\left\{\sqrt[4]{r}, \sqrt{\xi(2\sqrt{r})}\right\}\,.
\]
\end{Lem}

\begin{proof}
We start by choosing $0 < \gamma(m) \leq \frac{1}{8}$ and $\eta(m)$ such that
\begin{equation}
\label{prepgammaeta}
\eta(m) \defl 4 \gamma(m) - \gamma(m)^2(2\bc_{\ref{hatnohatdiff}} + 8 + 8(m+2)) > 0 \,.
\end{equation}
We define a continuous increasing function $\epsilon : (0,r_0] \to \R_+$ by $\epsilon(\rho) \defl \max\{\rho,\xi(\rho)\}$. Now either $|\gamma(m)\bb(\phi,r)| \leq r \sqrt{\epsilon(\sqrt{r})}$ or $|\gamma(m)\bb(\phi,r)| > r \sqrt{\epsilon(\sqrt{r})}$. We will subsequently derive an estimate for $|\bb(\phi,r)|$ in the latter case. We first observe that since $2r \leq \sqrt{r}$,
\begin{equation}
\label{prepgammaeta2}
r\epsilon(2r) \leq r\epsilon(\sqrt{r}) \leq r\sqrt{\epsilon(\sqrt{r})} < |\gamma(m)\bb(\phi,r)| \,.
\end{equation}
According to Lemma~\ref{trivialfirstmombound}, $|\bb(\phi,r)| \leq 4r$ and hence $|\gamma(m)\bb(\phi,r)| \leq \frac{r}{2} \leq r$ as well. Since $\epsilon(\sqrt{r}) \geq \sqrt{r}$ we see that
\begin{equation}
\label{prepgammaeta3}
|\gamma(m)\bb(\phi,r)| \leq r \leq \sqrt{r}\epsilon(2\sqrt{r}) \,.
\end{equation}
According to \eqref{prepgammaeta2} and \eqref{prepgammaeta3}, the intermediate value theorem applied to the function
\[
[r, \sqrt{r}] \to \R_+, \; \rho \mapsto \rho\epsilon(2\rho) 
\]
ensures that there exists some $r \leq R \leq \sqrt{r}$ with $R\epsilon(2R) = |\gamma(m)\bb(\phi,r)|$. For the point $x \defl \gamma(m)\bb(\phi,r)$ we have $x \in \B\left(0,\frac{r}{2}\right)$. Since $P_0(\phi,x,r) = V(\phi,0,r) = \hat V(\phi,0,r)$ we deduce from Lemma~\ref{hatnohatdiff}, Lemma~\ref{hathatdiff} and Lemma~\ref{trivialdevbound} together with $|x| = R \epsilon(2R)$ that
\begin{align*}
P_1(\phi,x,r) & + P_2(\phi,x,r) + P_3(\phi,x,r) + P_4(\phi,x,r) \\
 & = V(\phi,x,r) - P_0(\phi,x,r) \\
 & \leq \left|V(\phi,x,r) - \hat V(\phi,x,r) \right| + \hat V(\phi,x,r) - \hat V(\phi,0,r) \\
 & \leq \bc_{\ref{hatnohatdiff}}(m)\bnu(m)r^m(2r|x|^3 + r^2|x|^2) \\
 & \quad + \bc_{\ref{hathatdiff}}(m)\bnu(m)r^{m+4}\bigl(\balpha(m)m2^{m-1}\epsilon(2R)2 \\
 & \qquad \qquad \quad + (1 + \balpha(m))(\exc^{m*}(\phi,0,2R) + \exc^m_*(\phi,x,R))\bigr) \,.
\end{align*}
Define
\[
\bc(m) \defl 3\max\{\balpha(m)m2^m,\balpha(m) + 1\} \,,
\]
and divide the estimate above by $\bnu(m)r^{m+2}$. Recalling the definition of $\epsilon$, hypotheses (2), (3) and $2|x| \leq r$,
\begin{align}
\nonumber
\bP_1(\phi,x,r) & + \bP_2(\phi,x,r) + \bP_3(\phi,x,r) + \bP_4(\phi,x,r) \\
\nonumber
 & \leq |x|^2\bc_{\ref{hatnohatdiff}}(m)\left(\frac{2|x|}{r} + 1\right) \\
\nonumber
 & \quad + \bc_{\ref{hathatdiff}}(m)3^{-1}\bc(m)r^{2}\bigl(\epsilon(2R) + \exc^{m*}(\phi,0,2R) + \exc^m_*(\phi,x,R))\bigr) \\
\label{sumbound}
 & \leq |x|^22\bc_{\ref{hatnohatdiff}}(m) + \bc_{\ref{hathatdiff}}(m)\bc(m)r^{2}\epsilon(2R) \,.
\end{align}
We further observe that according to Lemma~\ref{densityintegralbound1},
\begin{align*}
\bP_2(\phi,x,r) & \geq -2|x|^2\bnu(m)^{-1}r^{-m-2}\int_{\B(0,r)}r^2 - |y|^2 \, d\phi(y) \\
 & \geq - 8|x|^2 \,,
\end{align*}
as well as,
\begin{align*}
|\bP_3(\phi,x,r)| & \leq 4|x|^2\bnu(m)^{-1}r^{-m-2}\int_{\B(0,r)}|x|^2|y|^2 \, d\phi(y) \\
 & \leq 8(m+2)|x|^2 \,,
\end{align*}
and $\bP_4(\phi,x,r) \geq 0$. Together with \eqref{sumbound} this yields
\begin{equation}
\label{p1bound}
\bP_1(\phi,x,r) \leq |x|^2(2\bc_{\ref{hatnohatdiff}}(m) + 8 + 8(m+2)) + \bc_{\ref{hathatdiff}}(m)\bc(m)r^{2}\epsilon(2R) \,.
\end{equation}
Finally recall that $x = \gamma(m)\bb(\phi,r)$ so that $|x|^2 = \gamma(m)^2|\bb(\phi,r)|^2$ and
\begin{align*}
\bP_1(\phi,x,r) & = \frac{4}{\bnu(m)r^{m+2}} \biggl\langle x, \int_{\B(0,r)} y(r^2 - |y|^2) \, d\phi(y) \biggr \rangle \\
 & = 4 \langle x, \bb(\phi,r) \rangle = 4\gamma(m)|\bb(\phi,r)|^2 \,.
\end{align*}
Therefore, by \eqref{prepgammaeta}, \eqref{p1bound} becomes,
\[
\eta(m)|\bb(\phi,r)|^2 \leq \bc_{\ref{hathatdiff}}(m)\bc(m)r^{2}\epsilon(2R) \,,
\]
and in turn with \eqref{prepgammaeta2}:
\begin{align}
\label{bbbound}
|\bb(\phi,r)| & \leq \sqrt{\eta(m)^{-1}\bc_{\ref{hathatdiff}}(m)\bc(m)}r\sqrt{\epsilon(2\sqrt{r})}  \,.
\end{align}
We recall that according to the initial dichotomy either \eqref{bbbound} holds true or
\[
|\bb(\phi,r)| \leq \gamma(m)^{-1}r \sqrt{\epsilon(\sqrt{r})} \,.
\]
This proves the lemma.
\end{proof}

The following proposition is the key estimate of these moment computations.

\begin{Prop}{\cite[Proposition~4.4.1]{DePauw3}}
	\label{quadratformprop}
There is a constant $\bc_{\theThm}(m) > 0$ with the following property. Let $x \in \Hi$, $0 < 2\sqrt{r} \leq r_0 \leq 1$ and $\xi : (0,r_0] \to [0,1]$ be a continuous increasing function and assume that
\begin{enumerate}
	\item $\Theta^m(\phi,0) = \Theta^m(\phi,x) = 1$,
	\item $|x| = r \max\left\{\sqrt[8]{r}, \sqrt[4]{\xi(2\sqrt{r})}\right\}$,
	\item $\exc_*^m(\phi,y,\rho) \leq \xi(\rho)$ for $ 0 < \rho \leq \sqrt{r}$ and $y \in \B(0,r)$,
	\item $\exc^{m}(\phi,y,\rho) \leq \xi(\rho)$ for $0 < \rho \leq 2 \sqrt r$ and $y \in \{0,x\}$.
\end{enumerate}
Then
\[
\left|\bQ(\phi,r)(x) - |x|^2\right| \leq \bc_{\theThm}(m)|x|^2 \max\biggl\{ \sqrt[8]{r}, \sqrt[4]{\xi\left(2\sqrt{r}\right)}\biggr\} \,.
\]
\end{Prop}

\begin{proof}
First note that as in the proof of Lemma~\ref{firstmombound} above, Lemma~\ref{hatnohatdiff} together with Lemma~\ref{hathatdiff} imply that
\begin{align}
\nonumber
|P_1(\phi,x,r) & + P_2(\phi,x,r) + P_3(\phi,x,r) + P_4(\phi,x,r)| \\
\nonumber
 & = |V(\phi,x,r) - P_0(\phi,x,r)| \\
\nonumber
 & \leq \left|V(\phi,x,r) - \hat V(\phi,x,r) \right| + \left|\hat V(\phi,x,r) - \hat V(\phi,0,r)\right| \\
\nonumber
 & \leq \bc_{\ref{hatnohatdiff}}(m)\bnu(m)r^m(r|x|^3(1 + \xi(r)) + r^2|x|^2\xi(2r)) \\
\label{sumbound2}
 & \quad + \bc_{\ref{hathatdiff}}(m)\bnu(m)r^{m+4}\bigl(|\dev^m(\phi,x,0,r)| + 2\xi(r) \bigr) \,.
\end{align}
Next we estimate $|\dev^m(\phi,x,0,r)|$ by
\begin{align}
\nonumber
|\balpha(m)^{-1}\dev^m(\phi,x,0,r)| & = \left|\frac{\phi(\B(0,r))}{\balpha(m)r^m} - \frac{\phi(\B(x,r))}{\balpha(m)r^m}\right| \\
\nonumber
 & \leq \left|\frac{\phi(\B(0,r))}{\balpha(m)r^m} - \Theta^m(\phi,0)\right| + \left|\Theta^m(\phi,x) - \frac{\phi(\B(x,r))}{\balpha(m)r^m} \right| \\
\nonumber
 & \leq \exc^m(\phi,0,r) + \exc^m(\phi,x,r) \\
\label{devbound}
 & \leq 2 \xi(r) \,.
\end{align}
To simplify the writing, we introduce the following notation:
\[
\eta(r) \defl \max\left\{\sqrt[8]{r}, \sqrt[4]{\xi(2\sqrt{r})}\right\} \,.
\]
Dividing \eqref{sumbound2} by $\bnu(m)r^{m+2}$ and using \eqref{devbound} and hypothesis (2), we obtain
\begin{align}
\nonumber
|\bP_1(\phi,x,r) & + \bP_2(\phi,x,r) + \bP_3(\phi,x,r) + \bP_4(\phi,x,r)| \\
\nonumber
 & \leq \bc_{\ref{hatnohatdiff}}(m)|x|^2\left(\frac{|x|}{r}(1 + \xi(r)) + \xi(2r)\right) \\
\nonumber
 & \quad + \bc_{\ref{hathatdiff}}(m)r^{2}\bigl(|\dev^m(\phi,x,0,r)| + 2\xi(r)\bigr) \\
\nonumber
 & \leq \bc_{\ref{hatnohatdiff}}(m)|x|^2\left(\eta(r)(1 + \xi(r)) + \xi(2r)\right) \\
\label{sumbound3}
 & \quad + \bc_{\ref{hathatdiff}}(m)2(1 + \balpha(m))|x|^{2}\eta(r)^{-2}\xi(r) \,.
\end{align}
According to Lemma~\ref{firstmombound} we also have that
\begin{align}
\nonumber
|\bP_1(\phi,x,r)| & = 4|\langle x, \bb(\phi,r) \rangle| \\
\nonumber
 & \leq 4 \bc_{\ref{firstmombound}}(m)|x|r\max\left\{\sqrt[4]{r}, \sqrt{\xi(2\sqrt{r})}\right\} \\
\label{p1estimate}
 & = 4\bc_{\ref{firstmombound}}(m)|x|^2 \eta(r) \,.
\end{align}
Furthermore,
\begin{align}
\nonumber
|\bP_3(\phi,x,r)| & \leq 4|x|^2\bnu(m)^{-1}r^{-m-2}\int_{\B(0,r)}|x||y| \, d\phi(y) \\
\nonumber
 & \leq 4|x|^2\bnu(m)^{-1}r^{-m-2}r\eta(r)r\phi(\B(0,r)) \\
\label{p3estimate}
 & \leq 4(m+2)|x|^2\eta(r)(1 + \xi(r)) \,,
\end{align}
as well as
\begin{align}
\nonumber
|\bP_4(\phi,x,r)| & = |x|^4\bnu(m)^{-1}r^{-m-2}\phi(\B(0,r)) \\
\nonumber
 & \leq |x|^2\eta(r)^2r^2\bnu(m)^{-1}r^{-m-2}\phi(\B(0,r)) \\
\label{p4estimate}
 & \leq (m+2)|x|^2\eta(r)^2(1 + \xi(r)) \,.
\end{align}
Plugging \eqref{p1estimate},\eqref{p3estimate} and \eqref{p4estimate} into \eqref{sumbound3}, and observing that $\xi(2r) \leq \eta(r)$ as well as $\eta(r)^{-2}\xi(r) \leq \sqrt{\xi(r)}$, we find that
\begin{align}
\nonumber
|\bP_2(\phi,x,r)| & \leq \bc_{\ref{hatnohatdiff}}(m)|x|^2\left(\eta(r)(1 + \xi(r)) + \xi(2r)\right) \\
\nonumber
 & \quad + 2(1 + \balpha(m))\bc_{\ref{hathatdiff}}(m)|x|^{2}\eta(r)^{-1}\xi(r) \\
\nonumber
 & \quad + 4\bc_{\ref{firstmombound}}(m)|x|^2 \eta(r) \\
\nonumber
 & \quad + 5(m+2)|x|^2\eta(r)(1 + \xi(r)) \\
\label{p2estimate}
 & \leq \bc(m) |x|^2\eta(r) \,,
\end{align}
for some $\bc(m) > 0$ depending only on $m$. Finally, recalling the definition of $P_2(\phi,x,r)$ and referring to Lemma~\ref{densityintegralbound1}, it is an easy matter to check that
\begin{equation}
\label{finalmoment}
4|\bQ(\phi,r)(x) - |x|^2| \leq 4\exc^m(\phi,0,x)|x|^2 + |\bP_2(\phi,x,r)| \,.
\end{equation}
Plugging \eqref{p2estimate} into \eqref{finalmoment} yields the expected estimate.
\end{proof}

The lack of local compactness of the Hilbert space $\Hi$ prevents us from showing that $\spt(\phi)$ is Reifenberg flat in a neighborhood of the origin as done in \cite{DePauw3}. But if $\phi = \|T\|$ is Ahlfors regular we have additional structure. First we know that tangent planes exist almost everywhere by Lemma~\ref{tangentplane} and a slicing argument as used in the proof of Lemma~\ref{orthogonalfamily} below allows us to find orthogonal frames in the support of $\phi = \|T\|$ at all small scales around a point that possesses a tangent plane. A priori, the closeness to a tangent plane at a given scale depends on the particular point, but the moment computations above can be used to make this scale uniform in some small neighborhood.

\subsection{Uniform closeness to planes}

We now show how to find an orthogonal family in the support of a rectifiable chain. This is similar to \cite[Proposition 4.6.2]{DePauw3}, although simpler, because of the additional structure of a rectifiable chain we don't need to assume that the support is Reifenberg flat. We actually will use this family in order to show that the support of some almost monotonic chain is Reifenberg flat. Given a Radon measure $\phi$ in $\Hi$, $x \in \Hi$, $r > 0$ and $W \in \bG(\Hi,m)$ we define
\begin{align*}
\bbeta_2(\phi,x,r,W) & \defl \biggl(r^{-m-2}\int_{\B(x,r)} |\pi_{W^\perp}(y - x)|^2 \, d\phi(y)\biggr)^{\frac{1}{2}} \,, \\
\bbeta_\infty(\phi,x,r,W) & \defl r^{-1}\sup\{|\pi_{W^\perp}(y - x)| : y \in \spt(\phi) \cap \B(x,r)\} \,.
\end{align*}

\begin{Lem}
\label{orthogonalfamily}
Let $T \in \cR_m(\Hi;G)$ and assume that $W \in \bG(\Hi,m)$ and $0 < \rho \leq (25\sqrt{m})^{-1}$ are such that
\begin{enumerate}
	\item $\spt(T) \subset \B(0,1)$ and $\spt(\partial T) \subset \partial\B(0,1)$,
	\item $\pi_{W\#} (T \res Z_W(0,2^{-1})) \neq 0$,
	\item $\bbeta_\infty(\|T\|,0,1,W) < \rho$.
\end{enumerate}
Let $\rho' \defl m^\frac{1}{4}\rho^\frac{1}{2}$. Then $\rho' \leq \frac{1}{5}$ and for every $s \in (2\rho',1]$ there is an orthonormal family $e_1,\dots,e_m \in \Hi$ with $s e_i \in \spt(T)$, $i=1,...,m$.
\end{Lem}

\begin{proof}
Note that if $v_1,\dots,v_m \in \Hi$ are orthonormal vectors with $|\pi_{W^\perp}(v_i)| \leq c$ and $V \defl \spa\{v_1,\dots,v_m\} \in \bG(\Hi,m)$, then 
\begin{equation}
\label{distbound}
\hdist(\B_V(0,1), \B_W(0,1)) \leq \sup \{|\pi_{W^\perp}(x)| : x \in \B_{V}(0,1)\} \leq \sqrt{m}c \,.
\end{equation}
By hypothesis (3) (and since $1 - \rho \leq \sqrt{1 - \rho^2}$), $\spt(\partial T) \cap Z_W(1-\rho) = \emptyset$ and the constancy theorem implies together with (2) that there is some $g_0 \neq 0$ with
\[
\pi_{W\#} (T \res Z_W(1-\rho)) = g_0 \curr{\B_W(0,1-\rho)} \,.
\]
For an $m$-plane $V$ and $r > 0$ let $\psi_{V,r} : \B(0,1) \setminus Z_V(r) \to V$ be the Lipschitz map given by $\psi_{V,r}(x) \defl |x| |\pi_V(x)|^{-1} \pi_V(x)$. This map preserves the norm and orthogonality in the following sense,
\begin{equation}
\label{normpreserve}
|\psi_{V,r}(x)| = |x| \,, \quad \mbox{ and for } y \in V \quad \psi_{V,r}(x) \perp y \Leftrightarrow x \perp y \,.
\end{equation}
The fist statement is obvious, for the second note that if $x = v + v^\perp$ with $v \in V$ and $v \in V^\perp$, then $\psi_{V,r}(x) = \lambda v$ for some $\lambda \neq 0$ and hence
\[
\langle \psi_{V,r}(x),y \rangle = \lambda \langle v,y \rangle = \lambda \langle x,y \rangle \,,
\]
for $y \in V$. Let $H_{V,r} : [0,1] \times (\B(0,1) \setminus Z_V(r)) \to V$ be the Lipschitz homotopy $H_{V,r}(t,x) = t\pi_V(x) + (1-t)\psi_{V,r}(x)$ between $\psi_{V,r}$ and $\pi_V$. By assumption $\pi_W : \spt(T) \cap Z_W(1-\rho) \to \B_W(0,1-\rho)$ is surjective. We next show that
\begin{equation}
\label{setinclusion}
B_W(0,1) \setminus \B_W(0,2\rho) \subset \psi_{W,\rho}(\spt(T) \setminus Z_W(\rho)) \,.
\end{equation}
This is essentially a consequence of the constancy theorem. More precisely, consider the rectifiable chain $S \defl \curr{0,1} \times (T \res Z_W(\rho)^c) \in \cR_{m+1}(\R\times\Hi;G)$ and the set $B_\rho \defl \B_W(0,1-\rho) \setminus \B_W(0,2\rho)$. Clearly, $H_{W,\rho\#} S = 0$ because the image lies in an $m$-dimensional plane. By the homotopy formula for $G$-chains (compare with the proof of Lemma~\ref{differentplane}),
\begin{align*}
0 & = \partial (H_{W,\rho\#} S) = H_{W,\rho\#} \partial S \\
 & = H_{W,\rho\#}\bigl(\curr 1 \times (T \res Z_W(\rho)^c) - \curr 0 \times (T \res Z_W(\rho)^c) - \curr{0,1} \times \partial(T \res Z_W(\rho)^c)\bigr) \\
 & = \pi_{W\#} (T \res Z_W(\rho)^c) - \psi_{W,\rho\#} (T \res Z_W(\rho)^c) - H_{W,\rho\#}\bigl(\curr{0,1} \times \partial(T \res Z_W(\rho)^c)\bigr) \,.
\end{align*}
Because of hypothesis (1) and (3), $H_{W,\rho\#}(\spt([0,1] \times \partial(T \res Z_W(\rho)^c)) \subset W \setminus B_\rho$ and on $B_\rho$, $(\pi_{W\#} (T \res Z_W(\rho)^c)) \res B_\rho = g_0\curr{B_\rho}$. So the same must be true for $\psi_{W,\rho\#} (T \res Z_W(\rho)^c)$, and since $g_0 \neq 0$, the inclusion in \eqref{setinclusion} holds by the constancy theorem because $\psi_{W,\rho}(\partial(T \res Z_W(\rho)^c)) \subset \partial \B_W(0,1) \cup \B_W(0,2\rho)$.

Pick some $w_1 \in W$ with $|w_1| = s$, where $s \in (2\rho',1]$. Since $2\rho \leq 2 \rho' < s$, \eqref{setinclusion} implies the existence of some $e_1 \in \Hi$ with $se_1 \in \spt(T)$ and $\psi_{W,\rho}(se_1) = w_1$. Because $\psi_{W,\rho}$ preserves the norm by \eqref{normpreserve}, $e_1$ is of unit length. Let $V_1$ be the orthogonal complement of $w_1$ in $W$. Again by \eqref{normpreserve}, the vector $e_1$ is also orthogonal to $V_1 \subset W$. Consider the new $m$-plane $W_1$ spanned by $e_1$ and $V_1$. There holds $|\pi_{W^\perp}(se_1)| \leq \rho$ because $se_1 \in \spt(T)$ and hence $\hdist(\B_W(0,1),\B_{W_1}(0,1)) \leq \sqrt{m}\rho s^{-1}$ by \eqref{distbound}. Accordingly, 
\begin{align*}
	\bbeta_\infty(\|T\|,0,1,W_1) & \leq \bbeta_\infty(\|T\|,0,1,W) + \hdist(\B_W(0,1),\B_{W_1}(0,1)) \\
 & < (1 + \sqrt{m}s^{-1})\rho \leq 2\sqrt{m} \rho s^{-1} \leq \sqrt{m} \rho \rho'^{-1} \\
 & = \rho' \leq 5^{-1} \,.
\end{align*}
Because of Lemma~\ref{differentplane}, the new plane $W_1$ satisfies the same hypotheses as $W$ with $\rho'$ in place of $\rho$.

Assume that for $k < m$ we have already constructed some orthonormal vectors $e_1,\dots,e_k$ orthogonal to some $m-k$ dimensional subspace $V_k \subset W$ with $se_i \in \spt(T)$ for all $i$. With $W_k$ we denote the $m$-plane spanned by $e_1,\dots,e_k$ and $V_k$. Now, pick a $w_{k+1} \in V_k$ with $|w_{k+1}| = s$ and $s \in (2\rho',1]$. The same calculation as for $W_1$ gives that
\[
\max\{\bbeta_\infty(\|T\|,0,1,W_k), \hdist(\B_W(0,1),\B_{W_k}(0,1))\} < \rho' \leq 5^{-1} \,,
\]
and as in \eqref{setinclusion} there is some $e_{k+1} \in \Hi$ with $se_{k+1} \in \spt(T)$ and $\psi_{W_k,\rho'}(se_{k+1}) = w_{k+1}$. Let $V_{k+1}$ be the orthogonal complement of $w_{k+1}$ in $V_k$. Since $w_{k+1}$ is orthogonal to each of the vectors $e_1,\dots,e_k$ and to $V_{k+1}$ the same holds for $e_{k+1}$ by \eqref{normpreserve} applied to the map $\psi_{W_k,\rho'}$. The new plane $W_{k+1}$ has the same properties we obtained for $W_k$. Proceeding this way we get the desired orthonormal vectors $e_1, \dots, e_m$.
\end{proof}

Lemma~\ref{densityintegralbound2} together with Proposition~\ref{quadratformprop} allow us to control $\bbeta_2$ with respect to some plane that is spanned by an orthogonal frame in the support of $\phi$. This corresponds to Lemma~4.7.2 and Proposition~4.7.3 in \cite{DePauw3}.

\begin{Prop}
	\label{betainftysmall}
There is a constant $\bc_{\theThm}(m) > 0$ with the following property. Let $\phi$ be a finite Borel measure on $\Hi$, $x_i,x_{i,k} \in \Hi$ for $i = 1,\dots,m$ and $k \geq 1$ be a sequence of points. Further assume that $0 < 4\sqrt{r} < r_0 \leq 1$ and $\xi : (0,r_0] \to (0,\frac{1}{2}]$ is a continuous increasing function such that
\begin{enumerate}
	\item $\phi(\partial \B(0,r)) = 0$,
	\item $\lim_{k \to \infty}|x_{i,k} - x_i| = 0$ for all $i$,
	\item $x_{i} \perp x_{j}$ for $i \neq j$, and $|x_{i}| = r \eta(r)$ for
	\[
	\eta(r) \defl \max\biggl\{ \sqrt[8]{r}, \sqrt[4]{\xi\left(2\sqrt{r}\right)} \biggr\} \,,
	\]
	\item $\Theta^m(\phi,0) = \Theta^m(\phi,x_{i,k}) = \Theta^m(\phi,y) = 1$ for all $i$, $k$ and $\phi$-almost all $y \in \B(0,r)$,
	\item $\exc_*^m(\phi,y,\rho) \leq \xi(\rho)$ for $0 < \rho \leq 4\sqrt{r}$ and $y \in \B(0,2r)$,
	\item $\exc^{m*}(\phi,y,\rho) \leq \xi(\rho)$ for $0 < \rho \leq 4\sqrt{r}$ and $y \in \{0,x_{i,k} : i,k \geq 1 \}$.
\end{enumerate}
Then
\[
\bbeta_\infty(\phi,0,r/2,W) < \bc_{\theThm}(m)\eta(r)^{\frac{1}{m+2}} \,,
\]
where $W \defl \spa\{x_1,\dots,x_m\}$.
\end{Prop}

\begin{proof}
If $k$ is big enough, then $W_k \defl \spa\{x_{1,k},\dots,x_{m,k}\}$ is an $m$-dimensional subspace of $\Hi$ and $W_k$ converges to $W$. Without loss of generality we assume that $W_k \in \bG(\Hi,m)$ for all $k$ and $0 < |x_{i,k}| < 2r\eta(r) \leq 2r$ for all $i$ and $k$. Let $e_i,f_{i,k} \in \Hi$ be the unit vectors such that $|x_{i,k}|f_{i,k} = x_{i,k}$ and $|x_{i}|e_{i} = x_{i}$ and let $e_{m+1,k},e_{m+2,k},e_{m+3,k},\dots$ be an orthonormal basis of $W_k^\perp$. By assumption $s \mapsto s \eta(s)$ is continuous, strictly increasing and satisfies $\lim_{s \downarrow 0} s \eta(s) = 0$. Hence there is a unique $r_{i,k} > 0$ for which $|x_{i,k}| = r_{i,k}\eta(r_{i,k})$ for some $0 < r_{i,k} < 2r$ and
\begin{equation}
\label{limit1}
\lim_{k \to \infty} \max_{1 \leq i \leq m}|r_{i,k} - r| = \lim_{k \to \infty} \max_{1 \leq i \leq m}|\eta(r_{i,k}) - \eta(r)| = 0 \,.
\end{equation}
Further, there is an orthonormal basis $e_{1,k}, \dots, e_{m,k}$ of $W_k$ with 
\begin{equation}
\label{limitofframe}
\lim_{k \to \infty} \max_{1 \leq i \leq m}|r \eta(r) e_{i,k} - x_{i,k}| = \lim_{k \to \infty} \max_{1 \leq i \leq m}|e_{i,k} - f_{i,k}| = 0 \,.
\end{equation}
$e_{i,k}$ can for example be constructed via the Gram-Schmidt procedure from $e_{i,k}$. From hypothesis (1) it follows that
\begin{equation}
\label{limitofframe2}
\lim_{\epsilon \downarrow 0} \phi(\B(0,r + \epsilon)) - \phi(\B(0,r - \epsilon)) = 0 \,.
\end{equation}

By the assumption on $\xi$ we have for all $0 < \rho \leq r$,
\[
\left| \frac{\phi(\B(0,\rho))}{\balpha(m) \rho^m} - 1 \right| \leq \max\{\exc^{m}_*(\phi,0,r), \exc^{m*}(\phi,0,r)\} \leq \xi(r) \,.
\]
Lemma~\ref{densityintegralbound2} then implies that for all $k$,
\begin{equation}
\label{tracebound}
\Biggl|m - \sum_{i \geq 1} \frac{1}{\bnu(m)r^{m+2}}\int_{\B(0,r)} \langle e_{i,k},y \rangle^2 \, d\phi(y)\Biggr| = |m - \tr \bQ(\phi,r)| \leq (m+4)\xi(r) \,.
\end{equation}
Additionally, Proposition~\ref{quadratformprop} implies that for all $i,k$,
\begin{align}
\nonumber
\Biggl|1 - \frac{1}{\bnu(m)r_{i,k}^{m+2}}\int_{\B(0,r_{i,k})} \langle f_{i,k},y \rangle^2 \, d\phi(y)\Biggr| & = \frac{1}{|x_{i,k}|^{2}}\left||x_{i,k}|^2 - \bQ(\phi,r_{i,k})(x_{i,k})\right| \\
\label{framebound}
 & \leq \bc_{\ref{quadratformprop}}(m) \eta(r_{i,k}) \,.
\end{align}
From \eqref{limit1}, \eqref{limitofframe} and \eqref{limitofframe2} we get
\begin{align}
	\nonumber
\epsilon_k & \defl \max_{1 \leq i \leq m} \biggl|\frac{1}{\bnu(m)r_{i,k}^{m+2}} \int_{\B(0,r_{i,k})} \langle f_{i,k},y \rangle^2 \, d\phi(y) \\
\label{errorbound}
 & \qquad \qquad - \frac{1}{\bnu(m)r^{m+2}}\int_{\B(0,r)} \langle e_{i,k},y \rangle^2 \, d\phi(y) \biggr| \to 0 \,,
\end{align}
for $k \to \infty$.

Since $\xi(r) \leq \eta(r)$ it follows from \eqref{limit1}, \eqref{tracebound}, \eqref{framebound} and \eqref{errorbound} that for $\bc(m) \defl \bnu(m)(m\bc_{\ref{quadratformprop}}(m) + 2m + 4)$,
\begin{align}
	\nonumber
\bbeta_2(\phi,0,r,W_k)^2 & = \frac{1}{r^{m+2}}\int_{\B(0,r)} |\pi_{W_k^\perp}(y)|^2 \, d\phi(y) \\
	\nonumber
& = \sum_{i > m} \frac{1}{r^{m+2}}\int_{\B(0,r)} \langle e_{i,k},y \rangle^2 \, d\phi(y) \\
	\nonumber
& \leq \bnu(m)\tr \bQ(\phi,r) - \sum_{i \leq m}\frac{1}{r^{m+2}}\int_{\B(0,r)} \langle e_{i,k},y \rangle^2 \, d\phi(y) \\
	\nonumber
& \leq \bnu(m)((m+4)\xi(r) + m\bc_{\ref{quadratformprop}}(m) \eta(r_{i,k}) + m\epsilon_k) \\
	\label{beta2est}
& \leq \bc(m) \max_{1 \leq i \leq m}\{\eta(r),\eta(r_{i,k}),\epsilon_k\} \to \bc(m) \eta(r) \,,
\end{align}
for $k \to \infty$. Let $\delta \defl \bbeta_2(\phi,0,r,W_k)$ and define
\begin{align*}
q_{m} & \defl \frac{2}{m+2} \,, \\
c_m & \defl 1 + (4\balpha(m)^{-1})^\frac{1}{m} \,, \\
\delta_{m} & \defl \left(2(c_m - 1) \right)^{- \frac{1}{q_m}} \,.
\end{align*}
If $\delta \geq \delta_m$, then
\[
\bbeta_\infty(\phi,0,r/2,W_k) \leq 1 \leq (\delta_m^{-1}\delta)^{q_m} \leq \delta_m^{-q_m} \delta^{q_m} \,.
\]
We therefore assume that $\delta < \delta_m$. Consider the set
\[
B \defl \left\{y \in \B(0,r) \cap \spt(\phi) : |\pi_{W_k^\perp}(y)| \geq \delta^{q_{m}}r \right\} \,.
\]
Observe that
\begin{align*}
	\delta^2 & = \frac{1}{r^{m+2}}\int_{\B(0,r)} |\pi_{W_k^\perp}(y)|^2 \, d\phi(y) \\
	& \geq \frac{1}{r^{m+2}} \phi(B) \delta^{2q_m}r^2 \,,
\end{align*}
and therefore
\begin{equation}
\label{measureest}
\phi(B) \leq r^m \delta^{2(1 - q_{m})} \,.
\end{equation}
Now assume there exists $y \in \B(0,r/2) \cap \spt(\phi)$ with
\[
|\pi_{W_k^\perp}(y)| > c_m \delta^{q_{m}} r \,.
\]
Put $\rho \defl (c_m - 1)\delta^{q_m}r$ and notice that $\B(y,\rho) \subset \oB(0,r)$ since
\[
\rho = (c_m - 1)\delta^{q_m}r < (c_m - 1)\delta_m^{q_m}r = (c_m - 1)\left(2(c_m - 1) \right)^{- 1}r \leq \tfrac{1}{2}r \,.
\]
By our assumption on the densities we can assume that $\Theta^m(\phi,y) = 1$. The above implies that $\spt(\phi) \cap \B(y,\rho) \subset B$ and with \eqref{measureest} we obtain
\begin{equation}
\label{measureest2}
\phi(\B(y,\rho)) \leq \phi(B) \leq r^m \delta^{2(1 - q_{m})} \,.
\end{equation}
Since $\xi(\rho) \leq \xi(r_0) \leq \frac{1}{2}$, the bound on $\exc_*^m(\phi,y,\rho)$ implies that
\[
\frac{\phi(\B(y,\rho))}{\balpha(m)\rho^m} \geq 1 - \xi(\rho) \geq \tfrac{1}{2} \,.
\]
Combining this with \eqref{measureest2} we obtain
\[
\tfrac{1}{2}\balpha(m)(c_m - 1)^m\delta^{mq_m}r^m = \tfrac{1}{2}\balpha(m)\rho^m \leq \phi(\B(y,\rho)) \leq r^m \delta^{2(1 - q_{m})} \,.
\]
Since $mq_m = \frac{2m}{m+2} = 2(1 - \frac{2}{m+2}) = 2(1 - q_{m})$, this gives a contradiction,
\[
1 \geq \tfrac{1}{2}\balpha(m)(c_m - 1)^m = \tfrac{1}{2}\balpha(m)4\balpha(m)^{-1} = 2  \,.
\]
Therefore
\[
\bbeta_\infty(\phi,0,r/2,W_k) \leq \max\{\delta_m^{-q_m},c_m\}\delta^{q_m} = \max\{2(c_m - 1),c_m\} \delta^{\frac{2}{m+2}}\,.
\]
Since $\delta = \bbeta_2(\phi,0,r,W_k)$, \eqref{beta2est} implies that
\[
\limsup_{k\to\infty}\bbeta_\infty(\phi,0,r/2,W_k) \leq \bc'(m)\eta(r)^\frac{1}{m+2} \,,
\]
for some constant $\bc'(m) > 0$ depending only on $m$. Because $W_k \cap \B(0,1)$ converges in Hausdorff distance to $W \cap \B(0,1)$, we get $\bbeta_\infty(\phi,0,r/2,W_k) \to \bbeta_\infty(\phi,0,r/2,W)$ and the proposition follows.


\end{proof}

This is the special case of the proposition above in case $\xi$ is the constant function.

\begin{Cor}
	\label{betainftysmall2}
There is a constant $\bc_{\theThm}(m) > 0$ with the following property. Let $\phi$ be a finite Borel measure on $\Hi$, $x_i,x_{i,k} \in \Hi$ for $i = 1,\dots,m$ and $k \geq 1$ and assume that $\epsilon, r_0,r > 0$ are such that:
\begin{enumerate}
	\item $0 < 4\sqrt{r} \leq r_0 \leq \epsilon \leq \frac{1}{2}$,
	\item $\phi(\partial \B(0,r)) = 0$,
	\item $\lim_{k \to \infty}|x_{i,k} - x_i| = 0$ for all $i$,
	\item $x_i \perp x_j$ for $i \neq j$ and $|x_i| = r \epsilon^\frac{1}{4}$,
	\item $\Theta^m(\phi,0) = \Theta^m(\phi,x_{i,k}) = \Theta^m(\phi,y) = 1$ for all $i$, $k$ and $\phi$-almost all $y \in \B(0,r_0)$.
	\item $\exc_*^m(\phi,y,r_0) \leq \epsilon$ for $y\in \B(0,2r_0)$,
	\item $\exc^{m*}(\phi,y,r_0) \leq \epsilon$ for $y \in \{0,x_{i,k} : i,k \geq 1\}$.
\end{enumerate}
Then
\[
\bbeta_\infty(\phi,0,r/2,W) < \bc_{\theThm}(m)\epsilon^{\frac{1}{4(m+2)}} \,,
\]
where $W \defl \spa\{x_1,\dots,x_m\}$.
\end{Cor}

\section{Regularity of almost minimizer}

First we define the main objects of this paper, namely almost mass minimizing rectifiable chains with respect to some gauge $\xi$. This is an adaptation of the original definition of Almgren in \cite{A} to chains. Almgren's original definition is harder to work with since competing surfaces have to be obtained by Lipschitz deformations of the original one and cut and paste constructions are not allowed. Due to slicing, such cut and paste constructions are easily available for rectifiable $G$-chains and this simplifies the arguments greatly. For this reason we follow Bombieri and his definition of almost mass minimizing integral currents \cite{B}.

\begin{Def}
	\label{almostminimal}
A rectifiable chain $T \in \cR_m(\Hi;G)$ is $(\bM,\xi,\delta)$-\emph{minimal} in a set $A \subset \Hi \setminus \spt(\partial T)$ if $\xi : (0,\delta] \to \R_+$ is a gauge and the following holds: For every $x \in A$, $0 < r < \min\{\delta,\dist\{x,\spt(\partial T)\}\}$ and every $S \in \cR_m(\Hi;G)$ with 
\begin{enumerate}
	\item $\spt(S) \subset \B(x,r)$,
	\item $\partial S = 0$,
\end{enumerate}
there holds
\[
	\bM(T \res \B(x,r)) \leq (1 + \xi(r))\bM(T \res \B(x,r) + S) \,.
\]
\end{Def}

We will further assume that $\xi$ is continuous and satisfies the Dini condition,
\[
	\int_0^{\delta} \frac{\xi(t)}{t} \, dt < \infty \,.
\]
The next result is a simple adjustment of \cite[Proposition~3.4.5]{DePauw3} to the setting of rectifiable $G$-chains. Note that because of the definition of almost minimality we use here, radius of balls instead of diameter of sets, there is no factor $2$ appearing in the lemma below.

\begin{Lem}
	\label{nearlymonlem}
Let $T \in \cR_m(\Hi;G)$ be $(\bM,\xi,\delta)$-minimal in $A \subset \Hi \setminus \B(\spt(\partial T),\delta)$ and define
\[
	\Xi(r) \defl m\int_0^r \frac{\xi(t)}{t} \, dt \,,
\]
for every $0 < r \leq \delta$. Then $T$ is almost monotonic in $A$ with gauge $\Xi$.
\end{Lem}

\begin{proof}
For $x \in A$ and $0 < r < \delta$ define $f_x(r) \defl \|T\|(\B(x,r))$. Let $d_x : \Hi \to \R$ be the distance function to $x$ and assume that $r$ is such that $\langle T, d_x, r\rangle \in \cR_{m-1}(\Hi;G)$ as well as $\partial (T \res \B(x,r)) = \langle T, d_x, r\rangle \in \cR_{m-1}(\Hi;G)$. This holds for almost all $0 < r < \delta$ because of \cite[Theorem~5.2.4]{PH} and by \eqref{slicemass},
\[
\int_0^r \bM(\langle T, d_x, s\rangle) \, ds \leq \bM(T \res \B(x,r)) = f_x(r) \,.
\]
Hence for almost all $r$,
\begin{equation}
\label{derivativemass}
\bM(\langle T, d_x, r\rangle) \leq f_x'(r) \,.
\end{equation}
Let $\curr x \cone \langle T, d_x, r\rangle \in \cR_m(\Hi;G)$ be the cone over $\langle T, d_x, r\rangle$ with center $x$. There holds $\partial(\curr x \cone \langle T, d_x, r\rangle) = \langle T, d_x, r\rangle$ and by \eqref{cone_equality},
\begin{equation}
\label{massconeestimate}
\bM(\curr x \cone \langle T, d_x, r\rangle) = \frac{r}{m}\bM(\langle T, d_x, r\rangle) \,.
\end{equation}
The almost minimality of $T$ implies in combination with \eqref{derivativemass} and \eqref{massconeestimate} that for almost all $r$,
\begin{align*}
f_x(r) & = \bM(T \res \B(x,r)) \leq (1 + \xi(r))\bM(\curr x \cone \langle T, d_x, r\rangle) \\
 & = (1 + \xi(r))\frac{r}{m}\bM(\langle T, d_x, r\rangle) \leq (1 + \xi(r))\frac{r}{m} f_x'(r) \,.
\end{align*}
Hence, for almost all $r$ with $f_x(r) > 0$,
\begin{align*}
(\log \circ f_x)'(r) & = \frac{f_x'(r)}{f_x(r)} \geq \frac{m}{r} \frac{1}{1 + \xi(r)} \geq \frac{m}{r}(1-\xi(r)) \,.
\end{align*}
Integrating shows that for all $0 < r_1 < r_2 < \delta$ with $f_x(r_1) > 0$,
\[
\log \left( \frac{f_x(r_2)}{f_x(r_1)}\right) \geq \int_{r_1}^{r_2} \frac{m}{r}(1-\xi(r)) \, dr = \log \left( \frac{r_2^m}{r_1^m}\right) - \Xi(r_2) + \Xi(r_1) \,,
\]
respectively that
\[
\exp(\Xi(r_1))r_1^{-m} f_x(r_1) \leq \exp(\Xi(r_2))r_2^{-m} f_x(r_2) \,.
\]
If $f_x(r_1) = 0$, the statement is trivial. Hence $T$ is almost monotonic in $A$ with gauge function $\Xi$.
\end{proof}

We will encounter a similar differential equation in connection with the epiperimetric inequality in Lemma~\ref{diffequation2}.

\subsection{Polyhedral approximation and a differential inequality}

Because of the formulation of Reifenberg's epiperimetric inequality with polyhedral chains we first need some results that justify this assumption. The required results about polyhedral approximation are contained in \cite{DePauw2}. With $\theta > 0$ and the group norm $\|\cdot\|$ we associate a new group norm $\|\cdot\|_\theta$ defined on $G$ by $\|0_G\|_\theta=0$ and $\|g\|_\theta \defl \max\{\|g\|,\theta\}$ whenever $g \neq 0_G$.

\begin{Lem}
\label{polyhedralcone}
Let $T \in \cR_m(\Hi;G)$ and $r_0 \in (0,1]$ such that $\spt(\partial T) \cap \B(0,r_0) = \emptyset$. Further let $\theta > 0$ and assume that,
\begin{enumerate}[(A)]
	\item $\Theta^m(\|T\|,x) \geq \theta$ for $\|T\|$-a.e.\ $x \in \B(0,r_0)$,
	\item \label{compspt} $\spt(T) \cap \B(0,r_0)$ is compact.
\end{enumerate}
Let $\bM_\theta$ be the mass on rectifiable $G$-chains induced by the norm $\|\cdot\|_\theta$ (defined right before the Lemma). For all $s > 0$ and almost every $r \in [0,r_0]$, $T_r \defl \partial (T \res \B(0,r)) \in \cR_{m-1}(\Hi;G)$ and there is a polyhedral chain $P_{r,s} \in \cP_{m-1}(\Hi;G)$ and a rectifiable chain $R_{r,s} \in \cR_m(\Hi;G)$ such that
\begin{enumerate}
	\item \label{boundaries} $\partial (R_{r,s} + \curr 0 \cone P_{r,s}) = T_r$,
	\item \label{sestimate} $\bM_\theta(R_{r,s}) < s$ and $\spt(R_{r,s}) \subset \oB(\spt(T_r),s)$,
	\item \label{conebounds} $\max\{\bM_\theta (\bar P_{r,s}) \res \B(0,r)), \bM_\theta(\curr 0 \cone P_{r,s})\} \leq \bM(\curr 0 \cone T_{r}) + s$, where $\bar P_{r,s}$ is the infinite cone generated by $\curr 0 \cone P_{r,s}$.
\end{enumerate}
\end{Lem}

\begin{proof}
Let $f : \Hi \to \R$ be the $1$-Lipschitz function given by $f(x) \defl |x|$. As noted before, the slice $\langle T,f,r\rangle$ exists for almost every $r \in \R$ and is an element of $\cR_{m-1}(\Hi;G)$. By \cite[Theorem~5.2.4]{PH} we have that $\langle T,f,r\rangle = \partial (T \res \{f < r\}) = \partial (T \res \B(0,r)) = T_r$ for almost every $r \leq r_0$. By the compactness assumption on $T$, the chain $T_r$ has also compact support. It follows from \cite[Theorem 4.2 (E)]{DePauw2} that there are $P_{r,s} \in \cP_{m-1}(\Hi;G)$ and $R_{r,s} \in \cR_m(\Hi;G)$ such that $P_{r,s} - T_{r} = \partial R_{r,s}$, $\bM_\theta(P_{r,s}) \leq \bM_\theta(T_{r}) + s$, $\bM_\theta(R_{r,s}) < s$ and $\spt(R_{r,s}) \subset \oB(\spt(T_{r}),s)$. In particular we have that $\partial P_{r,s} = 0$. \eqref{boundaries} and \eqref{sestimate} hold by construction. By the lower bound on the densities there holds $\bM(T_r) = \bM_\theta(T_r)$ for almost all $r$ because slices inherit the group elements from the original chain. From now on we also assume that $2s \leq r$.

As stated in \eqref{cone_equality} and \eqref{cone_inequality}, $\bM(\curr 0 \cone T_r) = \frac{r}{m}\bM(T_r)$ and also,
\begin{align}
\nonumber
\bM_\theta(\curr 0 \cone P_{r,s}) \res \B(0,r - s)) & \leq \frac{r-s}{m}\bM_\theta(P_{r,s}) \leq \frac{r}{m}\bM_\theta(T_r) + \frac{rs}{m} \\
\nonumber
 & = \bM(\curr 0 \cone T_r) + \frac{rs}{m} \\
\label{chainconemasscomp}
 & \leq \bM(\curr 0 \cone T_{r}) + s.
\end{align}
Let $\bar P_{r,s}$ be a large enough (or infinite) scaled version of $\curr 0 \cone P_{r,s}$ that has its boundary outside $\oB(0,2r)$. Since $\spt(P_{r,s}) \subset \B(\spt(T_r), s)$, we have $\spt(\curr 0 \cone P_{r,s}) \subset \spt(\bar P_{r,s}) \cap \B(0,r + s)$ and hence with \eqref{chainconemasscomp} (note that $2s \leq r$),
\begin{align*}
\max\{\bM_\theta(\bar P_{r,s} \res \B(0,r)), & \bM_\theta(\curr 0 \cone P_{r,s})\}\\
 & \leq \bM_\theta(\bar P_{r,s} \res \B(0,r + s)) \\
 & \leq \frac{(r + s)^m}{(r - s)^m} \bM_\theta(\bar P_{r,s} \res \B(0,r - s)) \\
 & \leq \frac{(r + s)^m}{(r - s)^m}\left(\bM(\curr 0 \cone T_{r}) + s\right) \,.
\end{align*}
This converges to $\bM(\curr 0 \cone T_{r})$ for $s \to 0$ and by replacing $s$ with a smaller value if necessary, we obtain \eqref{conebounds}.
\end{proof}

Next we use Reifenberg's epiperimetric inequality of Theorem~\ref{masscomparison2} to obtain a differential inequality for $f(r) \defl \bM(T \res \B(x,r))$, where $T$ is an almost minimizing $G$-chain.

\begin{Prop}
\label{rundifferentialequation}
There is a constant $0 < \epsilon_{\theThm}(m) \leq \frac{1}{4}$ with the following property. Let $T \in \cR_m(\Hi;G)$, $g_0 \in G \setminus \{0_G\}$, $0 < r_0 < 1$, $0 < \epsilon \leq \epsilon_{\theThm}(m)$ and assume that for any $0 < r \leq r_0$ there is some $W_r \in \bG(\Hi,m)$ such that:
\begin{enumerate}
	\item $\spt(\partial T) \subset \Hi \setminus \B(0,3r_0)$,
	\item $T$ is $(\bM,\xi,2r_0)$-minimal in $\B(0,r_0)$ for a continuous gauge $\xi$,
	\item $\Theta^m(\|T\|,x) \geq \frac{3}{4}\|g_0\|$ for $\|T\|$-a.e.\ $x \in \B(0,r_0)$,
	\item $\hdist(\B(0,2r) \cap \spt(T), \B(0,2r) \cap W_r) \leq \epsilon r$,
	\item $\Exc(T \res \B(0,2r),0,r,W_r) \leq \|g_0\|\epsilon r^m$,
	\item $\pi_{W_r\#} \left(T \res (\B(0,2r) \cap Z_{W_r}(r))\right) = g_0 \curr{\B_{W_r}(0,r)}$.
\end{enumerate}
If we set $f(r) \defl \bM(T \res \B(0,r))$ and $\lambda \defl \lambda_{\ref{masscomparison2}}(m)$, then for almost every $r \in [0,r_0]$,
\[
f(r) \leq (1 + \xi(r)) \frac{r}{m}\left(\lambda f'(r) + (1 - \lambda)\|g_0\|\balpha(m)mr^{m-1}\right) \,.
\]
\end{Prop}

\begin{proof}
As in Lemma~\ref{polyhedralcone} we use the notation $T_r \defl \partial (T \res \B(0,r))$ for $r \in [0,r_0]$. Assume that
\[
s \leq \min\left\{\|g_0\|\epsilon r^m, 2^{-1}r \right\} \,,
\]
and for the application of Lemma~\ref{polyhedralcone} let $\theta \defl \frac{3}{4}\|g_0\|$ (notice $\spt(T) \cap \B(0,r_0)$ is compact according to Lemma~\ref{compactlem}) and consider the two chains $R_{r,s}$ and $P_{r,s}$ as constructed there. The constant $\theta$ is justified by our assumption on the densities of $\|T\|$. We also abbreviate $\|g\|_\theta \defl \max\{\|g\|,\theta\}$ if $g \neq 0_G$ and $\bM_\theta$ the associated mass. Note that $\bM \leq \bM_\theta$, $\|g_0\|_\theta = \|g_0\|$ and $\bM(T \res B) = \bM_\theta(T \res B)$ for every Borel set $B \subset \B(0,r_0)$. Let $\bar P_{r,s}$ and $\bar T_{r}$ be the infinite cones generated by $\curr 0 \cone P_{r,s}$ and $\curr 0 \cone T_r$ respectively. The following holds for almost all $r$:
\begin{equation}
\label{containedianulus}
\spt(R_{r,s}) \subset \oB(\spt(T_r),s) \subset \oB(0,r+s)\setminus \B(0,r-s) \,,
\end{equation}
\begin{equation}
\label{containedinball}
\spt(\curr 0 \cone P_{r,s} + R_{r,s}) \subset \oB(0,r + s) \,,
\end{equation}
\begin{equation}
\label{boundaryofcones}
\partial (\curr 0 \cone P_{r,s} + R_{r,s}) = T_r \,,
\end{equation}
\begin{equation}
\label{massofcones}
\bM_\theta(\curr 0 \cone P_{r,s}) \leq \bM(\curr 0 \cone T_{r}) + s, \quad \bM_\theta(R_{r,s}) \leq s \,,
\end{equation}
\begin{equation}
\label{coneballestimate}
\bM_\theta(\bar P_{r,s} \res \B(0,r)) \leq \bM(\curr 0 \cone T_r) + s \,.
\end{equation}
We want to show that
\begin{equation}
\label{firstfformula}
f(r) \leq (1 + \xi(r)) \left(\lambda \bM(\curr 0 \cone T_{r}) + (1 - \lambda)\|g_0\|\balpha(m)r^m\right) \,,
\end{equation}
for $\lambda = \lambda_{\ref{masscomparison2}} \in (\frac{1}{2},1)$. In order to apply Theorem~\ref{masscomparison2} it is necessary that the cone $\bar T_r$ has small cylindrical excess, but this may not hold for almost all $r$. So we consider two cases. First assume that $\lambda\bM(\curr 0 \cone T_{r}) \geq \|g_0\|(\epsilon r^m + \lambda\balpha(m) r^m)$. For $W \defl W_{r}$, assumptions (5) and (6) imply,
\begin{align*}
f(r) & \leq \bM(T \res (\B(0,2r) \cap Z_W(r))) \\
 & \leq \Exc(T \res \B(0,2r),0,r,W) + \|g_0\|\balpha(m)r^m \\
 & \leq \|g_0\|(\epsilon r^m + \balpha(m)r^m) \\
 & \leq \lambda \bM(\curr 0 \cone T_{r}) + (1 - \lambda)\|g_0\|\balpha(m) r^m \\
 & \leq (1 + \xi(r)) \left(\lambda \bM(\curr 0 \cone T_{r}) + (1 - \lambda)\|g_0\|\balpha(m)r^m\right) \,.
\end{align*}
In the second case $\bM(\curr 0 \cone T_{r}) \leq \|g_0\|(\lambda^{-1}\epsilon r^m + \balpha(m) r^m) \leq \|g_0\|(2\epsilon + \balpha(m)) r^m$. By assumption (4), $|\pi_{W^\perp}(x)| \leq \epsilon r$ for $x \in \spt(T) \cap \B(0,2r)$ and hence $|x| \leq (1 + \epsilon)r$ for all $x \in \spt(T) \cap \B(0,2r) \cap Z_{W}(r)$. Since we assume $s \leq \|g_0\|\epsilon r^m$, it follows from (5) and \eqref{coneballestimate} that
\begin{align}
	\nonumber
\bM_\theta(\bar P_{r,s} \res Z_W(r)) & \leq \bM_\theta(\bar P_{r,s} \res \B(0,(1 + \epsilon)r)) \\
	\nonumber
 & = (1 + \epsilon)^m \bM_\theta(\bar P_{r,s} \res \B(0,r)) \\
	\nonumber
 & \leq (1 + \epsilon)^m (\bM(\curr 0 \cone T_{r}) + s) \\
	\nonumber
 & \leq (1 + \epsilon)^m \|g_0\|((2\epsilon + \balpha(m))r^m + \epsilon r^m) \\
	\label{excessboundpoly}
 & \leq \bc(m)\|g_0\|\epsilon r^m + \|g_0\|\balpha(m)r^m \,,
\end{align}
for some constant $\bc(m) > 0$. By the constancy theorem it is $\pi_{W\#}(\bar P_{r,s} \res Z_W(r)) = g\curr{\B(0,r)}$ for some $g \in G$. It is indeed $g = g_0$ as we will see. From \eqref{boundaryofcones} it follows that $\partial (\curr 0 \cone P_{r,s} + R_{r,s} - (T \res \B(0,r))) = 0$ and hence $\pi_{W\#} (\curr 0 \cone P_{r,s} + R_{r,s}) = \pi_{W\#} (T \res \B(0,r))$ by the constancy theorem. From assumptions (4), (6) and the fact that $\epsilon \leq \frac{1}{4}$ it follows that the chain $\pi_{W\#} (T \res \B(0,r))$ has weight $g_0$ on $\B_W(0,4^{-1}r)$. From \eqref{containedianulus}, (4), $s \leq \frac{r}{2}$ and $\epsilon \leq \frac{1}{4}$ it follows that
\[
\spt(\pi_{W\#} R_{r,s}) \subset \B_W(0,r+s) \setminus \B_W(0,r-s-\epsilon r) \subset \B_W(0,r+s) \setminus \B_W(0,4^{-1} r) \,.
\]
So $\pi_{W\#} (\curr 0 \cone P_{r,s})$ has the same weight as $\pi_{W\#} (T \res \B(0,r))$ on $\B_W(0,4^{-1} r)$ which is $g_0$. Hence, if $\bc(m)\epsilon_{\theThm} < \epsilon_{\ref{masscomparison2}}$, it follows from \eqref{excessboundpoly} that
\[
\Exc_1(\bar P_{r,s}, W) < \|g_0\|\epsilon_{\ref{masscomparison2}} \,,
\]
and if $\epsilon_{\theThm} < \frac{1}{4}\epsilon_{\ref{masscomparison2}}$, it follows from (4) and \eqref{containedianulus} that for all $x \in \spt(\bar P_{r,s}) \cap Z_W$,
\[
|\pi_{W^\perp}(x)| \leq \sup_{y \in \spt(T_r)} \frac{|\pi_{W^\perp}(y)| +s}{r-s-|\pi_{W^\perp}(y)|} \leq \frac{2\epsilon_{\theThm} r}{r - 2\epsilon_{\theThm} r} \leq 4\epsilon_{\theThm} < \epsilon_{\ref{masscomparison2}} \,.
\]
Hence we can apply Theorem~\ref{masscomparison2} to the cone $\bar P_{r,s}$. This gives a new chain $S_{r,s} \in \cR_m(\Hi;G)$ with $\partial S_{r,s} = \partial (\bar P_{r,s} \res \B(0,r-s)) = \partial ((\curr 0 \cone P_{r,s}) \res \B(0,r-s))$ and
\[
\bM_\theta(S_{r,s}) \leq \lambda \bM_\theta((\curr 0 \cone P_{r,s}) \res \B(0,r-s)) + (1-\lambda)\|g_0\|\balpha(m) (r-s)^m \,.
\]
Let $S'_{r,s} \defl S_{r,s} + R_{r,s} + (\curr 0 \cone P_{r,s}) \res \B(0,r-s)^c$. By \eqref{boundaryofcones}, $\partial S'_{r,s} = T_r$ and by \eqref{massofcones} and \eqref{coneballestimate},
\begin{align*}
e(r,s) & \defl \bM_\theta(R_{r,s} + ((0 \cone P_{r,s}) \res \B(0,r-s)^c)) \\
 & \leq s + ((1 + sr^{-1})^m - (1 - sr^{-1})^m)(\bM_\theta(\curr 0 \cone T_r) + s) \,,
\end{align*}
and this converges to $0$ for $s \to 0$. By the almost minimality of $T$, \eqref{containedinball} and \eqref{massofcones} it follows for almost all $r$ and small enough $s$,
\begin{align*}
f(r) & = \bM(T \res \B(0,r)) \leq (1 + \xi(r + s))\bM_\theta(S'_{r,s}) \\
 & \leq (1 + \xi(r + s))(\bM_\theta(S_{r,s}) + e(r,s)) \\
 & \leq (1 + \xi(r + s))\bigl(\lambda \bM_\theta((\curr 0 \cone P_{r,s}) \res \B(0,r-s)) \\
 & \qquad \qquad \qquad \qquad \quad + (1-\lambda)\|g_0\|\balpha(m) (r-s)^m + e(r,s)\bigr) \\
 & \leq (1 + \xi(r + s))\bigl(\lambda(\bM(\curr 0 \cone T_r) + s) + (1-\lambda)\|g_0\|\balpha(m) (r-s)^m + e(r,s)\bigr) \,.
\end{align*}
Taking the limit for $s \to 0$ we obtain \eqref{firstfformula} for almost all $r \in [0,r_0]$. $f(r) = \bM(T \res \B(0,r))$ is monotonically increasing and hence differentiable almost everywhere. By the integral slice estimate \eqref{slicemass} we obtain for all $0 < a \leq b \leq r_0$,
\[
\int_a^b \bM(T_r) \, dr \leq \bM(T \res (\B(0,b) \setminus \oB(0,a))) = f(b) - f(a) \,.
\]
If $r$ is a Lebesgue point of $s \mapsto \bM(T_s)$ and also a point of differentiability for $f$, then $\bM(T_r) \leq f'(r)$. Almost every $r$ is such a point, hence $\bM(\curr 0 \cone T_r) = \frac{r}{m}\bM(T_r) \leq \frac{r}{m}f'(r)$ for almost all $r$. Applied to formula \eqref{firstfformula} we obtain the result.
\end{proof}

Next we state a lemma that treats the differential inequality obtained in Proposition~\ref{rundifferentialequation}.

\begin{Lem}
	\label{diffequation2}
Let $f,\xi,\Xi : (0,r_0] \to [0,\infty)$ be gauges and $\lambda \in (0,1)$, $\theta > 0$. Assume that $\xi$ is continuous and for almost all $r \in [0,r_0]$,
\begin{align*}
\Xi(r) & = m\int_0^r \frac{\xi(t)}{t} \, dt < \infty \,, \\
\theta r^m & \leq \exp(\Xi(r))f(r) \,, \\
f(r) & \leq (1 + \xi(r)) \frac{r}{m}\left(\lambda f'(r) + (1-\lambda)\theta mr^{m-1}\right) \,.
\end{align*}
Assume further that $(1 + \xi(r_0))\sqrt{\lambda} \leq \lambda_0 < 1$, $\exp(\Xi(r_0)) \leq 2$, $\xi(r) \leq \Xi(r)$ and
\[
	\lambda_0 \leq \frac{\Xi(r)}{\Xi(r) + \xi(r)} \,, \quad \mbox{for all } r \,.
\]
Then for all $0 < r \leq r_0$,
\[
	\exp(\Xi(r))\frac{f(r)}{r^m} - \theta \leq \Xi(r)\left(\Xi(r_0)^{-1}\left(\exp(\Xi(r_0))\frac{f(r_0)}{r_0^m} - \theta\right) + 8\theta \right) \,.
\]
\end{Lem}

\begin{proof}
Define the function $e(r) \defl \exp(\Xi(r))f(r) - \theta r^m$. By assumption $e(r) \geq 0$ and w.l.o.g.\ we assume that $e(r) > 0$ for all $r > 0$. Otherwise we replace $\xi(r)$ with $\xi_a(r) \defl \xi(r) + ar$ for $a > 0$. Since the final estimate is continuous in $a$ we then can take the limit for $a \downarrow 0$ afterwards. 
	
Since $f$ and $\Xi$ are differentiable almost everywhere, the same is true for $e$ and therefore,
\begin{align*}
e'(r) & = \exp(\Xi(r))f'(r) + \Xi'(r)\exp(\Xi(r))f(r) - \theta mr^{m-1} \\
	& = \exp(\Xi(r))f'(r) + m\frac{\xi(r)}{r}(e(r) + \theta r^m) - \theta mr^{m-1} \\
	& = \exp(\Xi(r))f'(r) + m\frac{\xi(r)}{r}e(r) - (1-\xi(r))\theta mr^{m-1} \,.
\end{align*}
Hence,
\begin{align*}
e(r) & = \exp(\Xi(r))f(r) - \theta r^m \\
	& \leq \exp(\Xi(r))(1 + \xi(r))\frac{r}{m}\left(\lambda f'(r) + (1-\lambda)\theta mr^{m-1}\right) - \theta r^m \\
	& = (1 + \xi(r))\frac{\lambda r}{m}\bigl(e'(r) + (1-\xi(r))\theta mr^{m-1} - m\frac{\xi(r)}{r}e(r) \bigr) \\
	& \quad + \exp(\Xi(r))(1 + \xi(r))\frac{r}{m}(1-\lambda)\theta mr^{m-1} - \theta r^m \\
	& = (1 + \xi(r))\frac{\lambda r}{m}e'(r) - \lambda (1 + \xi(r))\xi(r)e(r)  \\
	& \quad + \theta r^m\bigl(\exp(\Xi(r))(1 + \xi(r))(1-\lambda) + (1-\xi(r))(1 + \xi(r))\lambda - 1 \bigr) \\
	& \leq (1 + \xi(r))\frac{\lambda r}{m}e'(r) - \lambda (1 + \xi(r))\xi(r)e(r) \\
	& \quad + \theta r^m(1-\lambda)\left(\exp(\Xi(r))(1 + \xi(r)) - 1 \right) \,.
\end{align*}
We look at two cases. First assume that for some $r$,
\begin{equation}
\label{firstassumption}
(1-\sqrt{\lambda}) e(r) \leq \theta r^m(1-\lambda)\bigl(\exp(\Xi(r))(1 + \xi(r)) - 1 \bigr) \,.
\end{equation}
Since we assume that $\exp(\Xi(r_0)) \leq 2$ and $\xi(r) \leq \Xi(r)$ for all $r$, we get
\begin{align}
	\nonumber
e(r) & \leq \theta r^m (1+\sqrt{\lambda})\bigl(\exp(\Xi(r))(1 + \xi(r)) - 1 \bigr) \\
	\nonumber
	& \leq \theta r^m (1+\sqrt{\lambda})\bigl(\exp(\Xi(r_0))\Xi(r) + 2\xi(r) \bigr) \\
	\nonumber
	& \leq 2\theta r^m\left(2\Xi(r) + 2\Xi(r) \right) \\
	\label{firsterbound}
	& \leq 8\theta r^m \Xi(r) \defr g(r) \,.
\end{align}
For this function $g$ the derivative calculates as
\begin{align}
	\nonumber
g'(r) & = 8\theta mr^{m-1}\Xi(r) + 8\theta r^{m-1}m\xi(r) = 8\theta r^{m-1}m(\Xi(r) + \xi(r)) \\
	\label{gbound}
	& = \frac{\Xi(r) + \xi(r)}{\Xi(r)}\frac{m}{r}g(r)\,.
\end{align}
	
For almost all $r$ at which \eqref{firstassumption} doesn't hold,
\[
	\sqrt{\lambda}e(r) \leq (1 + \xi(r))\frac{\lambda r}{m}e'(r) \,,
\]
because the sum of this inequality with the one in \eqref{firstassumption} has to be satisfied for almost every $r$ by our bound on $e(r)$ established before. In this case,
\[
	e(r) \leq (1 + \xi(r_0))\frac{\sqrt\lambda r}{m}e'(r) \leq \frac{\lambda_0 r}{m}e'(r) \leq \frac{\Xi(r)}{\Xi(r) + \xi(r)}\frac{r}{m}e'(r) \,.
\]
Let $h(r) \defl \max\{e(r),g(r)\}$. Then for almost every $r \in \{g \geq e\}$,
\[
h'(r) = g'(r) = \frac{\Xi(r) + \xi(r)}{\Xi(r)}\frac{m}{r}g(r) = \frac{\Xi(r) + \xi(r)}{\Xi(r)}\frac{m}{r}h(r) \,.
\]
and for almost every $r \in \{g < e\}$,
\[
h'(r) = e'(r) \geq \frac{\Xi(r) + \xi(r)}{\Xi(r)}\frac{m}{r}e(r) = \frac{\Xi(r) + \xi(r)}{\Xi(r)}\frac{m}{r}h(r) \,.
\]
Hence for almost every $r$,
\begin{align*}
\frac{\Xi(r) + \xi(r)}{\Xi(r)}\frac{m}{r}h(r) & \leq  h'(r) \,.
\end{align*}
Because $\Xi'(r) = \frac{m}{r}\xi(r)$,
\begin{align*}
\frac{\partial}{\partial r} \log(r^m\Xi(r)) & = \frac{\partial}{\partial r} \left(m\log(r) + \log(\Xi(r))\right) = \left(1 + \frac{\xi(r)}{\Xi(r)} \right) \frac{m}{r} \\
	& \leq \frac{h'(r)}{h(r)} = \frac{\partial}{\partial r} \log(h(r)) \,.
\end{align*}
Integrating gives for all $0 < r \leq r_0$,
\[
	\frac{r_0^m\Xi(r_0)}{r^m\Xi(r)} \leq \frac{h(r_0)}{h(r)} \,,
\]
respectively,
\[
	\frac{e(r)}{r^m} \leq \frac{h(r)}{r^m} \leq \frac{\Xi(r)}{\Xi(r_0)} \frac{h(r)}{r_0^m} \leq \frac{\Xi(r)}{r_0^m\Xi(r_0)}\left(\exp(\Xi(r_0))f(r_0) - \theta r_0^m + 8\theta r_0^m \Xi(r_0)\right) \,.
\]
This shows the lemma.
\end{proof}

Before we proceed we give some remarks on the assumptions on $\xi$ in the lemma above.

\begin{Lem}
	\label{technicalproperties}
The following properties hold for a gauge $\xi : (0,r_0] \to \R_+$:
\begin{enumerate}
	\item If $\xi$ is concave, then $\Xi(r) \geq m\xi(r) \geq \xi(r)$ for all $r$.
	\item If $\frac{\xi(r)}{r^\alpha} \geq \frac{\xi(s)}{s^\alpha}$ for  all $0 < r \leq s \leq r_0$ and some $0 < \alpha \leq 1$. Then the same holds for $\beta \in [\alpha,1]$ in place of $\alpha$ and $\Xi(r) \geq \frac{m}{\alpha} \xi(r)$ is satisfied for all $r$. Additionally,
	\[
		\lambda \leq \frac{\Xi(r)}{\Xi(r) + \xi(r)} \quad \mbox{ holds for all } r \,,
	\]
	if $\lambda \leq \frac{m}{m+\alpha}$, respectively, $\alpha \leq m\frac{1 - \lambda}{\lambda}$.
\end{enumerate}
\end{Lem}

\begin{proof}
If $\xi$ is concave, then for $0 < r \leq r_0$ and $0 < s \leq 1$ there holds $\xi(sr) \geq s\xi(r)$, respectively, $(sr)^{-1}\xi(sr) \geq r^{-1}\xi(r)$. This shows that
\[
	\Xi(r) = m \int_0^r \frac{\xi(t)}{t} \, dt \geq m \int_0^r \frac{\xi(r)}{r} \, dt = m\xi(r) \geq \xi(r) \,.
\]
Let $\alpha$ and $\beta$ be as in (2), then
\[
\frac{\xi(s)}{s^\beta} = \frac{\xi(s)}{s^{\beta - \alpha} s^\alpha} \leq \frac{\xi(r)}{s^{\beta - \alpha} r^\alpha} = \frac{r^{\beta - \alpha}}{s^{\beta - \alpha}} \frac{\xi(r)}{r^\beta} \leq \frac{\xi(r)}{r^\beta} \,.
\]
Similarly to the concave case,
\begin{align*}
\Xi(r) & = m \int_0^r \frac{\xi(t)}{t} \, dt \geq m \int_0^r t^{\alpha - 1}\frac{\xi(r)}{r^\alpha} \, dt \\
 & = \frac{m}{\alpha} \xi(r) \,.
\end{align*}
For all $r$ we get
\[
\frac{\Xi(r)}{\Xi(r) + \xi(r)} \geq \frac{\Xi(r)}{\Xi(r) + \tfrac{\alpha}{m}\Xi(r)} = \frac{m}{m + \alpha} \,.
\]
\end{proof}

\subsection{Upper excess bounds and Reifenberg flatness}

First we give conditions that allow to estimate the excess in a neighborhood of a point with small excess.

\begin{Lem}
	\label{weakdini}
Let $\phi$ be a finite Borel measure on $\Hi$, $0 < r_0,\epsilon,\eta \leq 1$, $\xi : (0,2r_0] \to (0,\epsilon]$ be a gauge and $|x| \leq r_0\eta$. Assume that
\begin{enumerate}
	\item $\Theta^m(\phi,0) = \Theta^m(\phi,x) = 1$,
	\item $\exc^m(\phi,0,2r_0) \leq 1$,
	\item for $0 < s \leq t \leq r_0$ and $y \in \{0,x\}$,
	 \[\exp(\xi(s))\frac{\phi(\B(y,s))}{\balpha(m)s^m} \leq \exp(\xi(t))\frac{\phi(\B(y,t))}{\balpha(m)t^m} \,. \]
\end{enumerate}
Then
\begin{align*}
\exc^m(\phi,x,r_0) & \leq \exc^m(\phi,0,r_0(1 + \eta)) + 2^{m+6}\max\{\eta,\epsilon\} \,.
\end{align*}
\end{Lem}

\begin{proof}
Because of (1) and (2), the following estimate holds for all $0 < r \leq 2r_0$,
\begin{equation}
\label{densitybound}
\frac{\phi(\B(0,r))}{\balpha(m)r^m} \leq \exc^m(\phi,0,2r_0) + 1 \leq 2 \,.
\end{equation}
Let $0 < r \leq r_0$ be such that $|x| = \eta r$ (if $|x| = 0$ there is nothing to show). We first assume that the functions in (3) are strictly increasing on $(0,r_0]$. Clearly, 
\begin{align*}
\frac{\phi(\B(x,r))}{\balpha(m)r^m} & \leq \frac{\phi(\B(0,|x| + r))}{\balpha(m)(|x| + r)^m} \frac{(|x| + r)^m}{r^m} \\
 & = \frac{\phi(\B(0,r(1+\eta)))}{\balpha(m)(r(1 + \eta))^m} (1 + \eta)^m \,.
\end{align*}
Hence for $0 < t \leq r$,
\begin{align}
	\nonumber
\exp(\xi(t))\frac{\phi(\B(x,t))}{\balpha(m)t^m} & \leq \exp(\xi(r))\frac{\phi(\B(x,r))}{\balpha(m)r^m} \\
	\label{aboveest}
 & \leq \exp(\xi(r))(1 + \eta)^m\frac{\phi(\B(0,r(1+\eta)))}{\balpha(m)(r(1 + \eta))^m} \,.
\end{align}
Since
\[
\lim_{s \downarrow 0}\exp(\xi(s))\frac{\phi(\B(x,s))}{\balpha(m)s^m} = \lim_{s \downarrow 0}\exp(\xi(s))\frac{\phi(\B(0,s))}{\balpha(m)s^m} \,,
\]
and these functions are strictly increasing, there is for any $s \in (0,r]$ some $s' \in (0,s]$ with
\[
\exp(\xi(s'))\frac{\phi(\B(0,s'))}{\balpha(m)s'^m} \leq \exp(\xi(s))\frac{\phi(\B(x,s))}{\balpha(m)s^m} \,.
\]
Thus,
\begin{equation}
\label{belowest}
-\frac{\phi(\B(x,s))}{\balpha(m)s^m} \leq -\exp(\xi(s') -\xi(s))\frac{\phi(\B(0,s'))}{\balpha(m)s'^m} \leq -\exp(-\xi(s))\frac{\phi(\B(0,s'))}{\balpha(m)s'^m} \,.
\end{equation}
If we assume that $\exp(\xi(r))(1+\eta)^m \leq 1 + \delta$ and $\exp(-\xi(r)) \geq 1-\delta$ for some $\delta > 0$, and combine \eqref{densitybound}, \eqref{aboveest} with \eqref{belowest}, we get for $s,t \in (0,r]$,
\begin{align*}
\frac{\phi(\B(x,t))}{\balpha(m)t^m} - \frac{\phi(\B(x,s))}{\balpha(m)s^m} & \leq (1+\delta)\frac{\phi(\B(0,r(1+\eta)))}{\balpha(m)(r(1 + \eta))^m} - (1-\delta)\frac{\phi(\B(0,s'))}{\balpha(m)s'^m} \\
 & \leq \exc^m(\phi,0,r(1 + \eta)) + 4\delta \,.
\end{align*}
Since $\eta \leq 1$ and $\xi(2r_0) \leq \epsilon \leq 1$, there holds
\begin{align*}
\exp(\xi(r))(1+\eta)^m - 1 & \leq (1 + 4\epsilon)(1 + 2^m\eta) - 1 \\
 & \leq (4 + 2^m + 2^{m+2})\max\{\epsilon,\eta\} \\
 & \leq 2^{m+4}\max\{\epsilon,\eta\} \,,
\end{align*}
and similarly, $\exp(-\xi(r)) \geq 1 - \epsilon$. Hence
\[
\exc^m(\phi,x,r) \leq \exc^m(\phi,0,r(1 + \eta)) + 2^{m+6}\max\{\epsilon,\eta\} \,.
\]
This shows the lemma assuming that $s \mapsto \exp(\xi(s))\frac{\phi(\B(x,s))}{\balpha(m)s^m}$ is strictly increasing. The general case follows by replacing $\xi(s)$ with $\xi(s) + as$ and taking the limit $a \downarrow 0$.
\end{proof}

The following bootstrap argument is an application of the moments computations. It is the key lemma for showing Reifenberg flatness of the support of a chain in a neighborhood of small excess.

\begin{Lem}
	\label{bootstraplem}
Let $T \in \cR_m(\Hi;G)$ and $0 < r_0,s_0,\epsilon \leq 1$ be such that:
\begin{enumerate}
	\item $0 < 4\sqrt{s_0} \leq r_0 \leq \epsilon \leq \frac{1}{81}$.
	\item $\spt(\partial T) \subset \Hi \setminus \B(0,2r_0)$.
	\item $\Theta^m(\|T\|,x) = \theta > 0$ for $\|T\|$-almost all $x \in \spt(T) \cap \B(0,2r_0)$.
	\item There is a plane $W \in \bG(\Hi;G)$ and $g_0 \in G$ with $\|g_0\| = \Theta^m(\|T\|,0) = \theta$, such that $\lim_{r \downarrow 0}\bbeta_\infty(\|T\|,0,r,W) = 0$ and for all sufficiently small $r$,
\[
\pi_{W\#} \left(T \res (\B(x,r) \cap Z_{W}(\pi_{W}(x),2^{-1}r))\right) = g_0 \curr{\B_{W}(\pi_{W}(x),2^{-1}r)} \,.
\]
	\item there holds,
	\begin{align*}
		\exc_*^{m}(\|T\|,x,r_0) & \leq \theta\epsilon \mbox{ for all } x \in \B(0,2r_0) \,, \\
		\exc^{m*}(\|T\|,x,r_0) & \leq \theta\epsilon  \mbox{ for } \|T\|-\mbox{almost every } x \in \B(0,2r_0\epsilon^\frac{1}{4}) \,.
	\end{align*}
\end{enumerate}
If
\[
\bc_{\ref{betainftysmall2}}(m)\epsilon^{\frac{1}{4(m+2)}} \leq \frac{1}{25\sqrt{m}} \,,
\]
then for all $r \in (0,s_0)$ there is a plane $W_r \in \bG(\Hi,m)$ such that
\begin{equation*}
	\hdist(\spt(T) \cap \B(0,r),W_{r} \cap \B(0,r)) < 2\bc_{\ref{betainftysmall2}}(m)r\epsilon^{\frac{1}{4(m+2)}} \,.
\end{equation*}
and
\begin{equation*}
	\pi_{W_r\#} \left(T \res \B(0,r) \cap Z_{W_r}(0,\tfrac{r}{2})\right) = g_0 \curr{\B_{W_r}(0,\tfrac{r}{2})} \,.
\end{equation*}
Further, for any $r \in (0,s_0)$ there is an orthogonal frame $x_1, \dots, x_m \in \spt(T) \cap \partial\B(0,r)$ with $x_i \perp x_j$ for $i \neq j$.
\end{Lem}

\begin{proof}
Let $\phi \defl \theta^{-1}\|T\|$. By assumption, $s_0$ satisfies the bounds of Corollary~\ref{betainftysmall2}. Consider the sequence $s_k$ defined by $s_{k+1} \defl 2s_k\epsilon^\frac{1}{4}$. The factor $2$ in this definition is to compensate for the fact that the $\bbeta_\infty$ estimate in Corollary~\ref{betainftysmall2} is at scale $1/2$. This is a monotone sequence with $\lim_{k\to\infty} s_k = 0$ since by definition
\[
s_{k+1} \leq 2s_k\epsilon^\frac{1}{4} \leq \tfrac{2}{3}s_k \,.
\]
Because of (4) there is some $k \geq 1$ such that for any $s \in (0,s_k)$,
\begin{equation}
\label{betabound}
\bbeta_\infty(\phi,0,s,W) < \frac{1}{25\sqrt{m}} \,,
\end{equation}
and
\begin{equation}
\label{projectiongood}
\pi_{W\#} \left(T \res (\B(0,s) \cap Z_{W}(0,\tfrac{s}{2}))\right) = g_0 \curr{\B_{W}(0,\tfrac{s}{2})} \,.
\end{equation}
From Lemma~\ref{orthogonalfamily} we obtain that for all $s \in (0,s_k)$ there is an orthogonal frame $x_1(s),\dots,x_n(s) \in \Hi$ such that
\begin{equation}
\label{orthframe}
x_i(s) \in \spt(T) \,, \quad |x_i(s)| = s \quad \mbox{and} \quad x_i(s) \perp x_j(s) \mbox{ if } i \neq j \,.
\end{equation}
For any $r \in (0,s_{k-1})$ set $s \defl 2r\epsilon^\frac{1}{4} < s_k$. Corollary~\ref{betainftysmall2} implies that for all but countably many $r$ (i.e.\ those $r$ with $\|T\|(\partial \B(0,r)) = 0$) there is some plane $W_{r} \in \bG(\Hi,m)$ with
\begin{equation}
\label{betainfinityapriori}
\bbeta_\infty(\|T\|,0,r,W_{r}) < \bc_{\ref{betainftysmall2}}(m)\epsilon^{\frac{1}{4(m+2)}} \leq \frac{1}{25\sqrt{m}} \,.
\end{equation}
More precisely, the plane $W_{r}$ is spanned by the vectors $x_i(s)$, $i=1,\dots,m$, from \eqref{orthframe} and because of hypothesis (3) we can approximate $x_i(s)$ as necessary for Corollary~\ref{betainftysmall2}. In order to proceed the argument we show that 
\begin{equation}
\label{projectiongood2}
\pi_{W_r\#} \left(T \res (\B(0,r) \cap Z_{W_r}(0,\tfrac{r}{2}))\right) = g_0 \curr{\B_{W_r}(0,\tfrac{r}{2})} \,.
\end{equation}
This is achieved with Lemma~\ref{differentplane} by estimating the distance from $W_r$ to $W$. Since $x_i(s) \in \spt(T) \cap \B(0,s)$ by \eqref{orthframe}, we obtain from \eqref{betainfinityapriori},
\begin{align*}
	|\pi_{W}(x_i(s)) - x_i(s)| = |\pi_{W^\perp}(x_i(s))| \leq \bbeta_\infty(\|T\|,0,s,W)s < \frac{s}{25\sqrt{m}} \,.
\end{align*}
For all $w \in W_r$, if we write $w = \sum_i \lambda_i x_i(s)$ with $|w|^2 = s^2\sum_i|\lambda_i|^2$,
\[
|\pi_{W}(w) - w| < \frac{s}{25\sqrt{m}} \sum_i |\lambda_i| \leq \frac{|w|}{25} \,.
\]
Hence $\|\pi_W-\pi_{W_r}\| \leq \frac{1}{5}$ from Lemma~\ref{lemma.norms.1} and Lemma~\ref{differentplane} implies together with \eqref{projectiongood} that \eqref{projectiongood2} holds for $W_r$ if $r \in (0,s_{k-1})$ and $\|T\|(\partial \B(0,r)) = 0$. Because the inequality in \eqref{betainfinityapriori} is strict, \eqref{betainfinityapriori} and \eqref{projectiongood2} hold for all $r \in (0,s_{k-1})$. Again from Lemma~\ref{orthogonalfamily} it follows that we can find an orthogonal frame as in \eqref{orthframe} for all $r \in (0,s_{k-1})$.

Proceeding this way we conclude by induction on $k$ that for all $r \in (0,s_0)$ there is such a plane $W_{r} \in \bG(\Hi,m)$. To obtain the estimate on the Hausdorff distance note that for all $x \in \spt(T) \cap \B(0,r)$,
\[
\dist(x,W_{r} \cap \B(0,r)) < \bc_{\ref{betainftysmall2}}(m)r\epsilon^{\frac{1}{4(m+2)}} \,.
\]
On the other side, \eqref{projectiongood2} implies that for any
\[
x \in W_{r} \cap \B\left(0,r\left(1 - \bc_{\ref{betainftysmall2}}(m)\epsilon^{\frac{1}{4(m+2)}}\right)\right)
\]
there is an $x' \in \spt(T) \cap \B(0,r)$ with $|x-x'| < \bc_{\ref{betainftysmall2}}r\epsilon^{\frac{1}{4(m+2)}}$. Hence
\begin{equation*}
	\hdist(\spt(T) \cap \B(0,r),W_{r} \cap \B(0,r)) < 2\bc_{\ref{betainftysmall2}}(m)r\epsilon^{\frac{1}{4(m+2)}} \,.
\end{equation*}
\end{proof}

The following result shows that almost monotonic chains have neighborhoods that are Reifenberg flat. The assumption on the almost constant densities will be justified below in Lemma~\ref{almostconstdensitieslem} for nearly monotonic chains if we have a discreteness assumption on $G$. Moreover an almost monotonic chain is nearly monotonic because of Lemma~\ref{nearlymonotonic}.

\begin{Thm}
	\label{reifflat}
	Let $T \in \cR_m(\Hi;G)$ and assume that there are $x_0 \in \Hi$, $\theta > 0, r_0 \in (0,1]$ with
	\begin{enumerate}
		\item $\spt(\partial T) \subset \Hi \setminus \B(x_0,2r_0)$,
		\item $\Theta^m(\|T\|,x) = \Theta^m(\|T\|,x_0) = \theta$ for $\|T\|$-almost every $x \in \B(x_0,r_0)$,
		\item $T$ is almost monotonic in $\B(x_0,r_0)$ for some gauge $\xi : (0,r_0] \to (0,\infty)$.
	\end{enumerate}
Then for any $\epsilon > 0$ there is a $r_\epsilon \in (0,r_0]$ such that for all $x \in \spt(T) \cap \B(x_0,r_\epsilon)$ and all $r \in (0,r_\epsilon]$ there is a plane $W \in \bG(\Hi,m)$ and $g_{x} \in G$ with the properties,
	\[
	\exc^m(\|T\|,x,r_\epsilon) \leq \theta \epsilon \,,
	\]
	\[
	\hdist\bigl(\spt(T) \cap \B(x,r), (x + W) \cap \B(x,r)\bigr) \leq r\epsilon \,,
	\]
	\[
	\pi_{W\#} \left(T \res (\B(x,r) \cap Z_{W}(\pi_W(x),\tfrac{r}{2})\right) = g_{x}\curr{\B_{W}(\pi_W(x),\tfrac{r}{2})} \,,
	\]
	and $\|g_{x}\| = \theta$. Further, for any $x \in \spt(T) \cap \B(x_0,r_\epsilon)$ and $r \in (0,r_\epsilon]$ there is an orthogonal frame $x_1, \dots, x_m$ with $|x_i| = r$ and $x + x_i \in \spt(T)$.
\end{Thm}

\begin{proof}
Fix some $\epsilon > 0$ with
\[
\epsilon \leq \frac{1}{2^{m+7}} \min\left\{\frac{1}{81}, \frac{1}{(25\sqrt{m}\bc_{\ref{betainftysmall2}}(m))^{4(m+2)}} \right\} \,.
\]
The finite Borel measure $\phi \defl \theta^{-1}\|T\|$ is almost monotonic in $\B(x_0,r_0)$ with respect to the same gauge as $\|T\|$. From Lemma~\ref{nearlymonotonic} it follows that $\exc^m_*(\|T\|,x,r) \leq \bc_{\ref{nearlymonotonic}}\xi(r)$ for all $x \in \B(x_0,r_0)$ and $0 < r \leq r_0$. Since $\Theta^m(\phi,x_0) = 1$ there is a $r_1 \in (0,\frac{r_0}{6}]$ such that
\[
\max\{\theta^{-1}\bc_{\ref{nearlymonotonic}}\xi(6r_1),\exc^m(\phi,x_0,6r_1)\} \leq \epsilon \,.
\]
According to Lemma~\ref{weakdini}, any $x \in \B(x_0,3r_1\epsilon)$ with $\Theta^m(\phi,x) = 1$ satisfies
\begin{align}
\nonumber
\exc^m(\phi,x,3r_1) & \leq \exc^m(\phi,x_0,6r_1) + 2^{m+6}\max\{\epsilon,\epsilon\} \leq 2^{m+7}\epsilon \\ 
\label{excessbound}
& \leq \min\left\{\frac{1}{81}, \frac{1}{(25\sqrt{m}\bc_{\ref{betainftysmall2}}(m))^{4(m+2)}} \right\} \,.
\end{align}
Let $r_2 \defl \frac{1}{16}(r_1\epsilon)^2$, respectively, $4\sqrt{r_2} \leq r_1\epsilon$. Since $\bc_{\ref{nearlymonotonic}}\xi(6r_1) \leq \frac{\theta}{2}$, it follows from Lemma~\ref{compactlem}, that $\|T\|$ is Ahlfors-regular in $\oB(x_0,6r_1)$ and hence tangent planes for $T$ exist for $\|T\|$-a.e.\ $x \in \oB(x_0,6r_1)$ by Lemma~\ref{tangentplane}. Consider a point $x \in \B(x_0,r_2)$ with $\Theta^m(\|T\|,x) = \theta$ at which such a tangent plane exists. Note that $\B(x,2r_1\epsilon) \subset \B(x_0,3r_1\epsilon)$ and $\max\{\exc^{m*}(\phi,y,r),\exc_*^{m}(\phi,z,r)\} \leq 2^{m+7}\epsilon$ if $0 < 4\sqrt{r} \leq r_1\epsilon$, $y \in \B(x,2r_1\epsilon)$ and $z \in \B(x,2r_1\epsilon)$ with $\Theta^m(\phi,x) = 1$. From Lemma~\ref{bootstraplem} and \eqref{excessbound} we obtain that for any $0 < r < r_2$ there is a plane $W_{x,r} \in \bG(\Hi,m)$ such that,
\begin{equation}
\label{hdistxr}
	\hdist(\spt(T) \cap \B(x,r),W_{x,r} \cap \B(x,r)) < 2\bc_{\ref{betainftysmall2}}(m)r\bigl(2^{m+7}\epsilon\bigr)^{\frac{1}{4(m+2)}} \leq \frac{2r}{25} \,.
\end{equation}
and
\begin{equation}
\label{projectxr}
	\pi_{W_{x,r}\#} \left(T \res (\B(x,r) \cap Z_{W_{x,r}}(\pi_{W_{x,r}}(x),\tfrac{r}{2})\right) = g_x \curr{\B_{W_{x,r}}(x,\tfrac{r}{2})} \,.
\end{equation}
with $\|g_x\| = \theta$. Moreover, there is an orthogonal frame centered at $x$ for all scales $0 < r < r_2$ as described in the theorem. The set of those points $x \in \spt(T) \cap \oB(x_0,r_2)$ at which tangent planes exist and $\Theta^m(\|T\|,x) = \theta$ is dense. Thus for all $x \in \oB(x_0,r_2) \cap \spt(T)$ and all $0<r<r_2$ there is a $W_{x,r} \in \bG(\Hi,m)$ such that \eqref{hdistxr} holds as well as \eqref{projectxr} for some $g_{x,r} \in G$ with $\|g_{x,r}\| = \theta$. $g_{x,r}$ does not depend on $r$ because of Corollary~\ref{cor.reif.angle.1} and Lemma~\ref{differentplane}. Lemma~\ref{compactlem} and Lemma~\ref{nearlymonotonic} imply that $\spt(T) \cap \B(x,r)$ is compact if $x \in \oB(x_0,r_2)$ and $r$ is small enough. Thus the statement about the orthogonal frames holds for all points in $\spt(T) \cap \oB(x_0,r_2)$. Let $x \in \spt(T) \cap \oB(x_0,r_2)$ and $x_i$ be a sequence in $\oB(x_0,r_2)$ with $\Theta^m(\|T\|,x) = \theta$ that converges to $x_0$. Let $0 < r < 3r_1$ and $\lambda > 1$ with $\lambda r < 3r_1$. Applying \eqref{excessbound} to the sequence $x_i$,
\begin{align*}
\frac{\|T\|(\B(x,r))}{\balpha(m)r^m} - \theta & \leq \limsup_{i\to\infty} \frac{\|T\|(\B(x_i,\lambda r))}{\balpha(m)r^m} - \theta \\
 & \leq \limsup_{i\to\infty} \biggl|\frac{\|T\|(\B(x_i,\lambda r))}{\balpha(m)(\lambda r)^m} - \theta\biggl| \lambda^m + |\theta \lambda^m - \theta| \\
 & \leq \exc^m(\|T\|,x,3r_1) + (\lambda^m - 1)\theta \\
 & \leq \theta 2^{m+7}\epsilon \lambda^m + (\lambda^m - 1)\theta \,.
\end{align*}
Similarly we obtain a lower bound of $-\theta 2^{m+7}\epsilon \lambda^{-m} + (\lambda^{-m} - 1)\theta$. Taking the limit for $\lambda \downarrow 1$, it follows that $\exc^m(\|T\|,x,r) \leq \theta 2^{m+8}\epsilon$ for all $x \in \spt(T) \cap \oB(x_0,r_2)$ and $0 < r < 3r_1$. This concludes the proof.
\end{proof}

\subsection{Main regularity results}

Combining Theorem~\ref{reifflat}, Lemma~\ref{nearlymonotonic}, Lemma~\ref{almostminimal}, Proposition~\ref{rundifferentialequation}, Lemma~\ref{diffequation2} and Lemma~\ref{technicalproperties} we obtain.

\begin{Prop}
	\label{improvedexcess}
Assume that $T \in \cR_m(\Hi;G)$, $x_0 \in \spt(T)$ and $r_0 > 0$ satisfy:
\begin{enumerate}
	\item $\spt(\partial T) \subset \Hi \setminus \B(x_0,2r_0)$.
	\item $T$ is $(\bM,\xi,r_0)$-minimal in $\B(x_0,r_0)$ for a continuous gauge $\xi$ with $\Xi(\delta) = m\int_0^{r_0} \frac{\xi(s)}{s} \, ds < \infty$.
	\item $\Theta^m(\|T\|,x) = \Theta^m(\|T\|,x_0) = \theta$ for $\|T\|$-almost every $x \in \B(x_0,r_0)$.
	\item Set $\lambda_0 \defl \sqrt[4]{\lambda_{\ref{masscomparison2}}}$. Assume that $(1 + \xi(r_0)) \lambda_0 \leq 1$, $\exp(\Xi(r_0)) \leq 2$ and $s \mapsto \frac{\xi(s)}{s^\alpha}$ is decreasing for some $\alpha \leq m \frac{1 - \lambda_0}{\lambda_0}$.
\end{enumerate}
Then there are constants $r_1 > 0$ and $\bc_{\theThm}(m,r_0,\Xi(r_0),\theta,\bM(T)) > 0$ such that for all $0 < r \leq r_1$,
\begin{align*}
\exp(\Xi(r))\frac{\|T\|(\B(x,r))}{\balpha(m) r^m} - \theta &  \mbox{ is increasing in } r \mbox{ if } x \in \oB(x_0,r_1)\,, \\
	\exc^{m}_*(\|T\|,x,r) & \leq \bc_{\theThm}\,\Xi(r)  \,, \mbox{ if } x \in \oB(x_0,r_1) \,, \\
\exp(\Xi(r))\frac{\|T\|(\B(x,r))}{\balpha(m) r^m} - \theta & \leq \bc_{\theThm}\, \Xi(r) \,, \mbox{ if } x \in \spt(T) \cap \oB(x_0,r_1) \,, \\
\exc^{m*}(\|T\|,x,r) & \leq \bc_{\theThm}\,\Xi(r)  \,, \mbox{ if } x \in \spt(T) \cap \oB(x_0,r_1) \,.
\end{align*}
In particular $\Theta^m(\|T\|,x) = \theta$ for all $x \in \spt(T) \cap \oB(x_0,r_1)$.
\end{Prop}

\begin{proof}
From Lemma~\ref{almostminimal} it follows that $T$ is almost monotonic in $\B(x_0,r_0)$ with gauge $\Xi : (0,r_0] \to \R_+$. From Lemma~\ref{nearlymonotonic} it follows that $\|T\|$ is nearly monotonic in $\B(x_0,r_0)$ with some gauge $\bc_* \, \Xi$ where $\bc_* \defl \bc_{\ref{nearlymonotonic}}(m,r_0,\Xi(r_0),\bM(T))$. This shows the first two conclusions.

From Theorem~\ref{reifflat} it follows that for any $0 < \epsilon < 1$ there is a scale $0 < r_\epsilon \leq \frac{r_0}{2}$ such that for all $x \in \spt(T)\cap \B(x_0,r_\epsilon)$ and $0 < r \leq r_\epsilon$ there is a plane $W = W_{x,r} \in \bG(\Hi,m)$ and $g_x \in G$ with $\|g_x\| = \theta$ and
\[
\exc^m(\|T\|,x,2r_\epsilon) \leq \theta\epsilon \,,
\]
\[
	\hdist(\spt(T) \cap \B(x,2r), (x + W) \cap \B(x,2r)) \leq r\epsilon \,,
\]
\[
	\pi_{W\#} \left(T \res (\B(\pi_{W}(x),2r) \cap Z_{W}(\pi_{W}(x),r)\right) = g_x \curr{\B_{W}(\pi_{W}(x),r)} \,.
\]
In order to apply Proposition~\ref{rundifferentialequation} we need to check that $T$ has small cylindrical excess over $W$. But this is rather simple using the two properties above. Let $x \in \B(0,r_\epsilon)$. For simplicity we translate the chain to the origin $T' \defl \tau_{-x\#} T$, where $\tau_{-x}(y) = y - x$. Then for any $y \in \spt(T') \cap \B(0,2r) \cap Z_{W}(0,r)$ there holds $|y| \leq r + r\epsilon$ and thus
\[
	\spt(T') \cap \B(0,2r) \cap Z_{W}(0,r) \subset \B\left(0, r\left(1 + \epsilon\right)\right) \,.
\]
Therefore,
\begin{align*}
	\frac{1}{\balpha(m) r^m} & \Exc(T' \res \B(0,2r),0,r,W) \\
	& = \frac{1}{\balpha(m) r^m} (\|T'\|(\B(0,2r) \cap Z_W(r)) - \theta\balpha(m) r^m) \\
	& \leq \frac{1}{\balpha(m) r^m} (\|T'\|(\B\left(0, r\left(1 + \epsilon\right)\right)) - \theta\balpha(m) r^m) \\
	& = \frac{\|T'\|(\B\left(0, r\left(1 + \epsilon\right)\right))}{\balpha(m) (1+\epsilon)^mr^m} (1 + \epsilon)^m - \theta \\
	& \leq (\theta + \exc^{m}(\|T'\|,0,2r_\epsilon))(1 + \epsilon)^m - \theta \\
	& \leq (\theta + \epsilon\theta)(1 + 2^m\epsilon) - \theta \\
	& \leq \theta(2^m + 2^m\epsilon + 1)\epsilon \,.
\end{align*}
So if $\epsilon$ is small enough we can apply Proposition~\ref{rundifferentialequation}. It follows that for all $x \in \B(x_0,r_\epsilon)$ with $\Theta^m(\|T\|,x) = \theta$ and for almost all $0 < r \leq r_\epsilon$,
\[
	f_x(r) \leq (1 + \xi(r)) \frac{r}{m}\left(\lambda f_x'(r) + (1 - \lambda)\theta\balpha(m) mr^{m-1}\right),
\]
where $f_x(r) \defl \|T\|(\B(x,r))$ and $0 < \lambda \defl \lambda_{\ref{masscomparison2}} < 1$. Note that (4) implies that $(1 + \xi(r_0))\sqrt{\lambda_{\ref{masscomparison2}}} \leq \lambda_0 = \sqrt[4]{\lambda_{\ref{masscomparison2}}}$. Hence with the help of Lemma~\ref{technicalproperties} we can apply Lemma~\ref{diffequation2} and obtain
\[
	\frac{f_x(r)}{\balpha(m) r^m} - \theta \leq \exp(\Xi(r))\frac{f_x(r)}{\balpha(m) r^m} - \theta \leq \bc^* \, \Xi(r) \,,
\]
 for all $0 < r \leq r_\epsilon$ and some $\bc^* = \bc^*(m,r_0,\Xi(r_0),\theta,\bM(T))$. Accordingly, $\exc^{m*}(\|T\|,x,r) \leq (\bc^* + \bc_*)\Xi(r)$ if $x \in \B(x_0,r_\epsilon)$ satisfies $\Theta^m(\|T\|,x) = \theta$ and $0 < r \leq r_\epsilon$. 

Now assume that $x \in \spt(T) \cap \oB(x_0,r_\epsilon)$ and let $0 < r < r_\epsilon$. Let $x_k$ be a sequence in $\B(x_0,r_\epsilon)$ with $|x_k -x| = r\delta_k \to 0$ and $\Theta^m(\|T\|,x_k) = \theta$ for all $k$. Since $\Xi$ is continuous,
\begin{align*}
\bc^*\Xi(r) \geq \frac{\|T\|(\B(x,r))}{\balpha(m) r^m} - \theta & = \left(1 + \delta_k\right)^m\frac{\|T\|(\B(x,r))}{\balpha(m) r^m(1 + \delta_k)^m} - \theta \\
 & \leq \left(1 + \delta_k\right)^m\frac{\|T\|(\B(x_k,r(1 + \delta_k)))}{\balpha(m) r^m(1 + \delta_k)^m} - \theta \\
 & \leq \left(1 + \delta_k\right)^m(\bc^* \Xi(r(1 + \delta_k)) + \theta) - \theta \\
 & \rightarrow \bc^*\Xi(r) \,,
\end{align*}
for $k \to \infty$. Hence $\Theta^m(\|T\|,x) = \theta$ for all $x \in \spt(T) \cap \oB(x_0,r_\epsilon)$ and we obtain the third conclusion. If $0 < r \leq s < r_\epsilon$ are such that $(\balpha(m) s^m)^{-1}f_x(s) \geq (\balpha(m) r^m)^{-1}f_x(r)$, then
\begin{align*}
\frac{\|T\|(\B(x,s))}{\balpha(m) s^m} - \frac{\|T\|(\B(x,r))}{\balpha(m) r^m} & = \left(\frac{\|T\|(\B(x,s))}{\balpha(m) s^m} - \theta\right) - \left(\frac{\|T\|(\B(x,r))}{\balpha(m) r^m} - \theta\right) \\
	& \leq \bc^*\Xi(s) + \bc_*\Xi(r) \leq (\bc^* + \bc_*)\Xi(s) \,.
\end{align*}
\end{proof}

Before we can apply Proposition~\ref{improvedexcess} we need to find an open set in $\spt(T)\setminus \spt(\partial T)$ where the densities are almost constant. This is possible assuming the normed Abelian group $(G,\|\cdot\|)$ is such that $\{\|g\| : g \in G\}$ is a discrete and closed subset of $\R$, or equivalently that
\[
	\delta_L(G) \defl \inf\{|\|g\|-\|h\|| : g \neq h, \|g\|,\|h\| \leq L \} > 0
\]
for all $L > 0$.

\begin{Lem}
\label{almostconstdensitieslem}
Let $\Hi$ be a Hilbert space, $(G,\|\cdot\|)$ be a normed Abelian group such that $\{ \|g\| : g \in G \}$ is discrete, and $T \in \cR_m(\Hi;G)$ be a rectifiable $G$-chain that is nearly monotonic in an open set $U \subset \Hi \setminus \spt(\partial T)$. Then there is a dense open subset $U_{d}$ of $\spt(T) \cap U$ with the property that the map $x \mapsto \Theta^m(\|T\|,x)$ is locally $\|T\|$-almost constant on $U_{d}$. Moreover, if $g \mapsto \|g\|$ is constant on $G \setminus \{0_G\}$, then $x \mapsto \Theta^m(\|T\|,x)$ is $\|T\|$-almost constant on $U$.
\end{Lem}

\begin{proof}
Without loss of generality we can assume that $\|T\|(U) > 0$, otherwise $\spt(T) \cap U = \emptyset$. It follows from Lemma~\ref{densityexist} that $\Theta^m(\|T\|,x)$ exists for all $x \in U$ and the function $x \mapsto \Theta^m(\|T\|,x)$ is upper semicontinuous on $U$. Let $A$ be the subset of those $x \in U$ for which $\Theta^m(\|T\|,x) = \|\bg(x)\| > 0$, where $\bg : \Hi \to G$ is some $\cH^m$-measurable $G$-orientation representing $T$. The set $A$ satisfies $\|T\|(U \setminus A) = 0$ as we have seen in Subsection~\ref{chains_subsection}. For $L \defl 2\cdot \operatorname{essinf}_{\|T\|}\{\|\bg(z)\| : z \in A\} > 0$ it holds that $0 < \delta_L(G) \leq \operatorname{essinf}_{\|T\|}\{\|\bg(z)\| : z \in A\}$ by the discreteness assumption on $G$. Now fix some point $x \in A$ with
\[
	\|\bg(x)\| < \delta_L(G)/2 + \operatorname{essinf}_{\|T\|}\{\|\bg(z)\| : z \in A\} \,.
\]
Because $x \mapsto \|\bg(x)\|$ is upper semicontinuous on $A$, there is some $r > 0$ such that $\B(x,r) \subset U$ and for $\|T\|$-a.e.\ $y \in \B(x,r) \cap A$,
\begin{align*}
	\operatorname{essinf}_{\|T\|}\{\|\bg(z)\| : z \in A\} & \leq \|\bg(y)\| < \|\bg(x)\| + \delta_L(G)/2 \\
	& < \delta_L(G) + \operatorname{essinf}_{\|T\|}\{\|\bg(z)\| : z \in A\} \,.
\end{align*}
Hence for $\|T\|$-a.e.\ $y \in \B(x,r)$,
\[
\Theta^m(\|T\|,y) = \|\bg(y)\| = \operatorname{essinf}_{\|T\|}\{\|\bg(z)\| : z \in A\} \,.
\]
This shows that $\Theta^m(\|T\|,y)$ is equal to some constant $\theta > 0$ for $\|T\|$-a.e.\ $y \in \B(x,r)$. The same argument shows that for any open subset $U' \subset U$ with $\|T\|(U') > 0$ we can find an open subset $V' \subset U'$ with $\|T\|(V') > 0$ on which $x \mapsto \Theta^m(\|T\|,x)$ is constant. This shows the first statement. If we further assume that $\|g\| = \theta > 0$ for all $g \neq 0_G$, then obviously $\Theta^m(\|T\|,x) = \|\bg(x)\| = \theta$ for $\|T\|$-almost every $x \in U$. Hence we can take $U_{d} = U$. 
\end{proof}

This assumption on the group is indeed a necessary one in order to obtain points of almost constant densities as we will see in Example~\ref{discreteness_example}. The following lemma is a standard result for representing sets as graphs. We say that a map $f : U \subset \R^m \to \Hi$ is of class $C^{1,\xi}$ for some gauge $\xi$ if there is a constant $C > 0$ such that for all $x,y \in U$ there is a linear map $Df_x : \R^m \to \Hi$ such that
\[
|f(y) - f(x) - Df_x(y-x)| \leq C|x-y|\xi(C|x-y|) \,.
\]

\begin{Lem} 
\label{manifoldlem}
Let $S \subset X$ be a closed set with $0 \in S$ and assume that $r_0 > 0$ and $\eta : (0,2r_0] \to \R_+$ is a gauge with the following properties:
\begin{enumerate}
	\item For all $x \in \B(0,r_0)$ and all $0 < r \leq r_0$ there is a plane $W_{x,r} \in \bG(\Hi,m)$ with
	\[
	\hdist(S \cap \B(x,r), (x + W_{x,r}) \cap \B(x,r)) \leq r\eta(r) \,.
	\]
	\item For $W \defl W_{0,r_0}$ assume that $\pi_W(S \cap \B(0,\frac{r_0}{4}) \cap Z_W(0,\frac{r_0}{8})) = \B_W(0,\frac{r_0}{8})$.
	\item $\int_0^{2r_0} \frac{\eta(r)}{r} \, dr \leq \frac{1}{120}$.
	\item $r \mapsto \frac{\eta(r)}{r}$ is decreasing.
\end{enumerate}
Set $\hat \eta(r) \defl \int_0^{r} \frac{\eta(s)}{s} \, ds$. Then there is a unique map $f : \B_W(0,\frac{r_0}{8}) \to W_0^\perp$ with $\graph(f) = S \cap \B(0,\frac{r_0}{4}) \cap Z_W(0,\frac{r_0}{8})$ and $f$ is of class $C^{1,\hat\eta}$.
\end{Lem}

\begin{proof}
Recall that for the distance on $\bG(\Hi,m)$ there holds $\|\pi_{V_1} - \pi_{V_2}\| = \hdist(\B_{V_1}(0,1),\B_{V_2}(0,1))$ as observed in Lemma~\ref{lemma.norms.1}. Since $r \mapsto \frac{\eta(r)}{r}$ is assumed to be decreasing,
\begin{equation}
\label{xiestimate}
\hat \eta(r) = \int_0^{r} \frac{\eta(s)}{s} \, ds \geq \int_0^{r} \frac{\eta(r)}{r} \, ds = \eta(r) \,.
\end{equation}
Fix a point $x \in S \cap \B(0,r_0)$ and some $0 < r \leq r_0$. Comparing the planes $W_{x,r}$ and $W_{x,r/2}$ it follows from Corollary~\ref{cor.reif.angle.1} and $\eta(r/2) \leq \eta(r)$ that
\begin{equation}
\label{samecenter}
\left\|\pi_{W_{x,r}} - \pi_{W_{x,r/2}}\right\| \leq \eta(r)(2 + r(r/2)^{-1}) = 4 \eta(r) \,. 
\end{equation}
Similarly, it follows from Lemma~\ref{lemma.reif.angle.2} and $\eta(r) \leq \eta(r_0) \leq \frac{1}{4}$ that for $x,y \in S \cap \B(0,r_0)$ and $0 < r \leq r_0$ with $|x-y| \leq r/2$,
\begin{equation}
\label{samescale}
\left\|\pi_{W_{x,r}} - \pi_{W_{y,r}}\right\| \leq 24 \eta(r) \,.
\end{equation}
To see this let $\lambda = \frac{1}{2}, \epsilon = \eta(r) \leq \frac{1}{4}$ and $\nu = 4$ in the setting of Lemma~\ref{lemma.reif.angle.2}. Let $r_k \defl 2^{-k} r_0$. It follows from \eqref{samecenter} and hypothesis (4) that for $0 \leq k \leq l$,
\begin{align}
\nonumber
\bigl\|\pi_{W_{x,r_k}} - \pi_{W_{x,r_l}}\bigr\| & \leq \sum_{i \geq k} \bigl\|\pi_{W_{x,r_i}} - \pi_{W_{x,r_{i+1}}}\bigr\| \leq 4 \sum_{i \geq k} \eta(r_i) \\
\label{cauchy}
 & = 4 \sum_{i \geq k} \int_{r_i}^{2r_i}\frac{\eta(r_i)}{r_i} \leq 4 \int_0^{2r_{k}} \frac{\eta(r)}{r} \, dr = 4 \hat \eta(2r_{k}) \,.
\end{align}
Because the Grassmannian $\bG(\Hi,m)$ is complete, there is for any $x \in S \cap \B(x_0,r_0)$ a plane $W_x = \lim_{k \to \infty} W_{x,k}$. Moreover, if $x,y \in S \cap \B(0,r_0)$ and $k \geq 1$  are such that $|x-y| \leq r_{k}/2$, then with \eqref{xiestimate}, \eqref{samescale} and \eqref{cauchy},
\begin{align}
\nonumber
\bigl\|\pi_{W_{x}} - \pi_{W_{y}}\bigr\| & \leq \bigl\|\pi_{W_{x,r_k}} - \pi_{W_{x}}\bigr\| + \bigl\|\pi_{W_{x,r_k}} - \pi_{W_{y,r_k}}\bigr\| + \bigl\|\pi_{W_{y,r_k}} - \pi_{W_{y}}\bigr\|\\
\label{tangentdifference}
 & \leq 24\eta(r_k) + 8 \hat \eta(2r_{k}) \leq 32 \hat\eta(2r_{k}) \,.
\end{align}
With \eqref{cauchy} and hypothesis (1),
\begin{align}
\nonumber
 \hdist(S \cap \B(x,r_k), &(x + W_{x}) \cap \B(x,r_k)) \\
\nonumber
 & \leq \hdist(S \cap \B(x,r_k), (x + W_{x,r_k}) \cap \B(x,r_k)) \\
\nonumber
 & \quad + \hdist((x + W_{x,k}) \cap \B(x,r_k), (x + W_{x}) \cap \B(x,r_k)) \\
\nonumber
 & \leq r_k \eta(r_k) + r_k\bigl\|\pi_{W_{x,r_k}} - \pi_{W_{x}}\bigr\| \\
\label{tangenplanelimit}
 & \leq 5r_k \hat\eta(2r_k) \,.
\end{align}
Further, if $|x| \leq r_0/2$, then with \eqref{xiestimate}, \eqref{samescale}, \eqref{cauchy} and hypothesis (3),
\begin{align}
\nonumber
\bigl\|\pi_{W} - \pi_{W_{x,r_k}}\bigr\| & \leq \bigl\|\pi_{W_{0,r_0}} - \pi_{W_{x,r_0}}\bigr\| + \bigl\|\pi_{W_{x,r_0}} - \pi_{W_{x,r_k}} \bigr\|\\
\nonumber
 & \leq  24 \eta(r_0) + 4 \hat \eta(2r_0) \leq 30 \eta(2r_0) \\
\label{projectto0}
 & \leq \frac{1}{4} \,.
\end{align}
So let $x,y \in S \cap \B(0,r_0/4)$ and $k \geq 1$ such that $r_{k+1} < |x-y| \leq r_{k}$. By assumption there is a $v \in W_{x,r_k}$ with $|x + v - y| \leq r_k\eta(r_k)$. Because of \eqref{projectto0},
\begin{align*}
\bigl|\pi_{W}(x-y)\bigr| & \geq \bigl|\pi_{W_{x,r_k}}(v)\bigr| - \bigl|\pi_{W_{x,r_k}}(x + v - y)\bigr| - \bigl\|\pi_{W} - \pi_{W_{x,r_k}}\bigr\||x-y| \\
 & \geq |v| - r_k\eta(r_k) - \tfrac{1}{4}|x-y| \\
 & \geq \tfrac{3}{4}|x-y| - 2r_k\eta(r_0) \\
 & \geq \tfrac{3}{4}|x-y| - 4\hat\eta(2r_0)|x-y| \\
 & \geq \tfrac{1}{\sqrt{2}}\bigl|x-y\bigr| \,.
\end{align*}
Together with hypothesis (2) this shows that $\pi_W : S \cap \B(0,r_0/4) \cap Z_W(0,r_0/8)) \to \B_W(0,r_0/8)$ is a bi-Lipschitz map. Let $f : \B_W(0,r_0/4) \to W^\perp$ be the map that represents $S \cap \B(0,r_0/4) \cap Z_W(0,r_0/8))$ as a graph over $\B_W(0,r_0/8)$. Since $\sqrt{2}|\pi_{W}(x-y)| \geq |x-y|$ for $x,y \in S \cap \B(0,r_0/4)$, we can estimate the Lipschitz constant by $\Lip(f) \leq 1$. \eqref{projectto0} also implies that for $v \in W_x$,
\[
\bigl|\pi_{W}(v)\bigr| \geq |v| - \tfrac{1}{4}|v| \geq \tfrac{1}{\sqrt{2}}\bigl|v\bigr| \,.
\]
This shows that $\pi_W : W_{x} \to W$ is injective and there is a linear map $L_x : W \to W^\perp$ with $\|L_x\| \leq 1$ which represents $W_x$ as a graph over $W$.

Let $w,w' \in \B_W(0,r_0/8)$ with $r_{k+1} < |w-w'| \leq r_{k}$ for some $k \geq 2$ and set $x \defl w + f(w),x' \defl w'+f(w') \in S$. Since $f$ is $1$-Lipschitz,
\begin{align*}
r_{k+1} & \leq |w-w'| \leq |x-x'| \leq |w-w'| + |f(w)-f(w')| \leq 2|w-w'| \leq r_{k-1} \,.
\end{align*}
In particular $x,x' \in S \cap \B(0,r_0/4)$ and by \eqref{tangenplanelimit} there is some $v \in W_x$ with $|x + v - x'| \leq 5r_{k-1} \hat\eta(2r_{k-1})$. Because $\pi_{W^\perp}(v) = L_x(\pi_W(v))$, we conclude
\begin{align*}
|f(w') - f(w) - L_x(w'- w)| & \leq |f(w') - f(w) - \pi_{W^\perp}(v)| + |L_x(\pi_W(v) + w - w')| \\
 & = |\pi_{W^\perp}(x' - x - v)| + |L_x(\pi_W(v + x - x'))| \\
 & \leq 2|x' - x - v| \leq 10r_{k-1} \hat\eta(2r_{k-1}) \\
 & \leq 40|w-w'| \hat\eta(8|w-w'|) \,.
\end{align*}
This shows the result.
\end{proof}

Note that if $\eta(r) = c r^\alpha$ for some $0 < \alpha \leq 1$, then $\hat\eta(r) = \frac{c}{\alpha}r^\alpha$ and hypothesis (4) is satisfied because of Lemma~\ref{technicalproperties}(2). Combining Proposition~\ref{improvedexcess}, Lemma~\ref{almostconstdensitieslem} and Lemma~\ref{manifoldlem} we obtain our main regularity result.

\begin{Thm}
	\label{reifflat2}
Let $\Hi$ be a Hilbert space with $\dim(\Hi) > m \geq 1$ and let $(G,\|\cdot\|)$ be a normed Abelian group such that $\{\|g\| : g \in G\}$ is discrete. There is a constant $0 < \alpha_0 < 1$ with the following property. Let $T \in \cR_m(\Hi;G)$ and $U \subset \Hi \setminus \spt(T)$ be an open set. Assume that $T$ is $(\bM,\xi,\delta)$-minimal in $U$ for some gauge $\xi(r) = cr^\alpha$, where $\alpha > 0$. Then there is an open dense subset $U_{\textrm{reg}}$ of $\spt(T) \cap U$ that is a $C^{1,\beta}$-submanifold of $\Hi$ with
\[
\beta \defl \frac{\min\{\alpha_0, \alpha\}}{8(m+2)} \,.
\]
Moreover, if $g \mapsto \|g\|$ is constant on $G \setminus \{0_G\}$, then 
\[
\cH^m(\spt(T) \cap U \setminus U_{\textrm{reg}}) = \|T\|(U \setminus U_{\textrm{reg}}) = 0 \,.
\]
\end{Thm}


\begin{proof}
We can assume that $\delta \leq 1$, and hence $T$ is almost minimal with respect to the gauge $r \mapsto cr^{\min\{\alpha_0, \alpha\}}$, where $\alpha_0 \defl m \frac{1 - \sqrt[4]{\lambda_{\ref{masscomparison2}}}}{\sqrt[4]{\lambda_{\ref{masscomparison2}}}} < 1$. By restricting to the smaller of the two values we can further assume that $\alpha \leq \alpha_0$. As before we define $\Xi(r) \defl m\int_0^r \frac{\xi(s)}{s} \, ds = \frac{m}{\alpha}r^\alpha$.

Without loss of generality we can assume that $\spt(T) \cap U \neq \emptyset$ and by exhaustion we may also assume that $\dist(U,\spt(\partial T)) > 0$. Due to Lemma~\ref{nearlymonlem} and Lemma~\ref{nearlymonotonic}, $T$ is nearly monotonic in $U$. Because of Lemma~\ref{almostconstdensitieslem} there is a dense open set $U_{d}$ of $\spt(T) \cap U$ that has locally $\|T\|$-almost constant densities. With Proposition~\ref{improvedexcess} and Lemma~\ref{technicalproperties} we deduce that for $\|T\|$-a.e.\ $x_0 \in U_d$ there are $c', r_1 > 0$ such that
\begin{align*}
	\exc^{m}_*(\|T\|,x,r) & \leq c' r^\alpha  \,, \mbox{ if } x \in \oB(x_0,r_1) \,, \\
\exc^{m*}(\|T\|,x,r) & \leq c' r^\alpha  \,, \mbox{ if } x \in \spt(T) \cap \oB(x_0,r_1) \,.
\end{align*}
The assumption $\alpha \leq \alpha_0$ is needed for Proposition~\ref{improvedexcess}. With Theorem~\ref{reifflat} we can further assume that $r_1$ is small enough such that we have orthogonal frames in the support of $T$ around all points $x \in \spt(T) \cap \oB(x_0,r_1)$ and for all scales $0 < r < r_1$. Let us assume that $\Theta^m(\|T\|,x) = \theta > 0$ for $\|T\|$-a.e.\ $x \in \oB(x_0,r_1)$. Applying the moments computations we obtain from Proposition~\ref{betainftysmall} that there is a scale $0 < r_2 \leq r_1$ and $c'' > 0$ such that for all $x \in \spt(T) \cap \oB(x_0,r_2)$ with $\Theta^m(\|T\|,x) = \theta$ and all $0 < r < r_2$ there is a plane $W \in \bG(\Hi,m)$ with
\begin{align*}
\bbeta_\infty(\|T\|,x,r,W)	& < \bc_{\ref{betainftysmall}} \max\biggl\{ \sqrt[8]{2r}, \sqrt[4]{\Xi\left(2\sqrt{2r}\right)} \biggr\}^\frac{1}{m+2} \\
	& \leq c'' r^\beta \,,
\end{align*}
and further if $r_2$ is small enough, Theorem~\ref{reifflat} and Lemma~\ref{differentplane} imply that
\begin{equation}
	\label{projectionfinal}
	\pi_{W\#} \left(T \res (\B(x,r) \cap Z_{W}(\pi_W(x),\tfrac{r}{2})\right) = g_x\curr{\B_{W}(\pi_W(x),\tfrac{r}{2})} \,,
\end{equation}
for some $g_x \in G$ with $\|g_x\| = \theta$. As in the last part of the proof of Lemma~\ref{bootstraplem} we obtain that
\begin{equation}
	\label{distancefinal}
\hdist(\spt(T) \cap \B(x,r), (x + W) \cap \B(x,r)) < 2c''r^{1+\beta} \,.
\end{equation}
These points $x$ are dense $\spt(T) \cap \oB(x_0,r_2)$ and as in the proof of Theorem~\ref{reifflat} we conclude that for all $x \in \spt(T) \cap \oB(x_0,r_2)$ and $0 < r < r_2$ there is a plane with \eqref{projectionfinal} and \eqref{distancefinal}. If $W$ is such a plane at $x_0$ and scale $0 < r < r_2$, the constancy theorem implies further that that $\pi_W : \spt(T) \cap \B(x_0,r) \cap Z_{W}(\pi_W(x_0),\frac{r}{2}) \to \B_{W}(\pi_W(x_0),\frac{r}{2})$ is surjective. This allows to apply Lemma~\ref{manifoldlem} and we see that $\spt(T) \cap \oB(x_0,r)$ is a $C^{1,\beta}$-submanifold of $\Hi$ if $r$ is small enough. Let $U_\textrm{reg}$ be the set of points in $\spt(T) \cap U$ with a relatively open neighborhood which is a $C^{1,\beta}$-submanifold of $\Hi$. The observations above show that any $x_0 \in U_d$ with $\Theta^m(\|T\|,x_0) = \theta$ is contained in $U_d$. Hence $\|T\|(U_d \setminus U_{\textrm{reg}}) = 0$ and since $\spt(T) \cap U_d$ is dense in $\spt(T) \cap U$, $U_\textrm{reg}$ is dense in $\spt(T) \cap U$. In case $g \mapsto \|g\| = \theta$ is constant on $G \setminus \{0_G\}$, the set of points $x_0 \in \spt(T) \cap U_d$ with $\Theta^m(\|T\|,x_0) = \theta$ forms a set of full $\|T\|$-measure, hence $\|T\|(\Hi \setminus U_{\textrm{reg}}) = 0$  because $\|T\|(\Hi \setminus U_d) = 0$.
\end{proof}

This theorem can be formulated for other gauge functions $\xi$. If we assume that $\xi$ is a continuous gauge such that:
\begin{enumerate}
	\item $r \mapsto \frac{\xi(r)}{r^{\alpha}}$ is decreasing, where $\alpha \leq \alpha_0 \defl \min\{m \frac{1 - \sqrt[4]{\lambda_{\ref{masscomparison2}}}}{\sqrt[4]{\lambda_{\ref{masscomparison2}}}}, \frac{1}{8(m+2)} \}$;
	\item $\Xi(r) \defl m\int_0^r \frac{\xi(s)}{s} \, ds < \infty$ for all $r$;
	\item $\hat \eta(r) \defl \int_0^r \frac{\eta(s)}{s} \, ds < \infty$ for all $r$ where $\eta(r) \defl \Xi\left(2\sqrt{r}\right)^\frac{1}{4(m+2)}$.
\end{enumerate}
Under the same assumptions as in the theorem above we obtain a $C^{1,\hat \eta}$ regularity result.

By hypothesis (1) and Lemma~\ref{technicalproperties} there is some constant $c > 0$ such that 
\[
\Xi(r) \geq m\xi(r) \geq \xi(r) \geq c r^{\frac{1}{8(m+2)}} \,,
\]
for all small $r > 0$. Hence 
\[
\eta(r) \geq \max\{1,c\}\max\biggl\{ \sqrt[8]{r}, \sqrt[4]{\Xi\left(2\sqrt{r}\right)} \biggr\}^\frac{1}{m+2} \,.
\]
In order to establish the technical assumption (4) in Lemma~\ref{manifoldlem}, note that,
\begin{align*}
\partial_r\frac{\eta(r)}{r} & = -\frac{1}{r^2}\Xi\left(2\sqrt{r}\right)^\frac{1}{4(m+2)} + \frac{1}{r}\frac{1}{4(m+2)}\Xi\left(2\sqrt{r}\right)^{\frac{1}{4(m+2)} - 1}\frac{m}{2\sqrt r}\xi(2\sqrt{r})\frac{1}{\sqrt r} \\
 & \leq -\frac{m^\frac{1}{4(m+2)}}{r^2}\xi\left(2\sqrt{r}\right)^\frac{1}{4(m+2)} + \frac{1}{r^2}\frac{m^{\frac{1}{4(m+2)}}}{4(m+2)}\xi\left(2\sqrt{r}\right)^{\frac{1}{4(m+2)} - 1}\xi(2\sqrt{r}) \\
 & = \frac{m^\frac{1}{4(m+2)}}{r^2}\xi(2\sqrt{r})^\frac{1}{4(m+2)}\left(\frac{1}{4(m+2)} - 1\right) \leq 0 \,.
\end{align*}
This shows that $r\mapsto \frac{\eta(r)}{r}$ is decreasing and with the help of Lemma~\ref{diffequation2}, Lemma~\ref{technicalproperties} and Lemma~\ref{manifoldlem}, the regularity result follows as in the proof of Theorem~\ref{reifflat2}.

\section{Examples and Counterexamples}
\label{examples}

In this section we give examples indicating the sharpness of our results. In the example below we show that in case $\{ \|g\| : g \in G\}$ is not discrete, there is in general no uniform lower density bound and mass minimizing $G$-chains may have no point of regularity.

\begin{Expl}
[Discreteness of the group]
\label{discreteness_example}
Let $G = (\Z/2\Z)^\infty$ be the Abelian group with coordinatewise addition in $\Z/2\Z$ and norm
\[
\|(a_1,a_2,a_3,\dots)\| \defl \sum_{i=1}^\infty \frac{1}{3^i}|a_i| \,,
\]
where $|1_{\Z/2\Z}| = 1$ and $|0_{\Z/2\Z}| = 0$. As a metric space $(G,\|\cdot\|)$ is bi-Lipschitz equivalent to the standard Cantor set and therefore totally disconnected but not discrete. Let $g_i \in G$ be the sequence $(a_1,a_2,a_3,\dots)$ with $a_i = 1$ and $a_j = 0$ for $j \neq i$. Let $\{p_i\}_{i \geq 1} \subset [0,1]$ be a countable dense subset and consider the $G$-chain 
\[
T = \sum_i g_i \curr{(0,p_i),(1,p_i)} \in \cR_1(\R^2;G) \,.
\]
Let $\pi(x,y) = x$ be the projection onto the first coordinate. A straightforward calculation shows that for $g = (1,1,\dots)$,
\begin{align*}
\bM(\pi_\# T) & = \bM(g \curr{(0,0),(1,0)}) = \|g\| = \sum_{i\geq 1} \|g_i\| \\
 & = \sum_{i\geq 1} \bM(g_i \curr{(0,p_i),(1,p_i)}) = \bM(T) \,.
\end{align*}
The constancy theorem implies that for any filling $S \in \cR_1(\R^2;G)$ of $\partial T$ there holds $\pi_\# S = g \curr{(0,0),(1,0)}$ and since $\pi$ is $1$-Lipschitz,
\[
\bM(T) = \bM(\pi_\# T) = \bM(\pi_\# S) \leq \bM(S) \,.
\]
Thus $T$ is a mass minimizing filling of $\partial T$, but $\spt(T) = [0,1]^2$ is not $\cH^1$-rectifiable and in particular contains no point of regularity and no lower bound on $1$-densities.
\end{Expl}

Under suitable conditions, the set of regular points of an almost minimizing chain is a $C^1$ submanifold. The converse is true too, even allowing for nice unions of such submanifolds as we show in the proposition below. Let $M \subset \Hi$ be an oriented $m$-dimensional submanifold of regularity $C^1$ possibly with boundary. At any $x \in M$ there is a tangent plane $\operatorname{Tan}(M,x) \in \bG(\Hi,m)$. Since $M$ is oriented, there is a unique choice of orientation for $\operatorname{Tan}(M,x)$ such that $\pi_{\operatorname{Tan}(M,x)} : M \cap \B(0,r) \to \operatorname{Tan}(M,x)$ is an orientation preserving homeomorphism onto its image for all small enough $r$. 

\begin{Prop}
\label{submanifoldprop}
Let $0 < \alpha \leq 1$ and $M_1,\dots,M_k \subset \Hi$ be compact oriented $m$-dimensional submanifolds in $\Hi$ of regularity $C^{1,\alpha}$. Let $T \defl \curr{M_1} + \dots + \curr{M_k} \in \cR_m(\Hi;\Z)$. Further assume that there are $C,\delta > 0$ such that:
\begin{enumerate}
	\item If $x \in M_i$ and $y \in M_j$ for some $i,j$, then
	\[
	\left\|\pi_{\operatorname{Tan}(M_i,x)} - \pi_{\operatorname{Tan}(M_j,y)}\right\| \leq C|x-y|^\alpha \,.
	\]
	\item If $x \in M_i$, $y \in M_j$ and $|x - y| \leq \delta$ for some $i,j$, then the orthogonal projection from $\operatorname{Tan}(M_j,y)$ to $\operatorname{Tan}(M_i,x)$ is orientation preserving.
\end{enumerate}
Then there are constants $C',\delta' > 0$ such that $T$ is $(\bM,\xi,\delta')$-minimal with $\xi(t) \defl C't^{2\alpha}$.
\end{Prop}

Note that if we consider only one submanifold, then these two conditions are trivially satisfied. See \cite[Lemma~2.2.3]{DePauw3} for two different characterizations of $C^{1,\alpha}$ manifolds in $\R^n$ which easily generalize to Hilbert spaces.

\begin{proof}
Since all the $M_i$ are H\"older regular and compact there are constants $C_1,\delta_1 > 0$ with the following property: If $x \in M_i \setminus \partial M_i$ for some $i$ and $W \in \bG(\Hi,m)$ is an $m$ plane with $d_{W,i,x} \defl \|\pi_{\operatorname{Tan}(M_i,x)} - \pi_{W}\| \leq \delta_1$, then
\begin{enumerate}
	\item for all $r > 0$ with	$Cr^\alpha \leq \delta_1$, the orthogonal projection $\pi_W : M_i \cap \B(x,r) \to B_{W,i,x,r} \subset W$ is a homeomorphism onto a neighborhood $B_{W,i,x,r}$ of $\pi_W(x)$ in $W$.
	\item If $\bw : B_{W,i,x,r} \to W^\perp$ is the map with the defining property $w + \bw(w) \in M$ in case $w \in U_{W,i,x,r}$, then
\begin{equation*}
\label{lipestimateofparam}
\sup_{w \in B_{W,i,x,r}} \left\|D\bw_w\right\| \leq C_1 (d_{W,i,x} + r^\alpha) \,.
\end{equation*}
\end{enumerate}
For (2) compare with the proof of Lemma~\ref{manifoldlem} or \cite[Lemma~2.2.3]{DePauw3}. If we assume that $W$ satisfies $d_{W,i,x} = \|\pi_{\operatorname{Tan}(M_i,x)} - \pi_{W}\| \leq Cr^\alpha \leq \delta_1$, then there is a constant $C_2 > 0$ such that for all $w \in B_{W,i,x,r}$,
\[
(\J \bw_w)^2 = \|D\bw_w\|_{\HS}^2 + \sum_{\#K \geq 2} \det((D\bw_w^*D\bw_w)_K) \leq C_2 r^{2\alpha} \,.
\]
See \eqref{jacobian_sum} and \eqref{norms}. Hence
\[
(1 + (\J \bw_w)^2)^\frac{1}{2} \leq (1 + C_2 r^{2\alpha})^\frac{1}{2} \leq 1 + C_2 r^{2\alpha} \,.
\]
Because $\pi_W : M_i \cap \B(x,r) \to B_{W,i,x,r} \subset W$ is a homeomorphism by (1), the cylindrical excess of a compact subset $A \subset M_i \cap \B(x,r)$ over the plane $W$ estimates as
\begin{align}
\nonumber
0 \leq \cH^m(A) - \cH^m(\pi_W(A)) & = \int_{\pi_W(A)} (1 + (\J \bw_w)^2)^{\frac{1}{2}}  - 1 \, dw \\
\nonumber
 & \leq C_2\int_{\pi_W(A)}r^{2\alpha} \, dw \\
\label{excessestimatesubman}
 & = C_2r^{2\alpha}\cH^m(\pi_W(A)) \,.
\end{align}

Let $0 < \delta_2 \leq \frac{1}{2}\delta$ be small enough such that $C(2\delta_2)^\alpha \leq \delta_1$ and consider some scale $0 < r \leq \delta_2$ and some $x \in \Hi$. Note that by the compatibility of orientations, $\spt(\partial T) = \partial M_1 \cup \dots \cup \partial M_k$. Assume that $\B(x,r)$ intersects $M_{i(1)}, \dots, M_{i(l)}$ in a set of positive measure and fix some points $x_{j} \in A_j \defl M_{i(j)} \cap \B(x,r) \setminus \partial M_{i(j)}$ as well as $W \defl \operatorname{Tan}(M_{i(1)},x_1)$. Since $2r \leq \delta$, for all $j = 1,\dots,l$ there holds,
\[
\|\pi_{\operatorname{Tan}(M_{i(j)},x_j)} - \pi_{W}\| \leq C|x_1 - x_j|^\alpha \leq C(2r)^\alpha \leq \delta_2 \,.
\]
Because $A_j \subset M_{i(j)} \cap \B(x_j,2r)$ it follows from \eqref{excessestimatesubman} that
\[
\cH^m(A_j) \leq (1 + C_2(2r)^{2\alpha})\cH^m(\pi_W(A_j)) = (1 + C_3r^{2\alpha})\cH^m(\pi_W(A_j)) \,,
\]
for $C_3 \defl 4C_2$. With the hypothesis on the orientation and relative position of the submanifolds, $\bM(T) = \bM(\curr{M_1}) + \cdots + \bM(\curr{M_k})$. More precisely, at every point where two submanifolds meet, they have the same tangent plane with compatible orientations and thus their multiplicities do not cancel. Since all the projections $\pi_W : A_j \to W$ are injective and orientation preserving,
\begin{align*}
\bM(T \res \B(x,r)) & = \bM(\curr{A_1}) + \cdots + \bM(\curr{A_l}) \\
 & \leq (1 + C_3r^{2\alpha})\bM(\pi_{W\#}\curr{A_1}) + \cdots + (1 + C_3r^{2\alpha})\bM(\pi_{W\#}\curr{A_l}) \\
 & = (1 + C_3r^{2\alpha})\bM(\pi_{W\#}(T \res \B(x,r))) \,.
\end{align*}
Let $S \in \cR_m(\Hi;\Z)$ with $\spt(S) \subset \B(x,r)$ and $\partial S = 0$. The constancy theorem implies $\pi_{W\#} S = 0$ and hence
\begin{align*}
\bM(T \res \B(x,r)) & \leq (1 + C_3r^{2\alpha})\bM(\pi_{W\#}(T \res \B(x,r))) \\
 & = (1 + C_3r^{2\alpha})\bM(\pi_{W\#}(T \res \B(x,r) + S)) \\
 & \leq (1 + C_3r^{2\alpha})\bM(T \res \B(x,r) + S) \,.
\end{align*}
This shows that $T$ is almost minimizing.
\end{proof}

The following example illustrates the proposition above showing that if $G = \Z$, then the set of regular points of an almost minimizing rectifiable $G$-chain is in general not of full measure. The same construction also works for an arbitrary normed Abelian group $G$ if there are $g_1,g_2 \in G \setminus \{0_G\}$ with $\|g_1\| + \|g_2\| = \|g_1 + g_2\|$.

\begin{Expl}
[Size of the regular set]
\label{regularsetexample}
Let $C \subset [0,1]$ be a topological Cantor set with $\cH^1(C) > 0$ and $\{0,1\} \subset C$. Let $\gamma : \R \to \R$ be a smooth function with $\gamma(t) > 0$ for $|t| < 1$ and $\gamma(t) = 0$ for $|t| \geq 1$. We could for example take
\[
\gamma(t) \defl
\left\{
	\begin{array}{ll}
		\exp(-1/(1-t^2))  & \mbox{if } |t| < 1 \,, \\
		0  & \mbox{if } |t| \geq 1 \,.
	\end{array}
\right.
\]
Fix an enumeration $U_i$ of the connected components of $[0,1] \setminus C$. Since $\{0,1\} \subset C$ each $U_i$ is of the form $(a_i-b_i,a_i+b_i)$ for some $0 < a_i,b_i < 1$. We define $\gamma_i : \R \to \R$ by $\gamma_i(t) \defl c_i\gamma(b_i^{-1}(t - a_i))$ for some $c_i > 0$ such that $|D^k\gamma_i| \leq 2^{-i}$ for all $k \leq i$. Set $f : \R \to \R$ to be the sum $f \defl \sum_i \gamma_i$. Then the partial sums $\sum_i D^k\gamma_i$ converge uniformly for all $k$ and hence $f$ is smooth and satisfies $f(t) = 0$ for $t \in C$ and $f(t) > 0$ for $t \in [0,1] \setminus C$.

Consider
\[
T \defl (\id_{\R} \times 0)_\#\curr{0,1} + (\id_{\R} \times f)_\#\curr{0,1} \in \cR_1(\R^2;\Z) \,.
\]
Clearly, $\partial T = (g_1 + g_2)(\curr 1 - \curr 0)$ and the set $C$ is the complement of the set of regular points. Since $\|T\|(C) = 2\cH^1(C) > 0$, the set of regular points doesn't have full measure. It remains to check that $T$ is almost minimizing. Since $f \in C^2$ has compact support, there is a constant $L > 0$ such that for all $x,y \in \R$,
\[
\left|f(y) - f(x) - f'(x)(y-x)\right| \leq L|y-x|^2 \,.
\]
Assume that $f'(x)^2 > 4Lf(x)$ for some $\epsilon > 0$. If $f'(x) > 0$ let $y \in \R$ be such that $x-y = f'(x)(2L)^{-1}$, and $y - x = f'(x)(2L)^{-1}$ otherwise. We get
\begin{align*}
f(y) & \leq f(x) + f'(x)(y-x) + L|y-x|^2 \\
 & < f'(x)(4L)^{-1} - f'(x)^2(2L)^{-1} + f'(x)^2(4L)^{-1} \\
 & = 0 \,.
\end{align*}
This contradicts $f(y) \geq 0$ and hence $f'(x)^2 \leq 4L f(x)$ for all $x \in [0,1] \setminus C$. Let $W_1 = \{(t,tf'(x)) : t \in \R\}$ be the tangent line at $(x,f(x))$ to $T$ and $W_0 = \{(t,0) : t \in \R\}$ the tangent line to $T$ at $(x,0)$. A simple calculation shows that
\[
\|\pi_{W_0} - \pi_{W_1}\| \leq f'(x) \leq (4L)^\frac{1}{2} f(x)^\frac{1}{2} \,,
\]
and hence for all $p,q \in \spt(T)$,
\[
	\|\pi_{\operatorname{Tan}(T,p)} - \pi_{\operatorname{Tan}(T,q)}\| \leq L'|p-q|^\frac{1}{2} \,,
\]
for some $L' > 0$. Proposition~\ref{submanifoldprop} now shows that $T$ is almost minimizing with a linear gauge $\xi(r) = L_4 r$. A more careful analysis shows that $T$ is actually almost minimizing with respect to a quadratic gauge function. But since Theorem~\ref{reifflat2} certainly applies to a linear gauge, this is enough to show that almost minimizing chains can have a branching set of positive measure.
\end{Expl}


\end{document}